\documentclass[12pt]{amsart}
\usepackage{geometry}
\usepackage{epsfig}
\usepackage{array}
\usepackage{xcolor}
\usepackage{tabularx}
\usepackage{mathabx}
\usepackage{txfonts,fancyvrb}
\usepackage{verbatim}
\usepackage{todonotes}
\usepackage{placeins}

\usepackage{thmtools}
\usepackage{thm-restate}

\usepackage{algorithm}
\usepackage[noend]{algpseudocode}
\makeatletter
\def\BState{\State\hskip-\ALG@thistlm}
\makeatother

\usepackage{subfig}

\usepackage{amscd,amssymb,mathabx,amsthm}
\usepackage[arrow,matrix,graph,frame,poly,arc,tips]{xy}
\usepackage{graphicx}
\usepackage{psfrag}

\usepackage{mathtools}
\usepackage[mathscr]{eucal}
\usepackage[colorlinks=false,plainpages]{hyperref}

\usepackage{tikz}
\usetikzlibrary{decorations.markings}
\usetikzlibrary{shapes}
\usetikzlibrary{arrows}
\usetikzlibrary{backgrounds}
\usetikzlibrary{calc}
\usepackage{mdwlist}

\usepackage{multirow}
\usepackage{enumerate}

\usepackage{verbatim}

\usepackage{etoolbox}
\AtEndEnvironment{proof}{\setcounter{claim}{0}}

\DeclarePairedDelimiter{\abs}{\lvert}{\rvert}

    \newcommand\ba{\begin{align*}}
    \newcommand\ea{\end{align*}}
    \newcommand\be{\begin{enumerate}}
    \newcommand\ee{\end{enumerate}}
    \newcommand\bp{\begin{proof}}
    \newcommand\ep{\end{proof}}
    \newcommand\bpp{\begin{prop}}
    \newcommand\epp{\end{prop}}
    \newcommand\bpb{\begin{prob}}
    \newcommand\epb{\end{prob}}
    \newcommand\bd{\begin{defn}}
    \newcommand\ed{\end{defn}}
    \newcommand\bh{\begin{hint}}
    \newcommand\eh{\end{hint}}

    \newcommand\bN{\mathbb{N}}
    \newcommand\N{\mathbb{N}}

    \newcommand\bZ{\mathbb{Z}}
    \newcommand\Z{\mathbb{Z}}

    \newcommand\CC{\mathcal{C}}

    \newcommand\LL{\mathcal{L}}

    \newcommand\aut{\operatorname{Aut}}

    \newcommand\Fill{\operatorname{Fill}}
    \newcommand\supp{\operatorname{supp}}

    \newcommand\Mod{\operatorname{Mod}}

    \newcommand\mC{\mathcal{C}}
    \newcommand\mL{\mathcal{L}}
    \newcommand\mM{\mathcal{M}}

    \DeclareMathOperator\Aut{Aut}

    \DeclareMathOperator\Stab{Stab}
    
    \DeclareMathOperator\qftp{qftp}
    \DeclareMathOperator\tp{tp}

    \newcommand\Or{\operatorname{Or}}

    \newcommand{\nin}[0]{\notin}
    
    \newcommand{\forkingdep}{
    	\;\raise.2em\hbox{$\mathrel|\kern-.9em\lower.35em\hbox{$\smile$}\kern-.7em\hbox{\char'57}$}\;}

    	\newcommand{\forkingind}{
    		\,\raise.2em\hbox{$\,\mathrel|\kern-.9em\lower.35em\hbox{$\smile$}$}}
    		\newcommand{\indep}[3]{#1 \mathop{\underset{#2}{ \forkingind}} #3}

    		\newtheorem{theorem}{Theorem}[section]
		\newtheorem{metatheorem}{Metatheorem}[section]
    		\newtheorem{lemma}[theorem]{Lemma}
		\newtheorem{sublemma}[theorem]{Sublemma}
    		\newtheorem{corollary}[theorem]{Corollary}
    		\newtheorem{proposition}[theorem]{Proposition}
    		
    		\newtheorem{observation}[theorem]{Observation}
    		\newtheorem{question}[theorem]{Question}
    		\newtheorem*{claim*}{Claim}
    		
    		\newtheorem{claim}{Claim}

    		\theoremstyle{remark}
    		
    		\newtheorem{remark}[theorem]{Remark}

    		\theoremstyle{definition}
    		\newtheorem{definition}[theorem]{Definition}
    		\newtheorem{prob}[theorem]{Problem}
    		\newtheorem{que}[theorem]{Question}

    		\newtheorem{example}[theorem]{Example}

    		\title{The model theory of the curve graph}

\begin{document}
    	   
    \author[V. Disarlo]{Valentina Disarlo}
\address{Institute of Science and Technology Austria (ISTA), Am Campus 1, 3400 Klosterneuburg (AT) }
    \email{Valentina.Disarlo@ist.ac.at}

    \author[T. Koberda]{Thomas Koberda}
    \address{Department of Mathematics, University of Virginia, Charlottesville, VA 22904-4137, USA}
    \email{thomas.koberda@gmail.com}
    \urladdr{https://sites.google.com/view/koberdat}

    \author[J. de la Nuez Gonz\'alez]{Javier de la Nuez Gonz\'alez}
    \address{School of Mathematics, Korea Institute for Advanced Study (KIAS), Seoul, 02455, Korea}
    \email{jnuezgonzalez@gmail.com}
    		\maketitle
    		
    		\newcommand{\pam}[0]{partial map } 
    		\newcommand{\pams}{partial maps } 
    		\newcommand{\D}[0]{\mathcal{D}}
    		\newcommand{\Dvee}[0]{\D^{\vee}}
    		\renewcommand{\P}[0]{\mathcal{P}}
    		\newcommand{\C}[0]{\mathcal{C}}
    		\newcommand{\tc}[0]{\tilde{\mathcal{C}}}
    		\newcommand{\tv}[0]{\pitchfork}
    		\newcommand{\sq}[0]{\mathcal{S}}
    		\newcommand{\E}[0]{\mathcal{E}}
    		\renewcommand{\mC}[0]{\mathcal{C}}
    		\renewcommand{\supp}[0]{\mathrm{supp}}
    		
    		\newcommand{\Th}[1]{\mathrm{Th}(#1)}
    		\newcommand{\wo}[0]{\mathcal{W}}
    		\newcommand{\alp}[0]{\mathcal{A}}
    		
    	\newcommand{\enum}[2]{\begin{enumerate}[\hspace*{0.5cm}#1] #2 \end{enumerate}}
    		\renewcommand{\N}{\mathbb{N}}
    		\newcommand{\M}{\mathcal{M}}
    		\newcommand{\subg}[1]{\langle #1 \rangle}

    		\newcommand{\va}[1]{\textcolor{blue}{#1}}
    		\newcommand{\tho}[1]{\textcolor{green}{#1}}
    		\newcommand{\lu}[1]{\textcolor{magenta}{#1}}
    		
    		\newcommand*{\greysquare}{\textcolor{gray}{\blacksquare}}
    		\newenvironment{subproof}[1][\proofname]{%
    		\renewcommand{\qedsymbol}{$\greysquare$}%
    		\begin{proof}[#1]%
    			}{%
    		\end{proof}%
    	}
    	
    	\newenvironment{subsubproof}[1][\proofname]{%
    	\renewcommand{\qedsymbol}{$\blacksquare$}%
    	\begin{proof}[#1]%
    		}{%
    	\end{proof}%
    }

    \renewcommand{\emph}[1]{\textbf{#1}}
    
    \newcommand{\mnote}[1]{\todo[inline, color=gray!50]{#1}}
    
    \begin{abstract}
    	This paper develops bridges between model theory, geometric topology,
    	and geometric group theory. We consider a surface $\Sigma$ of finite type and
    	its curve
    	graph $\mC(\Sigma)$, and we carry out a model-theoretic study of its
    	first-order theory. Among many other things, we are able to prove that the theory of the curve graph
	is $\omega$--stable, admits a natural version of quantifier elimination, and we compute its Morley rank.
	We also show that many of the complexes which are naturally associated to a surface are
    	interpretable in $\mC(\Sigma)$. This shows that these complexes are all $\omega$--stable
    	and admit certain {\it a priori} bounds on their Morley ranks. We are able to use Morley ranks
    	to prove that various complexes are not bi--interpretable with the curve graph.
	
	We then use the model theory of the curve graph to establish results about the geometric topology of surfaces and mapping class groups.
	We give a model-theoretic interpretation of the Ivanov Metaconjecture about automorphisms of sufficiently complex objects that are
	naturally associated to a surface.
    	As a consequence of quantifier elimination, we show that the algebraic intersection number predicate
    	is not definable in the first order theory of the curve graph. We 
	use the model-theoretic machinery we develop to show that the curve graph of a surface enjoys a novel phenomenon
    	that we call
    	{\it interpretation rigidity}. That is, if surfaces $\Sigma_1$ and $\Sigma_2$ admit curve graphs that are mutually interpretable,
    	then $\Sigma_1$ and $\Sigma_2$ are homeomorphic to each other. Finally, as a consequence of the proof of interpretation rigidity, we
	show that interpretations between curve graphs naturally give rise to homomorphisms
	between finite index subgroups of mapping class groups, addressing a classical question about injective maps between mapping class
	groups of surfaces.
	
	In the course of the proofs of the main results, we construct a class of auxiliary structures associated to the mapping class groups,
	called augmented Cayley graphs. These enriched structures are easier to investigate than the curve graph and other geometric
	graphs, and their value lies in the fact that they are bi-interpretable with various geometric graphs.
	Along the way, we show that the theories of augmented Cayley graphs enjoy many nice model-theoretic properties such as simple connectedness, total
	triviality, and weak elimination of imaginaries.
    \end{abstract}
    
    \maketitle
    \setcounter{tocdepth}{1}
    \tableofcontents

\section{Introduction}
This paper is an investigation of the curve graphs and mapping class groups of hyperbolic surfaces, from a model-theoretic point of view.
We thus provide a systematic framework for studying objects in combinatorial topology from the point of view of model theory; we
  employ model-theoretic techniques to study objects of classical interest in geometric topology, and investigate rigidity phenomena through which
  topological structure is reflected in model theory.
  
  The subject of this paper can be divided into two main sub-subjects as follows:
  \begin{enumerate}
  \item
  \emph{The model theory of the curve graph of a surface $\Sigma$ of hyperbolic type}; this is a unifying framework for the whole paper, and
  consists of a detailed investigation of certain properties of the theory of the curve graph. The work here is largely subordinate to the
 applications to geometric topology, and we do not claim to have conducted a complete model-theoretic study.
  \item
  \emph{The applications of model-theoretic ideas to geometric topology}; here, we investigate classical questions in surface topology, such as
  automorphisms of geometric graphs and injective maps between mapping class groups, using model-theoretic machinery.
  \end{enumerate}
  
  The authors hope that this paper will be of interest to both logicians and to geometric topologists. Results relevant to model theory
  are in Subsection~\ref{ss:mtcg} below. Results relevant to geometric topology are in Subsections~\ref{ss:ivanov},~\ref{ss:app},
  and~\ref{ss:other}.
  
  \subsection{Motivation: curve graphs, mapping class groups, and the Ivanov Metaconjecture}
  The primary topological motivation for this article is the \emph{curve graph}, which organizes homotopy classes of embedded, essential,
  compact $1$--submanifolds of a surface, and the \emph{mapping class group} of the surface, which consists of homotopy classes of homeomorphisms of the underlying surface. The curve graph and the mapping class group are central to the study of the topology
  and geometry of a surface; see Subsection~\ref{ss:mtcg} immediately below and Section~\ref{sec:background} for more detail.
  
The following is a well-known question in the geometric topology of surfaces,
  known as the Ivanov Metaconjecture~\cite{ivanov06}, and which finds its roots in ~\cite{ivanov_97}:
  \begin{prob}\label{p:ivanov}
  	Every object naturally associated to a surface $S$ with sufficiently rich structure has the extended mapping class group of $S$ as its group
  	of automorphisms. Moreover, the proof of this fact should reduce to Ivanov's theorem that the automorphism group of the curve graph
  	is the extended mapping class group.
  \end{prob}
  
  The combination~\cite{ivanov06,ivanov_97} has spurred a deluge of research, and these papers have accrued hundreds of citations
  and have spawned dozens of papers.
  Notable unifying approaches to the Ivanov Metaconjecture were proposed by 
  Brendle--Margalit in~\cite{brendle_margalit}; cf.~\cite{mcleay18,mcleay19}.
   
   One of the purposes of this article is to provide a unified account of the Ivanov Metaconjecture using model theory.
   Our approach in this paper is markedly different from previous ones, in that it
  avoids purely geometrically topological methods in favor of developing a robust framework based in logic.
  In particular we will articulate a precise formulation of Problem~\ref{p:ivanov} that uses the language of model theory.
  Our formulation naturally raises
  and answers a question that could be viewed as a broad philosophical generalization of Problem~\ref{p:ivanov}.
  
  \begin{que}\label{q:general}
  	Why does the curve graph play such a central role in the study of mapping class groups and the objects naturally associated to a surface?
  	Why is the behavior of these objects controlled by the structure of the curve graph?
  \end{que}

  \subsection{The main results of the paper}
  To pass from the philosophical to the mathematically precise, let us give our setup.
  We are concerned with a finite-type orientable surface $\Sigma$ (that is, a real
  two-dimensional manifold) of negative
  Euler characteristic and possibly with (finitely many) punctures,
  its (extended) mapping class group $\Mod^{\pm}(\Sigma)$, consisting of
  homotopy classes of homeomorphisms of $\Sigma$, and
  the  curve graph $\mC(\Sigma)$, which encodes the homotopy classes of simple closed
  curves on $\Sigma$.
  The reader is directed to Section~\ref{sec:background} (and
  especially Subsection~\ref{ss:mcg}) for precise definitions of the topological objects under consideration here.
  Our main results are about the model theory of the
  curve graph, thus investigating a geometrically/topologically defined
  object from a model--theoretic point of view.
  To guide the model-theoretic development of the theory, we draw inspiration and adapt some ideas from~~\cite{tent-jsl} and~\cite{BPZ17}, where a
  model--theoretic study of right-angled buildings  was carried out. The program we carry out can be summarized as follows:
  
  \begin{enumerate}
  	\item
  	Provide a rigorous framework for relating the theories of the mapping class group and various naturally defined complexes (such as
  	the curve graph, the pants graph, etc.) via the model-theoretic concept of interpretation.
  	\item
  	Control the complexity of these various theories, through quantifier elimination and $\omega$--stability.
  	\item
  	Build model-theoretically constructed, injective homomorphisms between finite index subgroups of mapping class groups of surfaces.
	\item
	Produce new examples of theories with explicitly computable Morley ranks that are infinite ordinals, as well as many other 
	auxiliary theories with desirable model-theoretic properties.
  \end{enumerate}

  \subsubsection{A model-theoretic framing of the Ivanov Metaconjecture}\label{ss:ivanov}
  Problem~\ref{p:ivanov}, and more generally Question~\ref{q:general}, are demystified by the fact that most natural geometric
  graphs associated to a surface are interpretable in the curve
  graph. Consequently the complexities of their first order theories are
  controlled by that of the curve graph. This includes properties like
  $\omega$--stability, well-definedness of Morley
  rank, and relative quantifier elimination. Objects that are bi-interpretable with the curve graph, or more generally a certain
  auxiliary structure that we call the \emph{augmented Cayley graph of the mapping class group}, automatically have the same
  automorphism group as the curve graph, which gives a precise formulation of Problem~\ref{p:ivanov}.
  
  \begin{metatheorem}[Model theoretic formulation of the Ivanov Metaconjecture]\label{meta:ivanov}
  The following makes precise, and organizes the instances of, Problem~\ref{p:ivanov}.
  \begin{itemize}
  	\item
  	``Naturally associated'' means interpretable in the curve graph.
  	\item
  	``Sufficiently rich structure'' means bi-interpretable with the curve graph, or more generally bi-interpretable with an augmented
  	Cayley graph of the mapping class group of $\Sigma$. This will be a certain auxiliary structure built out of the standard Cayley graph
  	of the extended mapping class group of $\Sigma$, which keeps track certain relations arising from subsurfaces of $\Sigma$.
  	\item
  	``Reduction to Ivanov's theorem'' is the general model-theoretic fact that structures that are bi-interpretable without parameters
  	have the same automorphism group, combined with the fact that the automorphism groups of the relevant augmented Cayley graphs
	are isomorphic to the extended mapping class group of $\Sigma$, a fact which ultimately reduces to Ivanov's result on automorphisms
	of the curve graph.
  \end{itemize}
  \end{metatheorem}
  
  Our formulation of the Ivanov Metaconjecture thus captures the spirit of Problem~\ref{p:ivanov}. In many cases,
  classically defined geometric
  graphs associated to  $\Sigma$ are bi-interpretable with an augmented
  Cayley graph of the mapping class group of $\Sigma$.

  \subsubsection{The model theory of the curve graph}\label{ss:mtcg}
  Here, we summarize the conclusions of our model-theoretic study of
  curve graphs of surfaces.
  We emphasize that the focus of the results here is subordinate to our interest in the geometric topology of surfaces and mapping class
  groups, and we therefore do not have the aim of giving new examples or phenomena within model theory itself.
  
    The curve graph was introduced by Harvey~\cite{harvey_1978},
    who defined it as a surface-theoretic analogue of a
    building associated to a symmetric space of nonpositive curvature. 
    It consists of a vertex for every homotopy class of essential, nonperipheral simple closed curves on a surface, 
    and the adjacency relation is via disjoint realizability. Perhaps the most influential work on the curve graph and its structure was carried out
    by Masur--Minsky in~\cite{masur_minsky_99,masur_minsky_00}
    
    The analogy with buildings is part of the inspiration for
    the present work, via the aforementioned work~\cite{tent-jsl,BPZ17} on right-angled buildings.
    The curve graph of a surface is usually a locally infinite graph of infinite
    diameter, with complicated structure on both a local and global scale. Thus,
    one  may suspect that reasonable notions of classification (e.g.~classifying the
    finite subgraphs of $\C(\Sigma)$ up to isomorphism, for { $\Sigma$ } fixed) might be so
    lacking in regular structure as to be intractable; cf.~\cite{KKIMRN,KKOsaka,bering_conant_gaster}.
    One of the purposes of this paper is to establish that,
    in spite of the apparent chaotic structure of the curve graph, its formal
    properties are tamer than one might initially imagine.
    
    We will view $\mC(\Sigma)$ as a structure in the language of graph
    theory, without any further structure.
    The language of (undirected) graphs has a single  binary relation $E$
    which is symmetric.
    In the context of the curve graph,
    the set $V$ of vertices of  $\mC(\Sigma)$ is the
    universe, and $E$ is interpreted as adjacency in $\mC(\Sigma)$.
    We will write
    $\Th{\mC(\Sigma)}$ for the~\emph{theory} of $\mC(\Sigma)$. Precisely,
    this consists of the set of first order logical sentences which
    are satisfied by $\mC(\Sigma)$.
    The adjective~\emph{first order} refers to the fact that scope of
    quantification, i.e.~allowable values for variables in formulae, are only permitted to be individual elements and not
    relations.
    
    The driving force of this paper is the following technical result
    to which we have alluded already,
    which we will not state completely precisely, and instead direct the reader to Section~\ref{sec:interpretations}
    (cf.~Corollary~\ref{bi-interpretability of the curve graph}).
    
    \begin{theorem}\label{thm:intro-bi-int}
    	The curve graph $\mC(\Sigma)$ is bi-interpretable with the augmented Cayley graph of the mapping class group of $\Sigma$.
    \end{theorem}
    
    Once Theorem~\ref{thm:intro-bi-int} is established, the study of the theory of the curve graph can be carried out in the better-behaved
    world of the augmented Cayley graph of the mapping class group.
    The augmented Cayley graphs of mapping class groups are
    a family of auxiliary structures {$\mathcal M^G_\D(\Sigma)$}, which depend on a finite index subgroup $G$ of the extended
    mapping class group $\Mod^\pm(\Sigma)$ and a class $\D$ of connected domains, equivalently curve graphs of connected subsurfaces, satisfying certain suitable hypotheses.
    The theories of some of these structures will have absolute quantifier elimination, and others will have quantifier elimination with respect
    to $\exists$--formulae.
    Here,
    (absolute) \emph{quantifier elimination} means that any first order predicate is equivalent
    in the theory $T$ to one which involves no quantifiers (often said ``quantifier-free modulo $T$'').
    A theory $T$ having relative quantifier
    elimination with respect to $\forall\exists$--formulae means that every first
    order predicate is equivalent mod $T$ to a Boolean combination of formulae in which all variables fall under the scope of one
    block of universal quantifiers, followed by one block of existential quantifiers.

    With this machinery in place, we will be able to
    show that under certain hypotheses, these auxiliary structures are bi--interpretable with the curve graph through mostly formal
    means; the increase in quantifier
    complexity in the theory of the curve graph is due to the quantifiers in the formulae furnishing the interpretations.
    We are thus able to use Theorem~\ref{thm:intro-bi-int} to establish results
    such as the following.
    
    \begin{theorem}\label{thm:quantifier elimination}
    	Let $\Sigma$ be an orientable finite-type surface that is not a torus with two {punctures} or a sphere with fewer than five punctures.
    	Then $\Th{\mC(\Sigma)}$ has quantifier elimination with
    	respect to the class of $\forall\exists$--formulae.
    \end{theorem}

    We note that one cannot hope for absolute quantifier elimination
    in $\Th{\mC(\Sigma)}$, since then every first order predicate about the graph
    $\mC(\Sigma)$ would be determined by finitely many adjacency relations.
    For example, it is not difficult to write a first order formula
    which expresses that
    $d_{\mC(\Sigma)}(a,b)=k$ and $d_{\mC(\Sigma)}(a,c)=k+1$, where here
    we are measuring graph metric distance in the curve graph,  and where $k\geq 2$ is an
    integer.
    Since the language of graph theory is equipped with a single binary relation, any quantifier-free formula in this language is simply
    a finite Boolean combination of adjacencies and non-adjacencies; thus, the quantifier--free
    types of the pair $(a,b)$ and the  pair
    $(a,c)$ are the same for $k$ sufficiently large. Therefore, the  quantifier--free type of a pair
    cannot determine the type of a pair.

    We can now state the second result about the model theory
    of the curve graph.
    
    \begin{theorem}\label{thm:stable}
    	Let $\Sigma$ be a surface of genus $g$ with $b$ punctures. Then $\Th{\mC(\Sigma)}$ is
    	$\omega$--stable. If $S$ has genus
    	$g$ and {$b$ punctures}, then the Morley rank of
    	the formula $x=x$ in the theory $\Th{\mC(\Sigma)}$ is equal to $\omega^{3g+b-3}$.
    \end{theorem}
     Here, a countable theory is \emph{$\omega$--stable} if there are only
    countably many  types over a countable set of parameters; see
    ~\cite{marker2006model}, Definition 4.2.17. 
    An analogue of Theorem~\ref{thm:stable} has recently
    been obtained for
    the Farey graph in~\cite{zahra-katrin}, using different
    methods. The authors of that paper also find an axiomatization
    of the theory of the Farey graph.

    \begin{remark}
    Previous work on stability inside of the curve graph of a surface was carried out by Bering--Conant--Gaster~\cite{bering_conant_gaster},
    where they show that the curve graph of a surface of genus $g$ and $b$ boundary components is \emph{$k$--edge stable} for $k=6g-5+2b$. This property can be formulated combinatorially
    by requiring that no \emph{half-graph} of height $\geq k$ occurs as an induced subgraph of the curve graph, and implies that quantifier--free
    formulae are stable in the model--theoretic sense; see Section 8.2 of Tent--Ziegler's book~\cite{tent_ziegler}, for instance. The authors of the
    paper~\cite{bering_conant_gaster} specifically suggest that understanding stability of arbitrary formulae in $\Th{\mC(\Sigma)}$ would
    require understanding quantifier elimination for some expansion of $\Th{\mC(\Sigma)}$, which is precisely what is carried out in this paper.
    \end{remark}

    \subsubsection{Further model-theoretic applications}
    Returning to $\omega$--stability,
    we will obtain control over the Morley rank of a large class of other formulae in $\Th{\mC(\Sigma)}$, coming
    from computations of the Morley rank of certain formulae in the language of the augmented Cayley graph; see Section~\ref{sec:morley}.
    We refer the reader to Subsection~\ref{ss:model}
    for an expanded discussion of this notion. In more intuitive terms, $\omega$--stability
    means that countable models for $\mC(S)$ have some hope of being classified, and do not admit too many types; for instance, stability
    precludes the interpretation of arithmetic in the underlying theory.
    
    The Morley rank of a theory
    is a notion of dimension. See the discussion in Section~\ref{sec:background}.
    We will prove equalities of Morley ranks for certain auxiliary structures
    that are bi--interpretable with the curve graph, and which we
    construct in the course of the paper.
    We note  that
    the general strategies  of our proofs of
    Theorem~\ref{thm:quantifier elimination}  and
    Theorem~\ref{thm:stable} are adapted from~\cite{BPZ17}
    and~\cite{tent-jsl}.

    \subsubsection{Applications to mapping class groups}\label{ss:app}
    In this subsection, we give some of the applications of the model theory of the curve graph to geometric topology.
       As a first consequence of quantifier elimination, we have the following result:
    
    \begin{corollary}\label{cor:alg-int}
    	Let $\Sigma$ be a surface of positive genus that is not a torus with less than two {punctures}. For each
    	natural number $k$, let $I_k(x,y)$ be
    	the predicate such that $I_k(a,b)$ if and only if the essential simple closed curves $a$ and $b$ have algebraic intersection number
    	$\pm k$.
    	Then $I_k$ is not a (parameter--free) definable relation in $\mC(\Sigma)$.
    \end{corollary}
    
    We will in fact prove a significantly more general statement; cf.~Corollary  \ref{cor:alg-int-body} below.
    Corollary~\ref{cor:alg-int}
    seems counterintuitive at first, since algebraic intersection number is a homological
    invariant which is easy to compute once curves are identified with conjugacy classes
    in $\pi_1(\Sigma)$. What the corollary is saying is that the
    first order theory of the graph structure of the
    curve graph is not well--suited to predicates encoding algebraic intersection
    number. Thus, Corollary~\ref{cor:alg-int} can be viewed as an answer to the following more philosophical question.
    
    \begin{question}
    	Why is the relationship between the intersection--theoretic behavior of curves on a surface and their homological relationship
    	so complicated?
    \end{question}
    
    The answer is, plainly, that the first order theory of the curve graph does not know about homology.
    
    A further, entirely novel direction in which we develop the theory here is with regard to \emph{interpretation rigidity}.
    Let $\mathcal L$ be a language and let $\mathcal X$ be a class of $\mathcal L$--structures.
    We will say that $\mathcal X$ enjoys interpretation rigidity if whenever $\mathcal A,\mathcal B\in\mathcal X$ are mutually interpretable
    (possibly with parameters), then
    $\mathcal A$ is isomorphic to $\mathcal B$. Here, we mean interpretation in the technical model--theoretic sense: roughly, an $\mathcal L$--structure $\mathcal A$ is interpretable
    in an $\mathcal L'$--structure $\mathcal B$ if the universe of $\mathcal A$, together with the interpretations of the non-logical symbols of $\mathcal L$ in $\mathcal A$, are given
    by a definable subset of a finite Cartesian product of copies of $\mathcal B$, modulo a definable equivalence relation. See
    Section~\ref{sec:background} for more details.
    
    \begin{theorem}\label{thm:interp-rigid}
    	Let $\mathcal X$ be the class of curve graphs of surfaces, excluding closed genus two surfaces, tori with at most two punctures,
    	and spheres with at most six punctures. Then the class $\mathcal X$ enjoys interpretation rigidity.
    \end{theorem}
    
    Thus, the model theory of curve graphs determines their isomorphism type within the class of curve graphs of 
    surfaces.
    As mentioned above,
    Theorem~\ref{thm:interp-rigid} relies on several results about the model theory of curve graphs (and the associated auxiliary
    structures) that we prove in this paper, namely elimination of quantifiers and weak elimination of imaginaries.
    
    Perhaps the most striking application of the model-theoretic methods in this paper to geometric topology is to the problem of classifying
    injective homomorphisms between mapping class groups.
      Generally, one would hope that such injective maps arise from canonical topological constructions such as inclusions of subsurfaces
    and covering spaces, though this is not always the case. We refer the reader 
    to~\cite{AraSoutoGT12,farb-survey,ivanov06,harvey-korkmaz,castel16,kordek-margalit} 
    for a survey of Problem~\ref{prob:injective} below and related results.
    The following is a difficult, generally open problem that has been a driving
    force in the theory of mapping class groups of surfaces, which in turn is closely related to superrigidity phenomena such as Margulis Superrigidity.
    
    \begin{prob}\label{prob:injective}
    Let $\Sigma_1$ and $\Sigma_2$ be surfaces. Under what conditions is there an injective map between (a finite index subgroup of)
    $\Mod(\Sigma_1)$ and $\Mod(\Sigma_2)$? Can these maps be understood geometrically?
    \end{prob}
    
    As a consequence of the methods involved in the proof of Theorem~\ref{thm:interp-rigid}, we obtain the following result, which gives a
    completely novel perspective on Problem~\ref{prob:injective}:
    
    \begin{theorem}\label{thm:induced-homo-intro}
    	Let $\Sigma_1$ and $\Sigma_2$ be surfaces such that $\Sigma_1$ is not a sphere with fewer than six punctures or a torus with
    	fewer than three punctures. Suppose that $\C(\Sigma_1)$ is interpretable
    	(possibly with parameters) in $\C(\Sigma_2)$. Then canonically associated
    	to the interpretation, there is a finite index subgroup
    	$G\leq\Mod^{\pm}(\Sigma_1)$ and an injective homomorphism $G\hookrightarrow \Mod^{\pm}(\Sigma_2)$.
    \end{theorem}
    
    Thus, the purely model-theoretic data of an interpretation of one curve graph in another yields an algebraic object, namely an
    injective homomorphism between the corresponding mapping class groups. The group $G$ in Theorem~\ref{thm:induced-homo-intro}
    is somewhat mysterious, and it is not clear how to explicitly describe it. 
    More precisely, the mapping class group of $\Sigma_1$ acts
    on some set of complete types of maximal Morley rank, a set which is finite by compactness. The group $G$ is the stabilizer of one of these
    types; a more explicit description of the group is unknown to the authors.
    
   It would be remarkable
    if interpretations between curve graphs of surfaces were always ``induced'' (in some appropriate sense)
    by topological operations, such as subsurface inclusions and
    covering maps.

  \subsubsection{Interpretability and non-interpretability of other  complexes}\label{ss:other}
    
    There are many graphs other than the curve graph which are naturally associated
    to a surface $\Sigma$ of finite
    type. A general theme that can be observed in  the literature on these complexes
    is that their automorphism groups tend to coincide  with that of the curve graph, except
    in some sporadic low--complexity cases and a few notable exceptions; cf. Problem~\ref{p:ivanov} above, and cf.~axioms for hyperbolicity
    in~\cite{masur_schleimer_13}.
    We would like to say, in precise terms,
    that most known graphs associated to $\Sigma$ are completely determined by
    the curve graph, and that in fact they can be reconstructed from a finite
    sequence of canonical, first order operations.
    
    For the remainder of this section, assume that $\Sigma$ is not a torus with two punctures or a sphere with fewer than five punctures.
    The following result is schematic in the sense that it is technical but has
    broad applicability to the first order theory of complexes associated to $\Sigma$. It may be viewed as a precise result answering
    Question~\ref{q:general}.
    We state it here for easy reference.

    \begin{corollary}\label{c:schematic}
    	Let $X(\Sigma)$ be a graph with vertices $V(X(\Sigma))$ and edges
    	$E(X(\Sigma))$. Let $G:=\Mod^{\pm}(\Sigma)$ denote the extended mapping class group of $\Sigma$, and suppose that
    	the  natural action of $G$ on $\C(\Sigma)$ induces an action on
    	$X(\Sigma)$ by graph isomorphisms. Assume that the following conditions hold:
    	\begin{enumerate}
    		\item There exists a constant $N \geq 1$ such that each $v \in V(X(\Sigma))$
    		is identified with a collection of at most
            $N$ curves or arcs;
    		\item The quotient $V(X(\Sigma))/G$ is finite;
    		\item The quotient $E(X(\Sigma))/G$ is finite.
    	\end{enumerate}
    	Then $X(\Sigma)$ is interpretable in $\C(\Sigma)$. Consequently,
    	$X(\Sigma)$ is $\omega$-stable.
    \end{corollary}
    
    Corollary~\ref{c:schematic} is the primary philosophical reason why we take interpretability in the curve graph of $\Sigma$ as a definition
    of the term ``naturally associated'' in Problem~\ref{p:ivanov}.
    Observe that the interpretations in Corollary~\ref{c:schematic} are without parameters; thus, automorphisms of the curve graph of $\Sigma$
    automatically
    induce automorphisms of the geometric graph $X(\Sigma)$. Consequently, the fact that the extended mapping
    class group acts on $X(\Sigma)$
    can be seen through a purely model-theoretic lens.
    
    As a consequence of Corollary~\ref{c:schematic}, we have the  following conclusions
    about the first order theory of graphs naturally associated to $\Sigma$.

    \begin{corollary}\label{c:interpretability}
    	All of the following graphs are interpretable in the curve complex $\C(\Sigma)$,
    	and are therefore $\omega$-stable:
    	\begin{enumerate}
    		\item the Hatcher-Thurston graph $\mathcal{HT}(\Sigma)$;
    		\item the pants graph $\mathcal P(\Sigma)$;
    		\item the marking graph $\mathcal{MG}(\Sigma)$;
    		\item the non-separating curve graph $\mathcal N(\Sigma)$;
    		\item the $k$-separating curve graph $\C_k(\Sigma)$;
    		\item the Torelli graph $\mathcal{T}(\Sigma)$;
    		\item the $k$-Schmutz  Schaller graph $\mathcal{S}_k(\Sigma)$;
    		\item the $k$-multicurve graph $\mathcal{MC}_k(\Sigma)$;
    		\item the arc graph $\mathcal A(\Sigma)$;
    		\item the $k$-multiarc graph $\mathcal{MA}_k(\Sigma)$;
    		\item the flip graph $\mathcal{F}(\Sigma)$;
    		\item the polygonalization graph $\mathcal{P}ol(\Sigma)$;
    		\item the arc-and-curve graph $\mathcal{AC}(\Sigma)$.
    	\end{enumerate}
    	
    \end{corollary}
    
     The decoration of a graph by a nonnegative integer $k\geq 0$ means that
    intersections between curves or arcs representing  vertices are  allowed in the
    edge relation, up to at most $k$ intersections, where (except for when noted to the contrary) a lack of decoration implies
    $k=0$.
    
    We briefly summarize these graphs for the convenience of the reader, and give
    some references for definitions and discussions of automorphism groups. The
    \emph{Hatcher-Thurston graph} of a surface of positive genus consists of
    cut systems, i.e. systems of curves whose complement is a connected surface
    of genus  zero, and  whose edges correspond to elementary
    moves~\cite{hatcher_thurston1980,irmak_korkmaz}.
    The~\emph{pants graph} consists of multicurves giving rise to a pants
    decomposition of the surface, with edges
    given by elementary moves \cite{margalit}. The~\emph{marking graph}
    consists  of markings on
    the surface, which are pants decompositions together with transversal data,
    and whose edges are  given  by elementary moves \cite{masur_minsky_00}.
    The \emph{non-separating curve
    graph} consists of simple closed curves on the surface whose complement is
    connected,  and edges are  given by disjointness \cite{irmak}. The
    \emph{separating  curve graph} has separating simple closed curves as
    its  vertices  and  the edge  relation is disjointness \cite{looijenga, kida}.
    The \emph{Torelli graph} consists of separating curves and bounding pairs, with
    the edge relation being disjointness \cite{kida}. The \emph{Schmutz Schaller
    graph} on a closed surface has nonseparating  curves as its vertices, with
    geometric intersection number one being the edge
    relation~\cite{schmutz_schaller}. The \emph{arc graph} and \emph{multiarc graph}
    consist of simple arcs or multiarcs with endpoints at distinguished marked
    points, with edge relation given by disjointness
    \cite{irmak_mccarthy10,erl-fan-2017,disarlo_0}. The \emph{flip graph} and
    \emph{polygonalization graph} consist of simple arc systems whose  complements
    are a triangulation or  a  polygonal decomposition of the surface,
    respectively, with endpoints of  the arcs lying at distinguished marked points.
    Edge  relations are given by elementary moves \cite{korkmaz_papadopoulos,
    aramayona_koberda_parlier, disarlo_parlier, bell_disarlo_tang_2018,DPTrans19}. The
    \emph{arc and curve graph} consists of simple closed curves and simple arcs,
    with adjacency given by disjointness \cite{korkmaz_papadopoulos, disarlo_0}.

    In the standard literature on surfaces, many of the graphs discussed
    in Corollary~\ref{c:interpretability}
    are referred to as ``complexes'' instead
    of as ``graphs''. We stick to ``graph'' for consistency of terminology. In each
    case, the passage from the graph to the complex is simply by taking the
    flag complex, which is completely determined by its $1$--skeleton. In
    particular,  the relation expressing that $k+1$ vertices bound a
    $k$--cell is a first order relation that  is  completely determined by
    the adjacency relation; thus  the first order theories of the
    graphs and the complexes  are the same, since the higher simplicial relation is parameter-free definable from the adjacency relation.

    Though the graphs in Corollary~\ref{c:interpretability} are interpretable in the curve
    graph (and so are ``naturally associated'' to the surface $\Sigma$),
    their theories are essentially different from that of the curve graph. This remark can
    be articulated precisely in terms of Morley rank; for such a discussion, we direct the reader to
    Sections~\ref{sec:interpretations} and~\ref{sec:morley} below. As we have already mentioned, we opt for a
    description of ``sufficiently rich structure'' in terms of bi--interpretability. As an example of the failure of
    bi--interpretability with the curve graph for a naturally associated object, we have:
    
    \begin{corollary}\label{c:non-bi}
    	Let $\Sigma$ be a surface of genus $g$ with {$b$ punctures}.
    	The curve graph {$\mathcal C(\Sigma)$ } is not interpretable in the pants graph
    	$\mathcal P(\Sigma)$, provided that {$3g + b > 4$}.
    	The curve graph $\mathcal C(\Sigma)$ is not interpretable
    	the separating curve graph $\mathcal{SC}(\Sigma)$, provided $g\geq 2$ and $b\leq 1$.
    	The curve graph $\mathcal{C}(\Sigma)$ is not interpretable
    	the arc graph $\mathcal A(\Sigma)$, provided $g\geq 2$ and $b= 1$.
    \end{corollary}
    
    However, many natural geometric graphs associated to $\Sigma$ are bi--interpretable with suitable augmented Cayley graphs,
    and these auxiliary
    structures do have the extended mapping class group of $\Sigma$ as their group of automorphisms. We have the following
    result, which is more precisely stated as Lemma~\ref{interpretation in B}; definitions of relevant terms will be given in the sequel.

    \begin{theorem}\label{t:bi-int-intro-geom}
    	Let $X(\Sigma)$ be a geometric graph associated to $\Sigma$, where $\Sigma$ is sufficiently complicated
    	(i.e.~not one of finitely many sporadic surfaces). Suppose furthermore that the natural action homomorphism
        \[
            \Mod^\pm(\Sigma)\longrightarrow \aut(X(\Sigma))
        \]
        is an isomorphism, that $X(\Sigma)$ admits an exhaustion by finite strongly rigid tuples, and that the
    	stabilizer of a domain in $\D_X$ is commensurable with the
    	stabilizer of a tuple of vertices of $X(\Sigma)$. Then $X(\Sigma)$ is bi--interpretable with an augmented Cayley graph of the mapping
    	class group of $\Sigma$.
    \end{theorem}

    In particular, Theorem~\ref{t:bi-int-intro-geom} gives a model-theoretic resolution of Problem~\ref{p:ivanov},
    via Metatheorem~\ref{meta:ivanov}. The theorem provides a clear
    illustration of
    why ``sufficiently complex'' is interpreted to mean bi--interpretable with an augmented Cayley graph of the mapping class group
    of $\Sigma$; moreover, that this latter object has its automorphism group given by $\Mod^{\pm}(\Sigma)$ is ultimately a corollary
    of Ivanov's result that $\Aut(\C(\Sigma))\cong\Mod^{\pm}(\Sigma)$ in~\cite{ivanov_97}.
    The extra hypotheses on $X(\Sigma)$ in Theorem~\ref{t:bi-int-intro-geom} are somewhat technical, but can be verified for the curve graph,
    the arc graph, the pants graph of a genus zero surface, and the nonseparating curve graph, for instance.
    
    We reiterate that significant previous progress on understanding Ivanov's Metaconjecture has been made
    by Brendle and Margalit~\cite{brendle_margalit}, wherein they give
    conditions under which Ivanov's Metaconjecture  is true  and false, and
    adapt it to the study of normal  subgroups of mapping class groups. Brendle--Margalit work with subcomplexes of complexes of
    domains, and they leave open the cases of complexes where the edge relation is more complicated and where their techniques do not
    generalize, such as arc complexes and complexes of multicurves; see Conjecture 1.9 in that paper.
    We remark that in that sense, our approach
    is more robust under modification of the adjacency relation, and it provides a novel and philosophically satisfying explanation for the
    observed phenomena, even though for certain technical reasons our
    methods do not apply to the complex of domains.
    
    A by-product of the proof of Theorem~\ref{t:bi-int-intro-geom} is a certain recipe for producing definable subsets of geometric graphs,
    or more generally of finite Cartesian products of them. Specifically, let $X(\Sigma)$ denote a geometric graph falling under the purview
    of Theorem~\ref{t:bi-int-intro-geom}, such as curve graphs of nonsporadic surfaces.
    
    \begin{corollary}\label{cor:definable-recipe}
    	Let $k$ be a nonzero natural number, and let $Y\subseteq X(\Sigma)^k$. Suppose that $Y$ is invariant under the action of the mapping
    	class group of $\Sigma$, and suppose that the quotient of $Y$ by the diagonal
    	action of the mapping class group is finite, or that
    	the quotient of the complement $X(\Sigma)^k\setminus Y$ by the mapping class group is finite. Then $Y$
    	is definable in the language of graph theory.
    \end{corollary}
    
    The quotient conditions on $Y$ in Corollary~\ref{cor:definable-recipe} are called \emph{finite} and \emph{cofinite}, respectively.
    For us, $X(\Sigma)$ is generally a graph, and so the conclusion of Corollary~\ref{cor:definable-recipe} says that $Y$ is definable in the
    language of graph theory.
    
    Finally, we note that some of the ideas and several arguments in this manuscript were inspired by the
    papers~\cite{tent-jsl} and~\cite{BPZ17}. Some of the ingredients involved in building an  auxiliary
    $\mL$--structure (Section~\ref{sec:framework}), applications of geometry
    to characterizing types (Section~\ref{sec:dt}), the notion of
    simple connectedness (Section~\ref{sec:sc}), and the use of weak
    convexity in order to ultimately  establish $\omega$--stability and
    versions of quantifier  elimination (Section~\ref{sec:core}, Subsection~\ref{sec:rel-qe},
    and Section~\ref{sec:full-qe}), are
    adapted from  the arguments  in~\cite{BPZ17}.
    That said, a large number of new ideas were required to execute arguments for the curve graph, and
    significant technical complications arose. For instance, group elements occur naturally within the context of the auxiliary $\mL$--structure
    (i.e.~the augmented Cayley graph),
    and this causes many of the resulting arguments to be more complicated. Numerous other
    modifications to the context of mapping class groups are required, and
    the interpretability results for other complexes (Section~\ref{sec:interpretations})
    is made possible through  the framework in which we operate. Among  the
    ingredients specific to the mapping class group situation is the use
    of the asymptotic geometry of  the mapping class  group  and the
    Behrstock inequality \cite{Be}. The development and establishment of interpretation rigidity for
    curve graphs of surfaces is original to this monograph, and requires a novel synthesis of model-theoretic ideas and ideas from geometric
    group theory. We also draw inspiration from various rigidity results for mapping
    class groups and curve graphs; see~\cite{ara-souto-rigid}.

  \subsection{Some ideas in the proofs}

In order to build the (bi-)interpretations underlying
Theorems~\ref{thm:intro-bi-int} and~\ref{t:bi-int-intro-geom},
the primary difficulty is articulating a good definition of a geometric graph.
Once such a definition is given, the proof of the relevant interpretability
statements becomes largely formal. After these interpretations have been
realized, the remaining mathematical work takes place in the augmented Cayley
graphs of mapping class groups.

In these auxiliary structures, we develop a combinatorial/model--theoretic
calculus for understanding types of tuples. This calculus is built from words in
group elements and domain symbols, and is developed in
Sections~\ref{sec:relational-theory}, \ref{sec:dt}, and~\ref{sec:sc}.
The central idea is that the type of a tuple of elements in the auxiliary
structure is, roughly, determined by its relative position in the augmented
Cayley graph. In the context of weakly convex subsets, developed in
Section~\ref{sec:core}, the argument ultimately becomes a back-and-forth
argument over controlled finite substructures. For more complicated tuples, such
an argument does not generalize in a straightforward way, which imposes the
necessity for a relativity hypothesis in the quantifier-elimination result.

This is the main mechanism behind Theorem~\ref{thm:quantifier elimination}.
Once these ideas are established, one can also prove $\omega$--stability of the
theory of the auxiliary structure. Moreover, one obtains that the type of a tuple
is determined by either its quantifier-free type or its existential type; this is
the fundamental reason that these auxiliary structures enjoy the
quantifier-elimination results proved later in the paper. Together with the
Morley-rank computation in Section~\ref{sec:morley}, this proves
Theorem~\ref{thm:stable}.

Finally, we investigate elimination of imaginaries in Section~\ref{sec:int-rigid}.
This is the key model-theoretic ingredient in the proof of interpretation
rigidity for curve graphs, Theorem~\ref{thm:interp-rigid}, and in the
construction of homomorphisms between finite-index subgroups of mapping class
groups arising from interpretations, Theorem~\ref{thm:induced-homo-intro}.
    
\subsection{Guide to the paper}
The paper is structured as follows. Section~\ref{sec:background} recalls the necessary background, both from the geometric group theory of mapping class groups and curve complexes, and from model theory. On the geometric side, we review the surface-theoretic objects that enter the construction, including curve graphs and mapping class group actions. On the model-theoretic side, we recall the notions of interpretation, bi-interpretation, imaginaries, quantifier elimination, stability, and Morley rank. These notions are used throughout the paper in a concrete way: interpretations provide the language for reconstructing one structure inside another, quantifier elimination gives control over definable relations, and Morley rank provides a dimension-like invariant to compare definable sets.

Section~\ref{sec:framework} introduces the basic geometric and model-theoretic objects used throughout the paper. In particular, we define domains, filling, orthogonality relations, and the augmented Cayley graphs of the mapping class group $\mathcal M^G_{\D}(\Sigma)$. The augmented Cayley graphs are the main auxiliary structures of the paper. They are designed to retain the surface-theoretic information carried by the curve graph, while being more convenient for model-theoretic analysis. Relations in these structures are indexed by domains, so that the topology of subsurfaces becomes visible in the first-order language.

In Section~\ref{sec:interpretations}, we introduce the notion of a geometric graph associated to a surface and prove general results about interpreting such graphs in the augmented Cayley graphs of the mapping class group. This terminology is one of the conceptual points of the paper: it is meant to give a precise first-order meaning to the phrase ``naturally associated to a surface'' in the Ivanov Metaconjecture. Rather than treating naturality as an informal or case-by-case condition, we isolate a broad class of surface graphs which can be defined in terms of finite mapping-class-group orbits and show that these graphs are interpretable in the curve graph and bi-interpretable with an appropriate augmented Cayley graph. The class of graphs includes all the examples known so far in geometric topology of surfaces. The interpretability statements Corollary~\ref{c:schematic}, Corollary~\ref{c:interpretability}, and Corollary~\ref{c:non-bi} are proved in this section. The same section also proves Corollary~\ref{cor:alg-int}, showing that algebraic intersection number is not definable in the first-order theory of the curve graph. Finally, Theorem~\ref{thm:quantifier elimination} is deduced in Section~\ref{sec:interpretations} from the relative quantifier-elimination theorem proved later in Subsection~\ref{sec:rel-qe}.

Section~\ref{sec:relational-theory} develops the combinatorial background needed to investigate the theories of the structures $\mathcal M^G_{\D}(\Sigma)$. The main purpose of this section is to introduce a calculus of words in group elements and domain symbols. Such words encode finite chains of relations in the augmented Cayley graph. A central point is that different words may describe the same relation, or may describe relations related by reduction or cancellation. The notions of quasi-permutation, reduction, strict reduction, and reduction without cancellation are introduced to distinguish  changes of representative from actual simplifications. This word calculus is the combinatorial foundation for the later analysis of types.

Sections~\ref{sec:dt} and~\ref{sec:sc} combine mapping-class-group-theoretic methods with the combinatorial tools from Section~\ref{sec:relational-theory}. These sections translate the control obtained from the word calculus into statements about displacement, types, and coherence of the resulting theory. In particular, they ensure that the combinatorial descriptions of paths and relations in the augmented Cayley graph behave well enough to support a quantifier-elimination argument.

Section~\ref{sec:core} develops the doubling technique and the weakly convex extension machinery. This is the technical core of the quantifier-elimination argument. The purpose of weak convexity is to single out finite controlled substructures of the augmented Cayley graph over which types can be compared by looking only at the domain relations already present in the language. The main extension results show that, when two tuples have the same quantifier-free data over a weakly convex set, one can enlarge the weakly convex set and extend the comparison so as to include new elements with prescribed domain-displacement behavior. In other words, this section supplies the back-and-forth step needed for quantifier elimination. The relative quantifier-elimination theorem is then proved in Subsection~\ref{sec:rel-qe}; it is the main technical input used earlier in Section~\ref{sec:interpretations} to prove Theorem~\ref{thm:quantifier elimination} for the curve graph.

In Section \ref{sec:full-qe} we strengthen the results to absolute quantifier elimination for a pure finite-index subgroup, and conclude the $\omega$--stability argument. This proves the $\omega$--stability part of Theorem~\ref{thm:stable}. The quantifier-elimination results show that definable relations in the augmented Cayley graph are governed by the prescribed finite relational data, while $\omega$--stability shows that the resulting theories are model-theoretically tame.
We then obtain Morley-rank bounds for the various theories under discussion in Section~\ref{sec:morley}. The Morley-rank computation, which completes the proof of Theorem~\ref{thm:stable}, should be understood as a model-theoretic dimension calculation. In the cases most relevant to the curve graph, the rank reflects the topological complexity of the underlying surface. These computations are also used to distinguish certain interpreted structures from one another.

Finally, Section~\ref{sec:int-rigid} proves that the class of curve graphs of surfaces enjoys interpretation rigidity. In that final section, the preceding model-theoretic analysis is used as a package: quantifier elimination, weak elimination of imaginaries, total triviality, and the Morley-rank computations together imply that, except in the low-complexity cases excluded in the statements, the first-order structure of the curve graph remembers the underlying surface. This proves Theorem~\ref{thm:interp-rigid}. The same arguments also associate to an interpretation between curve graphs an injective homomorphism between suitable finite-index subgroups of the corresponding mapping class groups, proving Theorem~\ref{thm:induced-homo-intro}.

\section*{Acknowledgements}
  The authors thank M.~Casals-Ruiz, I.~Kazachkov, S.-h.~Kim, J.~T.~Moore,   C.~Perin, and K.~Tent
  for helpful discussions. \\   
  The authors also thank anonymous reviewers for extensive, helpful comments. The authors thank Universit\"at
  Heidelberg, the University of Virginia,
  and University of the Basque Country (UPV/EHU) for hospitality and the working conditions provided while part of this research was carried out.
  The authors acknowledge support from U.S. National Science Foundation
  grants DMS 1107452, 1107263, 1107367 ``RNMS: Geometric structures and
  Representation varieties'' (the GEAR Network). \\ 
  The first author acknowledges support
  from the European Research Council under ERC-Consolidator grant 614733
  (GEOMETRICSTRUCTURES), the Olympia Morata Programme of Universit\"at Heidelberg and from the Deutsche
  Forschungsgemeinschaft (DFG, German Research Foundation) under Germany's Excellence Strategy EXC-2181/1 - 390900948 (the Heidelberg STRUCTURES Cluster of Excellence) and the Priority Program SPP 2026 ``Geometry at Infinity'' (DI 2610/2-1). This project has also
received funding from the European Union’s Horizon 2020 research and innovation programme under the Marie Skłodowska-Curie Grant Agreement No. 101034413 obtained by the first author. \\ 
  The second author was partially supported by
  an Alfred P. Sloan Foundation Research Fellowship, by NSF Grant DMS-1711488, and by
  NSF Grant DMS-2002596, and is partially supported by NSF Grant DMS-2349814. \\ The third author has received funding from the European Research
  Council under the European Union's Seventh Framework Programme (FP7/2007- 2013)/ERC
  Grant Agreements No. 291111 and No. 336983 and from the Basque Government Grant IT974-16.

\section*{AI Statement}
The ArXiv versions 1--3 of this manuscript were prepared with no AI 
assistance. In the present version, ChatGPT 5.5 was used to assist with 
the preparation of the seven notation tables, and polishing the human-
written expository text in a few points. No AI tools were ever used in 
developing or writing the proofs. The authors take full responsibility 
for the mathematical content and the final text of the manuscript.
  
\section{Background}\label{sec:background}

This section fixes the background conventions used throughout the paper. 
Subsection \ref{ss:mcg} recalls the surface-theoretic objects, especially mapping class groups and curve graphs. 
Subsection \ref{ss:model} recalls the model-theoretic notions needed later. 
The reader familiar with either side may wish to use this section mainly as a reference for conventions and notation.
  
  \subsection{Mapping class groups and curve graphs}\label{ss:mcg}
    A standard reference for the material in this section is the book by
    Farb and Margalit~\cite{farb_margalit_11}.
    
    \subsubsection{Generalities}
      Let $\Sigma$ be an orientable surface of finite type. We will always assume
      that the Euler characteristic $\chi(\Sigma)$ is negative, so that $\Sigma$
      admits a complete hyperbolic metric of finite area (possibly with cusps).
      \begin{definition}[The mapping class group $\Mod(\Sigma)$]
      	The \emph{extended mapping class group} {$\Mod^\pm(\Sigma)$} is defined to
      	be the group of the isotopy classes of the homeomorphisms of $\Sigma$.
      	We allow homeomorphisms to reverse the orientation of the surface and to
      	permute its punctures. The \emph{mapping class group} $\Mod(\Sigma)$ is the subgroup generated by orientation-preserving homeomorphisms. The \emph{puncture-fixing mapping class group} is the subgroup generated by mapping classes which fix the
	punctures pointwise; this subgroup is sometimes called the \emph{pure} mapping class group, but we will not use that terminology to
	avoid certain confusions.
	
      	A nontrivial mapping class $\psi$ is called \emph{pure} if it is of infinite order and admits a (possibly empty)
        canonical reduction system $\mathcal S$ such
      	that $\psi$ preserves every component of $\mathcal S$ and of $\Sigma\setminus \mathcal S$ up to isotopy, and restricts to the identity or to
      	a pseudo-Anosov mapping class on each component of $\Sigma\setminus \mathcal S$; cf.~\cite{BLM-duke}. It is a standard fact
      	that the mapping class group of a surface $\Sigma$ admits a finite index subgroup consisting of pure mapping classes. We will
	call a subgroup of the mapping class group consisting entirely of pure mapping classes a \emph{pure mapping class group}.
      \end{definition}
      
      Let $\gamma\colon S^1\to \Sigma$ be a homotopy class of maps of the circle into
      $\Sigma$, up to orientation reversal. It is a standard
      result that it suffices for us to work with smooth representatives. We say that
      $\gamma$ is a~\emph{simple closed curve}
      if in addition $\gamma(S^1)$, with an arbitrary choice of orientation, represents a nontrivial conjugacy class in
      $\pi_1(\Sigma)$, if some representative of $\gamma$ is an embedding, and
      if $\gamma(S^1)$ is not freely homotopic to a loop which encircles a puncture
      {or boundary component} of $\Sigma$. We will often conflate $\gamma$ with the
      image of a representative of $\gamma$, which can always be chosen to be geodesic
      in a fixed hyperbolic metric on $\Sigma$.
      The set of all simple closed curves on $\Sigma$ is organized into
      the~\emph{curve graph} of $\Sigma$.
      \begin{definition}[The curve graph $\C(\Sigma)$]\label{def:curve graph}
      	The \emph{curve graph} $\C(\Sigma)$ is a graph which consists of one vertex
      	for each simple closed curve, and where $\gamma_1$ and $\gamma_2$ are adjacent
      	if there are representatives of $\gamma_1(S^1)$ and $\gamma_2(S^1)$ which intersect the smallest possible
        number of
      	times in $\Sigma$, among all pairs of distinct simple closed curves;
	this latter condition is equivalent to the statement that geodesic
      	representatives of the corresponding loops intersect minimally in $\Sigma$. If $\Sigma$ has negative Euler characteristic
	then the adjacency relation is simply disjointness provided $\Sigma$ is not a torus with one puncture or a four-times punctured sphere.
	In the latter two cases, the adjacency relation is intersection in one and two points, respectively.
	
	If $D$ is an annular domain, then the definition of the curve graph is slightly different, since there are no curves in $D$ other than the
	core curve. We direct the reader to~\cite{masur_minsky_99} for a precise definition. The only relevant property of the curve graph in
	this case is that it is quasi-isometric to a line, and the Dehn twist about the core curve acts by a translation by a definite positive distance.
      \end{definition}
      
      We will often write $\mC_0(\Sigma)$ for the set of vertices of $\mC(\Sigma)$, and
      sometimes just $\mC_0$ if the underlying surface is clear.
      The curve graph has very complicated local and global structure. We record
      some of the facts which are useful for us.
      
      \begin{theorem}[See~\cite{farb_margalit_11,masur_minsky_99}]\label{thm:cs-basic}
      	Let $\Sigma$ be a non-sporadic surface and let $\mC(\Sigma)$ be the curve
      	graph of $\Sigma$. The following hold:
      	\begin{enumerate}
      		\item
      		The graph $\mC(\Sigma)$ is connected, locally infinite, and has infinite diameter.
      		\item
      		If $\Sigma'\subset \Sigma$ is an incompressible non-sporadic subsurface,
      		then there is an inclusion of subgraphs $\mC(\Sigma')\subset\mC(\Sigma)$,
      		and the image of this inclusion has diameter two in $\mC(\Sigma)$.
      	\end{enumerate}
      \end{theorem}
      
      Much more is known about the geometry of curve graphs, including their hyperbolicity and various hierarchical structures, though
      these properties will not be relevant for us here.
      
      By \emph{non-sporadic surfaces}, we generally
      mean ones which admit two disjoint,
      non-isotopic simple closed curves. Then, sporadic surfaces are spheres with at most
      four punctures and tori with at most one puncture. For sporadic surfaces, there
      are suitable modifications of  the definition of the curve graph which allow for
      analogues of Theorem~\ref{thm:cs-basic}. In the sequel, we will use the terminology
      \emph{sporadic surfaces} to refer to various finite lists of low-complexity surfaces whose exact
      members will depend on the context and will be specified
      appropriately.
      
      The mapping class group acts simplicially on $\C(\Sigma)$. Standard results
      from combinatorial topology of surfaces imply that, except for finitely
      many exceptions, the action of $\Mod(\Sigma)$ on $\mC(\Sigma)$ is faithful.
      \begin{theorem}[Ivanov \cite{ivanov_97}, Luo \cite{Luo}, Korkmaz \cite{korkmaz}]
      	Let $\Sigma$ be an orientable surface of genus $g$ and $b$ punctures:
      	\begin{enumerate}
      		\item If $\Sigma$ admits a pair of non-isotopic simple closed curves and if $(g, b)
      		\neq (1, 2)$, then any automorphism of $\C(\Sigma)$ is induced by a self-homeomorphism
      		of the surface.
      		\item Any automorphism of $\C(\Sigma_{1,2})$ preserving the set of
      		vertices represented by separating loops is induced by a self-homeomorphism of the surface.
      		\item There is an automorphism of $\C(\Sigma_{1,2})$ which is not
      		induced by any homeomorphism.
      	\end{enumerate}
      \end{theorem}
      
      The action of $\Mod(\Sigma)$ on $\mC(\Sigma)$ is cofinite, in the sense
      that the vertices and edges of $\mC(\Sigma)$ fall into finitely many orbits
      under the action of $\Mod(\Sigma)$. Since $\Mod(\Sigma)$ can invert edges
      of $\mC(\Sigma)$, we typically do not speak of a graph structure on the
      quotient $\mC(\Sigma)/\Mod(\Sigma)$.
      
    \subsubsection{Subsurface projections}\label{ss:projection} Let $\Sigma'$ be an essential subsurface of $\Sigma$.
      We denote the curve graph of $\Sigma'$ simply by $\C(\Sigma')$.
      Given $\alpha\in\C_{0}$, we denote by $\pi_{\Sigma'}(\alpha)$ the
      (possibly empty) projection of $\alpha$ to $\C(\Sigma')$. Following
      Masur--Minsky~\cite{masur_minsky_99}, the projection of a curve $\alpha$ to
      a  subsurface $\Sigma'\subsetneq\Sigma$ is obtained by taking  a geodesic
      representative of $\alpha$ and of $\partial\Sigma'$, and then taking the
      boundary of a small tubular neighborhood of
      \[\partial\Sigma'\cup(\alpha\cap\Sigma').\] The result is a finite collection
      of simple closed  curves of $\Sigma'$, some of which  may  be  inessential or
      peripheral. Discarding  those which are inessential or peripheral, we obtain a finite set of pairwise adjacent
      vertices in $\C(\Sigma')$, which we take to  be $\pi_{\Sigma'}(\alpha)$.
      Thus, projection gives a coarsely well-defined map from $\C(\Sigma)$ to
      $\C(\Sigma')$.
      We will write $d_{\Sigma'}(\alpha,\beta)$ as an abbreviation for
      \[d_{\C(\Sigma')}(\pi_{\Sigma'}(\alpha),\pi_{\Sigma'}(\beta)).\]
      We adopt the convention that the distance between the empty set and
      any set is infinite.  Subsurface projections will play an essential role
      in establishing
      simple connectedness of certain relational structures in
      Section~\ref{sec:sc} below.
      
  \subsection{Model theory}\label{ss:model}
    For standard references for the material contained in this
    section see~\cite{tent_ziegler,marker2006model}.
    
    \subsubsection{Languages, structures, theories, and models}
      We will work with several languages in this paper. The theory of the
      curve graph will be formulated in the first order language of graph
      theory, which consists of the single symmetric, binary relation $E$,
      interpreted as adjacency. For most of the discussion in this paper,
      we will work in an auxiliary structure adapted to a particular
      surface, and which will encode relations for subsurfaces and
      mapping classes.
      
      \begin{definition}[$\mL$--structure] A \emph{language} $\mL$ is a set
      	of constants, function symbols and relation symbols. An
      	\emph{$\mathcal{L}$--structure} $\mM$  is given by a set $M$ called
      	the \emph{universe}, together with an interpretation of the constants,
      	relations, and functions of $\LL$. We say that a structure is
      	\emph{relational} if its
      	underlying language has no functions and no constants.
      \end{definition}
      
      When talking about specific elements or collections of elements in the universe of a $\mathcal{L}$--structure, one often encounters
      \emph{tuples}. These will be thought of as subsets of the universe, indexed by an ordinal. If $A$ and $B$ are tuples of elements
      in a universe $M$, we will write $AB$ for the tuple $(A,B)$, indexed by the ordinal sum.
      
      One can use symbols in a language, together with logical connectives,
      variables, and quantifiers,
      in order to express conditions on tuples of elements in a structure. Such a condition
      among tuples is called a \emph{formula}. A \emph{sentence} is a formula without
      variables that are unbound by a quantifier. A formula or sentence in the language $\mathcal L$ is called an \emph{$\LL$--formula} or
      \emph{$\LL$--sentence}. If $A\subset M$, where $\mM$ is an $\mL$--structure, we call \emph{$\LL(A)$--formulae (resp.~sentences)}
      the $\LL$--formulae (resp.~sentences) with parameters in $A$.

      For every sentence $\phi$ we write $\mM \models \phi$ if $\phi$ holds in
      $\M$. We say that $\mM$ is a \emph{model of} $\phi$ and
      $\phi$ \emph{holds in} $\mM$. Similarly if $\Sigma$ is a set of sentences,
      then $\mM$ is a \emph{model of} $\Sigma$  if
      all the sentences of
      $\Sigma$ holds in $\mM$.
      
      \begin{definition}[Consistent $\mL$--theory]
      	A \emph{theory} $T$ is given by a set of $\LL$--sentences. A theory $T$
      	is \emph{consistent} if there is an $\LL$--structure $\mM$ where every
      	sentence in $T$ holds. We say that $\mM$ is a \emph{model} of $T$. The \emph{theory of $\mM$}, denoted $\Th{\mM}$, consists
      	of the maximal (with respect to inclusion) theory $T$ such that $\mM\models T$.
      \end{definition}
      The \emph{consequences} of $T$ are the sentences which hold in every model
      of $T$. If $\phi$ is a consequence of $T$, we say that $\phi$ \emph{follows
      from} $T$ (or $T$ \emph{proves} $\phi$) and write $T \vdash \phi$.
      If for all $\mL$--sentences
      $\phi$ either $T \vdash \phi$ or $T \vdash \neg \phi$, the theory $T$ is called \emph{complete}. G\"odel's Completeness Theorem
      implies that the proving relation $\vdash$ and the modeling relation $\models$ are, in an appropriate sense, equivalent.
      A theory is \emph{consistent} (sometimes called
      \emph{satisfiable}) if it admits a model. Equivalently,
      by G\"odel, a theory $T$ is consistent if it is not \emph{inconsistent}, which is
      to say $T\vdash \phi$ and $T\vdash \neg\phi$ for some sentence $\phi$. The following result is basic and will be used several times
      in this paper.

      \begin{theorem}[Compactness Theorem]
      	A theory $T$ is  consistent if and only if every finite subset of $T$
      	is consistent.
      \end{theorem}
      
      Let $\mM$ be an $\mL$--structure and $A \subseteq M$. We say that $b\in M$
      \emph{realizes} a set of $\mL(A)$--formulae $p(x)$ (containing at most
      the free variable $x$), if $b$ satisfies all formulae from $p(x)$; in
      this case we write $\mM \models p(b)$. We say that $p(x)$ is
      \emph{finitely satisfiable} in $\mM$ if every finite subset of $p$ is realizable in $\mM$.
      
      \begin{definition}[Types]
      	The set $p(x)$ of $\mL(A)$--formulae is a \emph{type over} $A$ if $p(x)$
      	is finitely satisfiable in $\mM$. If $p(x)$ is maximal with respect
      	to inclusion, we say that $p(x)$ is a \emph{complete type}.
      	We say that $A$ is the \emph{domain} (or set of \emph{parameters}) of $p$.
      	
      \end{definition}
      
      {An example} is the type determined by an element $m\in
      M$:
      $$\tp(m/A) = \{ \phi(x) ~|~\mM \models \phi(m), ~\phi \mbox{ is an }
      \mL(A)\mbox{-- formula} \} ~.$$
      
      An \emph{$n$--type} is a finitely satisfiable set of formulae in $n$
      variables $\{x_1, \ldots, x_n\}$. As for $1$--types,
      a maximal $n$--type  is  called \textbf{complete}. As for
      types of elements, we have the natural notion of the
      \emph{type of a tuple}, which is automatically a complete
      type.
      The set of $n$-types over $A$ is denoted by $S_n(A)$, and the set of
      complete types over $A$ is denoted by
      \[S(A):= \bigcup_{n < \omega}S_n(A).\]
      Let $p$  be a type and let $\phi\in p$. We say  that $\phi$
      \emph{isolates} $p$ if for all $\psi\in p$, we have
      \[\Th{\M}\models  \phi(x)\rightarrow\psi(x).\]
      
      The set of complete $n$--types with parameters in $A$ has a natural topology on it.
      A basis of open sets is given by formulae $\phi$ with $n$ free variables, and
      $p\in U_{\phi}$ if and only if $\phi\in p$. Completeness of types implies that these
      sets are in fact clopen. A type is isolated if and only if it is isolated in this
      topology. We call this topology the~\emph{Stone topology}.

      We will oftentimes need to discuss \emph{quantifier-free types} of tuples, over sets of parameters; these are defined in the 
      same way as types of tuples, except the formulae occurring are required to be quantifier-free. The quantifier-free type of an element
      $m\in M$ is written $\qftp(m)$.
      
      \begin{definition}[$\kappa$-saturated; cf.~Definition 4.3.1 in~\cite{marker2006model}]
      	Let  $\kappa$ be an  infinite cardinal.
      	A $\mL$--structure $\mathcal{M}$ is \emph{$\kappa$--saturated} if for all
      	$A \subset M$ and $p \in S(A)$ with $|A| < \kappa$,  the type $p$ is
      	realized in $\mM$. We say that $\mM$ is saturated if it is
      	$|M|$--saturated.
      \end{definition}
      
      To avoid having to switch models to realize types, we will often work in the
      \emph{monster model}, which can be thought of as a model of the theory containing as an elementary sub-model any model one might eventually need to consider. See Section 6.1 in~\cite{tent_ziegler}.
      A complete theory with infinite models admits a model $\mathscr M$
      such that all types over all subsets of $M$ are realized in $\mathscr M$.
      Up to some set-theoretic issues, $\mathscr M$ is unique up to isomorphism.
      It is well--known that the monster model of a theory $T$
      also enjoys all of the following properties:
      \begin{itemize}
      	
      	\item Any model of $T$ is elementarily embeddable in $\mathscr M$;
      	\item Any elementary bijection between two subsets of $\mathscr M$
      	can be extended to an automorphism of $\mathscr M$.
      \end{itemize}
      
    \subsubsection{Definability and interpretability}
      Let $X$ be an $\mL$--structure. A subset $Y\subset X$ is~\emph{definable} if
      there is a formula $\phi$ and a finite tuple of parameters $a$
      in $X$ such that
      \[Y=\{y\mid X\models \phi(y,a)\}.\]
      Thus, $Y$ is ``cut out'' by the formula $\phi$. We will say that
      $Y$ is \emph{$\emptyset$--definable}
      if \[Y=\{y\mid X\models \phi(y)\}.\]
      In this case, we also say that $Y$ is \emph{parameter--free definable}.
      We say a structure $Y$ is~\emph{interpretable} in a structure $X$ if,
      roughly, there is a definable subset $X_0$ of $X$ and a definable
      equivalence  relation $\sim$ on $X_0$ such that $Y$ is isomorphic to
      $X_0/\sim$. More precisely, we will define interpretability as follows.
      
      Let $f:A\to B$  be a function. We denote by $f^{\times r}$ the associated product
      function \[f\times f\cdots \times f:A^{r}\to B^{r}.\]
      Given an equivalence relation $E$ on a set $X$, let $E^{r}$ {be } the natural equivalence relation it induces on $X^{r}$.

      \begin{definition}[Interpretation] \label{d:interpretation} Given two structures $\M$ and
      	$\mathcal{N}$, a \emph{(parameter--free) interpretation}
      	of $\mathcal{M}$ in $\mathcal{N}$ is given by a tuple \[\bar{\eta}=(\eta,k,X,E)\colon\mathcal M\rightsquigarrow\mathcal N,\]
      	where the following holds:
      	\begin{itemize}
      		\item $k\in\N$;
      		\item $X\subseteq N^{k}$ is a $\emptyset$--definable set;
      		\item $E\subseteq X\times X$ a definable equivalence relation;
      		\item $\eta:M\to X/E$ a bijection such that for every $r\geq 1$
      		any definable set $Y\subseteq M^{r}$ is the preimage of a  unique,
      		$E$-invariant, $\emptyset$--definable set
      		$Y_{\eta}\subseteq X^{r}$.
      	\end{itemize}
      \end{definition}

      Notice that in order to verify a purported interpretation,
      it suffices to check the relevant properties for the relations in
      the language of $\M$, and for the graphs of function symbols in the language
      of $\mathcal{M}$.
      
      Given parameter-free interpretations \[\bar{\eta}_1=(\eta_{1},k_{1},X_{1},E_{1})\colon\M_1 \rightsquigarrow\M_2\]
      of $\M_{1}$ in $\M_{2}$ and \[\bar{\eta}_{2}=(\eta_{2},k_{2},X_{2},E_{ 2})\colon\M_2 \rightsquigarrow\M_3\]
      of $\M_{2}$ in $\M_{3}$, the composition
      \[\bar{\eta}_{2}\circ\bar{\eta}_{1}=(\eta_{3},k_{1},X_{1},E_{3})\colon\M_1 \rightsquigarrow\M_3\] is given as follows:
      \begin{itemize}
      	\item $X_{3}:=\eta_{2}[X_{1}]\subseteq X_{2}^{k_{1}}\subseteq M_{1}^{k_{1}k_{2}}$,
      	\item $E_{3}$ is the equivalence relation $\eta_{2}[E_{1}]\subseteq X_{3}\times X_{3}$.
      	Notice that this is automatically coarser than the restriction of
      	$E_{2}^{\times k_{1}}$ to $X_{3}$.
      	\item $\eta_{3}$ is the composition of $\eta_{1}$ with the map
      	$\tilde{\eta}_{2}: X_{1}/E_{1}\to X_{3}/E_{3}$ through which
      	\[\eta_{2}^{\times k_{1}}\restriction_{X_{1}}:X_{1}\to (X_{2}/E_{2})^{k_{1}}\]
      	factors in the natural way.
      \end{itemize}
      
      It can be shown that $\bar{\eta}_{3}$ is
      an interpretation of $\mathcal{M}_{1}$ in $\mathcal{M}_{3}$.
      \begin{definition}[Definable interpretation]\label{def:def-int}
      	We say that an interpretation \[\bar{\eta}=(\eta,k,X,E)\colon\M \rightsquigarrow\M\] of a structure
      	$\mathcal{M}$ in itself
      	is \emph{definable} if the relation
      	\[\Gamma_{\eta}=\{(x,y)\in M^{2}\,|\,\eta(x)=[y]_{E}\}\] is $\emptyset$-definable.
      \end{definition}
      
      \begin{definition}[Bi-interpretation] A \emph{bi-interpretation} between structures
      	$\M$ and $\mathcal{N}$ is a pair $(\bar{\eta},\bar{\zeta})$, where
      	$\bar{\eta}$ is a  interpretation of $\mathcal{M}\rightsquigarrow\mathcal{N}$,
      	where $\bar{\zeta}$ is an interpretation
      	of $\mathcal{N}\rightsquigarrow\mathcal{M}$, and where both
      	$\bar{\zeta}\circ\bar{\eta}$ and $\bar{\eta}\circ\bar{\zeta}$ are
      	$\emptyset$--definable. Accordingly, we say that the two structures are
      	bi-interpretable.
      \end{definition}
      
      It is a standard fact that if structures $\M$ and $\mathcal N$ are bi-interpretable (without parameters),
      then $\Aut(\M)\cong\Aut(\mathcal N)$.
      
    \subsubsection{Algebraic and definable closure}
      \newcommand\dcl{\mathrm{dcl}}
      \newcommand\acl{\mathrm{acl}}
      
      Let $X$ be an $\mL$--structure and let $A\subset X$ be a set of parameters. An element
      $a\in X$ is in the \emph{definable closure} of $X$ if $a$ is the unique element satisfying
      a formula $\phi$ with parameters in $A$. The definable closure $\dcl(A)$ consists of
      all such elements $a$. Similarly, the \emph{algebraic closure} of $A$ consists
      of elements $a\in X$ for which there is a formula $\phi$ with parameters in $A$
      which is satisfied by $a$ and which has finitely many solutions. We write
      $\acl(A)$ for the algebraic closure of $A$. Algebraic and definable closures are
      idempotent; see the exercises in Chapter 1 of~\cite{marker2006model}.
      
      Algebraic and definable closures are invariant under elementary extensions; that is, if $\mathcal N$ is an elementary
      extension of $\M$ and $a\in \M$ then the definable and algebraic
      closures of $a$ in $\mathcal N$ agree with those in $\M$. Moreover,
      if $X$ is sufficiently saturated then algebraic and definable closures are characterized
      in terms of types. Namely, $a\in\dcl(A)$ if and only if $\tp(a/A)$ has a unique
      realization in $X$, and $a\in\acl(A)$ if and only if $\tp(a/A)$ has finitely many
      realizations in $X$.
      
      If $X$ is the monster model of a theory, then $a\in\dcl(A)$ if and only if $a$ is fixed
      by every automorphism of $X$ which fixes $A$. Similarly, $a\in\acl(A)$ if and only
      if $a$ has finitely many orbits under the group of automorphisms of $X$ fixing $A$.
      
    \subsubsection{Imaginaries}\label{sssec:imaginaries}
      \newcommand\EQ{\mathrm{eq}}
      Let $X$ be an $\mL$--structure and let $x$ and $y$ be $n$--tuples for some $n$.
      An \emph{equivalence formula} $\phi(x,y)$ is a formula that is a symmetric and transitive
      relation. It is an equivalence relation on the set of tuples $a$ such that $\phi(a,a)$.
      An \emph{imaginary} is an equivalence formula $\phi$ together with an
      equivalence class $[a]_{\phi}$. We say that $X$ has \emph{elimination of imaginaries}
      if every imaginary $[a]_{\phi}$ is \emph{interdefinable} with a real tuple
      $\overline\rho$, i.e.~$\overline\rho$ is in the definable closure of $[a]_{\phi}$ and vice versa.
      A theory admits elimination of imaginaries if every model does.
      Not every theory admits elimination of imaginaries, though every theory $T$ can be embedded
      in a theory $T^{\EQ}$ which does. A model $X$ of a theory $T$ can be extended to a structure
      $X^{\EQ}$, which consists of $X$ (called the \emph{home sort}) together with all the
      definable equivalence relations on tuples in $X$ (called the \emph{imaginaries}). The
      algebraic and definable closure in $X^{\EQ}$ are written $\acl^{\EQ}$ and $\dcl^{\EQ}$
      respectively.
      
      A theory is said to have \emph{weak elimination of imaginaries} if every imaginary is in the definable closure of a finite tuple
      in the home sort, and if the tuple lies in the algebraic closure of the imaginary.

    \subsubsection{$\omega$--stability, Morley rank, and forking}
      {We will now recall the definition of $\omega$-stability and illustrate other related concepts, which will be essential for us.}
      
      \begin{definition}[$\omega$--stability]
      	Let $T$ be a complete theory with infinite models.
      	We say that $T$ is \emph{$\omega$--stable}
      	if in each model of $T$ and for each countable set of parameters $A$,
      	there are at most countably many complete $n$--types
      	for each $n$.
      \end{definition}

      \begin{example}
      	Let $\kappa$ be an uncountable cardinal.
        By the proof of Morley's categoricity theorem \cite{morley1965categoricity}, every countable $\kappa$--categorical theory is $\omega$--stable. A particularly notable example of this is the theory of algebraically closed fields in a fixed characteristic.
      \end{example}

      \begin{example}\label{ex:stable-interp}
          Suppose $X$ is an $\mL$--structure whose theory is
          $\omega$--stable, and suppose that $Y$ is an $\mL'$
          structure that is interpreted in $X$. Then the theory of
          $Y$ is also $\omega$--stable. This is because any countable
          collection of parameters $A$ in $Y$ lifts to a countable
          collection of parameters in $X$. If $\overline b$ is a
          tuple of elements in $Y$ then the type $\tp(\overline b/A)$
          is determined by the type of the lift of $\overline b$ over the lift of $A$ to $X$. Two distinct
          complete types in $Y$ (e.g.~of tuples $\overline b$ and $\overline b'$) over $A$ are mutually inconsistent, and so the types of the lifts of these tuples to $X$ will be distinct; in particular, there are only countably many
          complete types over $A$.
      \end{example}

      In  the  context of a complete theory  $T$,
      the \emph{Morley rank} is a notion of dimension
      for a formula with parameters in the monster model.
      \begin{definition}[Morley rank]
      	We define the \emph{Morley rank of a formula} $RM(\phi)$ by transfinite
      	induction as follows:
      	\begin{enumerate}
      		\item $RM(\phi(x)) \geq 0$ if $(\exists x)\phi(x)$.
      		\item For $\alpha = \beta + 1$, we have $RM(\phi(x))\geq \alpha$ if and only if
      		there exist{s} an infinite family of formulae $\{ \phi_i(x) \}$ which are pairwise
      		inconsistent
      		with $RM(\phi_i)\geq\beta$,
      		and $\forall x ( \phi_i(x) \to \phi(x) )$.
      		\item $RM(\phi(x)) \geq \delta $ for a limit ordinal $\delta$ if and only if
      		for each $\alpha < \delta$, we have $RM(\phi(x))\geq \alpha$.
      	\end{enumerate}
      	We set $RM(\phi(x))= \beta$ for the least $\beta$ such that $RM(\phi(x))
      	\ngeq \beta +1$. If there is no such $\beta$, then $\phi$ is unranked. The
      	\emph{Morley rank of a type} $p \in S(A)$ is the smallest Morley rank of a
      	formula $\phi \in p$ in that type.
      	\emph{The Morley rank of a countable theory} $T$ is defined as the Morley
      	rank of the formula $ x = x$, that is \[RM(T):= RM(x=x).\]
      	In the Stone topology on types, Morley rank coincides with Cantor--Bendixson rank.
      \end{definition}

      For  a countable complete theory, the notion of $\omega$--stability is
      characterized by the existence of an ordinal--valued rank function $r$, whose
      domain is  the set of definable subsets of a (sufficiently saturated) model,
      and which has the following property: if $X$ is a definable subset of the model,
      then $r(X)>\alpha$ if there exists a countably infinite collection of  pairwise
      disjoint definable sets $\{Y_i\subset X\mid i<\omega\}$ with
      $r(Y_i)\geq\alpha$. The smallest  such rank function  is the Morley rank.

      Let $T$ be an $\omega$--stable theory, and let $\M$ and $\mathcal{N}$ be models
      of $T$. Let $p$ be a type of $\M$ and $q$ be a type of $\mathcal{N}$ containing $p$.
      We say that $q$ is a \emph{forking} extension if the Morley rank of $q$ is smaller
      than that of $p$, and \emph{non--forking} if the Morley rank is the same.
      Terminologically and notationally, we say that a set $A$ is \emph{independent} from
      $B$ over $C$ if for every finite tuple $a$ in $A$, the type $\tp(a/BC)$ does not fork
      over $C$, and we write $\indep{A}{C}{B}$. For a particular tuple $a$ in $A$,
      we write $\indep{a}{C}{B}$.
      
    \subsubsection{Quantifier elimination}
      Let  $T$ be a $\mL$--theory and let  $x$  be a multi--variable. We say that
      formulae $\phi(x)$ and $\psi(x)$ are \emph{equivalent modulo $T$} if
      \[T\vdash (\forall x)(\phi(x)\leftrightarrow\psi(x)).\]
      \begin{definition}[Quantifier elimination]
      	A theory $T$  has \emph{quantifier elimination} if every $\mL$--formula in the
      	theory is equivalent modulo $T$ to a quantifier--free formula. A theory
      	$T$ has \emph{relative quantifier elimination} with respect to  a  class of
      	formulae $\mathcal{F}$ if  every $\mL$--formula in the theory is equivalent
      	modulo $T$ to a  quantifier free formula which allows elements of $\mathcal{F}$
      	as predicates.
      \end{definition}
      For example, quantifier elimination with respect to
      $\exists$--formulae means that a given first order formula will be equivalent modulo $T$ to
      a Boolean combination of $\exists$--formulae. Here, $\exists$--formulae 
      are formulae consisting of a quantifier-free part under the scope of a single block of existential quantifiers.
      
      One can always expand a given  language $\mL$ by adding predicates which  were
      definable in $\mL$, and this  expansion does  not  affect the
      absolute model theory of
      $\mL$--structures since it only depends on the class of  definable sets.
      Note that any theory can thus be embedded into  one which has quantifier
      elimination at  the  expense of enlarging the language.

      A useful criterion for proving that a theory has
      quantifier elimination relies on the following property.
      \begin{definition}[Back-and-forth property]
      	Let $\mM$ and $\mathcal{N}$  be $\mL$--structures, and let $I$ be a  set  of partial isomorphisms  between $\mM$ and $\mathcal{N}$. We say that $I$ has the  \emph{back--and--forth  property} if whenever $(\overline{a},\overline{b})\in I$ and
      	$c\in M$, there exists a $d\in N$ such that $(\overline{a}c,\overline{b}d)\in I$. Dually,
      	if $d\in N$, there exists a $c\in M$ such that $(\overline{a}c,\overline{b}d)\in I$.
      \end{definition}
      In the definition above,
      we view substructures as tuples. If $\overline{a}$ denotes a tuple
      of elements in $M$ and $c\in M$, then $\overline{a}c$ is the concatenation of
      $\overline{a}$ and $c$. The following result is a standard fact from model theory, which is essential for us.
      
      The following is a well-known fact. See for instance \cite{marker2006model}[Thm. 2.4.6].
      \begin{theorem}\label{thm:backandforth}
      	Let $T$ be a theory, and  let
      	$\M$ and $\mathcal{N}$ be  models of  $T$. Suppose
      	that $I$ is a family of  partial isomorphisms between $\mM$  and  $\mathcal{N}$ with the back and forth property.
      	Then for each $f\in I$, the types of  the domain and  the range  of $f$
      	are equal.
      	
      	In particular, if $I$ is the family of finite partial isomorphisms between subsets
      	of $\omega$--saturated models of $T$ and if $I$ has the back--and--forth property,
      	then the quantifier--free type of a tuple in an arbitrary
      	model of
      	$T$ determines its type. Therefore, $T$ has quantifier elimination.
      	The converse also holds.
      \end{theorem}
      
In our context, quantifier elimination will appear in two related forms.  We will enlarge the language of graph theory and introduce auxiliary structures whose theories have absolute quantifier elimination in suitable languages.  The curve graph itself cannot have absolute quantifier elimination in the pure graph language, as noted above.  However, by bi-interpreting the curve graph with these auxiliary structures, we will obtain a relative form of quantifier elimination for the curve graph, with respect to natural classes of formulae.

\section{Framework}\label{sec:framework}
  Adopting the notation from the previous section, let $\Sigma$ be a surface of genus $g$ with $b \geq 0$ punctures and
  $\chi(\Sigma) <0$,  and $\mathcal C(\Sigma)$
  (or $\mathcal{C}$ for short) its curve graph.

  \begin{definition}[Domains, filling]
  	   Let $A\subseteq\C_0$. The 
    (possibly disconnected)
  	subsurface $\Fill(A) \subseteq \Sigma$ \emph{filled} by
  	$A$ is by definition the isotopy class of the smallest essential subsurface of $\Sigma$ containing these curves.
  	A \emph{domain} $D$ is a subset of vertices $D \subseteq \C_0(\Sigma)$
  	with the property that if
  	$A\subset D$, and if
  	$\gamma\in\C_0$ is a curve that is isotopic (perhaps peripherally) into the (possibly disconnected)
  	subsurface $\Fill(A) \subseteq \Sigma$, then $\gamma\in D$. Note that when $\Fill(D)\subsetneq \Sigma$, the subsurface $\Fill(D)$ may have boundary components,
  	notwithstanding our standing assumption that $\Sigma$ itself may have punctures but no boundary components.
  	We define $\mathcal{D}_0=\mathcal{D}_0(\Sigma)$ to be the collection of all domains $D \subset
  	C_0(\Sigma)$.
  \end{definition}
  
  The reason for the naught--subscript in the definition of the set of all domains will become clear
  in the sequel. Note that a domain can be thought of as the
  vertices of a curve graph of a subsurface of $\Sigma$; see the
  definition of an associated domain below.
  
  \begin{example}
  	Here are some examples of domains:
  	\begin{enumerate}
  		\item The full set of vertices of the curve graph $\C_0(\Sigma)$ is a domain.
  		\item For every $\alpha \in \C_0(\Sigma)$, we have $\{\alpha\} \in \D_0$ (we write $\alpha \in \D_0$ for short).
  		\item For every pair of pants $P$ with essential boundary curves
  		$\alpha, \beta, \gamma \in C_0(\Sigma)$ then $D:=\{\alpha, \beta, \gamma\}
  		\subset C_0(\Sigma)$ is a domain in $D$. Note that the surface $\Fill(D)$
  		is given by three pairwise disjoint annuli of core curves $\alpha, \beta, \gamma$,
  		and not by $P$.
        \item If $D$ is a domain and $g$ is a mapping class of $\Sigma$, then $g(D)$ is naturally a domain as well.
  	\end{enumerate}
  	
  \end{example}
  
  \begin{remark}
  In~\cite{brendle_margalit}, much technical complication is caused by the presence of certain configurations of subsurfaces (or equivalently
  domains for us) called~\emph{holes}
  and~\emph{corks}. We remark that we do not rule out such configurations, and they are not relevant to the particular results we obtain
  and methods that we use.
  \end{remark}
  
  \begin{definition}[Domain associated to a subsurface]
  	For an essential subsurface $\Sigma' \subset \Sigma$
  	its \emph{associated domain}  $D_\Sigma' \subset \C_0(\Sigma)$ is defined to be
  	\[D_{\Sigma'}:=\C_0(\Sigma') \cup \partial \Sigma'.\]
  \end{definition}
  
  \begin{definition}[Connected domains] We say that a domain $D\in\D_0$ is
  	\emph{connected} if $D=D_{\Sigma'}$ for some connected subsurface
  	$\Sigma' \subseteq \Sigma$ which is not a pair of pants.
  \end{definition}
  Note that if $D$ is connected $D= D_{\Sigma'}$ and $\Fill(D) = \Sigma'$.
  \begin{definition}[Complexity of a domain]
  	For a connected $D\in\D_0$, we define the \emph{complexity} of $D$ as
  	\[k(D) := 3 g' + b' -2,\] where $g'$ and $b'$ are respectively the
  	genus and the number of boundary components and punctures of $\Sigma'=\Fill(D)$.
  \end{definition}

  \begin{lemma}\label{complexity length}
  	Let $\Sigma$ be an orientable surface of genus $g$ and with $b$ boundary components
  	and punctures with $3g-3+b\geq 1$. Suppose that
  	\[\emptyset\neq D_{0}\subsetneq D_{1}\subsetneq D_{2}\subsetneq\cdots\subsetneq
  	D_{k}=\C_0(\Sigma)\] is a
  	chain of connected domains. Then $k\leq 3g-3+b$, and any maximal chain has length
  	exactly $3g-3+b$.
  \end{lemma}
  \begin{proof}
  	Consider a proper inclusion of domains $D_i\subsetneq D_{i+1}$
  	for $i>0$, with underlying
  	topological realizations $\Sigma_i\subsetneq\Sigma_{i+1}$ obtained via
  	application of $\Fill$.
  	Suppose that there is an essential
  	arc $\alpha\subset\Sigma_{i+1}\setminus\Sigma_i$ which meets boundary curves
  	$B_1$ and $B_2$ of $\Sigma_i$. Note that it is possible that $B_1=B_2$. We have
  	that a tubular neighborhood of $B_1\cup\alpha\cup B_2$ is homeomorphic to an
  	essential pair of pants (and possibly to a torus with one boundary component
  	if $\Sigma_i$ is itself an annulus. In this case, we may view $B_1\cup\alpha\cup B_2$
  	as obtained from an annulus by attaching two cuffs of a pair of pants to the two
  	boundary components).
  	If no such arc $\alpha$ exists, then $\Sigma_i$ has two
  	boundary components $B_1$ and $B_2$ which are isotopic to each other in $\Sigma_{i+1}$.
  	
  	It follows that there exists an ascending chain of
  	essential subsurfaces
  	\[\emptyset\neq\Sigma_0\subsetneq\Sigma_1\subsetneq\cdots\subsetneq \Sigma_m=\Sigma\] and
  	a strictly increasing function \[f\colon\{0,\ldots,n\}\to \{0,\ldots,m\}\]
  	satisfying the following conditions:
  	\begin{itemize}
  		\item We have $\Sigma_{i+1}\setminus\Sigma_i$ is either
  		an essential pair of pants or an annulus
  		for all $i>0$.
  		\item The surface $\Sigma_{f(i)}$ is the surface filled by the domain
  		$D_i$.
  	\end{itemize}
  	Thus, we may assume that for $i\geq 1$, the surface
  	$\Sigma_{i+1}$ is obtained from $\Sigma_i$ by attaching a pair of pants or by
  	gluing together two boundary components of $\Sigma_i$. It
  	suffices to show that $m=3g-2+b$.
  	
  	Recall that gluing two surfaces
  	$S_1$ and $S_2$ along a single boundary curve results in a surface $S_3$ such that
  	\[\chi(S_3)=\chi(S_1\cup S_2).\]
  	
  	The Euler characteristic of $\Sigma$ is $2-2g-b$.
  	The Euler characteristic of a pair of pants
  	is $-1$, so that $\Sigma_m$ is built from $2g+b-2$ pairs of pants. We thus have
  	that $m$ can be estimated from the number of gluings that need to be made
  	between boundary curves of these $2g+b-2$ pairs of pants in order to reassemble
  	$\Sigma_m$.
  	
  	The total number of boundary curves among all  the pairs of pants
  	is $6g+3b-6$. A total of $b$ of these curves
  	correspond to the boundary components of $\Sigma$ and are not involved in any
  	gluings. The remaining $6g+2b-6$ are glued in pairs, which results in exactly
  	$3g+b-3$ gluings.
  	
  	Now, since $m$ is assumed to be maximal, we must have that
  	$\Sigma_0$ is an annulus. Applying the maximality of $m$ again and  the assumption
  	that  the corresponding domains are connected, we have that $\Sigma_1$ is either
  	a torus with one boundary component or a sphere with four boundary components.
  	In either case, $\Sigma_1$ is the surface obtained after the first gluing, so
  	that $m=3g+b-3$, as claimed.
  \end{proof}
  
  \begin{definition}[Domain spanned by $D_1$ and $D_2$]
  	Given $D_{1},D_{2}\in\D_0$, there exists a unique smallest $D\in\D_0$ such that
  	$D_{1},D_{2}\subseteq D$. We will denote it by $D_{1}\vee D_{2}$ and say that $D_{1}\vee
  	D_{2}$ is the \emph{domain spanned by} $D_1$ and $D_2$, or the \emph{join} of these domains.
  	Given a subset $C\subseteq\C_{0}(\Sigma)$, we write $D(C)=\bigvee_{\gamma\in C}\{\gamma\}$, i.e. the smallest domain containing $C$. 
  \end{definition}
  
  \begin{example}[Domains spanned by connected domains]
  	Let $D_1$ and $D_2$ be two
  	connected domains with respectively associated surfaces $\Sigma_1, \Sigma_2 \subset \Sigma$.
    The following conclusions hold:
  	\begin{enumerate}
  		\item if $\Sigma_1 \subseteq \Sigma_2$ then $D_1 \vee D_2 = D_2$;
  		\item if $\Sigma_1 \cap \Sigma_2 = \varnothing $ then $D_1 \vee D_2 = D_1 \cup D_2$;
  		\item if $\Sigma_1 \cap \Sigma_2 \neq \varnothing$ then $D_1 \vee D_2 = D_{\Sigma_1 \cup
  		\Sigma_2}$; here, the notation $\Sigma_1 \cup
  		\Sigma_2$ means the smallest essential subsurface of $\Sigma$ containing both $\Sigma_1$ and $\Sigma_2$.
  	\end{enumerate}
    Note that $\Sigma_{1},\Sigma_{2}$ are only well-defined up to isotopy; thus, disjointness of or nonempty intersection between $\Sigma_{1}$ and $\Sigma_{2}$ means that there are representatives which are disjoint, or that all representatives intersect, respectively.
  \end{example}

  \subsection{Orthogonality}\label{ss:ortho} Here we introduce one of several notions of orthogonality that we will use in this paper.
    \begin{definition}[Orthogonal domains]\label{def:orthogonal-domains}
    	Given a domain $D\subseteq\mathcal{C}_0$ and $\alpha\in\C_0$, we say that $\alpha$ is \emph{orthogonal to } $D$ and we denote
    	$$\alpha\perp D $$  if and only if   $i(\alpha,\beta)=0$  for all  $\beta\in D$.
    	Here, $i(\alpha,\beta)$ denotes the {geometric intersection  number}
    	of $\alpha$ and $\beta$, taken to be the minimal number of intersections over
    	all  representatives of $\alpha$  and $\beta$ in their respective
    	isotopy classes. Similarly, we say that a domain $D$ is \emph{orthogonal } to a domain $D'$
    	and we write $$D\perp D'$$ if $\alpha\perp D'$ for all $\alpha\in D$. Notice that these notions are symmetric.
    \end{definition}
    We remark that if $D=\emptyset$ then every $\alpha\in\mC_0$ is orthogonal to $D$. Furthermore, if $D$ is a finite union of pairwise disjoint simple closed curves, then for every $\alpha \in D$ we have $\alpha \perp D$.
    \begin{example}[Orthogonality for connected domains]
    	Let $D$ be a connected domain filling an associated surface $\Sigma'$. We have the following:
    	\begin{enumerate}
    		\item For every curve $\alpha \in \partial \Sigma'$ which is also essential in $\Sigma$,
    		we have $\alpha \perp D$.
    		\item If $\Sigma \setminus \Sigma'$ is essential then every
    		essential simple closed curve $\alpha$ on $\Sigma \setminus \Sigma'$ is
    		orthogonal to  the domain $D$.
    	\end{enumerate}
    \end{example}

    \begin{definition}[Orthogonal complements] For a
    	domain $D \in \D_0$, we define its \emph{orthogonal complement}
    	as \[D^{\perp}:=\{\alpha\in\C_0(\Sigma)\,|\,\alpha\perp D\}.\]
    	Equivalently, $D^{\perp}$ is the set of simple closed curves which are
    	(possibly peripherally) homotopic into $\Sigma\setminus \Fill(D)$.  For $\gamma \in \C_0$, we will often write simply $\gamma^{\perp}$,
    	instead of $\{\gamma\}^\perp$, and for a set $C\subseteq\C_{0}(\Sigma)$ we write $C^{\perp}$ instead of $D(C)^{\perp}$.
    \end{definition}
    
    \begin{example} The following hold:
    	\begin{enumerate}
    		\item if $D = D_\Sigma'$ is connected then
    		$\partial \Sigma' \subset D^\perp$ and $D^\perp = D_{\Sigma \setminus \Fill(D)}$;
    		\item we have $D = D^\perp$ if and only if $D$ is a maximal union of pairwise disjoint simple closed curves.
    	\end{enumerate}
    \end{example}

\begin{table}[hb!]
\centering
\footnotesize
\renewcommand{\arraystretch}{1.25}
\setlength{\tabcolsep}{4pt}

\begin{tabular}{|p{3.7cm}|p{5.3cm}|p{5.3cm}|}
\hline
\textbf{Notation} & \textbf{Meaning} & \textbf{Example} \\
\hline

$\alpha \perp D$,
where $\alpha$ is a curve and $D$ is a domain
&
The curve $\alpha$ has zero geometric intersection with every curve in $D$.
&
If $\alpha$ is an essential boundary component of $\textrm{Fill}(D)$, then $\alpha \perp D$. \\
\hline

$D \perp E$,
where $D,E$ are domains
&
Every curve in $D$ is orthogonal to every curve in $E$.
&
If $D$ and $E$ fill disjoint subsurfaces, then $D \perp E$. \\
\hline

$\operatorname{Fill}(D)$,
where $D$ is a domain
&
The subsurface filled by $D$: the smallest essential subsurface of $\Sigma$ containing representatives of all curves in $D$.
&
If $D=D_Y$ comes from a subsurface $Y$, then $\operatorname{Fill}(D)$ is $Y$, up to removing inessential components. \\
\hline

$D^\perp$,
where $D$ is a domain
&
The orthogonal complement of $D$: the set of curves orthogonal to every curve in $D$ (eq. curves homotopic into the complement of $\operatorname{Fill}(D)$).
&
Essential boundary curves of $\operatorname{Fill}(D)$ belong to $D^\perp$. Thus boundary curves may lie both in $D$ and in $D^\perp$. \\
\hline

$\partial D$,
where $D$ is a domain
&
The boundary of $D$, defined by $\partial D:=D\cap D^\perp~$.
&
For a connected subsurface domain, this is the collection of essential boundary curves of the subsurface. \\
\hline

$F=D_1\vee\cdots\vee D_k$,
where $F$ is possibly disconnected
&
Orthogonal decomposition into connected pairwise orthogonal components.
&
The components $D_i$ are connected domain letters. \\
\hline

$\mathcal D_\vee$,
where $\mathcal D$ is a collection of connected domains
&
The collection $D_1\vee\cdots\vee D_k$ of pairwise orthogonal domains $D_i\in\mathcal D$.
&
These joins package independent domain relations; the word letters remain the connected domains in $\mathcal D$. \\
\hline

Transposition ("commuting") of domain letters
&
If \(D\perp E\) and \(D\neq E\), adjacent domain letters \((D,E)\) may be replaced by \((E,D)\).
&
Orthogonal domain letters are independent. Instead, repeated letters are handled by reduction moves.\\
\hline

\end{tabular}

\caption{Guide to the basic orthogonality conventions.}
\label{tab:orthogonality-guide}
\end{table}

    We will now see some fundamental properties of orthogonal complements.
    \begin{proposition}
    	The following holds:
    	\begin{enumerate}
    		\item $D^{\perp}\in\D_0$;
    		\item $(D^{\perp})^{\perp}=D$;
    		\item $(D_1 \vee D_2)^\perp = D_1^\perp \cap D_2^\perp$;
    		\item $(D_1 \cap D_2)^\perp = D_1^\perp \vee D_2^\perp$;
    		\item (Orthogonal Decomposition)
    		For every $D \in \D_0$ there exist domains
    		$D_1, \ldots, D_k \in \D_0$ which are connected and pairwise
    		orthogonal such that
    		\[D= \bigvee_{i=1}^k D_i.\]
    		The orthogonal decomposition is unique up to permutation of its domains.
    		
    	\end{enumerate}
    \end{proposition}
    \begin{proof}
    	The first four points follow easily from our definitions. We will now prove (5).
    	Let $D \in \D_0$, we have that \[\Fill(D) = \Sigma_1 \cup \ldots \cup \Sigma_k,\]  where
    	each $\Sigma_i$ is a suitable essential subsurface of $\Sigma$,
    	and where for $i\neq  j$, we  have that $\Sigma_i\cap\Sigma_j=\varnothing$ (up
    	to isotopy). We now have $D = D_1 \vee \ldots \vee D_k$ with
    	$D_i = D_{\Sigma_i}$.
    	As the subsurfaces are disjoint up to isotopy, and can only have (isotopy classes of) boundary components in common,
    	it follows that $D_i\perp D_j$ for  $i\neq j$. The uniqueness of the domains of the decomposition follows immediately.
    \end{proof}
    
    With the notion of the orthogonal complement of a domain in place, we can make a few remarks about subsurfaces filled by arcs.
    If $\Sigma$ is a surface with at least one puncture, then it is natural for us to consider \emph{essential simple arcs} on $\Sigma$,
    which are simply isotopy classes of essential, properly embedded copies of $\mathbb{R}$ inside of $\Sigma$. If
    $\{\alpha_1,\ldots,\alpha_n\}$ is a collection of essential simple arcs on $\Sigma$, then the notion of
    $\Fill(\{\alpha_1,\ldots,\alpha_n\})$ continues to make sense, though some modification is needed to the definitions before one can talk
    directly about the domain associated to a collection of arcs. Since it is not particularly relevant to our purposes, we will avoid making such a modification; however, the orthogonal complement
    $D_{\Fill(\{\alpha_1,\ldots,\alpha_n\})}^{\perp}$ is naturally a domain, and consists of precisely those isotopy classes of
    simple closed curves on $\Sigma$ with zero  geometric intersection number with each of $\{\alpha_1,\ldots,\alpha_n\}$. This will be essential in our discussion of geometric graphs in the sequel.

    \begin{definition}[Boundary of a domain]
    	Given $D\in\D_0$, we define $\partial D:=D\cap D^{\perp}$ the \emph{boundary} of $D$.
    	This coincides with the collection of curves
    	$\alpha\in\mathcal{C}$ parallel to a boundary of $\Fill(D)$. If $\alpha\in\partial D$,
    	then we say that $\alpha$ is \emph{peripheral} in $D$.
    \end{definition}
    
    In contrast with orthogonality, we also introduce a notion of transversality.
    
    \begin{definition}[Transverse domains]\label{def:transverse}
    	If $D$ and $D'$ are both non-orthogonal and incomparable
    	with respect to  the inclusion relation,
    	we say that they are \emph{transverse} and we write $D \pitchfork D'$
    \end{definition}
    
    \begin{example}The following hold:
    	\begin{enumerate}
    		\item If $D$ is a union of pairwise disjoint simple closed curves, then we have $\partial D = D$. If $D$ is also maximal with respect to inclusion
            then $\partial D = D^\perp$.
    		\item If $D=D_{\Sigma'}$ is a connected domain then $D^\perp = D_{\Sigma \setminus \Sigma'}$, where here 
		$\Sigma \setminus \Sigma'$ is the largest (possibly disconnected) essential subsurface of the complement of $\Sigma'$.
    		\item Let $\Sigma$ be a once-holed torus and $\Sigma' \supset \Sigma$ be a once-holed surface with genus 2. Then $D_\Sigma \subset D_\Sigma'$, with both domains non-orthogonal and non-transverse.
    	\end{enumerate}
    \end{example}

  \subsection{The augmented Cayley graphs $\mathcal M_0^G(\Sigma)$ and $\mathcal M_{\mathcal D}^G(\Sigma)$ }
    In this section we will define the auxiliary $\mathcal L$--structures we will work with in the rest of the paper, namely
    $\M_0^G(\Sigma)$ and $\M_{\mathcal D}^G(\Sigma)$. In the sequel, we let $G$ be a
    finite index subgroup of the extended mapping class group $\Mod^\pm(\Sigma)$, which is allowed to be arbitrary unless otherwise specified.
    
    \begin{definition}[$D$-relation]
    	Given mapping classes $\sigma,\tau\in G$ and a connected domain $D\in\mathcal{D}_0$,
    	we say that $\sigma$ and $\tau$ are \emph{$D$-related} if
    	\[\sigma^{-1}\circ \tau\in \mathrm{Stab}_G({D^{\perp}}).\] That is,
    	$\sigma^{-1}\circ \tau$ preserves the isotopy class of each curve in $D^{\perp}$. We write $R_D(\sigma,\tau)$ for the $D$--relation.
    	If $D=D_1\vee D_2\vee\cdots\vee D_n$ is a disconnected domain with exactly $n$ mutually orthogonal components,
    	then we set \[R_D:=R_{D_1}\circ R_{D_2}\circ\cdots\circ R_{D_n}.\]
    \end{definition}
    
    For $D$ connected, write \[G[D]:=\mathrm{Stab}_G({D^{\perp}}),\]
    where here again the stabilizer is pointwise. In particular, if $G=\Mod^{\pm}(\Sigma)$ and $D$ is connected, we may
    write $D=D_{\Sigma'}$ for a subsurface $\Sigma'$ of $\Sigma$. In this case, $\mathrm{Stab}_G(D^\perp)$ is isomorphic to the subgroup
    of $\Mod^\pm(\Sigma')$ preserving each boundary component and fixing the isotopy class of every simple closed curve
    outside of $\Sigma'$.
    
    Observe that $g\in G[D]$ does not necessarily mean that $g$ restricts to the identity on the surface filled by
    $D^{\perp}$, since for low complexity surfaces, the hyperelliptic involution preserves isotopy classes of simple closed curves. A concrete
    example to keep in mind is when $\Sigma$ is the four-times punctured sphere, and where $D$ is an annular region corresponding
    to a simple closed curve $\alpha$. Then, $G[D]$ is allowed to contain involutions of $\Sigma$ that fix $\alpha$ and permute
    pairs of punctures of $\Sigma$.

    Note that for a general $D$, the definition of $R_D$ makes sense since the composition of the constituent relations is commutative.
    If \[D=\bigvee_{i=1}^n D_i\] is a join of $n$ pairwise orthogonal connected domains, we will write \[G[D]=
    \prod_{i=1}^n\Stab_G(D_i^{\perp}).\] Again, this product makes sense since the constituent subgroups commute with each
    other. With this notation, we have $$\sigma,\tau\in G \mbox{ satisfy } R_{D}(\sigma,\tau) \mbox{ if and only }
    \sigma^{-1}\tau\in G[D]~.$$

    \begin{example}($D$-relation for maximal finite domains) For this example, suppose that $\Sigma$ is not a four-times punctured
    	sphere or a once-punctured torus, and that $G$ is the full
        extended mapping class group of $\Sigma$.
    	When $D$ is maximal and finite (i.e.~a pants decomposition of $\Sigma$), we have
    	$D= D^\perp = \{ \gamma_1, \ldots, \gamma_k \}$, where $\gamma_i$ are pairwise disjoint simple closed curves.  The condition $\sigma$ and $\tau$ are $D$-related is equivalent to
    	$$\sigma^{-1}\circ \tau \in \mathrm{Stab}(\{\gamma_1, \ldots, \gamma_k
	\})~.$$ Here $\sigma^{-1}\circ \tau$ is a product of Dehn twists around the curves $\{\gamma_1,\ldots,\gamma_k\}$.
	Equivalently, $\sigma$ and $\tau$ are in the same left coset in the group generated by Dehn twists around these
	curves. Observe that if $\Sigma$ is a four-times punctured sphere or a once-punctured torus and $D= \{ \gamma \} $, then
    	\[\Stab(D)=\Stab(D^{\perp})=\Stab(\gamma),\] and this group is only virtually generated by the Dehn twist about $\gamma$.
    \end{example}

    \begin{example}[A domain relation as agreement off a subsurface]
    Let $Y\subset \Sigma$ be a connected essential subsurface which is not a pair of pants, and let
    $D=D_Y$. Then $D^{\perp}$ is represented by the essential curves which can be isotoped into
    $\Sigma\setminus Y$, together with the essential boundary curves of $Y$. Thus, for
    $\sigma,\tau\in G$, the relation $R_D(\sigma,\tau)$ says that $\sigma^{-1}\tau$ fixes every curve
    outside $Y$ and fixes the boundary of $Y$. In other words, up to the low-complexity kernel and
    boundary twists, $\sigma$ and $\tau$ agree on the complement of $Y$.

    In the annular case, if $D=\{\alpha\}$ is the domain corresponding to a curve $\alpha$, then
    $R_D(\sigma,\tau)$ says that $\sigma^{-1}\tau$ lies in the stabilizer of $\alpha$. In a pure subgroup,
    this is the same as saying that the difference between $\sigma$ and $\tau$ is supported on the
    annulus about $\alpha$, together with complementary components fixed pointwise on curves.
    Thus an $R_D$--step should be thought of as a move supported in the domain $D$, not as an edge in
    the curve graph.
    \end{example}

\begin{example}[A running example: a separating curve on a genus-three surface]
    \label{ex:running-separating-curve}
    Let $\Sigma=\Sigma_{3,0}$ be a closed surface of genus three, and let
    $\alpha$ be a separating curve which cuts $\Sigma$ into a once-holed torus
    $Y\cong \Sigma_{1,1}$ and a genus-two surface with one boundary component
    $Z\cong \Sigma_{2,1}$. We will return to this example several times. Let
    \[
        A=\{\alpha\},\qquad D=D_Y,\qquad E=D_Z.
    \]
    Thus $A$ is the annular domain about $\alpha$, while $D$ and $E$ are the
    domains associated to the two complementary subsurfaces. In this example
    \[
        D^\perp=E,\qquad E^\perp=D,\qquad D\cap E=A,\qquad D\vee E=\mathcal C_0(\Sigma).
    \]
    The annular domain $A$ lies in both $\partial D$ and $\partial E$.

    The relation $R_D(\sigma,\tau)$ says that $\sigma^{-1}\tau$ fixes every curve
    which is supported in $Z$, and fixes the boundary curve $\alpha$. Thus
    $\sigma$ and $\tau$ differ by a mapping class supported on the $Y$ side,
    possibly including powers of the boundary twist about $\alpha$. Similarly,
    $R_E(\sigma,\tau)$ says that $\sigma$ and $\tau$ differ on the $Z$ side. The
    annular relation $R_A$ records the possibility of changing by powers of the
    Dehn twist $T_\alpha$.

    Thus the three relations $R_A$, $R_D$, and $R_E$ should be thought of as
    moves supported respectively on the annulus about $\alpha$, on the $Y$ side,
    and on the $Z$ side. This example is a useful model for the rest of the paper:
    $A$ is lower complexity than both $D$ and $E$, but it occurs as boundary data
    for both of them.
    \end{example}

    \subsubsection{The language $\mathcal L(\mathcal D_0)$ and the augmented Cayley graph $\M_0^G(\Sigma)$ }
      We define $\mathcal{L}_{0}(\D_0)$, writing $\mL_0$ for short, as the first order language containing a
      binary relation $R_{D}$ for every domain $D\in \mathcal{D}_0$ and a binary relation symbol
      $R_{g}$ for every $g\in G$:
      $$\mathcal{L}_0(\D_0):=(\{R_{D}\}_{D\in \mathcal{D}_0}, \{R_{g}\}_{g\in G} )~.$$
      Given $x,y \in G$, the binary relations $R_D$ and $R_g$ are defined as follows:
      \begin{itemize}
      	\item $R_{g}(x,y)$ holds if and only if $y=xg$;
      	\item $R_{D}(x,y)$ holds for $D\in\D_0$ if and only if $x$ and $y$ are $D$-related as mapping classes.
      \end{itemize}
      Note that
      $\mathcal C = \mC_0(\Sigma)$ is also an
      element of  $\D_0$, and we have:
      \begin{itemize}
      	\item $R_{\mC}(1,g)$ holds for all $g\in G$.
      \end{itemize}
      
      Let $\wo_0$ be the collection of finite words in the alphabet $\alp_0=\D_0\cup
      G$. Consider the language $\mathcal{L}(\D_0)$
      obtained from $\mathcal{L}_{0}(\D_0)$ by
      adding a binary
      relation symbol $R_{w}$ for each tuple $w= (\delta_{1},\delta_{2},\dots,\delta_{k})$ of elements in $ \wo_0$, that is
      $$\mathcal{L}(\D_0):=(\{R_{D}\}_{D\in \mathcal{D}_0}, \{R_{g}\}_{g\in G}, \{R_w\}_{w \in \wo_0} )~.$$

      Given $x, y \in G$, the relation $R_w$ is defined as follows:
      \begin{itemize}
      	\item $R_{w}(x,y)$ holds iff $R_{\delta_{1}}\circ R_{\delta_{2}} \circ \dots \circ R_{\delta_{k}} (x,y)$ holds.
      \end{itemize}
      In other words, we have:
      \begin{align*}
      	R_{w}(x,y)\leftrightarrow\exists z_{0}\exists z_{1}\dots\exists z_{k}
      	\left((z_{0}=x)\wedge (z_{k}=y)\wedge\bigwedge_{i=1}^{k}R_{\delta_{i}}(z_{i-1},z_{i})\right).
      \end{align*}

      \begin{definition}(The augmented Cayley graph  $\M_0^G(\Sigma)$)
      	We define the \emph{augmented Cayley graph of $G$ } to be the $\mathcal L(\mathcal D_0)$-structure  $\mathcal M_0^G(\Sigma)$
      	with universe $M^G$, i.e.~the elements of $G$ considered as a set, and relations in $\mathcal L(\mathcal D_0)$ as defined above.
      \end{definition}
      
    \subsubsection{The language $\mathcal L(\mathcal D)$ and the augmented Cayley graph $\M_\D^G(\Sigma)$ } The notion of a downward closed collection of connected domains $\mathcal D \subseteq \D_0$ will be key for the following.
      \begin{definition}[Downward closed, $G$--invariant]
      	From now on, unless explicitly stated otherwise, a collection denoted by $\D$ consists only of connected domains.  A collection of connected domains $\D \subseteq \D_0$ is \emph{downward closed} if for every
      	connected domain $D\in\D$ such that $D \neq \mathcal C_0(\Sigma)$ and for
      	every connected domain $E\subsetneq D$, we have $E\in\D$ as well. If $G\leq
        \Mod^{\pm}(\Sigma)$, a collection of connected domains
        $\D\subseteq\D_0$ is called \textbf{$G$--invariant} if for all $g\in G$,  $D\in\D$, we have $g(D)\in\D$ as well.
      \end{definition}
      
      We will frequently need to mention a possibly disconnected domain all of whose connected components lie in $\D$.  We denote by
      \[
      \Dvee=\left\{D_1\vee\cdots\vee D_k\mid k\geq 0,\; D_i\in\D,\; D_i\perp D_j\text{ for }i\neq j\right\}
      \]
      the finite orthogonal closure of $\D$.  Thus elements of $\D$ are the connected domain letters used in words, while elements of $\Dvee$ are used only as abbreviations for the corresponding componentwise relations.  If $F=D_1\vee\cdots\vee D_k\in\Dvee$, then
      \[
          R_F:=R_{D_1}\circ\cdots\circ R_{D_k},\qquad
          G[F]:=G[D_1]\cdots G[D_k],
      \]
      and the product and composition are independent of the ordering of the components.  Similarly, $R_F^*$ means the type defined by a strict sequence through the connected components of $F$, in any order.  These conventions do not introduce disconnected domain letters into words.
      
      Oftentimes, we will use downward closed collections of connected domains which contain $\mathcal{C}_0(\Sigma)$, but which nevertheless
      do not consist of all connected domains; for example, $\Sigma$ could be a closed surface of genus two, and $\D$ could consist of all connected domains
      of complexity at most that of a torus with one boundary component, together with $\mathcal C_0(\Sigma)$ itself.

      Let $\D\subseteq\D_0$ be a $G$--invariant and downward closed collection of connected domains. Let $\wo$ be the collection of finite words in the restricted alphabet \[\alp=\alp_{\D}:=\D\cup G \subset \alp_0.\]
      The \emph{language adapted to $\D$} is the sublanguage $\mathcal{L}(\D)\subseteq\mathcal{L}_0(\mathcal D_0)$ defined as
      $$\mathcal{L}(\D):=(\{R_{D}\}_{D\in \mathcal{D}}, \{R_{g}\}_{g\in G}, \{R_w\}_{w \in \wo} )~.$$
      In the language $\mathcal{L}(\D)$, we consider only relations $R_w$ where $w$ is a finite word with all letters in the restricted alphabet $\alp$, i.e.~no disconnected domain occurs as a letter of $w$.
      We will often write $\mathcal{L}$ instead of $\mathcal{L}(\D)$ for short, when no confusion can arise.
      \begin{definition}(The $\mathcal D$-augmented Cayley graph $\M_\D^G(\Sigma)$)
      	\label{def-D}  The  \emph{$\mathcal{D}$-augmented Cayley graph of $G$} is the $\mathcal{L}(\D)$-structure $\M_\D^G(\Sigma)$ whose universe is $M^G$ (i.e.~the set $G$) and whose relations in $\mathcal{L}(\D)$ are defined as above.
      \end{definition}

      \begin{remark}\label{rmk:left-v-right}
      	The group $G$ can play several roles. It is
      	identified with the universe of $\M_\D^G(\Sigma)$, and its elements appear as relational symbols. In the latter
      	of these roles, it is natural for the group $G$ to act on words on the right. In the
      	sequel, we will also need to consider $G$ as a group of automorphisms of $\mC(\Sigma)$.
      	In this case, it is convenient to view curves as left cosets of their stabilizer, and
      	then the group $G$ acts naturally on the left.
      \end{remark}

\section{Interpretations and the Ivanov Metaconjecture}\label{sec:interpretations}
  
  In this section, we will prove some general results about interpretations and
  bi--interpretations of various  structures with $\M^G_0(\Sigma)$ and $\M^G_\mathcal{D}(\Sigma)$, where here $G\leq\Mod(\Sigma)$
  is a finite index subgroup and where $\D$ is a downward closed
  and $G$--invariant collection of connected domains; cf.~Definition~\ref{def-D} above.
  Throughout this section, we will make forward reference to technical results proved later in the paper. We have elected
  to bring this section forward because it can be read as a modular piece that addresses Problem~\ref{p:ivanov} and Question
  ~\ref{q:general}
  from the introduction, and wherein the technical results are treated as black boxes.
  In particular, we  will  note
  consequences concerning quantifier elimination and $\omega$--stability,  though
  we will  relegate proofs of those  properties to  later  sections.

  \subsection{The curve graph $\mathcal C(\Sigma)$ is interpreted by $\M_\D^G(\Sigma)$ for $G \leq \Mod^\pm(\Sigma)$ of finite index }
    In this subsection we will prove that many complexes associated to surfaces are interpretable in $\M_{\mathcal D}^G(\Sigma)$,
    as well as in the curve graph $\mathcal C(\Sigma)$.
    
    \subsubsection{Definition and examples of $(G, \mathcal D)$-structures }
      We recall that a structure is called \emph{relational} if its underlying language has no functions and no constants.
        Recall that if $X$ is a set equipped with an action by a group $G$, and if $Y\subset X$ is $G$--invariant, we say that the
      quotient by the action of
      $G$ on $Y$ is \emph{finite} if the quotient of $Y$ by $G$ is a finite set. We say that the quotient
      is \emph{cofinite} if the quotient of $X\setminus Y$
      by $G$ is a finite set. As is standard, we use the notation $A^{<\omega}$ to denote the collection of all finite tuples with entries in a given set $A$.
      
      \begin{definition}[$(G,\D)$-structure] Let $G< \Mod^\pm(\Sigma)$ be a finite index subgroup. Let
      	$\mathcal{D} \subset \D_0$ be a $G$-invariant and downward closed collection of connected domains which contains $\C_0(\Sigma)$.
      	A \emph{$(G,\D)$-structure} $\mathcal{B}$ over a language $\mathcal{L}$ is a relational structure with
      	the following properties:
      	\begin{enumerate}
      		\item  The universe $B$ of $\mathcal{B}$ is a set equipped with an action of $G$, and $B$ consists of
      		finitely many $G$--orbits.
      		\item For each symbol $\mathfrak R\in\mathcal{L}$ of arity $k$,
      		the set {$\mathfrak R_\mathcal{B}\subseteq B^{k}$} is $G$--invariant, and the quotient of this set by $G$
      		is either finite or cofinite.
      		\item For any $b\in B^{<\omega}$ there is a
      		(necessarily unique) $D_{b}\in\Dvee$ such that $\Stab_{G}(b)$ contains $G[D_b]$ with finite index.
      	\end{enumerate}
      	When $G=\Mod^\pm(\Sigma)$ we sometimes say that $\mathcal{B}$ is a $\mathcal D$-\emph{structure}.
      \end{definition}
      For the most part, $\mathcal{B}$ will be the result of some uniform (or almost uniform) construction that assigns a structure in a fixed language to any surface $\Sigma$.  
    
      In the definition of a $(G, \mathcal D)$-structure, it is convenient to imagine that each $b\in B$ is a union of curves and arcs on $\Sigma$ and that $D_b\in\Dvee$ is the orthogonal of the domain associated to the subsurface filled by $b$, that is
      $D_b := D_{\mathrm{Fill}(b)}^{\perp}$, decomposed into its connected components. It follows from the definition of a $(G, \mathcal D)$-structure that if
      $\mathcal B(\Sigma)$ is a $(G, \mathcal D)$-structure and if $H\leq G$ has finite
      index then $\mathcal B(\Sigma)$ is a $(H, \mathcal D)$-structure as well.

      It is straightforward to see that $\mathcal M_\mathcal D^G(\Sigma)$ is a $(G, \mathcal D)$-structure.

    \subsubsection{The motivating examples: geometric graphs}
      The first and motivating example for the previous definition is the one of graphs commonly associated to $\Mod^\pm(\Sigma)$ such as the curve complex, the pants graph, and the arc complex. We propose to study them all under the following unifying notion of
      a \emph{geometric graph}.
      \begin{definition}[Geometric graph]
      	Let $X(\Sigma)$ be a graph with sets of vertices $V(X(\Sigma))$ and
      	sets of edges $E(X(\Sigma))$. We say that $X(\Sigma)$ is a
      	\emph{geometric graph} if the following conditions are satisfied:
      	\begin{enumerate}
      		\item There exists a constant $N \geq 1$ such that each $v \in V(X(\Sigma))$ is identified with a collection of  $N$ isotopy classes of essential curves and/or arcs;
      		\item  A finite index subgroup  $G\leq \Mod^\pm(\Sigma)$ of the mapping  class group acts on $X(\Sigma)$ via its
      		action on  curves and arcs, and $V(X(\Sigma))$
      		consists of  finitely many $G$--orbits;
      
      		\item The set $E(X(\Sigma))$  consists of finitely many $G$--orbits.
      	\end{enumerate}
      \end{definition}
      
      Let $v = \{ \gamma_1, \ldots, \gamma_{m_v} \} \in V(X(\Sigma))$ be a vertex of $X(\Sigma)$. The \emph{domain associated to $v$} is the domain $$D_v := D^{\perp}_{\mathrm{Fill}(v)}$$
      associated to the (possibly disconnected) subsurface $\mathrm{Fill}(v) \subseteq \Sigma$ filled by the curves and/or arcs defining $v$.
      \begin{definition}[$\D_X$ for $X$ a geometric graph]
      	Let $X(\Sigma)$ be a geometric graph.  We set $\mathcal D_{X}$ to be the smallest downward closed collection of connected domains generated by the connected components, in the orthogonal decompositions, of the domains $D_v$ for $v\in X(\Sigma)$.
      \end{definition}
      
      With the convention above, each $D_v$ lies in $\D_X^\vee$, and it is straightforward from the definition to establish that $X(\Sigma)$ is a $\mathcal D_X$-geometric structure.

      \begin{corollary} All of the following graphs are geometric graphs:
      	\begin{enumerate}
      		\item the curve graph $\mathcal C(\Sigma)$;
      		\item the Hatcher-Thurston graph $\mathcal{HT}(\Sigma)$;
      		\item the pants graph $\mathcal P(\Sigma)$;
      		\item the marking graph $\mathcal{MG}(\Sigma)$;
      		\item the non-separating curve graph $\mathcal N(\Sigma)$;
      		\item the $k$-separating curve graph $\C_k(\Sigma)$;
      		\item the Torelli graph $\mathcal{T}(\Sigma)$;
      		\item the $k$-Schmutz  Schaller graph $\mathcal{S}_k(\Sigma)$;
      		\item the $k$-multicurve graph $\mathcal{MC}_k(\Sigma)$;
      		\item the arc graph $\mathcal A(\Sigma)$;
      		\item the $k$-multiarc graph $\mathcal{MA}_k(\Sigma)$;
      		\item the flip graph $\mathcal{F}(\Sigma)$;
      		\item the polygonalization graph $\mathcal{P}ol(\Sigma)$;
      		\item the arc-and-curve graph $\mathcal{AC}(\Sigma)$.
      	\end{enumerate}
      \end{corollary}
      
      That all these graphs are geometric follows immediately from their definitions.
      
    \subsubsection{Natural interpretations of $(G, \D)$-structures}
      In this subsection we will see that every $(G,\mathcal D)$-structure has a natural interpretation into the $\mathcal D$-augmented Cayley graph of $G$.
      \begin{lemma}[$(G, \D)$-structures are interpretable in $\mathcal{M}_\D^G(\Sigma)$]\label{interpretation in M}
      	Suppose that $G$ acts by isomorphisms on a $(G, \mathcal D)$-structure $\mathcal{B}(\Sigma)$ over a language $\mathcal{L}_B$.
      	Then there is a natural interpretation of $\mathcal B(\Sigma)$ in the $\mathcal D$-augmented Cayley graph of $G$:
      	$$\bar{\zeta} :
      	\mathcal{B}(\Sigma) \leadsto \mathcal{M}_\D^G(\Sigma).$$
      	This map sends each element $b\in B$ to the corresponding
      	imaginary element in $M^G/\Stab(b)$. Moreover,
        the formulae defining the interpretation can be taken
        to be quantifier-free.
      \end{lemma}
      
      The interpretation of $ \mathcal{B}(\Sigma)\rightsquigarrow \mathcal{M}_\D^G(\Sigma)$ will be called the \emph{natural interpretation} of
      $ \mathcal{B}(\Sigma)$ in $\mathcal{M}_\D^G(\Sigma)$.
      \begin{proof}[Proof of Lemma~\ref{interpretation in M}]
      	For $b\in B$, we define the equivalence relation $S_{b}$ on $G$ by
      	$(g,h)\in S_{b}$ if and only if $g(b)=h(b)$.
      	Notice that this  equivalence relation is definable in $\mathcal{M}^G_\D(\Sigma)$
      	without parameters, since we can express it as a finite union
      	of relations of the form $R_{\sigma,D_{b}}$, with $\sigma\in G$ and $D_b\in\Dvee$ interpreted componentwise. This results from the fact that $\Stab_G(b)$ contains
      	$G[D_b]$ with finite index. Observe that equivalence classes in $S_b$ are naturally identified with cosets of $\Stab_G(b)$;
        we will therefore abbreviate $\Stab_G(b)$ by $S_b$.
      	
      	If $\{b_{1},b_{2},\dots, b_{m}\}$ are representatives of the orbits of $B$
      	under $G$, then there is a natural bijection:
      	\[B \longleftrightarrow \coprod_{j=1}^{m}M^G/S_{b_{j}}.\]
      	The disjoint union on the right hand side can be encoded as the quotient of
      	a definable subset of $M^G$ by a definable equivalence relation. This completes the interpretation of the universe
        $\mathcal B$ of $B$ in $\M_{\D}^G(\Sigma)$.
      	
      	Now, let $\mathfrak R=\mathfrak R^{(k)}\in\mathcal{L}_B$ be a relation. For an
      	arbitrary $b \in B$, denote its $G$-orbit by $O(b)$. Without loss of generality,
      	we assume that
      	\[\mathfrak R_{\mathcal{B}}\slash G\subseteq B^{k}\slash G\] is finite,
        the cofinite case being a straightforward modification.
      	It suffices to show that for each fixed \[\bar{i}\in\{1,2,\dots, m\}^{k},\]
      	the intersection
      	\[\mathfrak R_{\mathcal{B}}\cap (O(b_{i_{1}})\times O(b_{i_{2}})
      	\times\dots \times O(b_{i_{k}}))\]
      	is the pullback of a
      	formula $\psi_{\bar{i}}^{\mathfrak R}(x_{1},\dots, x_{k})$ which is equivariant under
      	\[S_{b_{i_{1}}}\times\dots\times S_{b_{i_{k}}},\] upon restriction to
      	\[O(b_{i_{1}})\times O(b_{i_{2}})\times\dots\times O(b_{i_{k}}),\] as this will show that $\mathfrak R$ is defined by
        a parameter--free formula that depends only on the
        $S_{b_{i_{1}}}\times\dots\times S_{b_{i_{k}}}$ equivalence
        classes of the argument of the formula.
        
      	Let $\mathcal{T}_{\bar{i}}$ be a finite collection of $(k-1)$-tuples
      	$(g_1,\ldots,g_{k-1})$  of
      	mapping classes such that every tuple
      	\[(b_{1},b_{2},\dots, b_{k})\in \mathfrak R_\mathcal{B}\cap (O(b_{i_{1}})\times O(b_{i_{2}})\times\cdots\times
      	O(b_{i_{k}}))\] is the $G$--translate of a tuple of the form
      	\[(b_{i_{1}},g_{1}(b_{i_{2}}),g_{2}(b_{i_3}),\dots, g_{k-1}(b_{i_{k}})).\] The existence of $\mathcal{T}_{\bar{i}}$ results from
      	the assumption that $\mathfrak R_{\mathcal{B}}\slash G$ is finite.
        We write $\bar\tau=(\tau_1,\ldots,\tau_{k-1})$ for a general
        element of $\mathcal{T}_{\bar{i}}$.
      	
      	We can then take the (abbreviated) quantifier free formula
      	\begin{align*}
      		\psi^{\mathfrak R}_{\bar{i}}(x_1,\ldots,x_k)=\bigvee_{\bar{i}\in\{1,\dots,
      		m\}^{k}}\bigvee_{\bar{\tau}\in\mathcal{T}_{\bar{i}}}
      		\bigwedge_{j=2}^{k}(S_{b_{i_{1}}}\circ R_{\tau_{j-1}}\circ
      		S_{b_{i_{j}}})(x_{1},x_{j}),
      	\end{align*}
      	which has the  desired property.
      \end{proof}

      \begin{remark}\label{r:psi-R-i}
      	The formula $\psi^{\mathfrak R}_{\bar{i}}(x)$ implicitly defines an arbitrary $k$--ary relation on the universe $B$ that is stable under the action of the
      	mapping class group and has a well-behaved quotient, once these have been interpreted in the augmented Cayley graph.
      	Corollary~\ref{cor:definable-recipe} will follow once this formula can be transported back to $\mathcal B$, which will be possible for certain
      	types of $(G,\D)$--structures; see the discussion immediately after the proof of Lemma~\ref{interpretation in B} below.
      \end{remark}
      
      \begin{corollary}[$X(\Sigma)$ is interpretable in $ \mathcal{M}_{\mathcal D_X}^G(\Sigma) $]\label{cor:geom_interpr}
      	Every geometric graph $X(\Sigma)$ is interpretable in the $\mathcal D_X$-augmented Cayley graph of $G$:
      	$$\bar{\zeta} :
      	X(\Sigma) \leadsto \mathcal{M}_{\mathcal D_X}^G(\Sigma) ~.$$
      	In particular, the curve complex $\mathcal C(\Sigma)$ is interpretable in $\mathcal M_0^G(\Sigma)$.
      \end{corollary}

      In Subsection~\ref{ss:omega-stability} we will prove the following, which is one of the main results of this paper.
      \begin{restatable}[]{theorem}{oldversion}\label{thm:omega-stab}
      	Let $G< \Mod^\pm(\Sigma)$ be finite index and let
      	$\emptyset\neq\D\subset\D_0$ be a $G$--invariant, downward closed collection of connected domains.
      	Then $\Th{\M_\mathcal{D}^G(\Sigma)}$ is $\omega$-stable.
      \end{restatable}
      
      Combining Theorem~\ref{thm:omega-stab} with Lemma \ref{interpretation in M} and Example~\ref{ex:stable-interp}, we obtain the following, which establishes part of
      Theorem \ref{thm:stable}:
      \begin{corollary}\label{cor:b-omega-stab}
      	Let  $\mathcal{B}(\Sigma)$ be a $(G, \D)$-structure.
      	Then $\mathcal{B}(\Sigma)$ is $\omega$--stable.
      	In particular,
      	for every geometric graph $X(\Sigma)$, the first order theory $\mathrm{Th}(X(\Sigma))$ is  $\omega$--stable.
      \end{corollary}
      
  \subsection{The curve graph interprets the extended  mapping class group}
    In this subsection, we assume $G = \Mod^\pm(\Sigma)$.
    Suppose that $X(\Sigma)$ is a geometric graph. We will write $\mathcal M_X(\Sigma)$ for $\mathcal M^G_{\D_X}(\Sigma)$ for compactness of notation. We will abuse notation and identify $X$ with the underlying universe.

    \begin{definition}[Strongly rigid tuple]
    	Let $\mathcal G=(V,E)$ be a graph. We say that a tuple $\delta\in V^{<\omega}$ of vertices of a graph is \emph{strongly rigid}
    	if for any partial isomorphism between their induced subgraphs \[f\colon \delta\to \delta'\subseteq \mathcal G,\]
    	there exists a unique automorphism $F: \mathcal G \to \mathcal G$  extending $f$.
    \end{definition}
    The partial isomorphism $f$ must preserve adjacency and non-adjacency. In the context of curve graphs, Irmak refers to
    preservation of non-adjacency as \emph{superinjectivity}~\cite{irmak-superinjective}.
    The uniqueness of the automorphism of $\mathcal G$ extending the partial isomorphism implies that the stabilizer of a strongly
    rigid tuple is trivial. The following observation is straightforward from the definition of a strongly rigid tuple.

    \begin{observation}
    	\label{c: rigid tuple}\leavevmode
    	Let {$\mathcal{G}$} be a graph. The type
        $\tp(\delta)$ of a strongly rigid
    	tuple of vertices $\delta$ is isolated by a
    	quantifier-free formula, since the type is determined by which vertices of $\delta$ are adjacent to each other
    	and which are not.
    	If a graph {$\mathcal{G}$} admits an exhaustion by strongly rigid tuples, then the automorphism group
	orbit of a tuple $\alpha$ of vertices is the
    	set of solutions of a suitable existential formula $\phi_{\alpha}(x)$, which in turn isolates
    	$\tp(\alpha)$.
    \end{observation}

    Indeed, if $\alpha$ is an arbitrary tuple of vertices which can
    be embedded in a strongly rigid tuple $\delta$, then the
    formula $\phi_\alpha$ simply asserts the existence of vertices
    in $\delta\setminus\alpha$ which have the appropriate adjacency relations.

Now, suppose that  $X=X(\Sigma)$ is a geometric graph that admits an exhaustion
    by strongly rigid tuples, and suppose that the natural action homomorphism
    \[
        \rho_X\colon G\longrightarrow \aut(X(\Sigma))
    \]
    is an isomorphism.
    For an arbitrary finite tuple $\alpha$ of vertices in $X$, we let $Y(\alpha)\subset X^{|\alpha|}$ be the set of realizations of
    $\tp(\alpha)$. As above, we let $\phi_{\alpha}(x)$ be an existential formula isolating $\tp(\alpha)$, which is taken to be
    quantifier-free if $\alpha$ is itself strongly rigid. We define
    \[
        \eta_{\alpha}\colon G\to X^{|\alpha|},\qquad
        \eta_{\alpha}(g)=g\cdot\alpha.
    \]
    The model-theoretic properties of strongly rigid tuples as enumerated in Observation~\ref{c: rigid tuple} now give the following
    bi--interpretability result.

    \begin{lemma}[$X(\Sigma)$ and $\mathcal M_{X}(\Sigma)$ are bi-interpretable]\label{interpretation in B}
    	Let $X(\Sigma)$ be a geometric graph. Suppose that the natural action homomorphism
        \[
            \rho_X\colon G\longrightarrow \aut(X(\Sigma))
        \]
        is an isomorphism, and that $X(\Sigma)$ has the following two properties:
    	\begin{enumerate}
    		\item
    		$X(\Sigma)$ admits an exhaustion by strongly rigid tuples. 
    		\item
    		For every connected proper domain $D\in \D_X$, there is a tuple $c\in X^{<\omega}$ such that $\Stab_G(c)=G[D]$.
    	\end{enumerate}
    	Then for any strongly rigid tuple $\delta$ we have that
    	\[
            \bar{\eta}:=(\eta_\delta,|\delta|,Y(\delta),=) : \mathcal{M}_{X}(\Sigma) \rightsquigarrow X(\Sigma)
        \]
        is an interpretation of
    	$\mathcal{M}_{X}(\Sigma)$ in $X(\Sigma)$; cf.~ Definition \ref{d:interpretation}.
    	
    	Furthermore,
    	if $\bar{\zeta}: X(\Sigma) \rightsquigarrow \mathcal{M}_{X}(\Sigma)$ is the natural interpretation of $X(\Sigma)$ in
    	$\mathcal{M}_{X}(\Sigma)$,
    	then $$(\bar{\zeta},\bar{\eta}):  X(\Sigma) \leftrightsquigarrow \mathcal{M}_{X}(\Sigma) $$ furnishes a bi-interpretation of
    	$X(\Sigma)$ in $\mathcal{M}_{X}(\Sigma)$. Moreover, the
        graphs of the maps $\bar{\eta}\circ\bar{\zeta}$  and
         $\bar\zeta\circ\bar\eta$ underlying the bi-interpretation
         are both
        defined by parameter-free existential formulae.
    \end{lemma}
    
    The first hypothesis on $X$ is satisfied by many natural geometric graphs, and the second, while less commonly studied,
    is also enjoyed by many geometric graphs of interest. See Subsection~\ref{ss:applications} below.
    
     \begin{proof}[Proof of Lemma~\ref{interpretation in B}]
    	We will fix a strongly rigid tuple $\delta=(\delta_1,\ldots,\delta_k)$ and suppress it from the notation
    	whenever possible for the remainder of the proof. Write $Y=Y(\delta)$ and $\eta=\eta_\delta$.

        We first note that $\eta\colon G\to Y$ is a bijection. Indeed, let $y\in Y$. Since $y$ realizes
        $\tp(\delta)$, Observation~\ref{c: rigid tuple} gives an automorphism
        $F\in \aut(X(\Sigma))$ such that $F(\delta)=y$. Since $\rho_X$ is surjective, we have
        $F=\rho_X(g)$ for some $g\in G$, and hence $y=g\cdot\delta=\eta(g)$.
        Conversely, if $\eta(g)=\eta(h)$, then $\rho_X(h^{-1}g)$ fixes $\delta$. The stabilizer of a
        strongly rigid tuple in $\aut(X(\Sigma))$ is trivial, and so $\rho_X(h^{-1}g)=1$. Since
        $\rho_X$ is injective, we have $g=h$.
    	
    	For $g\in G$ and $D\in\D_X$, it  suffices to check that
    	the relations $R_{g}$ and $R_{D}$ are the preimages
    	of the set of solutions of
    	suitable definable relations $\psi_{g}$ and $\psi_{D}$ on $Y$ under
    	$\eta\times\eta$.
    	
    	For $g\in G$, we write
      \[\psi_{g}(x,y)\equiv\phi_{g\cdot
    	\delta,\delta}(y,x).\]
     Here $x$ and $y$ are $k$--tuples of variables, and $\phi_{g\cdot\delta,\delta}$ is the existential formula isolating the type of $(g\cdot\delta,\delta)$. It follows from the definition of a strongly rigid tuple that
     $\delta$ has a trivial stabilizer, and so:
     
    	\begin{align*}
    		X\models \psi_{g}(\eta(h),\eta(h'))\leftrightarrow X\models
    	 \phi_{g\cdot
    		\delta,\delta}(h'\cdot\delta,h\cdot\delta)\leftrightarrow\\
    		\leftrightarrow X\models \phi_{g\cdot
    		\delta,\delta}(h^{-1}h'\cdot\delta~,\delta)\leftrightarrow h^{-1}h'=g
    		\leftrightarrow \M_{X} \models R_{g}(h,h').
    	\end{align*}

    	Now suppose first that $D\in\D_X$ is connected and proper. Choose a finite tuple $c=(c_{1},c_{2},\dots,
    	c_{\ell})\in X^{<\omega}$ such that $G[D]$
    	coincides with the stabilizer of $c$, 
    	as is guaranteed by hypothesis. Let
    	\begin{align*}
    		\psi_{D}(x,y)\equiv\exists z\,\,
    		\phi_{c,\delta}(z,x)\wedge\phi_{c,\delta}(z,y).
    	\end{align*}
    	It is straightforward to verify that $\psi_{D}$ interprets $R_{D}$. Note that here $z$ is an $\ell$--tuple of variables, and $\exists z$ an abbreviation for $\ell$-many existential quantifiers on these variables.
        If $D=\C_0(\Sigma)$, then $R_D$ is interpreted by the tautological formula $x=x\wedge y=y$.
        This completes the interpretation of $\M_X(\Sigma)$ in $X(\Sigma)$.

        It remains to show that this interpretation together with
        the natural interpretation of $X(\Sigma)$ in
        $\M_X(\Sigma)$ furnish a parameter-free bi-interpretation.
    	Let \[\bar{\xi}_{1}=\bar{\eta}\circ\bar{\zeta},\quad
        \bar{\xi}_{2}=\bar{\zeta}\circ\bar{\eta}.\] We need to show that the underlying
    	maps $\xi_{1}$ and $\xi_{2}$ are definable without parameters in
    	$X$ and in $\M_X$, respectively.
    	
    	To this end, let $\{b_{1},b_{2},\dots, b_{m}\}$ be the representatives of
    	the orbits of $X$ under the action of $G$, as used in the construction of $\bar{\zeta}$.
    	Given $b\in X$, we denote by $\hat{b}$ the unique  representative
    	$b_{i}$ satisfying $O_{G}(b_{i})=O_{G}(b)$.

        Recall the notation for the fixed strongly rigid tuple
        $\delta=(\delta_1,\ldots,\delta_k)$.
    	On the one hand, we have  that $\bar{\eta}$ sends
    	$g\in M^G$ to $g\cdot\delta\in X^{k}$, which in turn gets sent
    	to a $k$-tuple of equivalence classes $(gS_{\hat{\delta_{j}}})_{j=1}^{k}$
    	by $\bar{\zeta}$. Note that by definition,
    	\[gS_{\hat{\delta_{j}}}=\{h\in M^G\,|\,S_{\hat{\delta_{j}}}(g,h)\},\] where here we are implicitly using the natural bijection between
    	equivalence classes of $S_{\hat{\delta_{j}}}$ and cosets of the stabilizer of $\hat{\delta_{j}}$.
    	
    	Let $\Gamma_{\xi_{2}}$ be the predicate defining the graph of $\xi_2$; see Definition~\ref{def:def-int}.
    	Since for each $j$ the subgroup $S_{\hat{\delta}_{j}}$ is definable in $\M_X$ without
    	parameters, the definability of $\Gamma_{\xi_{2}}$ also follows.

        On the other hand, $\bar{\zeta}$ sends $b\in X$ to the unique coset
    	\[g_{b}S_{\hat{b}}\subseteq M^G=G\] consisting of those elements sending $\hat{b}$
    	to $b$, which is in turn sent by $\bar\eta$ to the collection
    	\[Y_{b}:=\{h\delta\,|\,h\in g_{b}S_{\hat{b}}\}.\] Note, however,
    	that
    	\begin{align*}
    		Y_{b}=&\{g_{b} x\,|\,\tp(x,\hat{b})=\tp(\delta,\hat{b})\}=\\
    		=\{g_{b} x\,|\,\tp(g_{b}
    		x,g_{b}\hat{b})=&\tp(\delta,\hat{b})\}=
    		\{y\,|\,\tp(y,b)=\tp(\delta,\hat{b})\}.
    	\end{align*}
    	The first equality in this chain of equalities follows from Observation~\ref{c: rigid tuple}. The set $Y_b$ is defined
        by a parameter--free existential formula.
    	
    	We may now set
    	\[\Gamma_{\xi_{1}}(x,y)\equiv\bigvee_{j=1}^{m}(\phi_{b_{j}}(x)\wedge
    	\phi_{b_{j},\delta}(x,y)),\]
    	whence the desired conclusion follows.
    \end{proof}

    Combining Lemma~\ref{interpretation in B} with Remark~\ref{r:psi-R-i}, we obtain Corollary~\ref{cor:definable-recipe} as
    claimed in the introduction.
    The following lemma relies on Corollary~\ref{cor:ext-qe} which  says  that $\M_X$
    has relative quantifier elimination.

\begin{lemma}[$X(\Sigma)$ has relative QE]\label{l:quantifier elimination}
    	Assume that $X(\Sigma)$ is a geometric graph satisfying:
    	\begin{enumerate}
            \item
            The natural action homomorphism
            \[
                \rho_X\colon G\longrightarrow \aut(X(\Sigma))
            \]
            is an isomorphism.
    		\item
    		$X(\Sigma)$ admits an exhaustion by strongly rigid tuples.
    		\item
    		For every connected proper domain $D\in \D_X$, there is a tuple $c\in X^{<\omega}$ such that $\Stab_G(c)=G[D]$.
    	\end{enumerate}
    	Then $X(\Sigma)$ has quantifier elimination relative to the collection of
    	$\forall \exists$--formulae.
    \end{lemma}

    \begin{proof}
    	Let $\phi(x)$ be a (possibly multivariable)
    	parameter--free formula in $\mathcal{L}_B$.
    	The interpretation $\zeta[\phi(x)]$ is equivalent to a Boolean combination of
    	existential formulae in $\Th{\M_X}$, by Corollary~\ref{cor:ext-qe}.
    	
    	Lemma~\ref{interpretation in B} implies that under $\bar\eta$, relations in $\M_X$ are interpreted in $X$ via
    	existential formulae, since the formulae isolating types of tuples are existential. Thus, a universal formula in $\M_X$ can be
    	interpreted as a $\forall\exists$ formula in $X$, and an existential formula is interpreted as a $\exists\forall$ formula.

              Note that here we use the fact that the domain $Y(\delta)\subseteq X^{|\delta|}$ of the underlying map of the
      interpretation $\bar\eta$ is a quantifier-free definable subset of $X$, as it is given by the set of realizations of
      a strongly rigid tuple $\delta$.
      
      In particular, suppose that $\theta(x)$ is a
      quantifier-free formula $\theta(x)$ in $X(\Sigma)$. Then
      its interpretation $\theta'(x)$ under $\bar\eta\circ\bar\zeta$ (i.e. its pullback) is a Boolean combination of $\forall\exists$-formulae in $X$. To conclude, we relate the quantifier complexity of the definable sets given by $\theta$ and $\theta'$. It suffices to show that if $\theta'$ is either $\forall\exists$ or $\exists\forall$, then $\theta$ is also equivalent to a Boolean combination of $\forall\exists$ formulae.

    	To this end, we note two different ways of
    	expressing $\theta(x)$:
    	\begin{align*}
    		X\models \forall x\,(\theta(x)\leftrightarrow(\exists
    		y\,\theta'(y)\wedge\Gamma_{\eta\circ\zeta}(x,y))) \\
    		X\models \forall x\,(\theta(x)\leftrightarrow(\forall
    		y\,\neg\Gamma_{\eta\circ\zeta}(x,y)\vee\theta'(y))),
    	\end{align*} 
      where here we recall that the predicate $\Gamma_{\eta\circ\zeta}(x,y)$ is existential by Lemma~\ref{interpretation in B}.
    	The first of these expressions can be used for $\exists\forall$--formulae, and the
    	second for $\forall\exists$--formulae. This establishes the lemma.
    \end{proof}
    
    \subsubsection{Applications to geometric graphs}\label{ss:applications}
      
      In order to obtain results for specific geometric graphs associated to a surface $\Sigma$, it suffices for us to verify the hypotheses
      of Lemma~\ref{interpretation in B} for those geometric graphs. The first condition is the exhaustion of the relevant geometric graph
      by strongly rigid tuples (which are often called finite rigid sets in the literature). We have the following results.
      
      \begin{theorem}\label{thm:strong-rigid}
      	The following geometric graphs of connected surfaces admit exhaustions by strongly rigid tuples:
      	\begin{enumerate}
      		\item
      		The curve graph of a surface that is not a torus with two punctures~\cite{aramayona2016exhausting}.
      		\item
      		The arc graph of a surface with at least one puncture that is not a sphere with three punctures~\cite{ShinkleAGT}.
      		\item
      		The pants graph of a sphere with at least five punctures~\cite{MaungchaungJKTR}.
      		\item
      		The flip graph of a surface that is not a torus with at most one puncture nor a
      		sphere with at most two punctures~\cite{ShinkleTAMS}.
      		\item
      		The nonseparating curve graph of a surface of genus at least three~\cite{depool}.
      	\end{enumerate}
      \end{theorem}
      
      For certain other complexes, such as the pants graph for surfaces with genus or the graph of separating curves,
      strongly rigid tuples are known to exist but the
      existence of exhaustions by them appears to be unknown~\cite{HHLM19,HuangTshi}.
      
      For the second hypothesis of Lemma~\ref{interpretation in B}, we have the following.
      
      \begin{lemma}\label{lem:int-second-hyp}
      	Let $\Sigma$ be a connected surface of negative Euler characteristic that is not a pair of pants.
	Let $X(\Sigma)$ denote one of the following geometric graphs:
      	\begin{enumerate}
      		\item
      		The curve graph of a connected surface;
      		\item
      		The pants graph of a connected surface;
      		\item
      		The arc graph of a surface with at least one puncture;
      		\item
      		The flip graph of a surface with at least one puncture;
      		\item
      		The nonseparating curve graph of a surface of genus at least three.
      	\end{enumerate}
      	If $D\in\D_X$ is connected then there is a tuple $c\in X^{<\omega}$ such that $\Stab_G(c)=G[D]$.
      \end{lemma}
      
      Here, $\Stab_G(c)$ is the stabilizer of the tuple, so that elements of the tuple are pointwise fixed. For the flip graph,
      if $c$ is a tuple of vertices then $\Stab_G(c)$ is finite. Moreover, $D_{\Fill(v)}$ is empty for any vertex $v$. In particular, there are no
      proper nonempty domains in $\D$ for the flip graph, and Lemma~\ref{lem:int-second-hyp} holds vacuously.
      
      \begin{proof}[Proof of Lemma~\ref{lem:int-second-hyp}]
      	Consider first the curve graph. If $D\in\D_X$ is a connected proper domain then $D$ is simply the domain associated to
      	a proper, connected subsurface of $\Sigma$ that is not a pair of pants. We may then consider a finite collection $c$ of simple closed curves
      	that fill the surface $\Fill(D^{\perp})$. Here, we allow $c$ to include curves that are peripheral in $D$ but not in $\Sigma$.

      	Regardless of what $\D_X$ is, if $\psi$ is a mapping class that fixes $c$, then
      	$\psi$ restricts to the identity on $D^{\perp}$, and so $\psi\in G[D]$. Conversely, if $\psi\in G[D]$ then $\psi$ restricts to the identity
      	on $D^{\perp}$ and so $\psi$ stabilizes $c$. The argument for the arc graph and nonseparating curve graph is identical.
      	
      	Now, suppose that $X(\Sigma)$ is the pants graph of $\Sigma$. If $D\in \D_X$ is a proper, connected domain then $D$ is an annulus.
      	It suffices to find a finite collection of pants decompositions of $\Sigma$ whose stabilizer is exactly $G[D]$, which in turn is a subgroup
      	that is (virtually)
      	generated by a (power of a) Dehn twist about the core curve $\alpha$ of $D$. If $P$ is a pants decomposition extending $\alpha$
      	and $g\in G[D]$ then $g$ stabilizes $P$. Note, however, that an element of $G$ that stabilizes $P$ may permute the components of $P$.
      	
      	Note that if $P$ has exactly one component, then $\Sigma$ is a four-times punctured sphere or once-punctured torus. Otherwise,
      	$P$ has at least two components. Let $\beta$ be a curve in $P$ that is distinct from $\alpha$, and let $\beta'$ differ from $\beta$
      	by an elementary move~\cite{margalit,minsky-survey},
      	so that \[P'=(P\setminus\{\beta\})\cup\{\beta'\}\] is also a pants decomposition of $\Sigma$. Now, suppose
      	that $g\in G$ stabilizes both $P$ and $P'$. Then $g$ permutes the components of both $P$ and $P'$. It follows that $g$ stabilizes
      	both $\beta$ and $\beta'$, and thus the domain associated to the surface filled by $\beta$ and $\beta'$. An easy induction, using the
      	fact that the pants graph of a surface is connected, furnishes a finite tuple in the pants graph whose stabilizer coincides with $G[D]$.
      \end{proof}
      
      From the above discussion we deduce the following.

      \begin{corollary}[Geometric graphs with relative QE]\label{bi-interpretability of the curve graph}
      Let $\Sigma$ be a connected surface of negative Euler characteristic that is not a pair of pants. For each of the following
      geometric graphs $X(\Sigma)$, there exists a suitable $\D=\D_X$ such that $X(\Sigma)$ and $\M_{\D}$ are bi-interpretable, and
      	therefore have quantifier elimination with respect to the class of $\forall\exists$--formulae:
      	\begin{enumerate}
      		\item
      		The curve graph of a connected surface  that is not a torus with two punctures.
      		\item
      		The arc graph of a connected surface with at least one puncture that is not a sphere with three punctures.
      		\item
      		The pants graph of a sphere with at least five punctures.
      		\item
      		The flip graph of a surface that is not a torus with at most one puncture or a sphere with at most two punctures.
      		\item
      		The nonseparating curve graph of a surface of genus at least three.
      	\end{enumerate}
      \end{corollary}

        \begin{proof}
      	Let $X(\Sigma)$ be any of the graphs above. For each of these graphs, with the stated low-complexity exclusions, the natural action homomorphism
        \[
            G=\Mod^\pm(\Sigma)\longrightarrow \aut(X(\Sigma))
        \]
        is an isomorphism by the corresponding automorphism-rigidity theorem for $X(\Sigma)$.
        By Theorem~\ref{thm:strong-rigid}, $X(\Sigma)$ admits an exhaustion by strongly rigid tuples.
        Furthermore, by Lemma~\ref{lem:int-second-hyp}, if $D \in \mathcal D_X$ is connected then there is a tuple
      	$c \in X^{<\omega}$ such that $\mathrm{Stab}_G(c)=G[D]$. Thus, by Lemma~\ref{l:quantifier elimination} we conclude that $X(\Sigma)$ has quantifier elimination with respect to the class of $\forall\exists$--formulae.
      \end{proof}
      
      In the case of the flip graph, we can actually prove even stronger results. If $X(\Sigma)$ is the flip graph of a surface with at least one puncture, then the collection $\D_X$ is empty. In this case, the
      relational theory of $\M^G_{\D_X}$ is very simple, and so one might expect stronger results about its theory.
      Indeed, we will see at the end of Section~\ref{sec:full-qe}
      that the flip graph admits quantifier elimination with respect to the class of $\exists$--formulae; see Corollary~\ref{cor:flip-qe} below.
      
  \subsection{The Morley ranks of geometric graphs}\label{sec:geometric_graphs}  Recall
    that geometric graphs are examples of $(G, \D)$-structures, as defined in the previous section.
    \begin{definition}
    	Let $\D$ be a downward closed, $G$--invariant collection of connected domains. We define its \emph{length} $k(\D)$ as the maximal $n \in \mathbb N$ for which
    	there exists a chain
    	$$\varnothing \subsetneq D_{0} \subsetneq \ldots \subsetneq D_{n} = \C_0(\Sigma)~,$$
    	where for $1\leq i<n$ we have that $D_i\in \D$.
    	
    	When  $X(\Sigma)$ is a geometric graph and $\D=\D_X$, then we write $$k(X(\Sigma)):=k(\D_X).$$
    \end{definition}
    
    \begin{corollary}[$\omega$-stability of other geometric graphs]
    	\label{cor:other graphs}
    	Suppose that $X(\Sigma)$ is a geometric graph.
    	We have that $X(\Sigma)$ is interpretable in $\C(\Sigma)$ and its Morley rank is
    	bounded above as follows: $$RM(X(\Sigma)) \leq \omega^{k(X(\Sigma))} ~.$$
    \end{corollary}
    \begin{proof}
    	By Corollary \ref{cor:geom_interpr}
    	and Corollary~\ref{lower bound on Morley rank} we have:
    	$$ RM(X(\Sigma)) \leq RM (\M_X^G(\Sigma)) = \omega^{k(X(\Sigma))}~,$$ the desired conclusion.
    \end{proof}

    Corollary~\ref{cor:other graphs} implies Corollary~\ref{c:interpretability} in a straightforward way.
    On the other hand, we have the following result which implies Corollary~\ref{c:non-bi}.
    \begin{corollary}\label{cor:pants}
    	Let $\Sigma$ be a surface of genus $g$ and $b$ punctures, let
    	$\mathcal P(\Sigma)$ denote the pants graph of $\Sigma$, let $\mathcal {SC}(\Sigma)$ denote
    	the separating curve graph, and let $\mathcal{A}(\Sigma)$ denote the arc
    	graph. For $\mathcal{SC}(\Sigma)$, assume that $g\geq 2$ and $b\leq 1$, and for $\mathcal{A}(\Sigma)$, assume that
    	$g\geq 2$ and $b=1$. For $\mathcal P(\Sigma)$, assume that if $g=0$ then $b\geq 4$. 
        Then we have:
    	\[RM(\mathcal{A}(\Sigma))\leq\omega^{3g-4},\quad RM(\mathcal{SC}(\Sigma))\leq\omega^{3g-4},\quad RM(\mathcal{P}(\Sigma))\leq \omega,\] and
    	$\mC(\Sigma)$ is not interpretable in $\mathcal{SC}(\Sigma)$ nor in $\mathcal{A}(\Sigma)$, nor in
    	$\mathcal{P}(\Sigma)$ provided $3g+b>4$.
    \end{corollary}
    \begin{proof}
    	A vertex $v$ of $\mathcal {A}(\Sigma)$ is given by a properly embedded essential arc
    	$\alpha$, with both endpoints on
    	the boundary component. If $\alpha$ is nonseparating
    	then the result of cutting $\Sigma$ open along $\alpha$ is a
    	surface of genus $g-1$ and one boundary component. If $\alpha$ is separating then each
    	component of the result of cutting $\Sigma$ open along $\alpha$ has genus at most $g-1$
    	and has one boundary component. The longest chain of domains in $\D_{\mathcal A}$ has length $3g-4$,
    	so that $k(\mathcal{A}(\Sigma))\leq 3g-4$.
    	
    	A vertex $v$ of $\mathcal{SC}(\Sigma)$ is given by a separating curve $\gamma_v$ on $\Sigma$. We have that
    	$\Sigma\setminus\gamma_v$ is the disjoint union of surfaces $\Sigma_1$ and $\Sigma_2$,
    	and that both $\Sigma_1$ and $\Sigma_2$ must have positive genus.
    	If $\Sigma_1$ is exactly a torus with at most two boundary
    	components, then $\Sigma_2$ is homeomorphic to a surface with one boundary component
    	and genus lower by one. In any case, the length of the longest chain of domains in $D_\mathcal{SC}$ is
    	$3g-4$. Thus, $k(\mathcal{SC}(\Sigma))\leq 3g-4$.
    	
    	A vertex $v$ of $\mathcal{P}(\Sigma)$ is given by a pants decomposition of $\Sigma$. If $D$ is a nontrivial,
    	connected, proper domain contained in $D_{\mathcal P}$ then $D$ is
    	a component of the pants decomposition.
    	In particular, $k(\mathcal{P}(\Sigma))=1$.
    	
    	Suppose $\mC(\Sigma)$ were interpretable in $\mathcal{P}(\Sigma)$. Write $\D$ for the $G$--invariant
    	downward closed collection of connected domains generated by the annular components of pants decompositions of $\Sigma$.
    	Then $\M^G_\D(\Sigma)$ has Morley rank $\omega$ by Corollary~\ref{lower bound on Morley rank},
    	and $X(\Sigma)$ is interpretable in $\M^G_\D(\Sigma)$. We have a chain of interpretations as follows:
    	$\M^G_{\D_0}(\Sigma)$ is interpretable in $\mC(\Sigma)$, which is interpretable in $\mathcal{P}(\Sigma)$, which is interpretable in
    	$\M^G_\D(\Sigma)$, i.e.~
	\[ \M^G_{\D_0}(\Sigma)\rightsquigarrow \mC(\Sigma) \rightsquigarrow \mathcal{P}(\Sigma)\rightsquigarrow \M^G_\D(\Sigma),\]
	and so \[RM(\M^G_{\D_0}(\Sigma))=\omega^{3g+b-3}.\] Theorem~\ref{thm:morley-upper-bound} gives
    	the desired contradiction, since it follows that the rank of a definable set of
    	imaginaries is bounded from above by that of a definable set of real tuples.
    	The arguments for $\mathcal{A}(\Sigma)$ and $\mathcal{SC}(\Sigma)$ are analogous.
    \end{proof}
    
    A philosophical reason for these bounds on Morley rank can be formulated thus:
    in all three graphs, one is highly constrained in the
    edge relation: one does not have a maximal subsurface curve graph's worth of choices for a
    neighbor in these graphs. This is especially apparent in the pants graph.
    The bounds on the Morley rank can thus be thought of as
    a rigorous expression of this restriction.
    
    We remark that in Margalit's work~\cite{margalit}, it is shown that \[\Aut(\mathcal P(\Sigma))\cong\Mod^{\pm}(\Sigma)\] by exploiting
    what appears to be an interpretation of the curve graph in the pants graph. The basic idea is to take a curve $\alpha$ on
    $\Sigma$ and to look at all
    possible extensions of $\alpha$ to a pants decomposition of $\Sigma$. When $\Sigma$ is a five-times punctured sphere then
    there is a Farey graph $F_{\alpha}$ in the pants graph that corresponds to the possible extensions of $\alpha$ to a pants decomposition.
    One then applies automorphisms of $\mathcal P(\Sigma)$ to $F_{\alpha}$ to get new Farey graphs that correspond to other curves
    on $\Sigma$. Thus, automorphisms of $\mathcal P(\Sigma)$ induce automorphisms of $\C(\Sigma)$, after checking a large number of
    details.
    
    While the association $\alpha\mapsto F_{\alpha}$ appears to be something like an interpretation, a Farey graph is not definable
    by a first order predicate; Corollary~\ref{cor:pants} shows that no such interpretation can be achieved in first order logic. Indeed,
    the definition of $F_{\alpha}$ involves a transitive closure, which is generally not realizable in first order logic.
    
    As a consequence of the geometry of the curve graph and associated totally geodesic Farey subgraphs, we have the following lower
    bound, which will follow from Corollary~\ref{lower bound on Morley rank}:
    
    \begin{corollary}\label{cor:pants-lower}
    	Suppose $\Sigma$ is not a sphere with fewer than four punctures. Then the Morley rank of $\mathcal P(\Sigma)$ is at least $\omega$.
    \end{corollary}

    More geometric intuition behind the lower bound of $\omega$ for
    the pants graph can be gained from the main result of~\cite{zahra-katrin}, which shows that the Morley rank of
    the Farey graph is exactly $\omega$.

    Finally, for the flip graph, we have the following conclusion:
    \begin{corollary}\label{cor:flip-morley}
    	Suppose that $\Sigma$ is a surface with at least one puncture. Then the Morley rank of the flip graph is at most one, and is equal to
    	one if and only if the flip graph is infinite.
    \end{corollary}
    \begin{proof}
    	Clearly, the Morley rank of the flip graph of $\Sigma$ is zero if and only if it is finite. If $X(\Sigma)$ denotes the flip graph of $\Sigma$, then $\D_X$ is empty. Corollary~\ref{cor:other graphs} implies that the Morley rank of $X(\Sigma)$ is at most $\omega^0=1$, and so if
    	$X(\Sigma)$ is infinite then its Morley rank is one.
    \end{proof}
    
  \subsection{Non-definable sets in the curve graphs}
    To close this section, we give a proof of Corollary~\ref{cor:alg-int}, as announced in the
    introduction. We will in fact prove a somewhat more general statement.
    
    \begin{corollary}\label{cor:alg-int-body}
    	Let $\Sigma$ be a surface that is not a torus with two punctures, and let $d$ denote the graph distance function in $\C(\Sigma)$.
    	Let $Q\subseteq\mC(\Sigma)^2$ be a binary predicate with the property that for any
    	$n\in \N$, there exist pairs
    	$(\alpha,\beta)\in Q$ and $(\alpha',\beta')\in\mC(\Sigma)^2\setminus Q$ with
    	\[d(\alpha,\beta),d(\alpha',\beta')\geq n.\]
    	Then $Q$ is not $\emptyset$--definable in $\mC(\Sigma)$.
    \end{corollary}
    
    Corollary~\ref{cor:alg-int-body} follows from quantifier elimination in $\M^G$, where $G<\Mod^{\pm}(\Sigma)$ is the
    pure mapping class group of $\Sigma$, and the fact that $\C(\Sigma)$ can be interpreted therein,
    and thus from a weak version of quantifier elimination in $\C(\Sigma)$.
    
    The following is a
    consequence of Corollary~\ref{cor:alg-int-body} and of the existence of pseudo-Anosov elements in the Torelli subgroup of $G$
    (which requires us to exclude a torus with at most one boundary component).
    Such mapping classes have positive translation length in $\mC(\Sigma)$, and also preserve the algebraic
    intersection pairing $i_a$ of oriented curves.
    Since vertices of $\C(\Sigma)$ are not oriented, there is no canonical way to assign a value $i_a(\alpha,\beta)$ to two
    vertices $\alpha$ and $\beta$ of the curve graph. However,
    the absolute value of $i_a(\alpha,\beta)$ is well-defined and
    does not depend on any choice of orientations.

    \begin{corollary}
    	Suppose $\Sigma$ has positive genus and is not a torus with
    	fewer than three punctures, and let $S\subset\bN$ be a nonempty proper subset.
    	The predicate $I_S$ that states that two simple closed curves $\alpha$ and $\beta$ have
    	integral algebraic intersection number satisfying $|i_a(\alpha,\beta)|\in S$ is not $\emptyset$--definable. Similarly,
    	the predicate $Q$ expressing that two curves have algebraic intersection
    	number $0\pmod 2$ is not $\emptyset$--definable.
    \end{corollary}

    To see how this corollary follows, let $\alpha$ and $\beta$ be simple closed curves on $\Sigma$ with an
    algebraic intersection number whose absolute value lies in $S$,
    and let $\psi$ be
    a pseudo-Anosov mapping class that acts trivially on $H_1(\Sigma,\mathbb Z)$; it is a standard fact that there are
    pseudo-Anosov mapping classes in all nontrivial, non-central normal subgroups of $\Mod(\Sigma)$ (see~\cite{ivanov-book}, for instance),
    and so such a $\psi$ exists. Then for any $N,\Delta\geq 0$, there exists an $M$
    such that if $m\geq M$ then $d(\alpha,\psi^m(\beta))\geq \Delta$, but the algebraic intersection number of
    $\alpha$ and $\psi^m(\beta)$ is independent of $m$. We may repeat the same argument for a pair of curves
    whose algebraic intersection number's absolute value lies outside of $S$. All possible algebraic intersection numbers
    are achieved by pairs of curves; indeed, taking two curves that intersect exactly once and then performing Dehn twists about
    one of them, we obtain all possible integer values for the algebraic intersection number. We may now apply Corollary~\ref{cor:alg-int-body}.

    \begin{proof}[Proof of Corollary~\ref{cor:alg-int-body}]
    	Throughout, let $G$ denote the pure mapping class group of $\Sigma$.
    	Suppose the contrary of the conclusion of the corollary. Lemma~\ref{interpretation in M} implies that $\mC(\Sigma)\rightsquigarrow
	\M^G$. A predicate $Q$ as in the statement is defined by
    	a parameter free formula $\theta$ in
    	$\mathcal{L}(\D_0)$, in the sense that $Q(x,y)$ is given
    	by a conjunction of formulae interpreting $\mC(\Sigma)$ in $\M^G$, together with $\theta$ itself.
    	Since $\M^G$ has quantifier elimination by Theorem~\ref{t:M-fullqe}, we may assume that
    	$\theta$ is quantifier free. We may therefore take $\theta$ to be a Boolean combination of relations in
    	the language $\mathcal{L}(\D_0)$, some positive and some negated.
    	
    	If $D$ is a proper domain, then the relation $R_D$ gives
    	rise to a priori bounds on distances between curves that are $D$--related. In terms
    	of mapping classes and curves, this means precisely that if $g\in G[D]$ and if
    	$\gamma\in\mC_0(\Sigma)$, then there is a $K=K(D,\gamma)$ such that
    	$d_{\mC(\Sigma)}(\gamma,g(\gamma))\leq K$. A justification of this claim is a standard result about the geometry of
    	the curve graph, and will be revisited with proof in Lemma~\ref{lem-upper-bound-disp}.
    	
    	Expressing $\theta$ as a Boolean combination of relational formulae $R_w$, suppose
    	\[\{w_1,\ldots,w_k\}\] are the words which appear in relations, where each domain appearing
    	in each $w_i$ is proper. Let $\gamma\in\mC_0(\Sigma)$ be a simple closed curve.
    	From Lemma~\ref{lem-upper-bound-disp} below,
    	we have that there exists a bound $K=K(\gamma,w_1,\ldots,w_k)$ such that
    	\[\bigvee_{i=1}^k R_{w_i}(\zeta(\gamma),\zeta(\gamma'))\rightarrow
    	d_{\mC(\Sigma)}(\gamma,\gamma')\leq K.\] We remark that here we are slightly abusing
    	notation since $\zeta(\gamma)$ is an equivalence class, though it is not difficult
    	to see that a choice of representative within an equivalence class does not
    	effect any substantive change.
    	
    	By assumption, for any bound $K$, there exist curves
    	$\{\alpha,\beta,\alpha',\beta'\}$, with
    	\[d_{\mC(\Sigma)}(\alpha,\beta)> K,\quad d_{\mC(\Sigma)}(\alpha',\beta')> K,\]
    	and such that
    	$\mC(\Sigma)\models Q(\alpha,\beta)\wedge \neg Q(\alpha',\beta')$.
    	Thus, given a finite list of words $\{w_1,\ldots,w_k\}$ wherein all the domains
    	occurring in these words are proper, there exist curves
    	$\{\alpha,\beta,\alpha',\beta'\}$
    	such that \[\M^G\models \left(\bigwedge_{i=1}^k \neg R_{w_i}(\zeta(\alpha),\zeta(\beta))\right)
    	\wedge
    	\left(\bigwedge_{i=1}^k \neg R_{w_i}(\zeta(\alpha'),\zeta(\beta'))\right),\] and such that
    	$\mC(\Sigma)\models Q(\alpha,\beta)\wedge \neg Q(\alpha',\beta')$, and one can assume
    	that the distances between these curves in $\mC(\Sigma)$ exceed any prescribed bound.
    	
    	We are therefore reduced to two cases. If there exists a finite list of words
    	$\{w_1,\ldots,w_k\}$ such that
    	\[\M^G\models (\forall x\forall y)\left(\bigwedge_{i=1}^k \neg R_{w_i}(x,y)
    	\rightarrow \theta(x,y)\right),\] then there exist curves $\{\alpha,\beta,\alpha',\beta'\}$
    	such that
    	\[\M^G\models \theta(\zeta(\alpha),\zeta(\beta))\wedge
    	\theta(\zeta(\alpha'),\zeta(\beta')),
    	\quad
    	\mC(\Sigma)\models Q(\alpha,\beta)\wedge \neg Q(\alpha',\beta'),\]
    	which is a contradiction.
    	
    	Otherwise, we have \[\M^G\models (\forall x\forall y)(\theta(x,y)\rightarrow R_w(x,y)),\]
    	for some word $w$, wherein all domains are proper.
    	It follows in this case that there exist curves $\{\alpha,\beta,\alpha',\beta'\}$ such that
    	\[\M^G\models \neg\theta(\zeta(\alpha),\zeta(\beta))\wedge
    	\neg\theta(\zeta(\alpha'),\zeta(\beta')),
    	\quad
    	\mC(\Sigma)\models Q(\alpha,\beta)\wedge \neg Q(\alpha',\beta'),\] which is also
    	a contradiction.
    \end{proof}

    \newcommand{\rel}[3]{#1(#2)#3}
\section{The relational theory of $\M^G_{\D}(\Sigma)$}\label{sec:relational-theory}
  Throughout this section, $G\leq\Mod^{\pm}(\Sigma)$ denotes a finite index subgroup, unless
  otherwise noted.
  We now record a combinatorial calculus for manipulating the relations $R_w$
  occurring in the language $\mL:=\mL(\D)$, where $\D\subset\D_0$ is assumed to be $G$--invariant and downward closed.
  This section is rather technical, but the goal is simple: how does the relational theory of \(\M^G_{\D}\) determine the type of a tuple?

  \begin{table}[htp]
\centering
\small
\renewcommand{\arraystretch}{1.15}
\begin{tabular}{p{0.22\textwidth}p{0.70\textwidth}}
\hline
Notation & Meaning \\
\hline
$\Sigma$ & The fixed finite-type orientable surface. \\

$C(\Sigma)$, $C_0(\Sigma)$ & The curve graph of $\Sigma$, and its set of vertices. \\

$D_0=D_0(\Sigma)$ & The collection of all domains in $C_0(\Sigma)$. A domain is, informally, the set of curves carried by an essential subsurface, together with the relevant boundary curves. \\

$D_\Sigma$ or $D_{\Sigma'}$ & The domain associated to a subsurface: 
$D_{\Sigma'}:=C_0(\Sigma')\cup \partial \Sigma'$. \\

$D\subseteq E$ & Inclusion of domains as subsets of $C_0(\Sigma)$. \\

$D\perp E$ & Orthogonality of domains; geometrically, the corresponding subsurfaces are disjoint up to boundary. \\

$D\pitchfork E$ & Transversality of domains: the domains are non-orthogonal and incomparable. \\

$D^\perp$ & The orthogonal complement of $D$: the domain consisting of curves disjoint from $D$. \\

$\partial D$ & The boundary of $D$, defined as $D\cap D^\perp$: geometrically, they are the essential
boundary curves of the subsurface realizing $D$. \\

$G$ & A finite-index subgroup of $\operatorname{Mod}^{\pm}(\Sigma)$, fixed throughout the construction. \\

$G[D]$ & The subgroup $\operatorname{Stab}_G(D^\perp)$; equivalently, the subgroup of $G$ allowed to move inside the domain $D$. \\

$\mathcal D$ & A fixed $G$-invariant, downward-closed collection of connected domains. \\

$\mathcal A=\mathcal A_{\mathcal D}$ & The alphabet $\mathcal D\cup G$ used to form relation words. \\

\small{$w=(\delta_1, \ldots, \delta_k)$} & A finite word in the alphabet $\mathcal A$. Its letters may be domain letters or group letters. \\

$W(D)$ & The set of all words whose letters are properly supported in $D$. \\

$\mathcal D_\vee$ & The finite orthogonal closure of $\mathcal D$: its elements are joins
$D_1\vee\cdots\vee D_k$ of pairwise orthogonal connected domains
$D_i\in\mathcal D$.\\

$R_D$ & The binary relation on $G$ given by
$R_D(\sigma,\tau)$ iff $\sigma^{-1}\tau\in G[D]$. \\

$R_g$ & The binary relation on $G$ given by 
$R_g(x,y)$ iff $y=xg$. \\

$R_w$ & The binary relation associated to $w$:
$R_w=R_{\delta_1}\circ\cdots\circ R_{\delta_k}$. \\

$M^G_{\mathcal D}(\Sigma)$ & The $\mathcal D$-augmented Cayley graph of $G$; its universe is $G$, and its relations are the $R_D$, the $R_g$, and the word relations $R_w$. \\

$R_D^\ast$ & The generic $D$-relation: a type-definable strengthening of $R_D$ excluding all relations supported in proper subdomains of $D$. \\

$w_1\preceq w_2$ & The partial order on words obtained by replacing domain letters of $w_2$ by words supported inside them, up to quasi-permutation. \\

$\operatorname{Or}(w)$ & The ordinal associated to a word $w$, used to measure complexity and prove termination of reductions. \\
\hline
\end{tabular}
\caption{{Notation used in the relational theory of $M^G_{\mathcal D}(\Sigma)$.}}
\label{tab:section5-notation}
\end{table}

  \subsection{Preliminaries}\label{ss:5-prelim}

    We note some properties
    enjoyed by the relations $R_{w}$ in $\mM^G:=\M^G_{\D}$.
    All of them are elementary conditions, which is to  say they
    can be captured by a sentence of $\mL$ and thus will hold in any model of $\Th{\mM^G}$.
    
    \enum{(a)}{
    	\item \label{identity} $R_{1_{G}}$ is the identity relation, where $1_G$ denotes the identity element of $G$.
    	\item \label{composition} $R_{g,h}=R_{gh}$ for any $g,h\in G$.
    	\item \label{intersection}$R_{D}\cap R_{D'}=R_{D\cap D'}$.
    	\item \label{absorption} Suppose that $R_{D}(1_G,g)$ for $g\in G$. Then
    	$R_{D,g}=R_{g,D}=R_{D}$.
    	\item \label{inclusion}$R_{D}\subseteq R_{D'}$  if and only if $D\subseteq
    	D'$. In particular, in this situation $R_{D',D}=R_{D,D'}=R_{D'}$.
    	\item \label{commutation} Suppose $D_{1},D_{2}\in\D$ satisfy $D_{1}\perp
    	D_{2}$. If $x,y$ satisfy $R_{D_{1},D_{2}}(x,y)$, then also
    	$R_{D_{2},D_{1}}(x,y)$.
    	\item \label{conjugatoin} Suppose $D\in\D$ and $g\in G$ are arbitrary. Then
    	$R_{g,D}(x,y)$ is equivalent to $R_{g^{-1}(D),g}(x,y)$.}

    	\newcommand{\W}[0]{R^{*}}

    	Notice that Items (\ref{identity}) and (\ref{composition}) above imply that in any model of $\Th{\mM^G}$,
	for a  given $g\in G$ and $x\in M$,
    	there is a unique $y$ such that $R_{g}(x,y)$.
    	In this situation we might write $y=xg$ and $xG=\{xg\}_{g\in G}$.
    	Recall that $\alp=\alp_{\D}$ consists of the alphabet $\D\cup G$.
    	
    	\begin{definition}[Inclusion relation $\subseteq$ in $\alp$]
    		We extend the inclusion relation $\subseteq$ to  the whole of
    		$\alp$ by declaring the following:
    		\begin{itemize}
    			\item   $g\subseteq D$ if and only if $R_{D}(1,g)$;
    			\item $g\subseteq h$ if and only if $g=h$.
    		\end{itemize}
           ( Note that $g\subseteq D$ if and only if $g\in G[D]$.)
    	\end{definition}

    	Note that it is possible for $\Sigma_0=\Sigma$. By convention, we will declare the support of the identity element of $G$ to be empty. {Recall from Table~\ref{tab:orthogonality-guide} that two connected domain letters commute precisely when the corresponding domains are orthogonal.}
        
    	\begin{definition}[Orthogonal words]\label{def:orthogonal-words}
    		We say that $g\in G\setminus\{1\}$ is
    		\emph{orthogonal} to $D\in\D$, and write $$g\perp D,$$  if $g\in G[D^{\perp}]$.
    		Similarly, given $u\in\wo$ and $D\in\D$ we say that $u$ is
    		\emph{orthogonal} to $D$, and write
    		$$ u \perp D,$$
    		if each letter of $u$ is orthogonal to $D$, where orthogonality is as in Subsection~\ref{ss:ortho}.
    	\end{definition}

    	\begin{definition}[$R_D^\star$-relation]Let $D\in\D$ be given. Let \[\alp(D)=\{X\in\alp\,|\,X\subsetneq D\},\] and denote by
    		$\wo(D)$ the collection of all $w\in\wo$ containing only instances of letters in $\alp(D)$.
    		We will consider the following type-definable relation:
    		\begin{align}
    			\label{jump}\W_{D}(x,y)\equiv R_{D}(x,y)\wedge\bigwedge_{w\in\wo(D)}\neg R_{w}(x,y).
    		\end{align}
    		For $g\in G\setminus 1$, we set $\W_g=R_g$, and
            $\W_1$ is the equality relation.
    	\end{definition}
    	Observe that in particular that for all $D\in\D$, we have \[\W_D(a,b)\to \bigwedge_{g\in G}\neg R_g(a,b).\]
    	
    	The idea is that $\W_{D}(x,y)$ holds whenever $x,y$ are $R_{D}$-related in a
    	\emph{generic} fashion. In other words, they are not linked by any relation stronger
    	than $R_{D}$ (i.e.~with letters properly supported in
        $D$); we will later show that in this situation $R_{D}$ is in some sense not
    	only a minimal, but in fact {\it the} minimal relation between $x$ and $y$.
    	
    	Using well-established ideas from mapping class  group theory  (which we will
    	make  explicit  below),  it is  not difficult to see that the
    	$R_D^*$ is finitely  satisfiable  in the usual mapping class group. That is, for $\wo_0\subseteq\wo(D)$ finite,
    	there are mapping classes which are $D$--related but not
    	$w$--related for $w\in \wo_0$. However, it is generally not possible to realize $R_D^*$ in
    	the usual mapping class  group, and only in \emph{nonstandard models} whose existence follows from the compactness theorem. The relation $R_D^*$ cannot be
        realizable in the standard model, since mapping classes supported on a given
    	sufficiently complicated subsurface $\Sigma_0\subset\Sigma$ are often expressible
    	as products of mapping  classes
    	supported on proper subsurfaces of $\Sigma_0$.  This observation will be
    	later codified  by the statement that  $\W_D$ is consistent and non-algebraic;
    	see Corollary~\ref{c: consistency} and Lemma~\ref{l: main lemma}.

    \begin{example}[Generic steps in the running Example \ref{ex:running-separating-curve}]
    \label{ex:running-generic-step}
 A standard mapping
    class supported on $Y$ gives an $R_D$--step, but it does not literally realize
    $\W_D$ in the standard model. Indeed, if $f\in G[D]$ is an actual mapping
    class, then $R_f(1,f)$ holds, and the group letter $f$ is one of the lower
    relations forbidden by $\W_D$.

    The type $\W_D(1,x)$ should instead be regarded as a nonstandard generic
    move supported on $Y$: it is $D$--related to the identity, but it avoids every
    specified finite list of relations supported in proper subdomains of $D$. In
    practice, a finite fragment of $\W_D$ is realized by taking a mapping class
    supported on $Y$ with sufficiently large displacement in the curve graph
    $\mathcal C(Y)$. Compactness then gives a realization of the full type in a
    saturated elementary extension.

    The same applies to $\W_E$ on the $Z$ side and to $\W_A$ in the annular curve
    graph of $\alpha$. Thus a strict $DE$--sequence should be pictured as a generic
    move on the $Y$ side followed by a generic move on the $Z$ side, even though
    the  generic elements live in a saturated model rather than in the
    original mapping class group.
    \end{example}
    	
    	\begin{definition}[$w$-sequence]
    	Letting $\M:=\M_\D^G(\Sigma)$ for some suitable $\D$ and $G$, let $\mathcal N$ be a
    	model of $\Th{\M}$ with universe $N$.
    	Given $a,b \in N$ and $w=(\delta_1, \ldots, \delta_k) \in \wo$
    	such that $\mathcal{N} \models R_w(a,b)$,
    	we define a \emph{$w$-sequence} from $a$ to $b$ to be a sequence
    	\[s: a=a_{0},a_{1},\ldots, a_{k}=b\] such that
    	$a_i \in N$  for $0\leq i\leq k$, and such that
    	\[\mathcal{N} \models \bigwedge_{i=0}^{k-1}R_{\delta_{i+1}}(a_{i},a_{i+1}).\]
	\end{definition}
	
    	\begin{figure}[h!]
        \psfrag{a}{$a=a_0$}
    		\psfrag{b}{$a_1$}
    		\psfrag{c}{$a_{k-1}$}
    		\psfrag{d}{$a_k=b$}
    		\psfrag{R_1}{$R_{\delta_1}$}
    		\psfrag{R_2}{$R_{\delta_2}$}
    		\psfrag{R_k}{$R_{\delta_{k-1}}$}
    	\includegraphics[width=6cm]{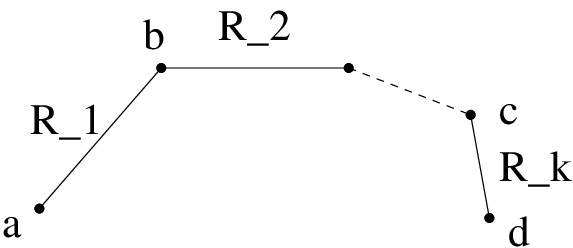}
        \caption{A $w$-sequence witnessing  $R_{w}(a,b)$}
    	\end{figure}
    	
    	A more refined notion of a $w$-sequence is the following:
    	\begin{definition}[Strict $w$-sequence]
    		Given $w\in\wo$, write
    		\[w=\delta_{1}\delta_{2}\cdots \delta_{k},\] with $\delta_i\in\alp$ for each $i$, and let $\W_{w}(x,y)$ stand for the
    		composition of relations \[\W_{\delta_{1}}\circ\W_{\delta_{2}}\circ\cdots \circ
    		\W_{\delta_{k}}.\]
    		In other words,
    		$\W_{w}(x,y)$ is equivalent to the existence of a sequence
    		\[x_{0}=x,x_{1},\ldots, x_{k}=y\] such that $\W_{\delta_{i}}(x_{i-1},x_{i})$
    		for all $1\leq i\leq k$. We will refer to any sequence of the form $\{x_0,\ldots,x_k\}$ as a
    		\emph{strict} \emph{$w$-sequence} from $x$ to $y$. If all the letters of $w$ are domains, i.e.  no group letter appears, we will call it a \textbf{strict domain-only} $w$-\textbf{sequence}.
    	\end{definition}
    	
    	Though it is not entirely obvious from the definition,
    	one can use compactness to show that $\W_w$ is expressible as a
    	parameter--free
    	type-definable condition as well.
    	
    	\begin{remark}[Terminology for strict sequences]
Throughout the paper we use the phrase \emph{strict $w$--sequence} for a
sequence witnessing the type-definable relation $R_w^*$. We will not use the
alternative phrase ``$w$--strict sequence''. When all letters of $w$ are
domain letters, i.e. when $w\in\D^*$, we call the sequence a
\emph{strict domain-only $w$--sequence}. The hull constructions in
Sections~\ref{sec:rel-qe}, \ref{sec:full-qe}, and \ref{sec:int-rigid} are
built only along strict domain-only sequences; group letters are dealt with
separately by moving endpoints within their $G$--orbits or by putting words in
normal form.
\end{remark}

    	\begin{remark}
    		We will use an expression such as $\phi(x,y)\in\W_{D}(x,y)$ 
            to indicate that $\phi$ is a finite
    		conjunction of formulae in \[\{R_{D}(x,y),\neg R_{w}(x,y)\,|\,w\precneq D\}.\]
    		Thus, the inclusion relations have as  their target the collection
    		of all finite conjunctions of relations
    		$\W_{D}(x,y)$.
    	\end{remark}

    	\begin{definition}[The inclusion relation $\subseteq$ on $\wo$]
    		Given $u,v\in\wo$ we will write $u\subseteq v$ if we can write
    		\[R_v=R_{v_{1}v_{2}\cdots v_{k}}\quad \textrm{with}\quad v_{j}\in\alp,\quad \textrm{and}\quad
    		R_u=R_{u_{1}u_{2}\cdots u_{k}},\] where
    		for all $1\leq j\leq k$ either:
    		\begin{itemize}
    			\item $v_{j}\in G$ and $u_{j}=v_{j}$
    			\item $v_{j}\in\D$ and either $u_{j}=v_j$ or $u_{j}\in\wo(v_j)$
    		\end{itemize}
    		
    		We write $u\subset v$ if $v_j\in\D$ and if $u_{j}\in\wo(v_j)$ for
    		at least one index $j$ above.
    	\end{definition}
    	
    	\begin{observation}
    		Let $w\in\wo$ and suppose that $\{z_{i}\}_{i=0}^{k}$ is
    		a $w$-sequence from $x$ to $y$. Then either $\{z_{i}\}_{i=0}^k$ is strict
    		and thus lies in $\W_{w}(x,y)$, or there exists a $w'\subseteq w$ (possibly equal to $w$)
    		and a finite refinement
    		$\{z'_{i}\}_{i=0}^{k'}$ of $\{z_{i}\}_{i=0}^k$ which is a $w'$-sequence from $x$ to $y$.
    	\end{observation}

        Note that the consecutive points in the refinement $\{z'_{i}\}_{i=0}^{k'}$ may not lie in the relations
        $R_g$ or $R_D$ for any domain or group element, even if they do for $\{z_{i}\}_{i=0}^k$; for instance,
        it may be the case that $R_D(z_i,z_{i+1})$ but these
        two points are not strictly $D$--related, so that
        $R_w(z_i,z_{i+1})$ for some potentially very long $w\in\wo(D)$. The word $w$ still satisfies $w\subseteq D$
        since $R_D=R_{DD}$.
    	In particular, we have:
    	\begin{lemma}[Refining a $w$-sequence to a strict sequence]
    		\label{strictness}
    		Let $\mathcal N$ be a model of $\Th{\M^G}$ with universe $N$, and let $x,y\in N$. Given a $w$-sequence
    		$\{z_{i}\}_{i=0}^k$ between points $x$ and $y$, there exists a
    		finite refinement $\{z'_{i}\}_{i=0}^{k'}$ of $\{z_{i}\}_{i=0}^k$ and a
    		$w'\subseteq w$ such that
    		$\{z'_{i}\}_{i=0}^{k'}$ is a strict $w'$-sequence from $x$ to $y$.
    	\end{lemma}

\begin{table}[htbp]
\centering
\small
\renewcommand{\arraystretch}{1.25}
\setlength{\tabcolsep}{4pt}
\begin{tabular}{|p{3.1cm}|p{5.5cm}|p{5.2cm}|}
\hline
\textbf{Terminology}
&
\textbf{Condition}
&
\textbf{Role}
\\
\hline

\(w\)--sequence with \(w=\delta_1\cdots\delta_k\)
&
A sequence
$a=a_0,a_1,\ldots,a_k=b$
such that
$R_{\delta_i}(a_{i-1},a_i)$
for $1\leq i\leq k$.
&
Witnesses the relation \(R_w(a,b)\). The single steps need not be
strict.
\\
\hline

Strict \(w\)--sequence
&
A sequence
$a=a_0,a_1,\ldots,a_k=b$
such that
$R_{\delta_i}^{\ast}(a_{i-1},a_i)$
for $1\leq i\leq k$.
&
Witnesses \(R_w^{\ast}(a,b)\). Each step realizes its
strict relation.
\\
\hline
Strict refinement
&
A \(w\)--sequence is replaced by a finite strict \(w'\)--sequence for
some \(w'\subseteq w\), as in Lemma~\ref{strictness}.
&
Replaces non-strict steps by strict sequences supported in the same
letters or in smaller domains.
\\
\hline
\end{tabular}
\caption{The sequence terminology used in the word calculus.}
\label{tab:w-sequences}
\end{table}
        
    Table~\ref{tab:w-sequences} summarizes the sequence terminology introduced above. Although every \(w\)--sequence admits a strict refinement, the existence of a strict \(w\)--sequence between \(x\) and \(y\) does not in general determine \(\tp(x,y)\). In order to obtain such control, we must require the word \(w\) to contain no
redundancies in the sense explained below.

  	\subsection{Word equivalence and cancellation}
    	
    	\newcommand{\woi}[0]{\W^{ind}}
    	\newcommand{\woc}[0]{\wo^{c}}
    	
    	From now on, we  will
    	implicitly assume (unless otherwise noted) that each word $w$ has \emph{connected letters},
    	that is if $D\in\D$ occurs in
    	$w$  then $D$ is a connected domain.
    	Consider a word \[w=(\delta_{1},\delta_{2},\delta_{3},\ldots,
    	\delta_{k})\in\wo.\] We can obtain a new $w'\in\wo$ by applying one of the following
    	 moves to $w$.
    	\begin{itemize}
    		\item [(Rm)] \label{removing1} Removing an instance of $1_{G}$.
    		\item  [(Cmp)]\label{merging} Replacing a subword of the form $g_{1},g_{2}$ by
    		$g_{1}g_{2}$.
    		\item [(Swp)]\label{swap}Replacing a subsequence of the form $(D_{1},D_{2})$ where
    		$D_{1}\perp D_{2}$ by $(D_{2},D_{1})$. We sometimes call this a \emph{transposition}.
    		\item [(Jmp)]\label{jumping} Replacing a subsequence of the form
    		$(D,g)$ by $(g,g^{-1}(D))$ or vice versa.
    		\item [($\mathrm{Abs}_{G}$)]\label{absorption1} Replacing a
    		subword of the form $(g,D)$ or $(D,g)$ with $D$, provided
    		$g\subseteq D$.
    		\item [($\mathrm{Abs}_{\subset}$)]\label{absorption2}
    		Replacing a subsequence $(D,D')$ or $(D',D)$  by $D$ in case $D'\subsetneq D$.
    		\item [($\mathrm{Abs}_{=}$)]\label{cancellation1} Replacing a subsequence
    		of the form $(D,D)$ by $D$.
    		
    	\end{itemize}
    	     Table \ref{tab:word-moves} summarizes the elementary moves introduced above. It is meant only as a guide for the reader, the precise meaning will be explained afterwords.

\begin{table}[htbp]
\centering
\footnotesize
\renewcommand{\arraystretch}{1.18}
\setlength{\tabcolsep}{2.5pt}
\begin{tabular}{|>{\centering\arraybackslash}p{1.35cm}|
                >{\raggedright\arraybackslash}p{2.45cm}|
                >{\raggedright\arraybackslash}p{2.65cm}|
                >{\raggedright\arraybackslash}p{2.95cm}|
                >{\raggedright\arraybackslash}p{2.35cm}|}
\hline
\textbf{Move} &
\textbf{Form} &
\textbf{Informal meaning} &
\textbf{Effect on the relation} &
\textbf{Example} \\
\hline

$(\mathrm{Rm})$
&
Delete or introduce the identity letter $1$.
&
Removes a vacuous step.
&
Does not change $R_w$.
&
{\scriptsize $(D,1,E)\sim (D,E)$.} \\
\hline

$(\mathrm{Cmp})$
&
Replace adjacent group letters by their product.
&
Compresses consecutive group moves.
&
Does not change $R_w$.
&
{\scriptsize $(g,h,D)\sim (gh,D)$.} \\
\hline

$(\mathrm{Swp})$
&
Swap adjacent orthogonal domains.
&
Independent domain moves commute.
&
Does not change $R_w$.
&
{\scriptsize If $D\perp E$, then $(D,E)\sim (E,D)$.} \\
\hline

$(\mathrm{Jmp})$
&
Move a group element past a domain.
&
Encodes equivariance under the group action.
&
Does not change $R_w$.
&
{\scriptsize $(g,D)\sim (gD,g)$; equivalently $gR_Dg^{-1}=R_{gD}$.} \\
\hline

$(\mathrm{Abs}_G)$
&
Absorb a group element supported in a domain.
&
Treats a standard $D$-supported move as part of a generic $D$-move.
&
Does not change $R_w$; used in quasi-permutation.
&
{\scriptsize If $g\subset D$, then $(g,D)\simeq D$.} \\
\hline

$(\mathrm{Abs}_{\subset})$
&
Replace a smaller domain by a larger one.
&
A $D$-move is also supported in any larger domain $E$.
&
Usually passes to a coarser relation.
&
{\scriptsize If $D\subset E$, then $D$ may be absorbed into $E$.} \\
\hline

$(\mathrm{Abs}_{=})$
&
Replace two equal adjacent domain letters by one copy.
&
Removes repetition of the same domain letter.
&
Usually passes to a coarser relation; used to detect reductions without using cancellation.
&
{\scriptsize $(D,D)$ may be replaced by $D$.} \\
\hline

$(\mathrm{C})$
&
Replace cancellable domain letters by a word supported in that domain.
&
Controls cancellation inside one support.
&
Not merely bookkeeping; used to produce reductions.
&
{\scriptsize A subword $(D,D)$ may be replaced by a word supported in $D$.} \\
\hline
\end{tabular}

\caption{{Summary of the elementary moves on words. The first group of moves should be viewed as bookkeeping moves: they change the representative of a word but not the relation it defines. Absorption/Cancellation moves are used to compare words and to pass to reductions.}}
\label{tab:word-moves}
\end{table}

    	We will occasionally say that two letters in a word $w\in\wo$ which can be transposed
    	\emph{commute} with each other. We remark that this is slightly misleading terminology,
    	since a domain is not orthogonal to itself and hence does not commute with itself in the
    	sense of transpositions. The move $\mathrm{Abs}_{=}$ provides a suitable
    	framework for dealing with instances of the same domain within a word.
    	We remark that if $D$ is a
    	non-annular domain and $E$ is an annular domain that is peripheral
    	to $D$, then $D$ and $E$ commute with each other.
        
        {The elementary moves below play two different roles. 
The moves $(\mathrm{Rm})$, $(\mathrm{Cmp})$, $(\mathrm{Swp})$, $(\mathrm{Jmp})$, and $(\mathrm{AbsG})$ are bookkeeping moves: they change the spelling of a word without changing the generic relation represented by it. 
The absorption moves $(\mathrm{Abs}\subset)$ and $(\mathrm{Abs}=)$, and the cancellation move $(\mathrm{C})$ introduced below, are reductions which actually simplify: they remove redundancies in the domain part of the word. 
The distinction between these two kinds of moves is the reason for separating quasi-permutation from reduction.}
    	
    	\begin{definition}[Permutation of $w$]
		We say that $w$ is a \emph{quasi-permutation} of $w'$ if it can be obtained from $w$
		by repeated applications of
		\[
		\{(\mathrm{Swp}), (\mathrm{Jmp}), (\mathrm{Cmp}), (\mathrm{Rm}), (\mathrm{Abs}_{G})\}
		\]
		and their inverses. We denote the fact that $w_{1}$ and $w_{2}$ are quasi-permutations of each other by $w_{1}\simeq w_{2}$. Given a word $w$, we write $[w]$ for the equivalence class of $w$ under this equivalence relation.
	\end{definition}

	The inclusion of $(\mathrm{Abs}_{G})$ in quasi-permutation is only a bookkeeping convention. Thus, if $g\subseteq D$, the words $(g,D)$ and $D$ lie in the same quasi-permutation class and define the same generic relation. This convention removes the otherwise artificial ambiguity in statements such as Lemma~\ref{uniqueness}; the group letter $g$ contributes no new displacement once it is supported in the adjacent domain $D$.

    \begin{example}[Word Calculus on the running Example~\ref{ex:running-separating-curve}]
    \label{ex:running-word-calculus}
    Since
    $A\subset D$ and $A\subset E$, the words
    $ AD, DA, DAD$
    all reduce, without cancellation, to words in which the annular letter $A$ has
    been absorbed into $D$. Similarly, $AE$ and $EA$ reduce to $E$. Geometrically,
    a twist about the boundary curve $\alpha$ is already allowed inside a move
    supported on either side of $\alpha$. On the other hand, $D$ and $E$ are orthogonal, so 
        $DE\simeq ED.$
    The word $DE$ is nevertheless reduced: neither $D$ nor $E$ contains the other,
    and no cancellation move applies. It represents the operation of moving
    generically on the $Y$ side and generically on the $Z$ side. The order of these
    two moves does not matter, because the two sides are disjoint up to the common
    boundary curve.

    Finally, if $T_\alpha$ is the Dehn twist about $\alpha$, then powers of
    $T_\alpha$ are absorbed next to either $D$ or $E$:
    \[
        T_\alpha^nD\simeq D,\qquad DT_\alpha^n\simeq D,\qquad
        T_\alpha^nE\simeq E,\qquad ET_\alpha^n\simeq E.
    \]
    This is the simplest example of why the moves $(\mathrm{Abs}_{G})$ and
    $(\mathrm{Abs}_{\subset})$ are not cosmetic: without them, the same geometric
    move would have many artificial word representatives. It is also a useful
    warning for Lemma~\ref{uniqueness}: uniqueness of a reduced word class is true
    only after these  representatives have already been identified. 
    \end{example}

    	Finally, consider the following class of moves:
    	\begin{itemize}
    		\item [($\mathrm{C}$)]\label{cancellation2} Replacing a
    		subsequence of the form $(D,D)$ by a word in $\wo(D)$.
    	\end{itemize}
    	
    	We refer to all the operations listed above as \emph{elementary moves}.

        The moves $(\mathrm{Rm})$, $(\mathrm{Cmp})$, $(\mathrm{Swp})$, $(\mathrm{Jmp})$, and $(\mathrm{Abs}_{G})$ are bookkeeping moves: they do not change the relation $R_w$ interpreted by the word.  The absorption moves $(\mathrm{Abs}_{\subset})$ and $(\mathrm{Abs}_{=})$ pass to a coarser relation, while the move $(\mathrm C)$ records the possibility that two consecutive copies of the same domain are replaced by a word supported in that domain.  Lemma~\ref{equivalence} below records the precise logical effect of these operations.

    	\begin{definition}[Reduced word]
		We say that $w$ is \emph{reduced} if there is no quasi-permutation of $w$ to which
		\[
		\{(\mathrm{Abs}_{\subset}),(\mathrm{Abs}_=),(\mathrm{C})\}
		\]
		can be applied, i.e.~if no quasi-permutation of $w$ contains two adjacent domain letters of the form $D,E$ with $D\subseteq E$.
	\end{definition}

	Thus reducedness is a property of the quasi-permutation class $[w]$. Applications of $(\mathrm{Abs}_{G})$ are treated as minor changes of representative and are never counted as strict reductions.
    	
    	\begin{definition}[Reduction of words, concatenation]
    		We say that $w'$ is a \emph{reduct} of $w$ (equivalently, $w$ reduces to $w'$) if
    		$w'$ can be obtained from $w$ by a successive application of elementary moves. If $w'$ is obtained from $w$ with
    		at least one application
    		one of the moves \[\{(\mathrm{Abs}_{\subset}),(\mathrm{Abs}_=),(\mathrm{C})\},\] then we say that $w'$ is a
    		\emph{strict reduct} of $w$.
    		We say that $w$ reduces to $w'$ \emph{without cancellation} if it reduces
    		to $w'$ applying only the moves $(\mathrm{Rm})$, $(\mathrm{Cmp})$, $(\mathrm{Swp})$, $(\mathrm{Jmp})$, $(\mathrm{Abs}_{G})$, $(\mathrm{Abs}_{\subset})$, and $(\mathrm{Abs}_{=})$; equivalently, without applying $(\mathrm C)$. Here, cancellation refers specifically to the replacement of two consecutive copies of a domain by a word supported in that domain; absorption moves may still occur.
    		    		Words $w_1$ and $w_2$ can be combined to a word $w_1w_2$ by \emph{concatenation}.
    	\end{definition}
    	
    	{Thus quasi-permutation identifies different spellings of the same generic relation, reduction removes redundancies, strict reduction means that a  domain simplification has occurred, and reduction without cancellation allows absorption but forbids the cancellation move $(\mathrm C)$.}

    	\begin{lemma}[Existence of reduced words]\label{lem:reduced-unique}
    		Up to quasi-permutation, there is a unique reduced word $w'$ that can be obtained
    		from any given word $w\in\wo$, provided the reduction is performed without
    		cancellation.
    	\end{lemma}
    	
    	\begin{proof}
    		This follows from a fairly straightforward induction on the length of $w$. The interested reader may find a detailed proof in Proposition 5.3
    		of~\cite{Baudisch-ample}.
    	\end{proof}

    	\begin{definition}[Partial order on words]
    		We say that $w_{1}\preceq w_{2}$ if for some
    		(possibly empty) collection of domains
    		$\{D_1,\ldots,D_k\}$
    		appearing in $w_2$, we can replace each $D_i$ by a word in $\wo(D_i)$ and
    		apply  a quasi-permutation to obtain $w_1$.
    	\end{definition}
    	Note that if $w_{1}$ is a reduct of $w_{2}$, then $w_{1}\preceq w_{2}$, but the converse is not true in general. We will be mostly interested in the case in which $w_{2}$ is already reduced.
    	The  partial order on words, together with the  complexity $k(D)$ of a
    	domain, allow us to define the associated ordinal of a word.
    	
    	\begin{definition}[Associated ordinal]
    		Let $w\in\wo$. We define the \emph{associated ordinal}
    		$\Or(w)$ of $w$ inductively.
    		\begin{enumerate}
    			\item
    			For $D=\varnothing$, we  define  $\Or(D)=0$.
    			\item
    			For $g\in G$, we define $\Or(g)=0$.
    			\item
    			For $\varnothing\neq D\subsetneq \mC_0$ a domain of complexity $k$,
    			we  define  $\Or(D)=\omega^k$.
    			\item
    			Let $w_1,w_2\in\wo$ be words for which $\Or$ is defined. We define the associated ordinal of
                the concatenation by
    			\[\Or(w_1w_2)=\Or(w_1)\oplus\Or(w_2);\] that is, we take the symmetric
    			sum  of the two ordinals.
    		\end{enumerate}
    	\end{definition}
    	
    	We recall for the convenience of the reader that the symmetric sum of two ordinals $\mu$ and $\nu$ is the order type of the
    	ordered pair $(\mu,\nu)$, where $\mu\geq \nu$.
    	
    	We remark that $\omega^{k_{\max}}$ is the foundation  rank of the partial
    	order on  words.
    	Notice that the moves
    	\[\{(\mathrm{Rm}),(\mathrm{Cmp}),(\mathrm{Swp}),(\mathrm{Jmp}), (\mathrm{Abs}_G)\}\]
    	preserve $\Or(w)$, while the moves
    	\[\{(\mathrm{Abs}_{\subset}),(\mathrm{Abs}_{=}),(\mathrm{C})\}\]
    	strictly decrease it.

    	\begin{observation}
    		\label{properly descending}
    		If $w_1$ is a reduct of $w_2$ then $\Or(w_{1})\leq\Or(w_{2})$. If $w_1$ is a strict reduct of $w_2$ then
    		$\Or(w_{1})<\Or(w_{2})$.
    	\end{observation}

{The following example separates the three phenomena:
commuting independent letters, 
absorbing redundant domain letters, 
cancelling two equal domain letters.
Bookkeeping moves such as quasi-permutation preserve the associated ordinal
$\operatorname{Or}(w)$, whereas actual reductions decrease it.}

\begin{example}[Word operations on the running example \ref{ex:running-separating-curve}]
\label{ex:running-word-operations}
{We continue the discussion from Example \ref{ex:running-word-calculus}.} Thus $A$ is the annular domain about
the separating curve $\alpha$, with $A\subset D$ and $A\subset E$, while
$D\perp E$.
The words
\[
DE \quad\text{and}\quad ED
\]
are quasi-permutations of one another: the two domain letters commute because
the domains are orthogonal. This is only a change of representative, not a
reduction.

By contrast, the words
\[
AD,\quad DA,\quad DAD
\]
reduce to $D$, and similarly $AE$ and $EA$ reduce to $E$. These are strict
reductions, since the annular letter $A$ is absorbed into the larger domain letter.
They are reductions without cancellation, because no move of type $(\mathrm C)$
is used.
\end{example}

  	\subsection{Straightening paths}
    	In this brief subsection, we elaborate on the relationship between the combinatorics of elementary moves and the structure of
    	models of $\Th{\M^G}$.
    	
    	Let $\mathcal N$ be a model of $\Th{\M^G}$ with universe $N$.
    	Let $u\in\wo$ and suppose that $\overline a=\{a_{i}\}_{i=0}^{k}\subset N$ is a $u$-sequence
    	from $a_{0}$ to $a_{k}$. That is, $u=u_1\cdots u_k$, and $R_{u_i}(a_{i-1},a_i)$ for all $i$.
    	We assume that $u$ contains a subword that admits the application of
    	some elementary move as above.
    	
    	If (Rm) can be applied to $u_{i}=1_{G}$, then there is a repetition $a_{i}=a_{i+1}$, and the result of deleting $a_{i}$
    	yields a $u'$ sequence from $a_0$ to $a_k$.

    	\begin{lemma}\label{operations strict sequences}
    	  Assume that $\{a_{i-1},a_{i},a_{i+1}\}$ is a subsequence
    	  corresponding to the subword $u_{i}u_{i+1}$, where $u_{i}u_{i+1}$ admits the application of  an elementary  move.
    	  Let $\overline a'$ be the result of deleting $a_{i}$ from the sequence $\overline a$.

        The following hold,
    	  according to the  types of moves  that  can be
    	  applied:    	
    	 \begin{enumerate}
    		\item \label{merging case}In case of
    		\[\{(\mathrm{Cmp}),(\mathrm{Abs}_{G}),(\mathrm{Abs}_{\subset})\},\]
    		the sequence $\overline a'$ is an $u'$-sequence from $x$ to $y$, where $u'$ is the result of the application of the particular move to $u$, along the
    		subword $u_iu_{i+1}$. In this case $\overline a'$ is a strict $u'$-sequence if $\overline a$ is a strict $u$-sequence.
    		\item \label{swap case}In  case of $(\mathrm{Swp})$ or $(\mathrm{Jmp})$  being
    		applied to the subword $u_{i}u_{i+1}$,
    		then by applying the commutation or conjugation properties of the relation $R_u$,
    		there exists an $a'_{i+1}\in N$ such that the result of replacing $a_{i+1}$ by
    		$a'_{i+1}$ is an $u'$-sequence, where $u$ is the result of applying
    		$(\mathrm{Swp})$ or $(\mathrm{Jmp})$ to $u$, respectively.
            
    		If $\overline a$ is  strict,  then so is any such  $\overline a'$.
    		\item \label{cancellation case}If $(u_{i},u_{i+1})=(D,D)$ for some $D\in\D$, then $\overline a'$
    		is an $u_{0}$-sequence from
    		$a_0$ to $a_k$, where $u_{0}$ is the result of applying move
    		$(\mathrm{Abs}_{=})$ to $u_{i}u_{i+1}$.
            
            If
    		$\overline a$ is strict there are two mutually incompatible situations:
    		\begin{itemize}
    			\item If $\W_{D}(a_{i},a_{i+2})$ then $a'$ is a strict $u_{0}$-sequence
    			from $a_0$ to $a_k$.
    			\item  $R_{w}(a_{i},a_{i+2})$ for some $w\in\wo(D)$, in which case $\overline a'$
    			is a (possibly non-strict)
    			$u_{1}$-sequence from $a_0$ to $a_k$, where $u_{1}$ is the result
    			of replacing $DD$ by $w$ via a move of type $(\mathrm{C})$.
    		\end{itemize}
    	\end{enumerate}
    	\end{lemma}
    
      \begin{proof}
      	For \ref{merging case} we consider only the move $(\mathrm{Abs}_{\subset})$, since the others can be handled in a very similar way. 
      	Here we are given $a_{i},a_{i+1},a_{i+2}$ with $R_{E}(a_{i},a_{i+2})$ and $R_{D}(a_{i+1},a_{i+2})$ and without loss of generality we may assume that $E\subsetneq D$. It is clear that $R_{D}(a_{i},a_{i+2})$, since $R_{E}$ is a refinement of $R_{D}$ and both are equivalent relations. It remains to show the claim regarding strictness. Assume, for the sake of contradiction that $\overline a$ is strict, but $\overline a'$ is not. Then $R_{w}(a_{i},a_{i+2})$ for some $w\subseteq D$. But then 
      	$R_{Ew}(a_{i+1},a_{i+2})$ holds, contradicting the strictness assumption $\W_{D}(a_{i+1},a_{i+2})$.

       For case \ref{swap case} we focus on the move $(\mathrm{Swp})$, since the proof of the claim is even easier for $(\mathrm{Jmp})$. For that purpose, fix domains $D_{1}$, $D_{2}$ which are orthogonal and not contained in each other. The following is clear: 
       \begin{observation}\label{obs:orthogonal-strict}
       	Let $E_{1}$ and $E_{2}$ be orthogonal domains. Suppose that we are given $u$ with $u\subseteq E_{1}$ and points $c,c',c''$ with $R_{u}(c,c')$ and $R_{E_{2}}(c',c'')$. Then there is a $d'$ with $R_{E_{2}}(c,d')$ and $R_{u}(d',c'')$. 
       \end{observation}
        Existence of the path $\overline a'$ in the statement follows as a particular case. For the claim about strictness, suppose for the sake of contradiction that $a_{i},a_{i+1},a_{i+2}$ is a strict $(D_{1},D_{2})$-sequence for domains $D_{1}, D_{2}$ orthogonal and not related by inclusion, and $(a_{i},a'_{i+1},a_{i+2})$ a $(D_{2},D_{1})$-sequence which fails to be strict. Without loss of generality, assume that $R_{u}(a'_{i+1},a_{i+2})$ for some $u\subsetneq D_{1}$. 
        
       An application of Observation~\ref{obs:orthogonal-strict} above implies the existence of $a''_{i+1}$ satisfying $R_{D_{1}}(a_{i},a''_{i+1})$ 
       and $R_{u}(a''_{i+1},a_{i+2})$. We claim that this contradicts our strictness assumption for the original sequence $(a_{i},a_{i+1},a_{i+2})$. 
       Indeed, by transitivity we have $R_{D_{1}}(a_{i+1},a''_{i+1})$ and $R_{D_{2}}(a_{i+1},a''_{i+1})$. It thus follows that $R_{D_{1}\cap D_{2}}(a_{i+1},a''_{i+1})$; cf.~the beginning of Subsection~\ref{ss:5-prelim} above. Writing $E=D_{1}\cap D_{2}$ we have that $a''_{i+1}$ witnesses that $R_{Eu}(a_{i+1},a_{i+2})$. 
       But $D_{1}\nsubseteq D_{2}$, so $E\subsetneq D_{2}$ and thus $Eu\subsetneq D_{2}$, which means $a_{i+1}$ and $a_{i+2}$ are not strictly $D_{2}$-related.

        Case \ref{cancellation case} is similarly easy and we omit further details.
      \end{proof}

       \begin{remark}
    		\label{any sequence} Notice that starting from any (strict)
    		$u$-sequence from $a_0$ to $a_k$, the lemma above allows one to lift each of the words $u_{i}$ in some reduction process $(u_{i})_{i=0}^{m}$ for $u$, where here $u_0=u$ and $u_m$ some reduced representative, to some $u_{i}$-sequence $\overline a^{i}$. Here,
            $\overline a^{m}$ is a strict $u_{m}$-sequence and each $\overline a^{i+1}$ differs from $\overline a^{i}$ only in positions corresponding to the  elementary move bringing $u_{i}$ to $u_{i+1}$. 
            This freedom of choice will be of
    		importance for Lemma \ref{gate property}.
    	\end{remark}

        \begin{comment}
    	\begin{remark}
    		\label{any sequence} Notice that given $u$ the reduction process sketched
    		above works starting from any (strict)
    		$u$-sequence from $a_0$ to $a_k$. This freedom of choice will be of
    		importance for Lemma \ref{gate property}.
    	\end{remark}
        \end{comment}
        
    	The
    	following lemma is straightforward, and we omit  a proof.
    	
    	\begin{lemma}
    		\label{equivalence}If $w$ can be obtained from $w'$ by repeated application
    		of the moves
    		from $(\mathrm{Rm})$ to $(\mathrm{Abs}_{=})$
    		and their inverses, then
    		\begin{align*}
    			\Th{\M^G}\vdash (\forall x\forall y)(R_{w}(x,y)\leftrightarrow
    			R_{w'}(x,y)).\\
    		\end{align*}
    		If $w\simeq w'$, then
    		\begin{align*}
    			\Th{\M^G}\vdash (\forall x\forall y)(\W_{w}(x,y)\leftrightarrow
    			\W_{w'}(x,y)).
    		\end{align*}
    		
    		For all $w,w'\in\wo$, if $w$ reduces to $w'$ then
    		\begin{align*}
    			\Th{\M^G}\vdash (\forall x\forall y)(R_{w'}(x,y)\rightarrow R_{w}(x,y)).
    		\end{align*}
    	\end{lemma}
    	
    	Combining Lemma \ref{strictness} with Observation \ref{properly descending},
    	one gets the following consequence:
    	\begin{corollary}\label{reduct cor}
    		Let $\mathcal{N}$
    		be  a  model of $\Th{\M^G}$ with universe $N$, and let $a,b\in N$  be
    		such that $R_{u}(a,b)$ for a  suitable $u\in\wo$.
    		Then there exists a reduct $w\in\wo$ of $u$ such that
    		$\W_{w}(a,b)$.
    	\end{corollary}

    	The following lemmas establish the predictability of the elementary moves
    	as operations  on  words in $\alp$. Let $D\in \D_0$ be a domain, and let $w$ be a word such
    	that $D\in w$. Note that under an application of an elementary move $w\mapsto w'$, the letter $D$ is unchanged, undergoes a swap,
    	is replaced by a domain in the $G$--orbit of $D$, is absorbed, or undergoes cancellation.

    	\begin{lemma}\label{lem:group-cancel}
    		Let $u,v\in\wo$ be reduced, and suppose that in the concatenation $uv$, we have that a letter $D$ of $u$ and a letter $E$ of $v$
    		undergo a cancellation or an absorption. Then for any quasi-permutation $u'\simeq u$ and any quasi-permutation $v'\simeq v$, the
    		corresponding letters $D'$ and $E'$ also undergo a cancellation or an absorption in the concatenation $u'v'$.
    	\end{lemma}
    	\begin{proof}
    		Write \[u=g_0D_1g_1D_2\cdots D_k g_{k}\] and
    		\[v=h_0E_1h_1E_2h_2\cdots E_jh_j,\] and suppose that $E=E_i$ cancels with or is absorbed by $D=D_j$. In order to
    		move $E_i$ into a position that is adjacent to $D_j$, we have that $E_i$ is conjugated by group elements via the swap move,
    		and the corresponding conjugates are either orthogonal to or absorb any domains that are encountered between $D_j$ and $E_i$.
    		
    		Observe that the conclusion of the lemma is straightforward, provided that no new domains or group elements appear
    		via the inverses of the absorption and composition rules.
    		
    		Appearances of new group elements occur via the inverse of the absorption
    		rule, or by the inverse of composition (i.e.~$gg^{-1}=1$). If a group element appears via the inverse of the absorption rule
    		with a domain $F$,
    		then this group element will be orthogonal to $E$ provided that $F$ is. If a pair of group elements $gg^{-1}$ appear
    		through the inverse of the composition rule, so that \[uv\simeq \cdots D\cdots gg^{-1} \cdots E\cdots,\] then in order to
    		move $E$ into a position adjacent to $D$ via absorptions, swaps, and commutations with orthogonal domains, we have that
    		either both or neither of $D$ and $E$ will be conjugated by $g$. This is simply because $E$ is absorbed by or cancels with $D$
    		if and only if $g(E)$ is absorbed by or cancels with $g(D)$.
    		
    		Appearances of new domains can occur only through inverses of the absorption rule, and if a domain yields a domain $FF'$
    		via the inverse of absorption then a domain $D$ will be orthogonal to $FF'$ if it is orthogonal to $F$. The conclusion of the lemma
    		is now clear.
    	\end{proof}
    	
    	We say that the displayed occurrence of $D$ \emph{survives in any reduct of $w$} if whenever a sequence of elementary moves is applied to $w$,
    	no absorption or cancellation involving that occurrence of $D$ occurs.
    	
    	\begin{lemma}
    		\label{surviving D}Suppose that a word $w$ can be written as $w=uDv$, where
    		$uD$ and $Dv$ are both reduced. Then the displayed occurrence of $D$ in $w$ survives
    		in any reduct of $w$. In particular $w$ cannot reduce to the identity.
    	\end{lemma}
    	\begin{proof}
    		Write $uD=g_1D_1g_2D_2\cdots g_kD_kD$ and 
    		$Dv=DE_1h_1E_2h_2\cdots E_jh_j,$ where $g_i,h_i\in G$, 
    		$D_i,E_i\in\mathcal{D}$, and both of these words are reduced up
    		to applications of the move (Rm). Suppose by contradiction that $u$ and $v$ are a counterexample to the lemma,
    		chosen to minimize $\Or(u)\oplus\Or(v)$, with the case $\Or(u)\oplus\Or(v)=0$ being the base case
    		in which $u$ and $v$ consist only of group elements.
    		
    		By Lemma~\ref{lem:group-cancel}, we may begin by moving all the group elements in $u$ to the left, and all the group
    		elements in $v$ to the right, to get words $g_u u'$ and $v'h_v$, where $u'$ and $v'$ have no group elements.
    		Moreover, it is necessarily
    		the case then that $u'Dv'$ is not reduced, and that a domain $E$ of $v'$ cancels with, absorbs,
    		or is absorbed by a domain $F$ of $u'$.
    		We may assume that in a sequence of reductions of $u'Dv'$, this absorption or cancellation is the first that occurs.
    		
    		Without loss of generality, $E$ is a domain orthogonal to $D$ and absorbs, is absorbed by, or cancels with
    		$F$, and $E$ can be moved
    		into a position adjacent to $F$ by a sequence of commutations of orthogonal domains. We may assume that
    		\[uDv=\cdots F\cdots D\cdots E\cdots\simeq \cdots FE\cdots D.\] After moving $E$ next to $F$ and performing a cancellation
    		or absorption, we obtain a (possibly non-reduced) word $w$, along with subwords
		$w_1D$ and $Dw_2$ given by reading $w$ from the left
    		up to $D$, and from $D$ to the right, respectively.
    		
    		Note first that $\Or(w_1)\oplus\Or(w_2)<\Or(u)\oplus\Or(v)$. Note moreover that $Dw_2$ is reduced, though $w_1D$ may not be.
    		If $F$ absorbs $E$ then $w_1D$ is reduced, and so we are done by the inductive hypothesis. If $F$ is absorbed by $E$ then
    		we may switch the roles of $u$ and $v$, and so we are again done by the inductive hypothesis. We may therefore assume
    		that $E=F$ and that a cancellation occurs.
    		
    		Applying Lemma~\ref{lem:group-cancel} again, we may move all the group elements that are created in the cancellation
    		to the far left, and we observe that all such group elements are necessarily orthogonal to $D$.
    		Let $Z$ be a domain in $w_1$ with which
    		such an absorption or cancellation involving $D$ occurs. We have that by a sequence of absorptions or cancellations that do
    		not involve $D$, the letter $D$ may be moved next to $Z$. It follows that all the domains between $D$ and $Z$ in $w_1$
    		must be orthogonal to $D$, including the ones created by the cancellation of $E$ with $F$. Note that here we are using the fact that
    		any group elements created in the cancellation of $E$ with $F$ were also orthogonal to $D$. Since $F\perp D$, we must have
    		that all the domains between $D$ and $Z$ in $u'$ were orthogonal to $D$, and so $u'$ was not reduced, a contradiction.
    		It follows that $D$ is not cancelled or absorbed in any reduct of $w_1D$. The lemma now follows by induction.
    	\end{proof}

    	\begin{corollary}\label{cor:associative}
    		The operation of concatenation
    		followed by reduction without cancellation, viewed
    		as an operation on equivalence classes $[w]$ of reduced words $w\in\wo$,
    		is well--defined
    		and associative.
    	\end{corollary}
    	\begin{proof}
    		It is straightforward to see that the operation of concatenation is well-defined on equivalence classes. If $w_1,w_1'\in [w_1]$ and
    		$w_2,w_2'\in [w_2]$ are representatives
    		of equivalence classes, then $w_1w_2$ and $w_1'w_2'$ also lie in the same equivalence class.
    		
    		For associativity, observe first that concatenation of formal strings of symbols is an associative operation. Associativity on the level
    		of equivalence classes then follows from Lemma~\ref{lem:reduced-unique}.
    	\end{proof}

  	\subsection{Symmetric decomposition}
    	
    	In the sequel, we will require a minor variation on Proposition 5.9 from~\cite{BPZ17},
    	which is called the \emph{Symmetric Decomposition Lemma}.
    	Even though the proof is nearly identical to what is given in that paper,
    	we will record the details
    	because the words that we consider may have group elements in them, in contrast
    	to~\cite{BPZ17}.

       \begin{definition}[Reduced product and absorption]
Let \(u\) and \(v\) be words. We write
$$u\ast v$$
for the word obtained by concatenating \(u\) and \(v\), and then reducing without cancellation. We say that \(u\) \emph{right-absorbs} \(v\) if
$u\ast v \simeq u$, and that \(u\) \emph{left-absorbs} \(v\) if
$v\ast u \simeq u$.
We say that \(u\) \emph{absorbs} \(v\) if it either left-absorbs or right-absorbs \(v\).
The absorption is \emph{proper} if the absorbed word \(v\) is not equivalent to the empty word.
\end{definition}

\begin{table}[htbp]
\centering
\footnotesize
\renewcommand{\arraystretch}{1.25}
\setlength{\tabcolsep}{4pt}

\begin{tabular}{|p{3.4cm}|p{5.2cm}|p{5.6cm}|}
\hline
\textbf{Notion} & \textbf{Meaning} & \textbf{Role in the word calculus} \\
\hline

Quasi-permutation
& An equivalence generated by the bookkeeping moves
$(\mathrm{Swp})$, $(\mathrm{Jmp})$, $(\mathrm{Cmp})$, $(\mathrm{Rm})$, 
$(\mathrm{AbsG})$
and inverses.
&
Identifies different spellings of the same generic relation. These moves do not represent simplifications of domain complexity. \\
\hline

Reduction
&
Applying the elementary moves which simplify a word, including absorption and cancellation moves.
&
Removes redundant letters and produces shorter or simpler representatives of the same relation. \\
\hline

Strict reduction
&
A reduction which strictly decreases the associated ordinal \(\operatorname{Or}(w)\).
&
A domain simplification has occurred (not just a change of spelling). \\
\hline

Reduction without cancellation
&
Reduction in which the cancellation move \((\mathrm{C})\) is forbidden, but absorption moves are still allowed.
&
This is the reduction used to define the product \(u\ast v\). It is well-defined on quasi-permutation classes and is associative. \\
\hline

Reduced product \(u\ast v\)
&
The word obtained by concatenating \(u\) and \(v\), and then reducing without cancellation.
&
Packages concatenation and reduction into a single operation on reduced word classes.  \\
\hline
Right/Left-absorption
&
A word \(u\) right-absorbs \(v\) if \(u\ast v\simeq u\). 

A word \(u\)  left-absorbs \(v\) if \(v\ast u\simeq u\).
&
Appending \(v\) on either side of \(u\) does not change the quasi-permutation class after reduction without cancellation. 
 \\
\hline

Proper absorption
&
An absorption is proper if the absorbed word is not equivalent to the empty word.
&
Rules out the trivial case and isolates the  absorption phenomena  in Lemma \ref{lem:sym-dec}. \\
\hline

\end{tabular}

\caption{Summary of the word operations. Quasi-permutation records changes of representative, while reduction and absorption record simplifications of words.}
\label{tab:word-calculus-summary}
\end{table}
    	
    	\begin{lemma}[Symmetric Decomposition Lemma]\label{lem:sym-dec}
    		Let $u$ and $v$ be reduced words. Then up to quasi-permutation, there are unique
    		decompositions \[u\simeq gu_1u'w,\quad v\simeq wv'v_1h,\] which satisfy the following
    		conditions:
    		\begin{enumerate}
    			\item
    			The letters $g$ and $h$ are group elements, and no group elements
    			occur in $u_1u'w$ and $wv'v_1$;
    			\item
    			The word $w$ consists of a sequence of mutually orthogonal domains;
    			\item
    			The word $u'$ is properly left-absorbed by $v_1$;
    			\item
    			The word $v'$ is properly right-absorbed by $u_1$;
    			\item
    			The words $\{u',w,v'\}$ pairwise commute;
    			\item
    			The word $u_1wv_1$ is reduced.
    		\end{enumerate}
    		Thus, $gu_1wv_1h$ is a reduct without cancellation of $uv$.
    	\end{lemma}
    	\begin{proof}
    		Applying the moves $\mathrm{(Jmp)}$ and $\mathrm{(Cmp)}$, we may assume that
    		$u=gu_0$ and $v=v_0h$, where $u_0$ and $v_0$ are reduced words with no occurrences of
    		group elements.
    		
    		We now proceed by induction on the sum of the lengths of $u_0$ and $v_0$.
    		If the word
    		$u_0v_0$ is already reduced then we set $u_1=u_0$ and $v_1=v_0$, and $v',u',w$
    		to be trivial.
    		
    		Since otherwise $u_0v_0$ is not reduced, without loss of generality, we may write
    		$u_0=u_0'D$, with $D$
    		left--absorbed by $v_0$, so that $D v_0\simeq v_0$. By induction, find words
    		$\{y_1,y',z,x',x_1\}$ which satisfy the conclusions of the lemma, and such that
    		\[u_0'=y_1y'z,\quad v_0=zx'x_1.\]
    		
    		Since $u_0$ is reduced, we have that $D$ cannot be left--absorbed by $z$
    		nor $x'$. It follows that $D$ is orthogonal to $zx'$; and in particular $zD$ is a (generally disconnected) domain,
    		and must be absorbed by $x_1$.
    		If $D$ is properly absorbed by $x_1$, then we set \[u'=y'D,\quad w=z,\quad v_1=y_1.\]
    		If $D$ is absorbed via move $\mathrm{(Abs_=)}$ then we may write (up to quasi-permutation)
    		$x_1=Dx_1'$, whereby we set \[u'=y', \quad w=zD,\quad v_1=x_1'.\]
    		
    		Now, clearly $gu_0v_0h\simeq gu_1wv_1h$. The word $w$ is the longest common terminal
    		segment of $u_0$ and $v_0^{-1}$, and $u_1w$ is the longest common initial subword
    		of $u_0$ and $[u_0]\ast [v_0]$, since otherwise $v_1=Dv_1'$ and $u'=Du''$, which
    		contradicts the fact that $u'$ is properly left--absorbed by $v_1$. Similarly,
    		$wv_1$ is the longest common terminal subword of $v_0$ and $[u_0]\ast [v_0]$. This
    		establishes uniqueness of the decomposition.
    	\end{proof}

	\section{Displacement and types }\label{sec:dt}

{
The goal of the next two sections is to use large-scale properties of the mapping class group of $\Sigma$ to obtain better control over types of tuples in the theories under consideration.  The word calculus developed in Section~5 gives finite descriptions of relations in the augmented Cayley graph; in the present section we begin to relate this calculus to model-theoretic types by associating to a word geometric data measuring where the corresponding movement is supported.  This displacement data is coarse enough to be invariant under the reductions introduced above, but fine enough to distinguish the relevant quantifier-free types.  The climax of this part of the paper is Corollary~\ref{qf characterization}, which characterizes the quantifier-free type of a pair.}

 	In this and the following section, the action of $G$, which as before will denote a finite index subgroup of $\Mod^{\pm}(\Sigma)$,
  	will generally
  	appear on the left (cf.~Remark~\ref{rmk:left-v-right}). We will generally
  	reserve the letters $\{g,h,k\}$, possibly with subscripts
  	and superscripts, for group elements, and all
  	other letters will denote elements of the universe on which group elements act.
  	
  	Let $\D\subseteq\D_0$ be a $G$--invariant
  	and downward closed family of domains.
  	Consider  a connected domain $D\in\D$.
  	By definition (see Section~\ref{sec:framework}), $D$  is  identified
  	with the curve graph of the underlying realized topological
  	surface $\Fill(D)$, together with some boundary  curves. We recall the notation
  	$\C(D)$, which denotes the curve graph  of the subsurface $\Fill(D)$; in particular, $\C(D)$ has infinite diameter,
	 provided $D$ is connected, non-annular,
  	and nonempty.
  	
  	If $g\in G[D]$, which is to say $R_{D}(1,g)$, then
  	$g$ acts on $\C(D)$.
  	The kernel \[K\unlhd G[D]\] of this action consists of all the
  	$g\in G$ such that $R_{\partial D}(1,g)$, which is  to say
  	the group generated by the Dehn twists over connected components of the \emph{inner
  	boundary} of $\Fill(D)$, i.e.~the boundary components of $\Fill(D)$ that lie in the interior of $\Sigma$.
  	
  	\begin{lemma}
  		\label{lem-upper-bound-disp}Let $D\in\D$ be connected. Given
  		$w\in\wo(D)$ and $\alpha\in\C(D)$, there is a constant
  		$K=K(w,\alpha)>0$ such that for each pair $g,g'\in M^G$ satisfying
  		\[R_{w}(g,g'),\quad R_D(1,g),\quad R_D(1,g'),\]
  		then we  have $d_{\C(D)}(g\cdot \alpha,g'\cdot \alpha)\leq K$.
  	\end{lemma}
  	\begin{proof}
  		We may assume that $D$ is not annular so that
  		$\abs{D}$ contains more than one curve,
  		since otherwise $w$ can only contain letters from $G$.
  		Fix $\alpha\in\C(D)$. The proof is by induction on $|w|$. Suppose that $w$ is
  		of the form $(v,D')$, where $D'\subsetneq D$. Choose $\beta$ in
  		$(D\cap\partial D')\setminus\partial D$, so that $\beta$ is a boundary curve
  		of
  		$D'$ which is non-peripheral in $D$.
  		Let $N=d_{\C(D)}(\alpha,\beta)$.
  		If $\M^G\models R_{w}(g,g')$, then there exists $h\in M^G$ such that both
  		$\M^G\models R_{v}(g,h)$  and $\M^G\models R_{D'}(h,g')$ hold. Since
  		$h\cdot \beta=g'\cdot \beta$, we have
  		\[d_{\C(D)}(h\cdot \alpha,g'\cdot \alpha)\leq 2N.\]
  		Using the triangle inequality and the induction hypothesis, it follows that
  		\[d_{\C(D)}(g\cdot\alpha,g'\cdot\alpha)\leq K_{0}+2N:=K\]
  		satisfies the required properties, where $K_{0}=K(v,\alpha)$ is the constant
  		provided by the induction hypothesis. The case in which $w$ is of the form
  		$vg$ follows by an identical argument, mutatis mutandis.
  	\end{proof}
  	
  	Since there are elements $g\in G[D]$ that act on $\C(D)$ with
  	arbitrarily large translation length, we obtain:
  	\begin{corollary}
  		\label{c: consistency}
  		If $D$ is a domain then the type $\W_{D}(x,y)$ is consistent.
  	\end{corollary}
  	
  	\begin{proof}
  		Suppose $\{w_1,\ldots,w_k\}\subset\wo(D).$ Since for each $i$ the word $w_i$ is
  		a finite word consisting of instances of $\mathcal{A}(D)$, Lemma~\ref{lem-upper-bound-disp} implies that
  		for each $\alpha\in\C(D)$, there  is an absolute
  		bound $C_i$ such that if  $R_{w_i}(1,g)$,  then
  		$d_{\C(D)}(\alpha,g\cdot\alpha)\leq C_i$. However, there exists an $h\in G$
  		such that $R_D(1,h)$
  		and \[d_D(\alpha,h\cdot\alpha)\geq C=\max_{1\leq i\leq k}C_i+1.\] Thus, the type
  		$R_D^*$ is finitely satisfiable and hence consistent.
  	\end{proof}
  	
  	A  converse result can be obtained using the fact that the
  	edge relation in a curve graph is the union of finitely many
  	distinct topological configurations.
  	
  	\begin{lemma}
  		\label{l: displacement}Let $D\in\D$ be a proper, connected domain, and let $F\subset\C(D)$ be a finite collection of vertices; assume furthermore that if $D$ is non-annular then no element of $F$ is peripheral. For all constants $K>0$, there exists
  		 $\phi\in\W_{D}$ such that $M^G\models \phi(1,g)$
  		implies \[d_{\C(D)}(F,g\cdot F)>K.\]
  	\end{lemma}
  	We do not consider the  case where $D=\mC_0$, and indeed if $D=\mC_0(\Sigma)$
  	then the  statement of the lemma  is not true unless $\D=\D_0$.
  	\begin{proof}[Proof of  Lemma~\ref{l: displacement}]
  		We argue the contrapositive, so that given $K$ we want to prove
  		the existence of a collection $\wo_{0}$ of finitely many words
  		in $\wo(D)$ such that \[d_{\C(D)}(F,g\cdot F)\leq K\] implies
  		$M\models R_{u}(1,g)$ for some $u\in\wo_{0}$.
  		
  		Since the mapping  class group acts by isometries on the curve graph and thus
  		preserves diameters of subsets of the curve  graph,
  		we may assume without loss of generality  that
  		$F$ consists of a single curve $\gamma$, which is
        non-peripheral whenever $D$ is non-annular.
  		
  		If $D$ is annular, then the Dehn twist $\tau$ about the core curve has the
  		property that $R_{D}(x,y)$ holds if and only if $y=x\tau^{n}$ for some $n$.
  		We have that $\C(D)$ is quasi-isometric to a line on which $\tau$ acts
  		as a loxodromic element, whence the result follows easily; cf.~Definition~\ref{def:curve graph}.
  		
  		Otherwise, $\Fill(D)$ is a surface with boundary of genus $g$ and
  		with $b$ boundary components, verifying the inequality $3g+b>3$. Here,
  		we remind the reader that a pair of pants is treated as three disjoint annuli
  		and is therefore not a connected domain. The graph $\hat{C}=\C(\Fill(D))$ is a locally
  		infinite graph of infinite diameter,
  		with vertices curves in $D\setminus\partial D$ and edges
  		between pairs of curves with minimal intersection. The vertices and edges of
  		$\hat{C}$ fall into finitely many orbits under the mapping class group $H$ of $\Fill(D)$.
  		
  		Choose $A\subset D$ a finite set of representatives from every orbit of vertices.
  		If $\gamma_1$ and $\gamma_2$ are minimally intersecting simple closed curves
  		on $\Fill(D)$, then there are only finitely many topological types of surfaces of the
  		form $\Fill(D)\setminus\{\gamma_1\cup\gamma_2\}$.
  		It follows that
  		there exists a finite collection
  		$\{h_{1},h_{2},\dots, h_{r}\}\subset G[D]$ such that
  		if $\{\alpha,\beta\}$ is an edge in $\hat{C}$ with $\alpha\in A$, then there is an $\alpha'\in A$
  		and \[g\in G[D]\cap\Stab(\alpha)\] such that
  		$\{\alpha,\beta\}=g\cdot \{\alpha,h_j\cdot \alpha'\}.$
  		for a suitable index $j$. We fix notation for the finite set of unordered pairs
  		$E:=\{\{\alpha,h_j\cdot \alpha'\}\mid \alpha\in A,\, 1\leq j\leq r\}.$
  		
  		We now claim
  		that there is a finite collection $\wo_0=\wo_0(K)$ such that if
  		\[d_{\C(D)}(\gamma,g\cdot \gamma)\leq K\] then $R_w(1,g)$  for some $w\in\wo_0$. If $K=0$, the conclusion is trivial. By
  		induction, suppose we have built a finite collection of such words $\wo_0(K-1)$,
  		and suppose that $\gamma'\in G[D]\cdot\gamma$ is at distance $K$ from $\gamma$.
  		We may then write
  		\[\gamma_{0}=\gamma,\gamma_{1},\ldots,\gamma_{K-1}, \gamma_{K}=\gamma',\]
  		where
  		$\{\gamma_{i},\gamma_{i+1}\}$ is a translate of an element of $E$.
  		For each $1\leq i\leq K$ there is an element $h_i\in G[D]$ and $\alpha_i\in A$
  		such that $\gamma_i= h_i\cdot \alpha_i$, and by induction we may suppose that
  		$R_w(1,h_i)$ for some  $w\in\wo_0(K-1)$, provided that $i<K$. The edge $\{\gamma_{K-1},\gamma_{K}\}$  is a
  		translate of $e_K\in E$ by an element $h_K'\in G[D]$. Note that $h_K$ and $h_K'$
  		differ by an element \[q_K\in\Stab_G(\gamma_K)=h_K\Stab_G(\alpha_K)h_K^{-1},\] so that
  		$h_K'= q_K\cdot h_K$. But then $q_K=h_Kq_K'h_K^{-1}$ for some
  		$q_K'\in\Stab_G(\alpha_K)$, so that \[h_K'=h_K\cdot q_K'.\]
  		
  		Writing $D'=D\cap\{\alpha_K\}^{\perp}$, notice that it is not necessarily true that
  		\[G[D']=\Stab_G(\alpha_K)\cap G[D],\] though we do obtain that
  		$G[D']\leq\Stab_G(\alpha_K)\cap G[D]$ with finite index; this can occur if $D'$ fails to be a connected domain.
		We therefore find finitely many elements
  		$\{\sigma_1,\ldots,\sigma_m\}\subset G[D]$ such that
  		\[\bigvee_{j=1}^m R_{D',\sigma_j}(1,q_K').\] Therefore, we  may enlarge $\wo_0(K-1)$ to a finite set $\wo_0(K)$ by adding all words
  		of the form $wD'\sigma_j$, where:
		\begin{enumerate}
		\item
		$w\in\wo_0(K-1)$;
		\item
		$D'=D\cap\{\alpha_K^{\perp}\}$ (with $\alpha_K$ ranging over
  		$A$);
		\item
		$\sigma_j$ ranges over
  		coset representatives of $G[D']$ in $\Stab_G(\alpha_i)\cap G[D]$.
		\end{enumerate}
		
		Thus we may
  		arrange for $R_w(1,h_K')$ for some $w\in\wo_0(K)$. This furnishes the finite set $\wo_0(K)$ as required by
		the contrapositive, and thus establishes the lemma.
  	\end{proof}
  	
  	Thus finite fragments of the generic relation \(R_D^\ast\) can be
chosen to force arbitrarily large displacement in \(\C(D)\). This is
the geometric input used in the Generic Perturbation Lemma.

	\section{Simple connectedness}\label{sec:sc}
The goal of this section is to prove that \(\Th{\M^G}\) has the
model-theoretic property of \emph{simple connectedness}, introduced
in~\cite{BPZ17}. In our setting, this property is established by
Lemma~\ref{l: main lemma}, together with the preceding discussion and
Lemma~\ref{uniqueness}. Its main consequence is the essential uniqueness
of the word relating two points: if \(a\) and \(b\) belong to a model of
\(\Th{\M^G}\) and \(R_u^*(a,b)\) holds, then the reduced class of \(u\)
is uniquely determined by \(a\) and \(b\). This class is denoted by
\(\delta(a,b)\) in Definition~\ref{def:strict}.

The argument proceeds as follows. The Behrstock inequality first provides
geometric witnesses for the failure of certain relations. These witnesses
are used in the Generic Perturbation Lemma, which supplies the main
geometric input for the proof of simple connectedness. Simple connectedness
then yields the uniqueness of \(\delta(a,b)\), from which we later recover
the quantifier-free type of the pair \((a,b)\).

In this section, we assume that
$G\leq \Mod^{\pm}(\Sigma)$
has finite index, that \(\Sigma\) satisfies
$3g-3+b\geq 1$,
and that \(\D\) is a \(G\)--invariant, downward-closed collection of
connected domains. We write \(\C_0\) for \(\C(\Sigma)\) when the
surface \(\Sigma\) is clear from the context.

  	\subsection{Certifying non-relatedness}
    	The following result is a rephrasing of Theorem 4.3 in
    	\cite{behrstock2006asymptotic},
    	and is commonly known as the \emph{Behrstock inequality}.
    	We  note that in the original,
    	the inequality is given in terms of complete markings. We refer the reader to Subsection~\ref{ss:projection} for background on
	subsurface projections.
    	\begin{theorem}
    		\label{t: projection}
    		There exists a constant $C\geq 0$ such that given any pair $X_{1}$  and
    		$X_{2}$
    		of transversely intersecting, essential, connected
    		subsurfaces  of $\Sigma$
    		which are not pairs of pants,
    		and a curve $\alpha\in\C_0$ with non trivial projection to both $X_{1}$ and
    		$X_{2}$,
    		we have:
    		\begin{align*}
    			\min\{d_{X_{1}}(\partial X_{2},\alpha),d_{X_{2}}(\partial X_{1},\alpha)\}\leq
    			C.
    		\end{align*}
    	\end{theorem}
    	
    	The following lemma is the key ingredient allowing one to describe the
    	structure of a general model of $\Th{\M^G}$. Here is an intuitive idea.
    	
    	Suppose we are given a reduced word $w$ and a curve $\alpha\in\C_0$ such that
    	$\alpha$ meets in an essential way at least one surface appearing in $w$, and
    	let
    	$h$ satisfy $R_w(1,h)$.
	\begin{definition}
	If $\alpha\neq h\cdot\alpha$ then we say
    	 that $\alpha$ is \emph{perturbed} by $h$.
	 \end{definition}
	The content of
    	the next
    	lemma is that in fact $\alpha$ will be perturbed by any $h$ satisfying
    	the relation $R_w(1,h)$ in a sufficiently generic way. More precisely,
    	given an arbitrary $\beta\in\C_0$, there is a formula
    	$\psi_{\alpha,\beta}(x,y)$ implied by $R^{*}_{w}(x,y)$ such that
    	$\psi_{\alpha,\beta}(1,h)$ implies $h\cdot \alpha\neq\beta$.
    	The same conclusion will hold equivariantly, with $(g,gh,g\cdot\alpha)$ in
    	place of $(1,h,\alpha)$, where here $g\in G$ is arbitrary.
    	
    	\begin{lemma}[Generic Perturbation Lemma]
    		\label{l: main lemma}
    		Suppose that \[w=D_{1}D_{2}\cdots D_{k}h\] is a given
    		reduced word with $h\in G$, that $g\in G$,
    		and that $\alpha\in\C_0$ with $h(\alpha)\not\perp D_{j}$
    		for some $1\leq j\leq k$.
    		Then there exist formulae $\phi_{i}(x,y)$ implied by $\W_{D_{i}}$ for
    		$1\leq i\leq k$ such that
    		\begin{align*}
    			M^G\models\forall x_{0}\forall x_{1}\,\cdots\forall
    			x_{k}\;\left(\left(\bigwedge_{i=1}^{k}\phi_{i}(x_{i-1},x_{i})\right)\rightarrow\neg
    			R_{g,\{\alpha\}^{\perp}}(x_{0},x_{k})\right).
    		\end{align*}
    	\end{lemma}
    	In Lemma~\ref{l: main lemma}, we implicitly assume that the domains occurring
    	in $w$ lie in $\D$. To demystify the statement of the lemma, consider mapping classes $\psi_1$ and $\psi_2$ such that
    	$\neg R_{g,\{\alpha\}^{\perp}}(\psi_1,\psi_2)$. The meaning of this is that \[\psi_1^{-1}\psi_2\notin g\Stab_G(\alpha),\] and in particular
    	$\psi_2(\alpha)\neq \psi_1\cdot g(\alpha)$.

Before the proof, let us spell out the reduction used in the argument. After
left-translating the variables, it is enough to prove the following equivalent
form. Fix a curve \(\beta\in\C_0\). We choose finite fragments
\(\phi_i\in R_{D_i}^*\) so that, whenever \(h_i\in G[D_i]\) realizes
\(\phi_i(1,h_i)\) for all \(i\), the product
\[
    H=h_1h_2\cdots h_k
\]
does not send the relevant curve to \(\beta\). In the statement of the lemma
one takes \(\beta=g\cdot\alpha\), with the terminal group letter \(h\) already
accounted for by the assumption that \(h(\alpha)\) meets one of the domain
letters. In the proof below we write \(H\) for the product
\(h_1\cdots h_k\), reserving \(g\) for the group element appearing in the
statement.
        
    	\begin{proof}[Proof of Lemma~\ref{l: main lemma}]
    		\newcommand{\J}[0]{\mathcal{J}}
    		We can assume that each of the $D_{j}$ is connected.
    		We need to show that for any $\alpha\in\C$ with
    		$\alpha\not\perp\bigvee_{i=1}^{k}D_{i}$
    		and $\beta\in\C$ arbitrary,
    		there exist formulae $\phi_{i}(x,y)$ as above such that for any sequence
    		\[h_{1},h_{2},\ldots, h_{k}\] of elements of $G$ satisfying
    		$\phi_{i}(1,h_{i})$
    		for $1\leq i\leq k$, the element
\[
    H=h_{1}h_{2}\cdots h_{k}\in G
\]
cannot map the curve \(\alpha\) to the curve \(\beta\).
    		
    		To  begin, let $C$ be the constant provided by Theorem \ref{t: projection}.
    		From Lemma \ref{l: displacement}, for each
    		$1\leq j\leq k$, there exists a formula \[\phi_{j}(x,y)\in\W_{D_{j}}(x,y)\]
    		with the property that for arbitrary $h\in M^G$, the condition
    		$\M^G\models\phi_{j}(1,h)$ implies $h\in G[D_j]$ and
    		\[d_{D_{j}}(A_{j},h\cdot A'_{j})> 2C,\] where  here
    		\begin{align*}
    			A_{j}=\pi_{D_{j}}\left(\{\beta\}\cup\bigcup_{\ell<j}\partial|D_{\ell}|\right) \\
    			A'_{j}=\pi_{D_{j}}\left(\{\alpha\}\cup\bigcup_{\ell>j}\partial|D_{\ell}|\right)
    		\end{align*}
		are the images of the Masur--Minsky projections onto $D_j$.
    		
    		Let $h_{1},h_{2},\dots, h_{k}$
    		be chosen so that $\phi_{i}(1,h_{i})$ for all $1\leq i\leq k$.
    		Write \[h_{i,j}=h_{i}h_{i+1}\cdots h_{j}\] for $0\leq i<j\leq k$.
    		
    		Let $j_{0}$ be the maximum index $1\leq j\leq k$ for which $\alpha$ is
    		not orthogonal to $D_{j}$.
    		Note that this implies $\alpha=h_{j_{0}+1,k}\cdot\alpha$.
    		
    		If for all $j<j_{0}$ we have $D_{j}\perp D_{j_{0}}$, then we obtain the
    		conclusion of the lemma. Indeed, in this case
    		$h_{1,j_{0}-1}$ fixes $D_{j_{0}}$, and thus
    		\begin{align*}
    			\pi_{D_{j_{0}}}(h_{1,j_{0}-1}\cdot\beta)=
    			\pi_{h_{1,j_{0}-1}\cdot D_{j_{0}}}(\beta)=
    			h_{1,j_{0}-1}\cdot\pi_{D_{j_{0}}}(\beta)=\pi_{D_{j_{0}}}(\beta).
    		\end{align*}
    		
    		It follows easily then that
    		\begin{align*}
    			d_{D_{j_{0}}}(H\cdot\alpha,\beta)=
    			d_{D_{j_{0}}}(
    			h_{1,j_{0}-1}h_{j_{0},k}\cdot\alpha,h_{1,j_{0}-1}\cdot\beta)=\\
    			=d_{D_{j_{0}}}(h_{j_{0},k}\cdot\alpha,\beta)=d_{D_{j_{0}}}(
    			h_{j_{0}}h_{j_{0}+1,k}\cdot\alpha,\beta)=
    			d_{D_{j_{0}}}(h_{j_{0}}\cdot\alpha,\beta)>C.
    		\end{align*}
    		Here, we implicitly allow the possibility that
		$\pi_{D_{j_{0}}}(\beta)=\emptyset$. In that case the displayed distance estimate is no longer meaningful, but the conclusion is immediate: since
		$w$ is not orthogonal to $\alpha$, we cannot have $H\cdot \alpha=\beta$.
    		
    		For the general case, we define $j_{0}$ as before.
    		We  inductively construct a descending sequence
    		\[j_{0}>j_{1}>\cdots>j_{t}\] of indices
    		by setting $j_{k+1}$ to be the maximum index $j$ which is smaller than $j_{k}$
    		and such that $D_{j}\not\perp D_{j_{k}}$.
    		Eventually, we obtain an index $j_t$
    		such that $D_{j_{t}}$ is orthogonal to $D_{j}$ for all $j<j_{t}$.
    		Necessarily, $D_{j_{\ell}}\tv D_{j_{\ell+1}}$ for $0\leq \ell\leq t-1$.
    		Indeed, $D_{j_{\ell}}$ and $D_{j_{\ell+1}}$
    		are not orthogonal by construction.
    		Moreover, they are incomparable since otherwise $w$ would not be reduced, as
    		every letter occurring in $w$ between these
    		surfaces is orthogonal to $D_{j_{\ell}}$.

The role of the sequence \(j_0>j_1>\cdots>j_t\) is the following. The index
\(j_0\) is the last domain letter which sees the curve \(\alpha\). If all
earlier domain letters are orthogonal to \(D_{j_0}\), then the large
projection in \(D_{j_0}\) already separates \(H\cdot\alpha\) from \(\beta\).
Otherwise, \(j_1\) is the last earlier domain letter which is transverse to
\(D_{j_0}\), and the Behrstock inequality transfers the obstruction from
\(D_{j_0}\) to \(D_{j_1}\). Iterating this process produces \(j_t\), where no
earlier domain letter interferes, and the same projection-separation argument
applies.
            
    		\begin{claim}\label{claim-sec7}
    			For all $1\leq \ell\leq t$, we have
    			$h_{j_{\ell}+1,k}\cdot\alpha\not\perp D_{j_{\ell}},$
    			and
    			$$d_{D_{j_{\ell}}}(h_{j_{\ell}+1,k}\cdot\alpha,\partial D_{j_{\ell-1}})\leq C.$$
    		\end{claim}
    		
    		The conclusion of the lemma follows from the case $\ell=t$ of Claim~\ref{claim-sec7}
    		above for the same choices of $\phi_{j}$ as in the case $t=0$ considered
    		previously.
    		Indeed on the one hand, since	$\phi_{j_{t}}(1,h_{j_{t}})$ holds, we have
    		\[d_{D_{j_{t}}}(\beta,h_{j_{t}}\cdot\partial D_{j_{t-1}})> 2C.\]
    		On the other hand, Claim~\ref{claim-sec7} asserts that
    		$d_{D_{j_{t}}}(h_{j_{t}+1,k}\cdot\alpha,\partial D_{j_{t-1}})\leq C,$
    		which in turn implies
    		$d_{D_{j_{t}}}(h_{j_{t},k}\cdot\alpha,h_{j_{t}}\cdot\partial D_{j_{t-1}})\leq C.$
    		This allows us to conclude that
    		\[d_{
    		h_{1,j_{t}-1}\cdot D_{j_{t}}}(\beta,H\cdot\alpha)=
    		d_{D_{j_{t}}}(\beta, h_{j_{t},k}\cdot\alpha)> C,\]
    		since $D_{i}\perp D_{j_{t}}$ for $i<j_{t}$.
    		Thus, we obtain $\beta\neq H\cdot\alpha$.
    		
    		\begin{proof}[Proof of Claim~\ref{claim-sec7}] We proceed by induction on $\ell$.
    		Suppose that for some $1\leq \ell<t$ we have already successfully shown that
    		\[d_{D_{j_{\ell}}}(h_{j_{\ell}+1,k}\cdot\alpha,\partial D_{j_{\ell-1}})\leq C.\]
    		Then, we
    		have
    		$d_{D_{j_{\ell}}}(h_{j_{\ell},k}\cdot\alpha,
    		h_{j_{\ell}}\cdot\partial D_{j_{\ell-1}})\leq C.$
    		The choice of $h_{j_{\ell}}$ implies that $d_{D_{j_{\ell}}}(\partial
    		D_{j_{\ell+1}},h_{j_{\ell}}\cdot\partial D_{j_{\ell-1}})> 2C.$
    		The triangle inequality
    		then implies that
    		$d_{D_{j_{\ell}}}(\partial D_{j_{\ell+1}},h_{j_{\ell},k}\cdot\alpha)> C.$
    		Theorem \ref{t: projection} then shows
    		$$d_{D_{j_{\ell+1}}}(h_{j_{\ell},k}\cdot\alpha,\partial D_{j_{\ell}})\leq C.$$
            
    		Since
    		$D_{s}\perp D_{j_{\ell}}$ for all $j_{\ell+1}<s<j_{\ell}$, the left
    		hand side of this
    		last inequality is equal to
    		$d_{D_{j_{\ell+1}}}(h_{j_{\ell+1}+1,k}\cdot\alpha,\partial D_{j_{\ell}}),$
    		which
    		establishes the claim.\end{proof}
		The lemma now follows.
    	\end{proof}
    	
  	\subsection{Parametrizing the quantifier--free type of a pair of elements}

    	For $w\in\wo$, we let $w^{-1}$ be the result of writing  the letter
    	occurring in the expression for
    	$w$ in reverse order, and by replacing each occurrence of
    	$g\in G$ with $g^{-1}$. It follows by  definition that
    	$R_{w}(x,y)$ if and only if $R_{w^{-1}}(y,x)$.
    	
    	For the remainder of this subsection $\mathcal N$ denotes a model of $\Th{\M^G}$, with universe $N$.
    	The following result gives strong restrictions on words which can induce the same generic relatedness in models of $\Th{\M^G}$.

    	\begin{lemma}
    		\label{uniqueness}
    		Suppose we are given two elements $a,b\in N$,
    		and let $u$ and $v$ be reduced words such that
    		$\W_{u}(a,b)$ and $\W_{v}(a,b)$. Then $u\simeq v$.
    	\end{lemma}
    	\begin{proof}
Recall that quasi-permutation includes the move \((\mathrm{Abs}_{G})\). Thus,
if \(g\subseteq D\), the words \((g,D)\), \((D,g)\), and \(D\) are already
quasi-permutation equivalent. Such group letters therefore do not produce
spurious counterexamples to the statement.

We argue by induction on \(\Or(u)\oplus\Or(v)\). Consider the concatenation
\(uv^{-1}\). If \(uv^{-1}\) is irreducible and contains a domain letter, then
we may write
\[
    uv^{-1}=D_1\cdots D_kh
\]
after quasi-permutation, with \(k>0\). Choose a curve \(\alpha\) such that
\(h(\alpha)\not\perp D_1\). Applying Lemma~\ref{l: main lemma} to the closed
generic sequence obtained from \(R_u^*(a,b)\) and \(R_v^*(a,b)\) gives
    $\neg R_{\{\alpha\}^{\perp}}(a,a)$, 
which is impossible. If \(uv^{-1}\) is irreducible and contains no domain
letters, then it consists only of group letters. Since the right \(G\)-action
on each \(G\)-orbit is free, \(R^*_{uv^{-1}}(a,a)\) forces the corresponding
group element to be trivial, and hence \(u\simeq v\).

We may therefore assume that \(uv^{-1}\) is reducible and that both \(u\) and
\(v\) contain domain letters. Reducibility means that, after quasi-permutation,
a terminal domain letter of \(u\) interacts with a terminal domain letter of
\(v\). Thus we may write
\[
    u\simeq u_0D,\qquad v\simeq v_0E
\]
with \(D,E\in\D\) comparable; after switching \(u\) and \(v\) if
necessary, assume \(E\subseteq D\).

Choose \(c,d\in N\) such that
\[
    R^*_{u_0}(a,c),\quad R^*_D(c,b),\qquad
    R^*_{v_0}(a,d),\quad R^*_E(d,b).
\]
If \(R^*_D(c,d)\) holds, then \(R^*_{u_0Dv_0^{-1}}(a,a)\). By
Corollary~\ref{reduct cor}, some reduct \(w\) of \(u_0Dv_0^{-1}\) satisfies
\(R^*_w(a,a)\). Lemma~\ref{surviving D} shows that this reduct is not
quasi-permutation equivalent to the empty word, contradicting
Lemma~\ref{l: main lemma}.

Thus the only remaining possibility is that \(E=D\) and
\(R^*_{u_1}(c,d)\) for some \(u_1\in\wo(D)\). Now
\[
    R^*_{u_0u_1}(a,d)
    \qquad\text{and}\qquad
    R^*_{v_0}(a,d).
\]
The associated ordinal has decreased, since \(u_1\) is supported in the
proper word calculus below the final \(D\)-letter. By the induction
hypothesis,
\[
    u_0u_1\simeq v_0.
\]
Consequently, we have the following and we conclude
\[
    u\simeq u_0D\simeq u_0u_1D\simeq v_0D\simeq v.
\]

\end{proof}
    	
    	\newcommand{\pa}[1]{\delta(#1)}
\begin{definition}[Reduced class between two points]\label{def:strict}
    We denote the unique reduced class $[u]$ such that
    $\W_{u}(a,b)$ by $\pa{a,b}~.$ Since quasi-permutation preserves the associated ordinal, we write
$\Or(\pa{a,b})$
for the associated ordinal of any representative of the reduced class
\(\pa{a,b}\).
\end{definition}
The uniqueness of the class $\pa{a,b}$ is one of the most important consequences of simple connectedness.

    	\begin{observation}
    		Given a triple $(a,a',a'')\in N^3$, we have that words in $\pa{a,a''}$ are reducts of
    		concatenations of representatives of $\pa{a,a'}$  and  $\pa{a',a''}$.
    	\end{observation}
	
	Recall that $\qftp$ denotes \emph{quantifier-free type}.
    	
    	\begin{corollary}
    		\label{qf characterization}Given tuples $a=(a_{i})_{i\in I}$
    		and $a'=(a'_{i})_{i\in I}$ of elements from $N$, we have $\qftp(a)=\qftp(a')$
    		if and only if $\pa{a_{i},a_{j}}=\pa{a'_{i},a'_{j}}$ for all $i,j\in I$.
    		In particular,
    		the class $\delta(a,b)$ determines the  quantifier-free type $\qftp^{a,b}(A)$.
    	\end{corollary}
    	\begin{proof}
    		Combining Lemma \ref{uniqueness} and Corollary \ref{reduct cor},
    		we  have that for any $u\in\wo$, the validity in $N$ of
    		$R_{u}(a,b)$ is equivalent to any representative of $\pa{a,b}$
    		being a reduct of $u$. Since the language contains only binary relations, the result follows.
    	\end{proof}
    	
	\section{Weakly convex sets and their extensions}\label{sec:core}

  	\newcommand{\wc}[0]{weakly convex } %tc
  	\newcommand{\wcp}[0]{weakly convex.} %tc
  	
  	\newcommand{\cc}[1]{#1''}
  	\renewcommand{\c}[1]{#1'}
  	
{We retain the notation that \(G\leq\Mod^{\pm}(\Sigma)\) is a finite index subgroup,
\(\D\) is a \(G\)-invariant and downward closed family of connected domains, and
\(\M^G=\M^G_{\D}\).}

The goal of this section is to prove Lemma~\ref{core qe}, the
\emph{\(D\)--Step Extension Lemma}.  This is the main extension result needed
for the back-and-forth argument which later yields quantifier elimination and
stability for \(\Th{\M^G}\).  Informally, the lemma says that a finite
configuration extended by one generic \(D\)-step can be matched over any other
configuration with the same relevant type data. The proof requires a controlled convexity notion adapted to the relations
\(R_D\).  Two issues have to be addressed.  First, words may change under
quasi-permutation, absorption, and cancellation, so it is necessary to keep track
of which domains occur in a minimizing sequence.  Second, annular and
boundary components create cases in which ordinary orthogonality is too weak.
This motivates the refined orthogonality notions introduced below. We begin by defining weak convexity and proving its basic properties.  We then
study how weakly convex sets behave under extensions by domains which are
orthogonal, strongly orthogonal, strictly orthogonal, or transverse to the
domains already present.  These extension results are then assembled in the proof
of the \(D\)--Step Extension Lemma.

  	\subsection{Transitivity and parallel lifting}
    	
    	Let $\mathcal A$ be a $\mL$-structure such that
    	$\Aut(\mathcal A)$ acts transitively on its
    	universe $A$.
    	It is immediate that for an arbitrary $\mL$--formula in one free
    	variable $\phi(x)$, the theory $\Th{\mathcal A}$ contains the sentence
    	\[(\exists x)\phi(x)\rightarrow (\forall x)\phi(x).\]
    	One can use this observation to prove that if $\mathcal{N}$ is a model of
$\Th{\mathcal A}$ and if $p(x,y)$ is a consistent type in two variables,
then the types $p(x,a)$ and $p(a,x)$ are consistent for arbitrary $a\in N$.
    	
	In our context, we can extend this observation further to show that certain sets of a model of $\Th{\M^G}$ can be ``reflected''
	across a domain. This will be crucial in performing back-and-forth constructions.

    	\begin{lemma}
    		\label{parallel lifting}
    		Let $\kappa$ be an infinite cardinal.
    		Suppose that we are given $|I|\leq \kappa$, a tuple
    		$a=(a_{i})_{i\in I}$ in a $\kappa^+$--saturated model $\mathcal{N}$
    		of $\Th{\M^G}$, a basepoint $a_{i_{0}}$ indexed by $i_{0}\in I$, and
    		$D\in\D$. Suppose furthermore that for all $i\in I$, we have
$R_{u_i}(a_{i_0},a_i)$
for some $u_i\perp D$. Let $a'_{i_0}\in N$ be such that
    		$R_{D}(a_{i_{0}},a'_{i_{0}})$. Then $a'_{i_{0}}$ extends to a tuple
    		$(a'_{i})_{i\in I}$ satisfying
    		both \[(a'_{i})_{i\in I}\equiv (a)_{i\in I},\quad
    		\textrm{and}\quad  R_{D}(a_{i},a'_{i}) ~\mbox{ for all } i\in I .\]
    	\end{lemma}
    	
    	Here, the symbol $\equiv$ is used to denote elementary equivalence, so that the types of these tuples coincide.
    	\begin{proof}[Proof of Lemma~\ref{parallel lifting}]
    		\newcommand{\BA}[0]{A_{0}}
    		Let $x=(x_i)_{i\in I}$ be a tuple of variables.
    		Fix an arbitrary formula $\psi(x)$ in the 
    		type $\tp_{x}(a)$.
    		Let $I_{0}$ be the finite subset of $I$ consisting of indices appearing
    		in $\psi$. We may assume by hypothesis
    		that $i_{0}\in I_{0}$, and that for all
    		$i\in I_{0}$, the formula $\psi$ implies $R_{u_{i}}(x_{i_{0}},x_{i})$
    		for a suitable $u_{i}\perp D$.
    		Let $(g_{i})_{i\in I_{0}}$ be a tuple of points of $M^G$, viewed
    		as elements of $G$, which witness $\psi(x_{I_{0}})$.
    		Pick an arbitrary $h\in G$ such that $\M^G\models R_{D}(1,h)$.
    		For $i\in I_{0}$, we set \[g'_{i}:=h^{g_{i_{0}}^{-1}}g_{i},\] where here the exponentiation notation denotes conjugation.
    		As the tuple $(g'_{i})_{i\in I_{0}}$ is in the orbit
    		of $(g_{i})_{i\in I_{0}}$ under the action of $\Aut(\M^G)$, it clearly
    		satisfies $\psi$ as well. On the other hand a simple calculation yields
    		$\M^G\models R_{h_i}(g_{i},g'_{i}),$  where here
    		$$h_i=h^{(g_{i_{0}}^{-1}g_{i})}.$$ 
    		
    		Now $k=g_{i_{0}}^{-1}g_{i}\in G[D^{\perp}]$ and $h\in
    		G[D]$, so that
    		$h^{k}\in G[D]$ as well.
    		Hence $R_{D}(g_{i},g'_{i})$ for all $i\in I$.
    		The desired result follows by compactness.
    	\end{proof}
    	
    	Note that in the proof of Lemma~\ref{parallel lifting}, it may be the case that $h\neq h^k$, since the orientations of the boundary curves
    	of $D^{\perp}$ may be reversed by $k$.
    	
  	\subsection{Weak convexity}        
    	We will work in the universe $N$ of a fixed, sufficiently  saturated model
    	$\mathcal{N}$ of the theory $\Th{\M^G}$. Note that the relations $R_{g}$, $g\in G$ induce a free action on the universe $N$ of $\mathcal{N}$, albeit not by automorphisms.
        We now introduce the notion of a \emph{weakly convex} set, which is a crucial technical
	concept that will allow us to perform back-and-forth constructions. In~\cite{tent-jsl,BPZ17}, such sets are called \emph{nice}.
    	An expression such as $x_{A}$ for $A\subseteq N$ will denote the (possibly infinite) tuple of
    	variables $(x_{a})_{a\in A}$.
    	The expression $\qftp^{x_{A}}(A)$ will denote the quantifier-free type of $A$
    	with $x_{a}$ in place of $a\in A$. That is, $\qftp^{x_{A}}(A)$ is the type comprised of the formulae
    	$R_{u}(x_{a},x_{a'})$ for which
    	$R_{u}(a,a')$ holds, and the formulae $\neg R_{u}(x_{a},x_{a'})$ otherwise, for all pairs
    	$a,a'\in A$.\\

    	\begin{definition}\label{def_11}
    		We say that a subset $A\subset N$ is \emph{\wc} if
    		\begin{itemize}
    			\item The set $A$ is a union of orbits, i.e. $A\cdot G= A$.
    			\item For all pairs $a,a'\in A$ and every reduced word $w$, if there is a strict $w$--sequence in $N$ between $a$ and $a'$, then there is a strict $w$--sequence between $a$ and $a'$ entirely contained in $A$.
    		\end{itemize}
    	\end{definition}

	  	\begin{figure}[htbp]
    		\begin{center}
    			\psfrag{aG}{\small $aG$}
    			\psfrag{a}{\small $a$}
    			\psfrag{a'}{$a'$}
    			\psfrag{a'G}{$a'G$}
    			\includegraphics[width=6.5cm]{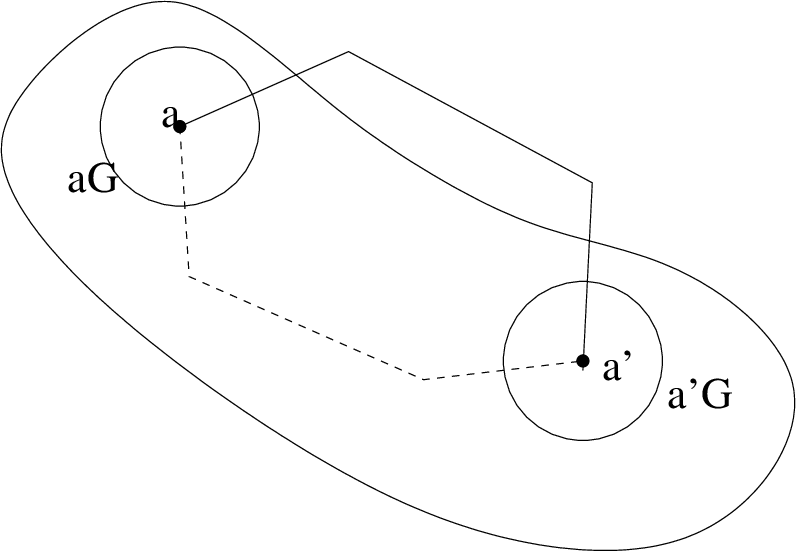}
    			\caption{Definition \ref{def_11}. Weak convexity asserts the existence
        		of the dotted path.}
    		\end{center}
    	\end{figure}

    	\begin{definition}[$D$--step away]
    		\label{one step away} Let $A\subset N,\quad a_{0}\in A$, and $D\in\D.$
    		We let $p_{a_{0},A}^{D}(x,y_{A})$ be the type stating that
    		\[R_{D}(x,y_{a_{0}})\wedge \neg R_{u}(x,y_{a}),\] for all $a\in A$ and $u\in\wo(D)$, as well as the type $\qftp^{y_{A}}(A)$ in the variables $y_{A}$.
    		We say that $b\in N$ is
		\emph{$D$--step away} from $A$ if there is $a_{0}\in A$ such that $\mathcal{N}\models p^{D}_{a_{0},A}(b,A)$.
    		In such a setup, we will refer to $a_{0}$ as a \emph{basepoint} for $b$ in $A$.

    	\end{definition}

    	\begin{lemma}
    		\label{step consistency} For all sets $A\subset N$ of parameters,
    		basepoints
    		$a_{0}\in A$, and domains $D\in\D$, we have that the type $p^{D}_{a_{0},A}$
    		is consistent.
    	\end{lemma}
    	\begin{proof}
            We argue by finite satisfiability.  It is enough to realize an arbitrary finite fragment of the type in the standard model.  Thus, after translating the finite tuple of parameters, we must find an element $x$ which is $D$--related to the chosen basepoint and which avoids a finite collection of relations $R_u$ with $u\in\wo(D)$ to each of finitely many other parameters.

        When $D$ is non-annular, Lemma~\ref{l: displacement} supplies, for each finite family of such words, a finite displacement bound in the curve graph $\CC(D)$.  Choosing an element of $G[D]$ with displacement larger than all these bounds gives the desired realization.  When $D$ is annular, the same argument uses a sufficiently large power of the Dehn twist about the core curve.  Hence every finite fragment is realized, and compactness gives consistency.
    	\end{proof}
    	    	
    	\begin{lemma}[Gate Property]
    		\label{gate property}Suppose that $A\subset N$ is weakly convex and  that
    		$b\in N$ is $D$--step away from $A$, with basepoint $a_{0}$.
    		Then for all $a\in A$, the class $\pa{b,a}$ consists of all the reducts without cancellation of $Du$ for some (any) $u\in\pa{a_{0},a}$.
    	\end{lemma}

    	In the case where $\D\neq \D_0$ is restricted,
    	we insist  that these subdomains lie in $\D$.
    	\begin{proof}[Proof of Lemma~\ref{gate property}]
    		Let $a\in A$ and let $u\in\pa{a_{0},a}$. Suppose  that
    		in the process of reduction of $Du$ to $\pa{b,a}$,
    		some cancellation takes place. Then
    		one can write $u\simeq D u_{2}=u'$ (cf. Remark \ref{any sequence}).
    		Since $A$ is weakly convex, there is a strict
    		$u'$--sequence $a_{0},a_{1},\ldots, a_{k}=a$ that is
    		contained entirely in $A$. Observe that $R_{D}(b,a_{1})$
    		holds by transitivity of
    		$R_{D}$. If we have $\W_{D}(b,a_{1})$, then
    		$\pa{b,a}=[Du_{2}]$, which contradicts our assumption that
    		cancellation occurs in reducing $Du$ to a word in $\pa{a,b}$.
    		Otherwise, the point $a_{1}\in A$ witnesses
    		$\mathcal{N}\not\models p^{D}_{a_{0},A}(b,A)$.
    	\end{proof}

    \begin{example}[The gate for a $D$--step in the running Example \ref{ex:running-separating-curve}]
    \label{ex:running-gate}
    Let $B$ be a
    weakly convex set. Suppose that $b$ is a $D$--step away from $B$ with basepoint
    $b_0\in B$. Thus $b$ is obtained from $B$ by adding a new generic move on the
    $Y$ side.
    The Gate Property says that, for any $c\in B$, the relative position from $b$
    to $c$ is obtained by first crossing the new $D$--step from $b$ back to $b_0$,
    and then following the old relative position from $b_0$ to $c$, except for
    absorptions forced by lower-complexity domains. Symbolically, representatives
    of $\pa{b,c}$ are obtained by reducing
        $D\,\pa{b_0,c}$
    without cancellation.    The boundary annulus $A=\{\alpha\}$ is the basic source of absorption in this
    example. If the old word $\pa{b_0,c}$ begins with an $A$--move, then the new
    initial $D$--move absorbs this annular move, because $A\subset D$. This is why
    the boundary-avoidance lemma is needed later: when one builds weakly convex
    extensions, boundary annuli have to be separated from new
    $D$--directions.
    \end{example}

\begin{definition}[Minimizing sequences and minimizing words]
Given a set \(A\subset N\) and \(b\in N\), we write
\[
\Or(b,A)=\min\{\Or(\pa{b,a})\}_{a\in A}.
\]
If \(\Or(w)=\Or(b,A)\) and
\[
\bar b=\{b=b_{0},b_{1},\dots,b_{k}=a\}
\]
is a strict \(w\)--sequence from \(b\) to some \(a\in A\), then we say that
\(\bar b\) is a \emph{minimizing sequence} from \(b\) to \(A\), and that \(w\)
is a \emph{minimizing word} from \(b\) to \(A\).
\end{definition}

    	\newcommand{\cn}[0]{con^{\perp*}}
    	\newcommand{\tcn}[0]{tcon^{\perp\circ}}
    	\newcommand{\wcn}[0]{con^{\perp\circ}}
    	\newcommand{\wtcn}[0]{wtcon^{\perp\circ}}

    	  Orthogonality of domains
    	is supposed to capture ``disjointness'', and for non-annular domains, it succeeds. The reader may have noted that
    	if $D$ is a domain and $\alpha\in\C_0$ is a boundary curve of $D$ then $\alpha\perp D$. This phenomenon is simply a feature of
    	domains, and we will need to rule it out in certain cases. We introduce a fix here for purely technical
    	purposes; we will call it
    	\emph{strong orthogonality}.

    	\begin{definition}[Strong orthogonality]\label{def:strong-orthogonality}
    		Given $D_1, D_2 \in \D_0$, we will say that $D_1$ is \emph{strongly orthogonal}
    		to $D_2$ and write $D_1\perp^*  D_2$ if the following holds:
    		\begin{enumerate}
    			\item $D_1\perp D_2$;
    			\item $D_i\not\subset \partial D_{3-i}$ for $i=1,2$.
    		\end{enumerate}
    	\end{definition}

    	In between orthogonality and strong orthogonality is a third notion which we will require for certain technical purposes, and
    	which we call \emph{strict orthogonality}. This last relation is not
    	symmetric.
    	\begin{definition}[Strict orthogonality]\label{def:strict-orthogonality}
    		Given $D_1, D_2 \in \D_0$, we will say that $D_1$ is \emph{strictly orthogonal}
    		to $D_2$ and write $D_1\perp^{\circ}  D_2$ if the following holds:
    		\begin{enumerate}
    			\item $D_1\perp D_2$;
    			\item $D_2\not\subset \partial D_{1}$
    		\end{enumerate}
    	\end{definition}

    	The definitions above are formulated in such a way that if neither $D_1$ nor $D_2$ are unions of annular domains, then
    	 orthogonality and strong orthogonality coincide.

    \begin{example}[Ordinary versus strong orthogonality in the running example \ref{ex:running-separating-curve}]
    \label{ex:running-strong-orthogonality}
    In Example~\ref{ex:running-separating-curve}, the annular domain
    $A=\{\alpha\}$ is orthogonal to $D=D_Y$: every curve in the annulus about
    $\alpha$ has zero geometric intersection with every curve carried by $Y$.
    However, $A$ is not strongly orthogonal to $D$, because $A\subset \partial D$.
    The same comments apply to $A$ and $E=D_Z$.

    This is not a minor distinction. A word supported in $A$ commutes with a
    $D$--move, but it can also be absorbed into a $D$--move. Thus ordinary
    orthogonality is enough to commute disjoint supports, while strong
    orthogonality is what prevents a boundary annulus from disappearing into a
    larger domain. The weak-convexity arguments use precisely this distinction
    when choosing cones which avoid boundary components.
    \end{example}
	
	\begin{definition}[Orthogonality to words]
\label{def:orthogonality-words}

We extend the strong and strict orthogonality relations to words in the
obvious way. Specifically, we write
\[
w\perp^* D
\]
if every group element appearing in \(w\) is orthogonal to \(D\), and every
domain occurring in \(w\) is strongly orthogonal to \(D\). Similarly, we write
\[
w\perp^\circ D
\]
if every group element appearing in \(w\) is orthogonal to \(D\), and every
domain \(E\) occurring in \(w\) satisfies \(E\perp^\circ D\).

For equivalence classes of words, we write
$[w]\perp D$
if \(w'\perp D\) for some \(w'\in[w]\). Similarly, we write
$[w]\perp^* D \mbox{ (respectively }
[w]\perp^\circ D \mbox{ )}$
if \(w'\perp^*D\) for some \(w'\in[w]\) (respectively if
\(w'\perp^\circ D\) for some \(w'\in[w]\)).
\end{definition}

In the definitions of \([w]\perp D\), \([w]\perp^*D\), and
\([w]\perp^\circ D\), we require orthogonality only for a representative
of the equivalence class, since even if \(w\perp D\), we have
\(w\simeq wgg^{-1}\) for an arbitrary \(g\in G\), which itself need not be
orthogonal to \(D\).

The following Table~\ref{tab:section8-orthogonality} summarizes all the
orthogonality conventions we use.
\begin{table}[hbp]
\centering
\scriptsize
\renewcommand{\arraystretch}{1.15}
\setlength{\tabcolsep}{3pt}

\begin{tabularx}{\textwidth}{|>{\raggedright\arraybackslash}p{1.85cm}|
                                >{\raggedright\arraybackslash}X|
                                >{\raggedright\arraybackslash}X|}
\hline
\textbf{Notation}
& \textbf{Definition }
& \textbf{Use and reference} \\
\hline

$D \perp E$
&
Ordinary orthogonality: every curve in $D$ is disjoint from every curve in $E$.
&
Definition~\ref{def:orthogonal-domains}. It is the geometric reason behind $(\mathrm{Swp})$ in Table~\ref{tab:word-moves}: orthogonal domain letters may be transposed. \\
\hline

$D \perp^{\ast} E$
&
Strong orthogonality: $D\perp E$, and neither domain is contained in the boundary of the other.
&
Introduced in Definition~\ref{def:strong-orthogonality} to rule out boundary-annulus absorption. It is tested on normal forms in Lemma~\ref{perp and equivalence}, and used in the boundary-avoidance construction of Lemma~\ref{l: efficient convexity} and Lemma~\ref{core qe}. \\
\hline

$D \perp^{\circ} E$
&
Strict orthogonality: $D\perp E$ and $E\not\subseteq \partial D$. This is one-sided.
&
Introduced in Definition~\ref{def:strict-orthogonality}. It is the condition used to define the cone in Definition~\ref{def:cones}; Lemma~\ref{l: cone} proves that the corresponding twisted cone is weakly convex, which is then used in Lemma~\ref{core qe}. \\
\hline

$g \perp D$
&
The group element $g$ is supported away from $D$, i.e. $g\in G[D^\perp]$.
&
 Definition~\ref{def:orthogonal-words}. It is needed in Lemma~\ref{parallel lifting}, where all connecting words must be orthogonal to $D$, and in the normal-form criteria of Lemmas~\ref{orthogonality characterization} and~\ref{perp and equivalence}. \\
\hline

$w \perp D$
&
Every letter of $w$ is orthogonal to $D$: domain letters by $\perp$, group letters by support away from $D$.
&
 Definition~\ref{def:orthogonal-words}. It is the hypothesis allowing reflection across a $D$--step in Lemma~\ref{parallel lifting}. Lemma~\ref{orthogonality characterization} and Corollary~\ref{concatenation perp} control this condition at the level of word classes. \\
\hline

$w \perp^{\ast} D$
&
Every group letter is orthogonal to $D$, and every domain letter is strongly orthogonal to $D$.
&
 Definition ~\ref{def:orthogonality-words}. The main criterion is Lemma~\ref{perp and equivalence}: strong orthogonality of a class can be checked on normal forms. \\
\hline

$w \perp^{\circ} D$
&
Every group letter is orthogonal to $D$, and every domain letter $E$ satisfies $E\perp^{\circ}D$.
&
Defined in Definition ~\ref{def:orthogonality-words}. It is used in Definition~\ref{def:cones} to define the strict cone; Lemma~\ref{l: cone} shows that this cone construction preserves weak convexity. \\
\hline

$[w]\perp D$
&
Some representative $w'\in[w]$ satisfies $w'\perp D$. Similarly for $[w]\perp^{\ast}D$ and $[w]\perp^{\circ}D$.
&
Introduced in Definition ~\ref{def:orthogonality-words} because orthogonality should be a property of the represented relation, not of one spelling of the word. Lemmas~\ref{orthogonality characterization} and~\ref{perp and equivalence} give the normal form tests; Corollary~\ref{concatenation perp} gives stability under reduction. \\
\hline

$D \pitchfork E$
&
Transversality: $D$ and $E$ are neither nested nor orthogonal.
&
 Definition~\ref{def:transverse}. It enters indirectly through the Behrstock inequality, Theorem~\ref{t: projection}, and Lemma~\ref{l: main lemma}, which support the simple-connectedness machinery used later. \\
\hline

\end{tabularx}

\caption{{Guide to the orthogonality conventions used in Section~8. }}
\label{tab:section8-orthogonality}
\end{table}

\FloatBarrier

    \subsection{Fellow traveling}
    	The purpose of this subsection is to control how a reduced word changes when group letters are moved without permuting its domain letters. This will allow us to compare strict sequences pointwise up to \(G\)-orbits and to isolate the position at which a new \(D\)-letter is absorbed.
    	\begin{definition}[$\equiv$--relation on words]
    	Given $w,w'\in\wo$ we say that $w\equiv w'$
    	if it is possible to obtain one from the other by applying the moves
    	\[\{(\mathrm{Cmp}),(\mathrm{Jmp}),(\mathrm{Rm}),(\mathrm{Abs}_{G})\}\]
    	and their inverses. The $\equiv$--equivalence class of a word $w$ will be written $[[w]]$.
	\end{definition}
	
    	Observe that the primary difference between the equivalence relations $\simeq$ and $\equiv$ is that $\simeq$ allows for permutation
    	of domains, whereas $\equiv$ does not; under $\equiv$, only group elements may be moved.
        Thus $[[w]]\subseteq [w]$.
    	
	\begin{definition}[Fellow Traveling]
    	Given two finite sequences $s,s'$ of points in $N$, we say that they
    	\emph{fellow travel}
    	with each other if for every $s:= a_{0},\dots, a_{k}$ and $s':= a'_{0},\dots, a'_{k'}$ and there are non-decreasing surjections $f:\{0,\dots k\}\twoheadrightarrow \{0,\dots m\}$ and $f':\{0,\dots k'\}\twoheadrightarrow \{0,\dots m\}$ such that if $f(\ell)=f'(\ell')$ for \[0\leq \ell\leq k,\quad 0\leq \ell'\leq k',\] then $a_{\ell}G=a'_{\ell'}G$. 
	\end{definition}
    	
    	\begin{observation}
    		\label{small wiggling} If $w\equiv w'$, then for any strict $w$--sequence $p$,
    		there is a unique strict $w'$--sequence $p'$
    		between the same endpoints that fellow travels with $p$.
    	\end{observation}

    	\begin{lemma}
    		\label{ends}
    		The following statements hold.
    		\begin{enumerate}
    			\item
    			Let $vD$ and $w$ be reduced words and suppose that the word
    			$vDw$ reduces to the identity. Then $w\simeq Dw'$ and the reduction involves an application of $(\mathrm{C})$ to the two $D$ domains in $vDDw'$. 
    			\item
    			For any reduced word $w$ there exists a unique maximal
    			possibly disconnected domain $E\in\Dvee$ such that for any reduced word of the form $uD$ equivalent to $w$, we have that $D$ is a union of components of $E$.
    		\end{enumerate}
    	\end{lemma}
	
	  	\begin{definition}\label{def:ends}
    		We denote the possibly disconnected domain $E$ in Lemma~\ref{ends}
    		the \emph{right end} of $w$. We define the \emph{left end} of $w$ symmetrically.
    	\end{definition}

    	\begin{proof}[Proof of Lemma~\ref{ends}]
    		We begin with the first statement.
    		By Lemma~\ref{surviving D}, we have that the word $Dw$ is not reduced. If $w_0$ is a reduct of $Dw$ then either
    		the leftmost letter $D$
    		is absorbed by a domain in $w$, or it cancels with a copy of $D$ (possibly conjugated by a group element)
    		occurring in $w$. In the latter case, we may write $w\simeq Dw'$,
    		and the conclusion of the first part of the lemma holds.
    		
    		Thus, we may assume that the letter $D$ is absorbed by a letter $E$ occurring in $w$. Since $vDw$ reduces to the identity and
    		since $w$ is itself reduced, we have that $E$ either cancels with or is absorbed by a letter $F$ in $v$, possibly after
    		conjugating by a group element. To eliminate the complications of group conjugation, we write $v=v_0Fv_1$, with
    		$v_1$ orthogonal to $E$ and with $v_1$ containing no group elements.
    		Since $D$ is absorbed by $E$, we have that $v_1$ and $D$ are already orthogonal, so that $vD$ is not
    		reduced, a contradiction.
    		
    		We now consider the second statement. Suppose that \[w\simeq g_0u_0 D_0\simeq g_1u_1 D_1,\]
    		with $D_0\neq D_1$ and with
    		$u_0$ and $u_1$ containing no group elements. We must have that
    		$D_0$ occurs as a letter in $u_1$ and $D_1$ occurs as a letter in
    		$u_0$. It is immediate then that $u_0\simeq u_0' D_1$ and $u_1\simeq u_1' D_0$, and that $D_1$ and $D_0$ must be orthogonal
    		to each other. They cannot be one inside the other, since in that case $w$ would not be reduced, as witnessed by a move of type $(\mathrm{Abs}_{\subset})$ or $(\mathrm{Abs}_{=})$. It follows that $F=D_0D_1$ is a domain with components $D_0$ and $D_1$, and that $w\simeq g_F u_F F$ for
    		a suitable group element $g_F$ and word $u_F$.
    		
    		By induction on $\Or(w)$, we may simply take $E$ to be the concatenation of all domains $D$ for which $w\simeq u D$. We showed above that $E$ consists of mutually orthogonal domains, and there are only
    		finitely many such domains by the well-ordering of the class of ordinals. Then $E$ is canonically defined and hence unique.
    	\end{proof}

    	We can now state and prove the following application of fellow traveling.
    	
    	\begin{lemma}[Fellow Traveling Lemma]\label{lem:equiv-fellow}
    		Let $w$ be a reduced word and $h\in G$, suppose $w\simeq wh$, and let \[s:=a_1,\ldots,a_k\] be a $wh$-sequence. Then there
    		exists a fellow traveling $w$-sequence \[s':=a'_1,\ldots,a'_{k}\] with $a_1=a_1'$ and $a_k=a_{k}'$.
    	\end{lemma}
    	\begin{proof}
    		The proof is by induction on the length of $w$. By virtue of Observation \ref{small wiggling} we may assume that $w$
    		contains no group elements. Let $E$ be the right end of $w$, and write $w\simeq vE$. Then $wh$ is equivalent to the word $hh^{-1}(vE)$ and it is clear that $h^{-1}(E)$ is the right end of $vEh$.
    		
    		Now, the word $whw^{-1}$ is reducible. By Lemma \ref{surviving D}, this implies that each connected component of $E$ has to cancel with some connected component of $h^{-1}(E)$, and so we may conclude $h^{-1}(E)=E$.
    		
    		Since we have $vE\simeq vEh\simeq vhE,$ there exists some $vE$-sequence $a''_{0},a''_{1},\dots, a''_{k-1}$ with $a''_{k-1}=a_{k}$.
    		Choose $a$ such that $R_{h}(a_{k-1},a)$ and
    		$\W_E(a,a_{k})$. It follows immediately that $R_{E}(a''_{k-2},a)$, and that
    		the word $vh$ is a reduct of $v\pa{a''_{k-2},a}$.
    		This means that there is a word in right normal form \[ug\in\pa{a''_{k-2},a}\cap\wo(E)\cup\{E\},\]
    		that is right-absorbed by $v$. Observe that, moving from left to right, each letter in $u$ can be permuted to be adjacent to a letter of
    		$v$ into which it can be subsumed using (Abs). No cancellation between letters of $u$ and letters of $v$ can take place: indeed,
    		since $u\in\wo(E)\cup\{E\}$, there would otherwise be a permutation of $vE$ to which the move $(\mathrm{Abs}_{=})$, $(\mathrm{Abs}_{\subset})$, or $(\mathrm{C})$ could be applied.
    		
    		Now, we have $v\simeq v(hg^{-1})$. By induction, there exists a $v$-sequence
    		\[a'_{0}=a_{0},\dots, a'_{k-2}=ag^{-1}\] from $a_{0}$ to $ag^{-1}$ that fellow travels with
    		$a_{0},\dots, a_{k-1}$. Since $g\in G[E]$, we have $R^{*}_{E}(ag^{-1},a_{k})$ and thus letting
    		$a'_{k-1}=a_{k}$ we obtain the desired $vE$-sequence.
    	\end{proof}
    
    	The next corollary records exactly what can happen when an initial \(D\)-letter
is reduced against a word \(v\). There are two possibilities. Either the
initial \(D\)-letter survives, after commuting past a strictly orthogonal
initial part of \(v\), or it is absorbed by a larger domain already present in
\(v\). This dichotomy is used repeatedly in the weak-convexity and
one-step-extension arguments below.

    \begin{corollary}
        \label{D tracking}Let $v\in \wo$ be a reduced word, let $D\in\D$ be a domain, and let $w$ be a reduct without cancellation of $Dv$.
        Then one of the following holds:
        \begin{itemize}
            \item There exist words $u,u'',u'$ with $uu''u'\in [v]$, $u''\in\wo(D)$ possibly empty, $u\perp^{\circ}D$, and $uDu'\in [[w]]$. Here $u'$ is also allowed to be empty.
            \item There exist words $u,u'$ and a domain $E\in\D$ with $D\subseteq E$, such that $u\perp^{\circ}D$ and $uEu'\in [v]\cap [[w]]$.
        \end{itemize}
    \end{corollary}
    \begin{proof}
        We argue by induction on the number of domain letters in $v$. The case in which $v$ has no domain letters is immediate. After applying bookkeeping moves, we may assume that $v$ is in right normal form and that its first domain letter is a component $F$ of the left end of $v$; write $v=Fv'$.

        If $F$ is neither orthogonal to $D$ nor comparable with $D$ under inclusion, then $Dv$ is already reduced. In this case the first alternative holds with $u=1$, $u''=1$, and $u'=v$. Thus we may assume that exactly one of the following situations occurs: $D\subseteq F$, $F\subsetneq D$, or $F\perp^{\circ}D$.

        If $D\subseteq F$, then the initial $D$ is absorbed by $F$, so that $w\simeq v$. The second alternative follows immediately by taking $u=1$, $E=F$, and $u'=v'$.

        If $F\subsetneq D$, then the reduction of $Dv$ begins by absorbing $F$ into $D$, and the remaining reduction is a reduction without cancellation of $Dv'$. Applying the inductive hypothesis to $(D,v')$, the second alternative cannot occur, since it would put a domain containing $D$ in a representative of $v'$ and hence would contradict the reducedness of $v=Fv'$. Thus the first alternative occurs for $v'$: say $v'\simeq uu''u'$ with $u\perp^{\circ}D$, $u''\in\wo(D)$, and $uDu'\in[[w]]$. Since $F\subsetneq D$ and $u\perp^{\circ}D$, the word $F$ commutes past $u$, and we have $v\simeq u(Fu'')u'$. The word $Fu''$ lies in $\wo(D)$, so the first alternative follows for $v$.

        Now assume $F\perp^{\circ}D$. Since $F$ can be transposed past $D$, applying the inductive hypothesis to $(D,v')$ and then moving $F$ back to the left gives one of the two alternatives above. In the first case, if $v'\simeq uu''u'$, then $v\simeq (Fu)u''u'$ and $Fu\perp^{\circ}D$. In the second case, if $uEu'\in[v']\cap[[w]]$, then $(Fu)Eu'\in[v]\cap[[w]]$ and $Fu\perp^{\circ}D$.
    \end{proof}

    \begin{definition}\label{def:normal form}
        We say that a reduced word $w\in\wo$ is a left (right) normal form if $w$ has at most one occurrence of a letter from $G$, and this occurrence is the leftmost (rightmost) letter. We refer to this group letter, when it is present, as the \emph{$G$--term}.
    \end{definition}

    The $G$--term of a reduced word is not generally intrinsic to the class $[w]$, since applying jump moves can change which group element appears at the end of a chosen normal-form representative.

    \iffalse
\begin{lemma}
    		\label{normal forms} Given a reduced $w\in\wo$,
    		there exits a unique left (right) normal form $w'\in [[w]]$.
    		Given two left (right) normal forms $w\equiv w'$, we can pass from one
    		to the other by an iterated application of $(\mathrm{Jmp})$. If
    		$w\perp^* D$, then $w'\perp^* D$, where here $w'$ is the corresponding
    		left (right) normal form.
    	\end{lemma}
    	The  proof of Lemma~\ref{normal forms} is straightforward.
    	\fi

    	\begin{lemma}
    		\label{orthogonality characterization} Let $[w]$ be an equivalence
    		class of reduced words and let $D\in\Dvee$ be an arbitrary possibly disconnected domain.
    		The following are equivalent:
    		\begin{enumerate}[(a)]
    			\item  $w'\perp D$ for all $w'\in[w]$ in left (right) normal form;
    			\item  \label{condition perp}$[w]\perp D$;
    			\item  For all $w'\in[w]$, we have
    			$\Th{\M^G}\vdash(\forall x\forall y)\,\,R_{w'}(x,y)\rightarrow R_{D^{\perp}}(x,y)$.  
    		\end{enumerate}
    	\end{lemma}
    	\begin{proof}
    		The implication $(a)\Rightarrow(b)$ is clear. For $(b)\Rightarrow(c)$,
    		notice that the conclusion is clearly true for $w'=w$.
    		Since Lemma~\ref{equivalence} implies
    		\[\Th{\M^G}\vdash (\forall x\forall y)\,R_{u}(x,y)\leftrightarrow R_{u'}(x,y),\]
    		for all pairs of equivalent (reduced) words $u$ and $u'$, the conclusion follows.
    		
    		Let us now show the implication  $\neg(a)\Rightarrow\neg(c)$.
    		Assume  there exists
    		$w'\in[w]$ in normal form such that $w'\not\perp D$.
    		One of the following two cases occurs.
    		\begin{enumerate}
    			\item  $D\perp E$ for any $E\in\D$ occurring in
    			$w'$, but the $G$--term of $w'$ is not in
    			$G[D^{\perp}]$;
    			\item  There exists a $E\in\D$ occurring in $w'$ such that $E\not\perp D$.
    		\end{enumerate}
    		
    		The first case is in contradiction with (c), since under this assumption we have
    		\[\Th{\M^G}\vdash \forall x\forall y\,\,R_{w'}(x,y)\rightarrow\neg R_{D^{\perp}}(x,y).\]
    		
    		Consider the second case. Write $w=gu$ where $u\in\D^{*}$. There
    		is some $\alpha\in D$ which is not orthogonal to some $E$ appearing in $u$.
    		By Lemma \ref{l: main lemma} there exists an
    		$h\in G$ such that $R_{u}(1,h)$ and such that $h(\alpha)\neq g^{-1}(\alpha)$. The element
    		$h'=gh$ satisfies $R_{w}(1,h')$; however, $\M^G\models\neg R_{D^{\perp}}(1,h')$  since $h'$
    		does not fix $\alpha$, contrary to $(c)$.
    	\end{proof}
    	
    	We note the following consequence:
    	
    	\begin{corollary}
    		\label{concatenation perp} Let $D\in\Dvee$. If $w$ is a reduct
    		of a concatenation of words $v_{1},v_{2}$ satisfying $[v_{i}]\perp D$ for $i\in\{1,2\}$, then
    		$[w]\perp D$.
    	\end{corollary}
    	
    	The following is a consequence of Lemma
    	\ref{orthogonality characterization},  and provides a natural refinement of
    	Corollary~\ref{concatenation perp}.
    	\begin{lemma}
    		\label{perp and equivalence}Given a word $w\in\wo$ and $D\in\Dvee$, assume first that $w$ is reduced. Then
    		$[w]\perp^{*}D$ if and only if for any $w'\in[w]$ in left (right) normal form, we have that $w'\perp^{*}D$. The same conclusion holds with
    		$\perp^*$ replaced by  $\perp$.

             Assume now that $w$ is some word, not necessarily reduced, satisfying $w\perp D$ and that $w'$ is some reduct of $w$ in left/right normal form. Then one may choose a reduction process from $w$ to $w'$ in which no inverse of $(\mathrm{Cmp})$ is used. Consequently, every word appearing along this chosen reduction remains in $\wo(D^{\perp})$.
    	\end{lemma}
    	\begin{proof}
    		Suppose that $[w]\perp D$. Lemma \ref{orthogonality characterization}
    		implies that $w'\perp D$ for any left (right) normal form of $w$. Thus, if $w'\in [w]$ is arbitrary, then any left or right normal form in $[[w]]$ must be orthogonal
    		to $D$.

    		Now suppose that $w\perp^*D$. Then, all group elements occurring in $w$ are orthogonal to $D$
    		and all domains are strongly orthogonal.
    		If $w'\simeq w$ is already strongly orthogonal to $D$ then there is nothing to show. If not, then $w'$ was obtained by a sequence
    		of elementary moves, involving at least one
    		composition of the inverses of the moves $\{(\mathrm{Cmp}),(\mathrm{Rm})\}$ resulting in the appearance of a subword of the form
    		$gg^{-1}$, with $g\not\perp D$. If there is a domain $E$ occurring in $w'$ that is not orthogonal to $D$, then $E$ is obtained by
    		conjugating a domain occurring in $w$ by a group element that is not orthogonal to $D$. If there is a
    		connected domain $F$ occurring in
    		$w'$ that is in the boundary of $D$, then $F$ is obtained by conjugating an annular domain occurring in $w$ by a group
    		element that is not orthogonal to $D$.
    		
    		By assumption, if $E$ is a domain of $w$ then $E\perp^* D$. Consider the sequence of moves required to obtain $w'$ from $w$. For
    		each move that is an inverse of $(\mathrm{Rm})$ followed by an inverse of $(\mathrm{Cmp})$ occurs either on the right or
    		on the left of $E$. Up to applying a sequence of moves
    		$\{(\mathrm{Cmp}),(\mathrm{Rm})\}$ and their inverses, we may write
    		\[w'=h g_1 u_1g_2u_2\cdots g_k u_k \sigma(E)u_{k+1} q_1
    		v_1q_2\cdots v_{\ell}q_{\ell+1}, \mbox{ where }\]:
    		\begin{enumerate}
    			\item
    			The element $h\in G$ is the $G$--term of $w$;
    			\item
    			We have $\sigma\in G$;
    			\item
    			For each $i$, both $g_i$ and $q_i$ are elements of $G$;
    			\item
    			For each $i$, the words $u_i$ and $v_i$ contain no group elements;
    			\item
    			In $G$, the word $g_1g_2\cdots g_kq_1\cdots q_{\ell+1}$ is equal to a word of the form
    			$\prod_{i=1}^j c_i h_i c_i^{-1}$ for elements $\{c_1,\ldots,c_j\}\subset G$,
    			where each $h_i$ is a conjugate by $c_i^{-1}$ to a group element that is orthogonal to $D$;
    			\item
    			In $G$, we have $q_{\ell+1}^{-1}\cdots q_1\sigma$ and $g_1\cdots g_k\sigma$ are equal to a group element that is orthogonal to $D$.
    		\end{enumerate}
    		Thus, moving all the group elements to the far left or right by applications of the move $(\mathrm{Jmp})$ and applying
    		a sequence of the moves $(\mathrm{Cmp})$ and $(\mathrm{Rm})$, we obtain a word whose $G$--term is orthogonal to $D$. The
    		non-annular domains remain orthogonal to $D$, and hence are strongly orthogonal to $D$. Finally, all annular domains are conjugated
    		by group elements that are orthogonal to $D$, and therefore remain strongly orthogonal to $D$.
    	\end{proof}

    	\begin{corollary}\label{cor:bdy-disj}
    		Let $E\in\D_0$ be a non-annular domain and $F\subset \partial E$ be a subdomain, necessarily a union of annular domains.
    		Suppose $w_1,w_2\in\wo_{r}(E)$ are such that no domain occurring in either $w_1$ or $w_2$ contains a component of $F$ as a proper subdomain. If $u$ is a reduct of $w_{1}w_{2}$ which lies in $\wo(E)$, then no domain $E'$ occurring in $u$ contains a component of $F$ as a proper subdomain.
    	\end{corollary}
       
         \begin{remark}
          Note that we work with words in which the letters are all connected domains, so that $E'$ cannot appear as a union of two domains from each side.
         \end{remark}

    	\begin{definition}
    		Given $D\in\D_0$, and sequences $q,q'$ in $N$ we say that $q$ and $q'$ are \emph{$D$-parallel}
    		if for any $a\in q$, there is some $a'\in q'$ such that $R_{D}(a,a')$, and vice versa. If in addition we can take $a'\in aG$, then
    		we say that $q$ and $q'$  \emph{$D$-fellow travel}.
    	\end{definition}

    	\begin{lemma}
    		\label{l: checking wc} Let $C_{0}\subseteq N$ and $E\in\D_0$ be such that $R_{E}(c,c')$ for any $c,c'\in C_{0}$.
    		Assume further that for any $c,c'\in C_{0}$ and any \[w\in\pa{c,c'}\cap(\wo(E)\cup\{E\}),\]
    		there exists a $w$-sequence from $c$ to $c'$ that is contained in $C_{0}$. Then $C:=C_{0}\cdot G$ is weakly convex.
    	\end{lemma}
    	\begin{proof}
    		Let $c,c'\in C$ and $u\in\pa{c,c'}$. Choose $d,d'\in C_{0}$
    		and $g,g'$ such that $R_{g}(c,d)$ and $R_{g'}(c',d')$.
    		
    		The reduced word
    		$v=g^{-1} u g'$ lies in $\pa{d,d'}$.
    		By Lemma~\ref{orthogonality characterization}, applied with $D=E^{\perp}$, the hypothesis $R_E(d,d')$ implies that the class $\pa{d,d'}$ has a representative in $\wo(E)\cup\{E\}$. Choose such a representative
            \[v'\in [[v]]\cap(\wo(E)\cup\{E\}).\]
            By assumption there is some $v'$-sequence $p_{1}$ from $d$ to $d'$ entirely contained in $C_{0}$.
    		On the other hand, we know by Observation~\ref{small wiggling} that $p_{1}$ fellows travels with a unique strict $v$-sequence $p_{2}$ from
    		$d$ to $d'$. Clearly $p_{2}\subseteq C$. Dropping the first and last points in $p_{2}$ yields a strict $u$--sequence
    		from $c$ to $c'$ entirely contained in $C$.
    	\end{proof}

	\subsection{More on the structure of weakly convex sets}
We now use the preceding normal-form results to construct weakly convex
subsets whose connecting sequences avoid prescribed boundary annuli.
These subsets will provide the cones used in the \(D\)--Step Extension
Lemma. 

The following lemma is a crucial place where boundaries of domains come into play. The content of the lemma is that for each
	point in a weakly convex set (dominated by a single domain $E$), every point can be made to lie in another smaller
	weakly convex set in which strict sequences avoid prescribed boundary domains of $E$.
    	
    	\begin{lemma}[Weak convexity avoiding boundaries]
		\label{l: efficient convexity}
		Let \[E\in\D_0, \quad F\subseteq\partial E, \quad C\subseteq N\] be such that
		\begin{enumerate}
			\item $R_E(c,c')$ for all $c,c'\in C$;
			\item $C\cdot G$ is weakly convex;
			\item $C\cdot G[E]=C$.
		\end{enumerate}
		Then for each $b_{0}\in C$ there exists a subset $\hat{C}\subseteq C$ containing $b_{0}$
    		such that:
    		\begin{enumerate}
    			\item
    			For all $c,c'\in \hat{C}$, no representative of $\pa{c,c'}$ lying in $\wo(E)\cup \{E\}$ contains components of
    			$\partial E$ that belong to $F$;
    			\item \label{weakly convex}
    			$\hat{C}\cdot G$ is weakly convex;
    			\item \label{bringing down sequences} Let $u\in \wo(E)\cup \{E\}$ in which no subdomain of $F$ appears.
    			Then for any $u$--sequence
    			\[c_0,\ldots,c_m\subset C\] with $c_{0}\in \hat{C}$, there exists another $u$-sequence
    			\[c_{0}=c'_{0},c'_{1},\dots, c'_{m}\subset \hat{C}\] such that $R_{F}(c_i,c_i')$ for all $1\leq i\leq m$.
    		\end{enumerate}
    	\end{lemma}

The proof is an induction on the number of components of \(F\). At one
inductive step we remove a single boundary annulus \(\alpha\). We first
partition \(C\) according to whether two points can be connected inside \(E\)
without using a domain properly containing \(\alpha\). Inside each equivalence
class we choose representatives which avoid the annulus \(\alpha\), and then
we use weak convexity to replace arbitrary sequences by parallel sequences
through these representatives. Corollary~\ref{cor:bdy-disj} is what prevents
the annulus \(\alpha\) from reappearing after reduction.
        
    	\begin{proof}[Proof of Lemma~\ref{l: efficient convexity}]
    		We prove this by induction on the number of components of $F$. The base case $F=\emptyset$ is trivial.
    		Assume now $\alpha\in F$. For $c,c'\in C$ we write $c\sim c'$ if some (equivalently, any)
    		representative of $\pa{c,c'}$ in $\wo(E)\cup\{E\}$ does not contain any domain properly containing
    		$\alpha$. Note that if $g$ is a group element occurring in such a representative
    		of $\pa{c,c'}$ then $g$ must stabilize the boundary components of $E$. So if $E'\subset E$ is a domain that does not
    		properly contain $\alpha$, then neither does $g(E')$.
    		
    		Notice that $\sim$ is an equivalence relation, since given $c,c',c''$ such that $c\sim c'$ and $c'\sim c''$ we have that
    		$\pa{c,c''}$ results from the reduction of a concatenation of words, none of whose domains contains
    		$\alpha$ properly; the fact that $\sim$ is an equivalence relation now follows from Corollary~\ref{cor:bdy-disj}.
    		
    		Let $\mathcal{B}=C/\sim$. For each $B\in\mathcal{B}$ choose a representative $c(B)\in B$; for the class containing $b_0$, we choose $c(B)=b_{0}$.
    		Let $\hat{B}$ be the collection of $c'\in B$ (where $B$ is regarded as a set) such that $\alpha$ does not appear in any (equivalently some) representative
    		of $\pa{c(B),c'}$ in $\wo(E)$; here we implicitly assume that all domains in a word are connected.
    		We claim that for any $c\in B$ there exists $c'\in\hat{B}$ such that $R_{\alpha}(c,c')$. In other words, $\hat B$ may be
    		properly contained in $B$, but an arbitrary element of $B$ is $\alpha$--related to
    		an element of $\hat B$.
    		
    		To show this, assume $c'\notin\hat{B}$. Since $\alpha$ commutes with all domains contained in
    		$E$ and with all group elements in $G[E]$ there is
    		a representative $v$ of $\pa{c(B),c'}$ of the form $w\alpha$, where
    		$w\in\wo(E)$ does not contain the letter $\alpha$.
    		
    		Since $C\cdot G$ is weakly convex, there is a $v$-sequence \[c_{0}=c(B),c_{1},\ldots,c_{m},c_{m+1}=c',\] with $c_{i}\in C\cdot G$ for
    		$0\leq i\leq m+1$. Using the fact that $v\in \wo(E)$ and $C=C\cdot G[E]$, one sees that $c_{i}$ is in $C$.
    		In fact, we have $c_{i}\in \hat{B}$ for $0\leq i\leq m$, since $\pa{c(B),c_{i}}$ has representatives in $\wo(E)$,
    		none of whose domains contains $\alpha$.
    		In particular, $R_{\alpha}(c_m,c')$, as claimed.
    		
    		Now, let \[\hat{C}=\bigcup_{B\in\mathcal{B}} \hat{B}.\] For $c,c'\in\hat{C}$, no reduced representative $w$ of
    		$\pa{c,c'}$ in $\wo(E)\cup\{E\}$ can contain $\alpha$. Indeed, if on the one hand $c\in B$ and $c'\in B'$ with $B\neq B'$,
    		then by the definition of the equivalence
    		relation $\sim$, the reduced word $w$ contains a domain properly containing $\alpha$.
    		Then, there cannot be any occurrence of $\alpha$ in $w$, since this occurrence would be absorbed into the domain properly
    		containing $\alpha$, contradicting that $w$ is reduced.
    		
    		If on the other hand $c,c'\in B$ for some $B\in\mathcal{B}$, then $w$ does not contain $\alpha$. This follows from
    		Corollary~\ref{cor:bdy-disj},
    		because $w$ can be obtained as
    		a reduction
            of a word in $\pa{c,c(B)}\pa{c(B),c'}$, which
    		in turn does
    		not include any domain containing $\alpha$.
    		
    		To complete the proof, we have the following sublemma.
    		
    		\begin{sublemma}
    			\label{l: hatC0 is weakly convex}
    			The following hold.
    			\begin{enumerate}
    				\item \label{first item}
                    For any $c,c'\in \hat C$ and any $u\in(\wo(E)\cup\{E\})\cap \pa{c,c'}$, there exists a $u$-sequence
    					$c=c_{0}, c_{1},\dots, c_{m}=c'$ contained in $\hat{C}$.
    					\item \label{second item}
    				$\hat {C}\cdot G$ is weakly convex.
    			\end{enumerate}
    		\end{sublemma}
    		\begin{subproof}
    			We start by treating a single equivalence class $B\in\mathcal B$. Let
                \[p\colon c_0,c_1,\ldots,c_{m'}\]
                be a $v$--sequence in $B$ such that $v$ contains no occurrence of $\alpha$. For each $i$, choose $d_i\in\hat B$ with $R_\alpha(c_i,d_i)$; if $c_i\in\hat B$, choose $d_i=c_i$. For consecutive indices, the word $\pa{d_i,d_{i+1}}$ is a reduct of
                \[\alpha\,\pa{c_i,c_{i+1}}\,\alpha.
                \]
                Since reduced representatives between points of $\hat B$ contain no occurrence of $\alpha$, the two boundary letters cancel. Equivalently, after adjusting the intermediate choices inside their $G[\alpha]$--orbits, we obtain a $v$--sequence
                \[p'\colon c'_0,c'_1,\ldots,c'_{m'}\]
                contained in $\hat B$ and satisfying $R_\alpha(c_i,c'_i)$ for all $i$. When one or both endpoints of the original sequence already lie in $\hat B$, the choices above keep those endpoints fixed.

                In particular, if both endpoints lie in the same set $\hat B$, then the adjusted sequence has the same endpoints and has type $vh$ for some $h\in G[\alpha]$. Since $vh\simeq v$ by Lemma~\ref{uniqueness}, Lemma~\ref{lem:equiv-fellow} gives a fellow-traveling $v$--sequence with the same endpoints, contained in $\hat C\cdot G$. This proves the first assertion of the sublemma in the single-class case; the second assertion then follows from Lemma~\ref{l: checking wc} in this case.

                Now, let $c,c'\in\hat{C}$ be an arbitrary pair of points that do not lie in a single class $\hat B$. We will show that for any
    			\[w\in\pa{c,c'}\cap(\wo(E)\cup\{E\}),\] there is some $w$-sequence from $c$ to $c'$ that is entirely contained in $\hat{C}$.
                This, together with Lemma~\ref{l: checking wc}, proves the sublemma.
    			
    			By assumption, there is a $w$-sequence \[p\colon c=c_{0},c_{1},\dots, c_{n}=c'\] from $c$ to $c'$ contained in $C\cdot G$,
    			since the latter is weakly convex. We claim that in fact $p\subset C$. Indeed, $c_{i}=d_{i}\cdot g_i$ for some $d_{i}\in C$ and $g_i\in G$. We know that $R_{E}(c_{0},d_{i})$ holds, since $c,d_{i}\in C$. Therefore $R_{E}(c_{i},d_{i})$, whence $g_i\in G[E]$ and $c_i\in d_i\cdot G[E]\subseteq C$.
    			Notice that $\pa{c_{i'},c_{j'}}$ is a subword of $\pa{c_{i},c_{j}}$ for $i\leq i'<j'\leq j$; in particular, the latter contains a domain properly containing $\alpha$ whenever the former does. It follows that $\{i\in [0,n]\,|\,c_{i}\in B\}$ consists of consecutive integers
    			for any $B\in\mathcal{B}$.
    			
    			By the discussion above, we may assume that $c$ and $c'$ are not contained in a common $\hat B$; hence
    			there are distinct sets $\{B_{1},B_{2},\dots, B_{r}\}$ with $r\geq 2$, and
    			\[m_{1}=-1,m_{2},\dots, m_{r+1}=n\] such that $c_{i}\in B_{\ell}$ for $i\in[m_{\ell}+1,m_{\ell+1}]$.
    			We can replace each subsequence \[p_{i}=(c_{i})_{m_{\ell}+1\leq i\leq m_{\ell+1}}\] with some new sequence
    			\[p'_{i}=(c'_{i})_{m_{\ell}+1\leq i\leq m_{\ell+1}}\subseteq \hat{B}_{i},\] as we have already
    			argued previously, ensuring that
    			\[c'_{0}=c_{0}=c\in B_{1}\quad \textrm{and}\quad c'_{n}=c_{n}=c'\in B_{r}.\]
    			We claim that the resulting sequence $(c'_{i})_{i=0}^{n}$ is still a
    			$w$-sequence from $c$ to $c'$.
    			
    			Indeed, let $0\leq i\leq n-1$.
    			If \[\{i,i+1\}\subseteq [m_{l}+1,m_{l+1}]\] for some $1\leq \ell\leq r$ then $\pa{c_{i},c_{i+1}}=\pa{c'_{i},c'_{i+1}}$, as
    			follows directly from the choice of the $c'_{j}$. Otherwise, any \[w\in\pa{c_{i},c_{i+1}}\in\wo(E)\cup\{E\}\]
    			contains a domain $D$ properly containing $\alpha$. Since
    			$R_{\alpha D\alpha}(c'_{i},c'_{i+1})$ and since $R_{\alpha D\alpha}$ is equivalent to $R_{D}$, the lemma now follows.
    		\end{subproof}

            For each $B\in\mathcal B$ we have $\hat B\cdot G[E]=\hat B$, since the presence of $\alpha$ in a representative of $\pa{c(B),b}$ in $\wo(E)$ is preserved by the action of $G[E]$ on $b$. Therefore $\hat C\cdot G[E]=\hat C$, and
    		Lemma \ref{l: hatC0 is weakly convex} allows us to apply the inductive hypothesis to $\hat{C}$ and $F\setminus\{\alpha\}$,
    		which yields $C'\subseteq C$ with $C'\cdot G$ weakly convex, $C'\cdot G[E]\subseteq C'$, and such that no
    		sequence \[w\in\pa{c,c'}\cap(\wo(E)\cup\{E\})\] with connected domains contains any $\beta\in F\setminus\{\alpha\}$;
    		thus it follows that
    		no sequence $w$ contains a letter $\beta\in F$ by the choice of $\hat{C}$.
    	\end{proof}
    	
    	\begin{definition}[Cones and Twisted Cones]\label{def:cones}
    		Let \[D\in\D_0,\quad  A\subset N, \quad a_{0}\in A.\]
    		We define the \emph{cone}
    		\[
    			\wcn_{A}(a_{0},D)=\{a\in A\,|\,\pa{a_{0},a}\perp^{\circ} D\},\] and the \emph{twisted cone}
    			\[\tcn_{A}(a_{0},D)=\{ag\,|\,a\in\wcn_{A}(a_{0},D),\,g\in G\}.
    		\]
    	\end{definition}

    	\begin{lemma}
		\label{l: cone}If $A$ is weakly convex, then $\tcn_{A}(a_{0},D)$ is weakly convex.
	\end{lemma}
	\begin{proof}
		It suffices, by Lemma~\ref{l: checking wc} applied with $E=\C_0(\Sigma)$, to prove the corresponding assertion for the cone $\wcn_A(a_0,D)$. Let $c,c'\in \wcn_A(a_0,D)$. Choose representatives
		\[
		w_1\in \pa{c,a_0},\qquad w_2\in \pa{a_0,c'}
		\]
		which are strictly orthogonal to $D$. Every representative of $\pa{c,c'}$ is obtained by reducing a concatenation of $w_1w_2$. Corollary~\ref{cor:bdy-disj} shows that there is a representative $u\in\pa{c,c'}$ with $u\perp^{\circ}D$. Since $A$ is weakly convex, there is a strict $u$--sequence from $c$ to $c'$ contained in $A$.
		If $b$ is an intermediate point of this sequence, then $\pa{a_0,b}$ is represented by a reduct of $w_1^{-1}$ followed by an initial segment of $u$. Both factors are strictly orthogonal to $D$, so Corollary~\ref{cor:bdy-disj} implies that $b\in\wcn_A(a_0,D)$. Thus the strict sequence lies in the cone, and Lemma~\ref{l: checking wc} gives that the twisted cone is weakly convex.
	\end{proof}

    \begin{definition}[Chosen boundary-avoiding cones]
\label{def:chosen-boundary-avoiding-cones}
Let $A\subseteq N$ be weakly convex, let $a_0\in A$, and let $D\in\D$.
Recall that
\[
    \wcn_A(a_0,D)=\{a\in A\mid \pa{a_0,a}\perp^\circ D\}.
\]
We apply Lemma~\ref{l: efficient convexity} to the set
\[
    C=\wcn_A(a_0,D),
\]
with
\[
    E=D^\perp,\qquad F=\partial D,
\]
and with the distinguished point $b_0=a_0$. A subset
\[
    C_0\subseteq \wcn_A(a_0,D)
\]
obtained from Lemma~\ref{l: efficient convexity} in this way will be called a
\emph{chosen boundary-avoiding cone} for the pair $(a_0,A)$ in the direction
$D$. We denote such a choice by
\[
    \cn_D(a_0,A):=C_0.
\]
Thus
\[
    \cn_D(a_0,A)\subseteq \wcn_A(a_0,D),
    \qquad
    \cn_D(a_0,A)\cdot G\subseteq \tcn_A(a_0,D).
\]

The notation $\cn_D(a_0,A)$ is not canonical: the choice of $C_0$ is part of
the data of a one-step extension. When we say that two one-step extensions,
or two weakly convex hulls, use the same cone, we mean that the same chosen
subset $\cn_D(a_0,A)$ is used before taking its $G$--saturation.
\end{definition}

\begin{table}[htbp]
\centering
\small
\renewcommand{\arraystretch}{1.25}
\setlength{\tabcolsep}{4pt}
\begin{tabular}{|p{2.5cm}|p{6.4cm}|p{5.0cm}|}
\hline
\textbf{Notation}
&
\textbf{Meaning}
&
\textbf{Role}
\\
\hline
Strict cone
\(\operatorname{con}^{\perp^\circ}_A(a_0,D)\)
&
$\operatorname{con}^{\perp^\circ}_A(a_0,D) := \{a\in A\mid \delta(a_0,a)\perp^\circ D\}$
&
The part of \(A\) whose relative position from \(a_0\) is strictly
orthogonal to \(D\).
\\
\hline
Twisted cone
\(\operatorname{tcon}^{\perp^\circ}_A(a_0,D)\)
&
$\operatorname{tcon}^{\perp^\circ}_A(a_0,D):=\operatorname{con}^{\perp^\circ}_A(a_0,D)\cdot G.$
&
The \(G\)--saturation of the strict cone; it is weakly convex by
Lemma~\ref{l: cone}.
\\
\hline

\(\cn_D(a_0,A)\)
&
A chosen boundary-avoiding subcone of
\(\operatorname{con}^{\perp^\circ}_A(a_0,D)\).
&
It avoids the boundary components of \(D\) and is used in the one-step
extension construction.
\\
\hline

\(\cn_D(a_0,A)\cdot G\)
&
The \(G\)--saturation of the chosen boundary-avoiding cone.
&
The cone whose parallel copy is adjoined in a one-step
\(D\)--extension.
\\
\hline
\end{tabular}
\caption{Cones used in the one-step \(D\)--step extension argument.}
\label{tab:cones}
\end{table}
    	
    	The following lemma is one of the key technical results of the entire paper, and it establishes
    	the back--and--forth property needed to prove quantifier elimination in
    	$\mathcal{M}^G$. Specifically, the quantifier--free type of a weakly convex set determines its type; see Theorem~\ref{t:Mqe}
	below.
    	
    	\begin{figure}[h!]
    		\psfrag{a}{$a_0$}
    		\psfrag{a'}{$a_0'$}
    		\psfrag{C}{$C_0$}
    		\psfrag{\hat{C}}{$C_0'$}
    		\psfrag{D}{$D$}
    		\psfrag{A}{$A$}
    		\psfrag{A'}{$A'$}
    		\includegraphics[width=4cm]{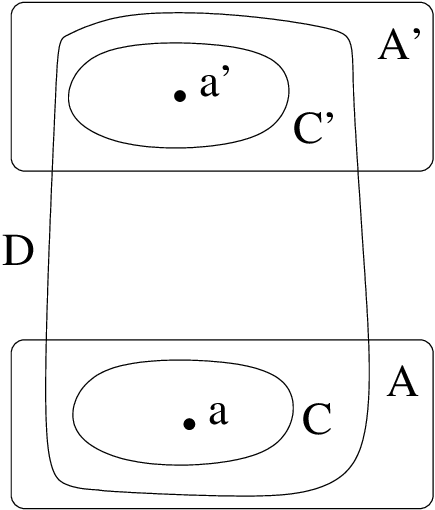}
    		\caption{Lemma~\ref{core qe}}
    	\end{figure}
    	
    	\begin{lemma}[$D$--Step Extension Lemma]
    		\label{core qe}
    		Suppose that $A\subseteq N$ is a weakly convex set of cardinality
less than $\kappa$ for some infinite cardinal $\kappa$, let $a_0\in A$ be a
basepoint, and let $D\in\D$. Let
\[
    C_0=\cn_D(a_0,A),
    \qquad
    C=C_0\cdot G\subseteq \tcn_A(a_0,D)
\]
be the chosen boundary-avoiding cone from
Definition~\ref{def:chosen-boundary-avoiding-cones}.
    		Consider the type:
    		\begin{align*}
    			q^{D}_{a_{0},A}(x_{C},y_{A}):=p^{D}_{a_{0},A}(x_{a_{0}},y_{A})\cup
    			\qftp^{x_C}(C)\cup\qftp^{y_A}(A)\cup\{R_{D}(x_{c},y_{c})\,|\,c\in C_{0}\}.
            \end{align*}
             The following conclusions hold:
    		\enum{(i)}{
    			\item \label{item extension}
    			Given a subset $\cc{A}=(\cc{a})_{a\in A}$ inside
    			$N$ that is isomorphic to $A$ and an $\c{a}_{0}$ such that
    			$\c{a}_{0}\cc{A}\models p^{D}_{a_{0},A}$, there exists a copy $\c{C}$ of $C$
    			extending $\c{a}_{0}$ such that $\c{C}\cc{A}\models q^{D}_{a_{0},A}$.
    			\item \label{item uniqueness}
    			The type  $q^{D}_{a_{0},A}$
    			implies the quantifier-free type $\qftp^{x_{C},y_{A}}(C,A)$.
    			\item \label{item convexity}
    			Let $\c{C}\cc{A}$ be a tuple in an arbitrary model of $\Th{\M^G}$ such that
    			\[\c{C}\cc{A}\models q^{D}_{a_{0},A}.\] Then $\c{C}\cc{A}$
    			is \wc as a set.
    		}
    	\end{lemma}

    	Before giving the proof of Lemma~\ref{core qe}, we explain its meaning.
    	The type
    	$p^{D}_{a_{0},A}(x_{a_{0}},y_{A})$  is the type of a new basepoint,
    	whose role is
    	analogous to that of $a_0\in C_0$. The tuple $y_A$ corresponds
    	to a copy of $A$,
    	which is thought of as  the ``original'' copy of  $A$.
    	The type $\qftp^{x_{C}}(C)$
    	corresponds to a new copy of  $C$  in the $x$--variables.
    	The type  $\qftp^{y_{A}}(A)$
    	corresponds to the original copy of $A$ in the $y$--variables, and the type
    	\[\{R_{D}(x_{c},y_{c})\,|\,c\in C_{0}\}\] says  that
    	each point in  the new  copy
    	of $C_0$  (in the $x$--variables)  is $D$--related
    	to the corresponding point
    	in the original copy of $C_0$.

        The consistency of $q^{D}_{a_{0},A}(x_C,y_A)$ follows from Lemma~\ref{step consistency}, together with the parallel-lifting construction of Lemma~\ref{parallel lifting}.

    	\begin{proof}[Proof of Lemma~\ref{core qe}]
    		Item (\ref{item extension}) is a particular instance of Lemma \ref{parallel lifting},
    		and follows from the consistency of the relevant types.
    		
    		For (\ref{item uniqueness}),  let $(c,a)\in C\times A$.
    		By Corollary~\ref{qf characterization}, it suffices to show that
    		$q^{D}_{a_{0},A}$ completely determines the value of $\pa{x_{c},y_{a}}$.
    		Now, for arbitrary choices of $g\in G$ and $c,d\in N$,
    		the type $\pa{cg,d}$ is completely determined by
    		$\pa{c,d}$. By virtue of this, the fact that reducts without cancellation are unique and Lemma \ref{gate property},
    		it is enough to prove that
    		\[q^{D}_{a_{0},A}(x_{C},y_{A})\vdash p^{D}_{c,A}(x_{c},y_{A})\] for all 
    		$c\in C_{0}$.
    		
    		So, suppose that we are given $\c{C}$ and $\cc{A}$ such that
    		\[\c{C}\cc{A}\models q^{D}_{a_{0},A}(x_{C},y_{A}).\]
    		Clearly, $R_{D}(\c{c},\cc{c})$ for any $c\in C_{0}$, where $\cc{c}$ denotes the element of $\cc{A}$ corresponding to $c$.
    		Now, assume for a
    		contradiction that there exists a
    		$c\in C_{0}$ and $a\in A$ such that $\pa{\c{c},\cc{a}}=[u]$
    		for some $u\in\wo(D)$.
    		Then we have that
    		$\pa{\c{a}_{0},\cc{a}}$ is a reduct without cancellation of
    		\[D\pa{a_{0},a}\] by Lemma \ref{gate property}, and a reduct
    		of \[\pa{\c{a}_{0},\c{c}}u=\pa{a_{0},c}u.\]
    		Now, by construction \(C_0=\cn_D(a_0,A)\) is contained in the strict cone
\(\wcn_A(a_0,D)\). Hence \(\pa{a_0,c}\perp^\circ D\). Since $u\in\wo(D)$, we have that if
    		the component of $D$ being tracked is annular, then every letter of $u$ supported in that component is a group element. It follows that
    		$D$ cannot  occur in $[\pa{a_{0},c}u]$ in  the  case  where
    		$D$ is  non-annular, and $D$ cannot be absorbed by any domain
    		occurring in $[\pa{a_{0},c}u]$.
    		This contradicts the assumption that
    		$\pa{\c{a}_{0},\cc{a}}$ is a reduct without cancellation of $D\pa{a_{0},a}$.
    		
    		Now consider part (\ref{item convexity}).
    		We need to show that for all pairs of elements
    		$e'_{1}\in \c{C}$ and $e''_{2}\in\cc{A}$ corresponding to points  $e_1\in C$ and
    		$e_2\in A$ respectively,
    		and for all $w\in\pa{e_{1},e_{2}}$,
    		there exists a strict $w$--sequence from
    		$e'_{1}$ to $e''_{2}$ which is contained entirely in $\c{C}\cc{A}$.
    		
    		By replacing $e_1$ by an element in its $G$--orbit,
    		we may clearly assume that $e_1\in C_{0}$.
    		We know that $\pa{e'_1,e''_2}$ is the result of reducing
    		$D\pa{e_1,e_2}$ without cancellation, by Lemma~\ref{gate property}.

    		Suppose  first that $D$ is not absorbed by any domain occurring
    		in $\pa{e_1,e_2}$.
    		By Observation \ref{small wiggling} and Corollary \ref{D tracking},
    		we may assume that $w$ is of the form
    		$uDu'$, where $u\perp^{\circ} D$ and where
    		\[v=uu''u'\in\pa{e_1,e_2}.\]
    		Here, the word $u''$ is absorbed  by $D$.
    		Let $d_1,\ldots,d_k$  be a strict $uu''u'$--sequence in $A$, with $d_1=e_1$
    		and  $d_k=e_2$, and let
    		$d_i$ and $d_j$ denote the starting and ending points, respectively, of the sequence $u''$. Notice that
    		\[\{d_{1},\dots, d_{i}\}\subset \wcn_{A}(a_{0},D),\] since $u\perp^{\circ}D$. Now, we observe that there is a $u$-sequence
    		\[d_1=f_1,f_2,\ldots,f_i\subset C_0,\] satisfying
    		$R_{\partial D}(d_{\ell},f_{\ell})$ for $2\leq \ell\leq i$, by Lemma~\ref{l: efficient convexity}.

    		Since  $u''\in\wo(D)$, we have $R_D^*(f_i', d_j'')$ and so \[f_1',f_2',\ldots,f_i',d_j'',\ldots d_{k}''\] is a strict $w$--sequence from
    		$e_1'$ to $e_2''$.

    		Finally, we suppose that $D$  is  absorbed by a domain in $\pa{e_1,e_2}$,
    		so that we may assume that prior  to application of the
    		absorption  move $(\mathrm{Abs}_{\subset})$, the word $w$ is of the form
    		$uDEu'$, so that we may write
    		\[w=uEu'\in\pa{e'_1,e''_2},\]
    		with $u\perp^{\circ} D$. As before, let $e_1=d_1,\ldots,d_k=e_2$ be a strict $w$--sequence from $e_1$ to $e_2$, and
    		write \[d_1,\ldots,d_i\subset\wcn_A(a_0,D)\] for a $u$-sequence. Choose
    		a $u$-sequence \[d_1=f_1,f_2,\ldots,f_i\subset C_0\] satisfying
    		$R_{\partial D}(d_{\ell},f_{\ell})$ for $2\leq \ell\leq i$, as allowed by Lemma~\ref{l: efficient convexity}. There is a point $n\in \c{C}\cc{A}$
    		such that $R_D(f_i',n)$ and $R_E(n,d_i'')$, and since $R_E(f_i',d_{i+1}'')$, we have
    		that \[f_1',\ldots,f_i', d_{i+1}'',\ldots,d_k''\] is a strict $w$--sequence in $\c{C}\cc{A}$ between $e_1'$ and $e_2''$.
    	\end{proof}

    \begin{definition}[One-step extensions and weakly convex hulls]
\label{def:one-step-extensions}
Let \(A\subseteq N\) be weakly convex, let \(a_0\in A\), and let
\(D\in\D\). Suppose that \(b\) is \(D\)--step away from \(A\) over the
basepoint \(a_0\). Let
\[
C_0=\cn_D(a_0,A),
\qquad
C=C_0\cdot G
\]
be the chosen boundary-avoiding cone from
Definition~\ref{def:chosen-boundary-avoiding-cones}.

A weakly convex set \(A^+\) is a
\emph{one-step \(D\)--extension of \(A\) through \(b\) over \(a_0\)}
if
\[
A^+=A\cup C',
\]
where \(C'=(c')_{c\in C}\) is indexed by \(C\), with
\(c'_{a_0}=b\), and
\[
C'A\models q^D_{a_0,A}(x_C,y_A),
\]
where \(q^D_{a_0,A}\) is the type constructed in
Lemma~\ref{core qe}.

Equivalently, \(A^+\) is obtained from \(A\) by adjoining a parallel
copy of the twisted cone \(C\), with the copy of \(a_0\) replaced by
the new \(D\)--generic point \(b\), and with corresponding points
\(D\)--related over the boundary-avoiding subcone \(C_0\).

\smallskip

Let $p:a_0,a_1,\ldots,a_k=b_0$ be a strict domain-only sequence whose terminal
point \(b_0\) belongs to a weakly convex set \(B\). For \(0\leq i<k\), let
\(D_i\) denote the domain letter of the step from \(a_i\) to \(a_{i+1}\).

We denote by \(\mathcal H(p,B)\) the collection of weakly convex sets \(H\)
for which there are weakly convex sets
\[
    H=H_0\supseteq H_1\supseteq\cdots\supseteq H_k=B
\]
such that, for every \(0\leq i<k\), the set \(H_i\) is a one-step
\(D_i\)--extension of \(H_{i+1}\) through \(a_i\) over the basepoint
\(a_{i+1}\). An element of \(\mathcal H(p,B)\) is called a
\emph{weakly convex \(p\)--hull of \(B\)}.

When \(p\) is a strict domain-only sequence from \(a\) to \(b\), we write
\[
    \mathcal H_p=\mathcal H(p,bG).
\]
For \(a,b\in N\), we write
\[
    \mathcal H_{a,b}=\bigcup_p \mathcal H_p,
\]
where \(p\) ranges over all strict domain-only sequences from \(a\) to an
element of \(bG\). Equivalently, if a reduced representative of
\(\delta(a,b)\) has a terminal group letter, we absorb that group letter into
the terminal \(G\)-orbit before forming the hull.
\end{definition}

\begin{remark}[Terminology for \(D\)-steps]
We use the phrase \emph{\(D\)--step away} only for the type
\(p^D_{a_0,A}\): a point \(b\) is \(D\)--step away from a weakly convex set
\(A\) over \(a_0\) when it realizes this type. We use
\emph{one-step \(D\)--extension} for the weakly convex set obtained by
adjoining the corresponding lifted cone. Thus a weakly convex hull is obtained
by iterating one-step \(D\)--extensions along a strict domain-only minimizing
sequence.
\end{remark}
	\begin{corollary}[Finite weakly convex extension]
		\label{cor:finite-wc-extension}
		Let $A\subseteq N$ be weakly convex and let $F\subseteq N$ be finite. Then there is a weakly convex set $B\subseteq N$ containing $A\cup F\cdot G$ which is obtained from \(A\) by finitely many one-step \(D\)--extensions in the sense of Definition~\ref{def:one-step-extensions}.
	\end{corollary}
	\begin{proof}
		Adjoin the elements of $F$ one at a time. For a new point $f$, choose a minimizing sequence from $f$ to the weakly convex set already constructed. Lemma~\ref{gate property} and Lemma~\ref{core qe} allow one to add the successive steps of this sequence while preserving weak convexity. Since the sequence and $F$ are finite, the process terminates.
	\end{proof}

	\subsection{Relative quantifier elimination}\label{sec:rel-qe}

	The preceding extension lemma gives the back-and-forth argument needed for relative quantifier elimination.
	We record the argument here, before turning to the stronger absolute quantifier-elimination theorem for the pure subgroup.
	The proofs in this subsection are close to the corresponding arguments for right-angled buildings in~\cite{BPZ17}, with the $D$--step construction of Lemma~\ref{core qe} replacing the building-theoretic extension step.

    \begin{definition}[Basepoints]

    Let $\mathcal{N}$ be an $\omega$--saturated model of $\Th{\M^G}$ with universe $N$. Given $a,b\in N$, we define $\Or(a,b)=\Or(\pa{a,b})$.

    	Generalizing basepoints from Definition~\ref{one step away},
	we say that $a_{0}\in A$ is a \emph{basepoint} for $b$ in $A$ if
    	$\Or(b,a_0)=\Or(b,A)$.
	\end{definition}
    	
    	The following lemmas record how basepoints behave when one follows a minimizing sequence toward a weakly convex set.

    	 \begin{lemma}
    		\label{several steps away}Let $A\subset N$ be a weakly convex set, and
    		let $a_{0}\in A$ be a basepoint for a point
    		$b\in N\setminus A$.
    		Let $w\in\pa{b,a_{0}}$, and suppose that $b'$ is the
    		penultimate point in a strict $w$--sequence
    		from $b$ to $a_{0}$. Suppose furthermore that $\W_{D}(b',a_{0})$
    		for a domain $D$.
    		Then $b'$ is $D$--step away from $A$ with basepoint $a_{0}$.
    		
		Moreover, suppose $\{a_{0}\}\cup A\subseteq B$ is weakly convex, and that $B$ is constructed by iterated $D$--step extensions
		as in Lemma~\ref{core qe}. Then
    		the element $b'$ is a basepoint for $b$ in $B$.
    	\end{lemma}
	\begin{proof}
		Write $w=w_{0}D$ so that $w_0$ is represented by the initial segment of the chosen strict sequence from $b$ to $b'$. If $b'$ were not $D$--step away from $A$ over $a_0$, then for some $a\in A$ there would be a word $u\in\wo(D)$ with $R_u(b',a)$. Concatenating the initial segment from $b$ to $b'$ with this $u$--sequence would give a relation from $b$ to $A$ whose associated ordinal is strictly smaller than $\Or(w)$, because the final letter $D$ has been replaced by a word supported in $D$. This contradicts the assumption that $a_0$ is a basepoint for $b$ in $A$.

		For the final assertion, suppose that $b'$ were not a basepoint for $b$ in the weakly convex set $B$ obtained by iterated step extensions. Choose $c\in B$ with $\Or(b,c)<\Or(b,b')$. Tracing the construction of $B$ backwards and using the Gate Property at the last step at which a minimizing sequence enters the new part of $B$, one obtains a point of the preceding weakly convex set with the same smaller ordinal. Iterating this argument eventually gives a point of $A$ closer to $b$ than $a_0$, again contradicting the choice of $a_0$.
	\end{proof}

    	\begin{lemma}
    		\label{gate property many steps}If $A\subset N$ is weakly convex and
    		if $a_{0}\in A$ is a basepoint for $b\in N\setminus A$, then for all
    		$a\in A$, the class $\pa{b,a}$ is the unique equivalence class of
    		reduced words that can be obtained by reduction from
    		$\pa{b,a_{0}}\pa{a_{0},a}$ without using the move (C).
    	\end{lemma}
    	\begin{proof}
    		This follows from repeated application of Lemma~\ref{several steps away}.
    	\end{proof}

    	Our preliminary result on quantifier elimination is as follows.
    	\begin{theorem}\label{t:Mqe}
    		The quantifier-free type of a weakly convex set of a model of $\Th{\mathcal{M}^G}$
    		determines its type.
    	\end{theorem}
    	\begin{proof}
    		Let $\kappa$ be an infinite cardinal, and
    		let  $\mathcal{N}$ and  $\mathcal{N}'$ be two  given $\kappa$-saturated models
    		of $\Th{\mathcal{M}^G}$. Let \[f:A\longrightarrow \tilde{A}\] be a  given  isomorphism
    		between two weakly convex sets, each of  which is the union of less than $\kappa$ many
    		$G$-orbits,  and let $c\in N\setminus A$.
    		We will show  that the map $f$ extends to a partial isomorphism
    		\[\tilde{f}:B\longrightarrow\tilde{B},\] where here $B$ and $\tilde{B}$ are  both
    		weakly convex with less than $\kappa$ many $G$-orbits and $B$ contains $c$, thus showing the back-and-forth property.
    		Notationally, we will write $\tilde a=f(a)$ for $a\in A$.

    		By a straightforward induction on the length of a
    		minimal sequence from $c$ to $A$, we may assume that
    		$c$ is a $D$-step away from $A$ over a basepoint
    		$a_{0}$, for a suitable $D$ and
    		$a_0$.
    		By Lemma \ref{step consistency} and the fact that $\mathcal{N}'$ is
    		$\kappa$-saturated, there exists
    		$\tilde{c}\in N'$ where \[\mathcal{N}'\models
    		p^{D}_{a_{0},A}(\tilde{c}).\]
    		Part (\ref{item extension})
    		of Lemma \ref{core qe} yields realizations $B$ and $\tilde{B}$ of
    		$q^{D}_{a_{0},A}$ in $\mathcal{N}$ and $\mathcal{N}'$ extending $cA$ and
    		$\tilde{c}\tilde{A}$ respectively.
    		
    		Both  $B$ and $\tilde{B}$ are weakly convex, by item (\ref{item convexity})
    		of Lemma \ref{core qe}. Finally, item (\ref{item uniqueness})
    		provides the desired unique extension of $f$  to \[\tilde{f}:B\longrightarrow\tilde{B}~.\] 
    	\end{proof}
    	
    	\begin{corollary}\label{cor:ext-qe}
    		The first order theory of $\M^G$ has quantifier elimination relative to existential formulae.
    	\end{corollary}
    	\begin{proof}
	As before, let $\mathcal N$ be a sufficiently saturated model of $\Th{\M^G}$ with universe $N$.
    		Let $a,b$ be finite tuples in $N$ with the same existential type, i.e. for every quantifier-free formula $\phi(x,y)$, we have \[\mathcal N\models
    		\exists x \phi(x,a)\leftrightarrow \exists x\phi(x,b).\] Now, Corollary~\ref{cor:finite-wc-extension} shows that $a$ can be extended to a finite
    		tuple $(a,a')$, with the property that $(a\cup a')\cdot G$ is weakly convex, so therefore the quantifier--free type of $(a\cup a')\cdot G$
    		determines its type by Theorem~\ref{t:Mqe}.
    		Since $a$ and $b$ have the same existential type, we have that there exists a tuple $b'$ such that
    		$(a,a')\equiv (b,b')$, and so the types of $(a,a')$ and $(b,b')$ coincide.
    	\end{proof}

\section{Quantifier elimination in $\M^G$}\label{sec:full-qe}
  	In this section, we improve the relative quantifier-elimination result of Subsection~\ref{sec:rel-qe} to absolute quantifier elimination in $\Th{\M^G}$, where $G$ is a certain
  	finite index subgroup consisting of pure mapping classes. In order to carry this task out, we will need to investigate the properties of
  	definable equivalence relations on $\M^G$, and the relationship between definable and algebraic closures.
  	Throughout, $\emptyset\neq\D\subset\D_0$ is a $G$--invariant, downward
  	closed collection of connected domains.
  	
  	\subsection{Pure mapping classes and orthogonality}
    	
    	Recall that if $g$ is a mapping class of a surface $\Sigma$,
    	then there is a unique, smallest
    	(possibly empty) collection of isotopy classes of pairwise disjoint simple closed curves $C_g\subset\mC_0(\Sigma)$ such
    	that $g(C_g)=C_g$, called a \emph{canonical reduction system}. See~\cite{BLM-duke} for details and background. A mapping class $g$
    	will be called \emph{pure} if $g$ fixes $C_g$ elementwise, and if the restriction of $g$ to each component of $\Sigma_g$ is either
    	trivial or a pseudo-Anosov mapping class. It is a standard fact that if $G<\Mod^{\pm}(\Sigma)$ denote the kernel of the natural map
    	\[G\leq \Mod^{\pm}(\Sigma)\longrightarrow \Aut(H_1(\Sigma,\bZ/3\bZ)),\]
	then $G$ consists of pure mapping classes~\cite{farb_margalit_11}. For the
	rest of this section, we will fix $G$ to be this kernel.

	The purity hypothesis is used from this section onward, principally in Lemmas~\ref{lem:pure-ortho} and~\ref{lem:pure-ortho-contain} and in the support computations that follow. The relative quantifier-elimination result of Subsection~\ref{sec:rel-qe} and the $\omega$--stability result of Subsection~\ref{ss:omega-stability} remain valid for arbitrary finite-index subgroups; absolute quantifier elimination, the canonical-base analysis, and the interpretation-rigidity arguments below use this pure finite-index subgroup.
    	
    	The following result relates pure mapping classes to orthogonality.
    	\begin{lemma}\label{lem:pure-ortho}
    		Let $E,F\in\D$ be distinct domains; let $g\in G$. If $g(E)=F$ then $E\not\perp F$.
    	\end{lemma}
    	\begin{proof}
    		Clearly it suffices to show that if $E$ is proper and nonempty then either $g(E)\not\perp E$ or $g(E)=E$.
    		It is a standard fact that if $g\in G$ and $\alpha\in\mC_0(\Sigma)$
    		then either $g(\alpha)\cap\alpha\neq\emptyset$ or $g(\alpha)=\alpha$
    		(see~\cite{BestBromFuji}, for example). This immediately implies the conclusion for annular regions.
    		
    		If $E$ is non-annular and $g(E)\neq E$ then either \[\Fill(E)\cap g(\Fill(E))\subset\partial E,\] or there exists an annular region
    		$D\subset \partial E$ such that $D\not\perp g(D)$. In the second case, we are clearly finished. In the first case, suppose that
    		there is at least one curve $\alpha\in E$ that is not peripheral in $E$. Then $\alpha$ is not fixed by $g$ and $g(\alpha)\perp\alpha$,
    		which is a contradiction. Thus we may conclude that $E$ has no non-peripheral curves and so must be a pair of pants.
    		Since $E$ consists exactly of its boundary curves, we must
    		then have that $g(E)=E$.
    	\end{proof}
    	
    	By an almost identical argument, one can establish the following:
    	
    	\begin{lemma}\label{lem:pure-ortho-contain}
    		Let $E,F\in\D$ be distinct domains, and let $g\in G$. If $E\perp F$ then $g(E)\not\subseteq F$.
    	\end{lemma}
    	\begin{proof}
    		We argue by considering connected components. If a connected component of $E$ is annular and maps into $F$, then its core curve is carried to a curve disjoint from the original core curve. This is impossible for a pure mapping class unless the curve is fixed. Similarly, if the relevant component of $E$ is a pair of pants, then its boundary multicurve cannot be carried into an orthogonal domain unless it is fixed componentwise.
    		
    		It remains to consider the case in which a connected component $E_0$ of $E$ has positive complexity and $g(E_0)$ is contained in a connected component of $F$ of positive complexity. Since $E\perp F$, every nonperipheral curve of $E_0$ is disjoint from every curve of $F$. Choose a nonperipheral curve $\alpha\in E_0$ such that $g(\alpha)$ is nonperipheral in $g(E_0)\subseteq F$. Then $\alpha$ and $g(\alpha)$ are disjoint. Since $g$ is pure, the standard purity dichotomy for curves gives $g(\alpha)=\alpha$, contradicting $\alpha\in E$ and $g(\alpha)\in F$ with $E$ and $F$ distinct and orthogonal. Thus $g(E)$ cannot be contained in $F$.
    	\end{proof}
    	
    \subsection{Imaginaries and rigidity of weakly convex sets}\label{ss:imagine-wc}
    	We retain notation from the previous section, so that $\mathcal{N}$ denotes a countable
    	$\omega$--saturated model of $\Th{\M^G}$.
    	Denote by $\hat{N}$ the collection of all imaginaries of the form $[a]_{D}$, i.e.~the $R_D$ equivalence class of $a$,
    	where $a\in N$ and $D\in\D$. Given a subset $A\subseteq N$, we denote by
    	$\hat{A}$
    	the collection of all the classes of the form $[a]_{D}$ with $a\in A$.
    	
	\begin{definition}[Imaginary algebraic closure]
	We will write $\widehat{\acl}(A)$ for the \emph{imaginary algebraic closure} of $A$, defined to be
	the intersection of $\acl^{\EQ}(A)$ with $\hat{N}$.
	\end{definition}
    	
    	Given $D\in\D$ and $a\in N$ we let $\subg{[a]_{D}}_{G}$
    	be the collection of all classes of the form $[a']_{E}$, where
    	$a'=a g$ for some $g$ such that $g^{-1}(D)\subseteq E$. Notice that
    	\[\subg{[a]_{D}}_{G}\subseteq \dcl^{\EQ}([a]_D).\]
	
	We direct the reader to Subsection~\ref{sssec:imaginaries} for a discussion of $\dcl^{\EQ}$ and $\acl^{\EQ}$.
    	
    	\begin{lemma}
    		\label{l: wc and acl}If $[a']_{E}$ with $a'=a g$ for some $g\in G$ is such that
    		$D\nsubseteq g(E)$,
    		then the orbit of $[a']_{E}$ under the action of the group of
    		automorphisms of $a G$ which preserve
    		$[a]_{D}$ is infinite.
    	\end{lemma}
    	\begin{proof}
    		Let $h\in G[D]$. There is a unique automorphism $\phi_{h}$ of the orbit
    		$aG$ sending $a$ to $ah$, which sends $ak$ to $ahk$ for any $k\in G$. These are exactly the automorphisms of the $G$--orbit $aG$ which preserve the class $[a]_D$. If we let $Q=G[E]$, then
    		\[\phi_h(a')=\phi_h(ag)=ahg\]
    		and $\phi_{h'}([a']_E)=\phi_h([a']_E)$ if and only if $hgQ=h'gQ$, i.e.~if and only if $(h^{-1}h')^g\in Q$, i.e.~if and only if
    		$h^{-1}h'\in Q^{g^{-1}}=G[g(E)]$.
    		Thus the orbit in question has cardinality at least that of
    		\[G[D]/(G[D]\cap G[g(E)]).
    		\]
    		Since $D\nsubseteq g(E)$, there is a connected component, or a positive-complexity subdomain, of $D$ which is not contained in $g(E)$. If this component is annular, take a nonzero power of the corresponding Dehn twist lying in $G[D]$; if it has positive complexity, take a pseudo-Anosov supported in that component, again after passing to a power lying in $G[D]$. No nonzero power of the chosen element lies in $G[g(E)]$, because its support is not contained in $g(E)$. Hence the displayed quotient is infinite, and so the orbit of $[a']_E$ is infinite.
    	\end{proof}

   \begin{definition}[Weakly convex hull notation]
We use the notation \(H(p,B)\), \(H_p\), and \(H_{a,b}\) introduced in
Definition~\ref{def:one-step-extensions}. Thus, a weakly convex \(p\)--hull is obtained by iterating
the one-step extensions supplied by Lemma~\ref{core qe} along the letters of the
sequence \(p\). We recall that \(\delta(a,b)\) denotes the reduced class recording the relative
position of \(a\) and \(b\), and that
\[
\Or(a,b)=\Or(\delta(a,b)).
\]
When \(p\) is a strict sequence from \(a\) to \(b\), we write \(H_p\) for
\(H(p,bG)\); for \(a,b\in N\), we write \(H_{a,b}\) for the union of \(H_p\)
over all strict sequences \(p\) from \(a\) to \(b\).
\end{definition}

The following is an easy observation which is left to the reader:
    	\begin{observation}
    		\label{fellow travelling convexity} Let $p$ and $B$ be as above and let $p'$
    		be a sequence (of the same length) which fellow travels with $p$.
    		Then $\mathcal{H}(p,B)=\mathcal{H}(p',B)$.
    	\end{observation}

    	\begin{observation}
    		\label{isomorphic paths}   	Given two $w$-sequences with endpoints $a,b$ and $a',b'$
    		respectively, an iterated application of Lemma \ref{core qe} yields the existence
    		of an isomorphism between an arbitrary $H\in\mathcal{H}_{p}$ and an arbitrary $H'\in\mathcal{H}_{p'}$
    		that sends $(a,b)$ to $(a',b')$. Thus, $p\equiv p'$ by relative quantifier elimination (Corollary~\ref{cor:ext-qe}).
    	\end{observation}
    	
     	The following lemma shows that a strict, minimizing sequence from an arbitrary point to a weakly convex set admits an essentially
	    canonical decomposition into two sequences which are orthogonal to each other, one of which lies entirely in the weakly convex set.

In the next lemma, $\cn_D(a,H)$ denotes the chosen boundary-avoiding cone
from Definition~\ref{def:chosen-boundary-avoiding-cones}. The choice is part
of the construction of the relevant one-step extension. We keep the notation
explicit because, in Part~\ref{part3}, two hull constructions are compared
only after requiring the same chosen cones to be used at each layer.

    Part~\ref{part3} of the following lemma is only used in
Section~\ref{sec:int-rigid}. The reader may ignore it on a first reading.
Parts~\ref{part1} and~\ref{part2} say that a point in a weakly convex
\(p\)-hull can be projected to a point on the original minimizing sequence,
up to a \(G\)-translate, and that the transverse displacement from the
projected point is strongly orthogonal to the remaining tail.
    
\begin{lemma}\label{lem-convex-extensions}
    Let
    \[
        p:a_0,a_1,\ldots,a_k=b_0
    \]
    be a strict minimizing sequence from $a_0$ to a weakly convex set $B$, with
    basepoint $b_0\in B$, and suppose that $p$ has domain-only type
    \[
        w=D_0D_1\cdots D_{k-1}\in \D^* .
    \]
    Let $H\in \mathcal H(p,B)$, and fix a construction
    \[
        H=H_0\supset H_1\supset\cdots\supset H_k=B
    \]
    such that, for each $0\leq \ell<k$, the set $H_\ell$ is the one-step
    $D_\ell$--extension of $H_{\ell+1}$ through $a_\ell$ over the basepoint
    $a_{\ell+1}$, using the chosen set $\cn_{D_\ell}(a_{\ell+1},H_{\ell+1})$ for that step.
    Then, for every $c\in H$, there exist elements
    \[
        c_1\in cG,\qquad c_0\in H,\qquad b_1\in B
    \]
    such that:
    \begin{enumerate}[(I)]
        \item\label{part1}
        The element $c_0$ lies on a strict domain-only sequence from $a_0$ to
        $b_0$.
        \item\label{part2}
        We have
        \[
            \pa{a_0,c_1}=\pa{a_0,c_0}*\pa{c_0,c_1},
        \]
        and
        \[
            \pa{b_0,b_1}=\pa{c_0,c_1}\perp^*\pa{c_0,b_0}.
        \]
        Moreover, the classes
        \[
            \pa{a_0,c_0},\qquad \pa{c_0,b_0},\qquad \pa{c_0,c_1},
            \qquad \pa{c_1,b_1}
        \]
        have representatives in $\D^*$.
        \item\label{part3}
        Suppose, in addition, that
        \[
            p':a'_0,a'_1,\ldots,a'_k=b'_0\in B
        \]
        is another strict sequence of type $w$ such that the same filtration
        $H=H_0\supset\cdots\supset H_k=B$ also realizes $H$ as a weakly
        convex $p'$--hull of $B$, using the same sets $\cn_{D_\ell}(a_{\ell+1},H_{\ell+1})$
        at every one-step extension.  Equivalently, at the $\ell$th step,
        $a'_{\ell}$ is the $D_\ell$--lift of $a'_{\ell+1}$ in the same one-step
        extension used to lift $a_{\ell+1}$ to $a_\ell$.
        Assume also that the parallel displacement
        \[
            \pa{a'_{\ell},a_\ell}=[w_\perp]
        \]
        is independent of $\ell$ and is strongly orthogonal to the remaining
        tail of the sequence.  Then the same elements $c_1\in cG$ and
        $b_1\in B$ can be used for $p'$: there is a point $c'_0\in H$ such that
        \[
            \pa{a'_0,c_1}=\pa{a'_0,c'_0}*\pa{c'_0,c_1},
            \qquad
            \pa{b'_0,b_1}=\pa{c'_0,c_1}\perp^*\pa{c'_0,b'_0},
        \]
        and the classes involved have representatives in $\D^*$.  Moreover,
        one may choose $c'_0$ so that
        \[
            \pa{a'_0,c'_0}=\pa{a_0,c_0},\qquad
            \pa{c'_0,b'_0}=\pa{c_0,b_0},\qquad
            \pa{c_0,c'_0}\perp^* w .
        \]
    
    \end{enumerate}
\end{lemma}
\begin{proof}
    We prove the statement by induction on $k$, carrying the strengthened
    assertion in Part~\ref{part3} as part of the induction.  If $k=0$, then
    $H=B$ and the result is immediate by taking $c_0=b_0$, $c_1=c$, and
    $b_1=c$ after replacing $c$ by an element of its $G$--orbit if necessary.

    Assume now that $k>0$, and write
    \[
        q:a_1,a_2,\ldots,a_k=b_0
    \]
    for the tail of $p$.  Thus $H_1\in \mathcal H(q,B)$, and $H_0$ is the
    one-step $D_0$--extension of $H_1$ through $a_0$ over $a_1$.

    There are two cases. If \(c\) already lies in the hull \(H_1\), the
claim follows by induction and the Gate Property. If \(c\) lies in the new
part \(H_0\setminus H_1\), then \(c\) is a \(D_0\)-lift of a point of the
chosen cone \(\cn_{D_0}(a_1,H_1)\), and the only extra work is to remove
possible boundary-annulus ambiguity using the defining property of the chosen
cone.

    First suppose that $c\in H_1$.  By the induction hypothesis applied in
    $H_1$ to the tail $q$, there are points $d_0,d_1\in H_1$ and $b_1\in B$
    such that $d_1\in cG$, the point $d_0$ lies on a strict sequence from
    $a_1$ to $b_0$, and
    \[
        \pa{a_1,d_1}=\pa{a_1,d_0}*\pa{d_0,d_1},
        \qquad
        \pa{b_0,b_1}=\pa{d_0,d_1}\perp^*\pa{d_0,b_0}.
    \]
    Set $c_0=d_0$ and $c_1=d_1$.  Since $p$ is minimizing from $a_0$ to $B$,
    the Gate Property, Lemma~\ref{gate property}, says that the classes from
    $a_0$ to points of $H_1$ are obtained by reducing
    $D_0$ followed by the corresponding class from $a_1$, and this reduction
    is without cancellation.  Hence
    \[
        \pa{a_0,c_1}=D_0*\pa{a_1,c_1}
        =D_0*(\pa{a_1,c_0}*\pa{c_0,c_1})
        =\pa{a_0,c_0}*\pa{c_0,c_1},
    \]
    where the last equality uses associativity of $*$ and uniqueness of
    reduction without cancellation.  This proves Parts~\ref{part1} and
    \ref{part2} in this case.  The same argument, applied to the tail of
    $p'$ and using the same one-step extension at the first layer, proves the
    strengthened assertion in Part~\ref{part3}.

    It remains the case $c\in H_0\setminus H_1$.  By the definition
    of a one-step extension, after replacing $c$ by an element $\tilde c\in cG$
    there is a point $d\in\cn_{D_0}(a_1,H_1)$ such that
    \begin{equation}\label{eq:lifted-point-in-cone}
        \pa{a_0,\tilde c}=\pa{a_1,d},\qquad
        D_0\in\pa{\tilde c,d},\qquad
        \pa{a_1,d}\perp^*D_0 .
    \end{equation}
    Here $d$ is the point of $H_1$ whose $D_0$--lift is $\tilde c$.

    Apply the induction hypothesis in $H_1$ to the point $d$.  We obtain
    points $d_0,d_1\in H_1$ and $\tilde b_1\in B$, with $d_1\in dG$, such
    that
    \begin{enumerate}[(a)]
        \item $d_0$ lies on a strict sequence from $a_1$ to $b_0$;
        \item $\pa{a_1,d_1}=\pa{a_1,d_0}*\pa{d_0,d_1}$;
        \item $\pa{b_0,\tilde b_1}=\pa{d_0,d_1}\perp^*\pa{d_0,b_0}$;
        \item the classes $\pa{a_1,d_0}$, $\pa{d_0,b_0}$, and
        $\pa{d_0,d_1}$ have representatives in $\D^*$.
    \end{enumerate}
    Write $d=d_1h$ for $h\in G$.  From
    \eqref{eq:lifted-point-in-cone}, Lemma~\ref{orthogonality characterization},
    and the fact that $\pa{a_1,d_1}$ has a domain-only representative, we get
    $h\in G[D_0^\perp]$ and
    \[
        \pa{d_0,d_1}\perp^\circ D_0.
    \]

    There can still be boundary annuli of $D_0$ which occur in a domain-only
    representative of $\pa{b_0,\tilde b_1}$.  The defining property of the chosen
    set $\cn_{D_0}(a_1,H_1)$ removes this ambiguity.  More precisely, write a
    representative of $\pa{b_0,\tilde b_1}$ in the form
    \[
        Eu\simeq uE,
        \qquad E\subseteq \partial D_0,
        \qquad u\perp^*D_0,
        \qquad u\in\D^*.
    \]
    The components of $E$ are the boundary components of $D_0$ which are
    absorbed into the earlier part of the sequence.  By weak convexity of
    $B=H_k$ and $H_1$, choose $b_1\in B$ and $\tilde d_0\in H_1$ such that
    \[
        \pa{b_0,b_1}=[u],\qquad \pa{b_1,\tilde b_1}=E,
        \qquad
        \pa{d_0,\tilde d_0}=E,\qquad \pa{\tilde d_0,d_1}=[u].
    \]
    Now use the chosen set $\cn_{D_0}(a_1,H_1)$ to
    replace $\tilde d_0$ by a point $\hat d_0\in\cn_{D_0}(a_1,H_1)$ with
    \[
        \pa{a_1,\hat d_0}=\pa{a_1,d_0},
        \qquad
        \pa{\hat d_0,d_1}=[u].
    \]

    Let $c_0$ be the $D_0$--lift of $\hat d_0$ in the construction of $H_0$,
    and set
    \[
        c_1=\tilde c h^{-1},
    \]
    so that $c_1$ is the corresponding lift of $d_1$.  Then
    \[
        \pa{c_0,c_1}=[u],
        \qquad
        \pa{c_0,b_0}=D_0\pa{d_0,b_0},
        \qquad
        \pa{b_0,b_1}=\pa{c_0,c_1}\perp^*\pa{c_0,b_0}.
    \]
    The point $c_0$ lies on a strict sequence from $a_0$ to $b_0$, since it
    is the $D_0$--lift of a point lying on the corresponding tail sequence from
    $a_1$ to $b_0$.  Also
    \[
        \pa{a_0,c_1}=\pa{a_0,c_0}*\pa{c_0,c_1},
    \]
    by the same reduction-without-cancellation argument as in the first case.
    This proves Parts~\ref{part1} and~\ref{part2}.

    Finally assume that the parallel sequence $p'$ is present.  Since the same sets $\cn_{D_\ell}(a_{\ell+1},H_{\ell+1})$ are used at every layer,
    the $D_0$--extension just described can equally be read as the first
    one-step extension in the construction from $p'$.  Applying the strengthened induction hypothesis
    to the tail sequences in $H_1$ gives the same terminal translate $c_1$ and
    the same point $b_1\in B$, and gives a point $d'_0\in H_1$ with
    \[
        \pa{a'_1,d'_0}=\pa{a_1,d_0},\qquad
        \pa{d'_0,b'_0}=\pa{d_0,b_0},\qquad
        \pa{d_0,d'_0}\perp^*D_1\cdots D_{k-1}.
    \]
    Taking the $D_0$--lift of the boundary-adjusted version of $d'_0$ gives a
    point $c'_0\in H_0$ satisfying the identities in Part~\ref{part3}.  The
    equality of the terminal points $c_1$ and $b_1$ follows from the fact that
    the two constructions use the same one-step extension and the same set
    $\cn_{D_0}(a_1,H_1)$; the displayed orthogonality follows by adjoining
    the first step $D_0$ to the orthogonality obtained in the tail.
\end{proof}

			\begin{definition}\label{d:homomorphism}
    		Let $B,\hat{B}\subseteq N$. A map $f:B\to \hat{B}$ is a \emph{homomorphism} if
    		\[\pa{f(b_1),f(b_2)}\preceq\pa{b_1,b_2}\] for any $b_1,b_2\in B$. If
    		$A\subseteq\hat{B}$, we say that $f$ is an \emph{$A$--homomorphism} if it
    		preserves each class in $A$.
    		
    		Given
    		$A\subseteq\hat{B}$ we say that $B$ is
    		\emph{strongly incompressible} over $A$ if
    		all $A$--homomorphisms $f:B\to \hat{B}$ are isomorphic embeddings.
    	\end{definition}
    	
    	\begin{lemma}\label{retract}
    		Suppose $A\subsetneq B\subseteq N$, where $A$ is weakly convex. Then there is a
    		homomorphic retraction $f:B\to A$. In particular, if $B$ is
    		weakly convex and strongly incompressible over some $A\subseteq B$, then
    		any $A$-homomorphism from $B$ to itself is an isomorphism.
    	\end{lemma}
    	\begin{proof}
    		We follow the proof of Lemma 7.11 in \cite{BPZ17}. We first note that there is a
    		homomorphic retraction from $A$ to itself. Let $A\subset H\subset N$ be a
    		maximal weakly convex subset that admits a retraction to $A$,
    		and let $C=H\cap B$. We claim that $C=B$.
    		If not, let $b\in B\setminus C$. 
    		Lemma~\ref{core qe} furnishes a weakly convex extension
    		$H'$ of $H$ containing $b$ which retracts to $H$, and which by composition retracts
    		to $A$. This violates the maximality of $H$.
    	\end{proof}
    	
    	We now proceed to adapt the construction of strongly incompressible sets over a finite set of parameters in \cite{BPZ17}.
    	\newcommand{\ti}[0]{\mathcal{H}^{t}}
   
    	\begin{definition}[Twisted $e$--hull]
    		Let $e=e_{1}e_{2}\dots e_{k}\subset \hat{N}$ be a finite tuple. We define the \emph{twisted $e$--hull}
		$\ti(e)\subseteq\mathcal{P}(N)$
    		by induction on $k$:
            \begin{enumerate}
            \item for $k=1$, if $e_1=[a]_D$, set $\ti(e_1)=\{aG\}$; 
            \item for $k>1$, we set $\ti(e)$ to be the union of all
    		$\mathcal{H}(p,H)$, taken over all pairs $(p,H)$ with $H\in\ti(e_{2},\dots, e_{k})$ and with $p$ a minimizing sequence
    		from some $a\in e_{1}$ to $H$, for which $\Or(\pa{a,H})$ is minimized among all $a$ and $H$.
            \end{enumerate}
    	\end{definition}
    	
    	The following technical result will be necessary to prove rigidity of weakly convex sets in Lemma~\ref{lem: incompressibility}
    	below.
    	
    	\newcommand{\uv}[0]{u_{|}}
    	\newcommand{\uh}[0]{u_{-}}
    	\begin{lemma}
    		\label{absorption-orthogonality}Let $\uv,\uh,v,w\in\wo$, and $g_{1},g_2,g_3,h\in G$ be such that
    		\begin{enumerate}
          \item \label{uvuhinD} $\uv,\uh\in\D^{*}$,
    			\item \label{uvperpuh}$\uv\perp^{*}\uh$,
    			\item \label{vperpwuv} $v\perp^{*}w\uv$,
          \item \label{isreduced} $vg_{3}w\uv$ is reduced,
    			\item \label{equality of two forms}$[g_{1}w]*[\uh h]\simeq g_{2}vg_{3}w\uv $.
    		\end{enumerate}
    		Then $\uv$ is the trivial word.
    	\end{lemma}
	
	Here, the subscripts in the symbols $\uv$ and $\uh$ are intended to evoke ``verticality'' and ``horizontality''.
    	\begin{proof}[Proof of Lemma~\ref{absorption-orthogonality}]
    		We argue by induction on the number of domains in a reduced representative of $[g_{2}vg_{3}w\uv]$. The base case of zero
    		is trivial.
    		We may assume that the reduction without cancellation of $[g_{1}w]*[\uh h]$ involves only absorption of letters in $w$ by letters in $\uh$, and not the other way around, since we can always remove letters from $\uh$ absorbed by $w$ without changing the properties of $\uh$.
    		
    		Let $K_{1},K_{2},\dots, K_{m}$ be the components of the right end of $w$, and let $J_{|}$ be the collection of $i\in\{1,\dots,m\}$ such that $K_i\perp \uv$.
    		Let $J_{-}$ be the collection of indices $i\in\{1,\dots,m\}$ for which the component $K_i$ survives, after the reductions in $[g_1w]*[\uh h]$, as a component $h^{-1}(K_i)$ of the right end.
    		Equivalently, $K_i$ is not absorbed or cancelled by the part of $\uh$ which can be moved to its right.

    		Let $F^{|}_{1},\dots, F^{|}_{r_{|}}$ be the components of the right end of $\uv$ and $F^{-}_{1},\dots, F^{-}_{r_{-}}$ those components of the right end of $\uh$ which are not properly absorbed on the right by $w$. Also, let $F^{v}_{1},\dots, F^{v}_{r_{v}}$ be the components of the right end of $g_{3}^{-1}(v)$ that are orthogonal to $w\uv$.
    		
    		Let \[\zeta=[g_{1}w]*[\uh h]=[g_{2}vg_{3}w\uv].\]
    		On the one hand, using the representative $[g_1w]*[\uh h]$, the right end $E$ of $\zeta$ must consist of the components
    		$$
    		\{h^{-1}(K_{i})\}_{i\in J_{-}}\cup\{h^{-1}(F^{-}_{1}),\dots, h^{-1}( F^{-}_{r_{-}})\},
    		$$    		
    		while on the other hand, using the representative $g_2vg_3w\uv$, it must consist of the components
    		$$
    		\{K_{i}\}_{i\in J_{|}}\cup\{F^{|}_{1},\dots, F^{|}_{r_{|}}\}\cup\{F^{v}_{1},\dots, F^{v}_{r_{v}}\}.
    		$$
    		By Lemma \ref{ends}, we know that these two collections must coincide. Note that the two sets are listed in the expressions above without repetitions. In the first list, because of the definition of $J_{-}$. In the second, this is a result of our assumption that $vg_{3}w\uv$ is reduced. 
            
    		Combining Lemma \ref{lem:pure-ortho} with the fact that $\uv\perp^*\uh$ (for the first item below) and $v\perp^* w\uv$ (for the second), 
    		we conclude
    		\begin{itemize}
    			\item $\{F^{|}_{1},\dots, F^{|}_{r_{|}}\}\subseteq \{h^{-1}(K_{i})\}_{i\in J_{-}}$;
    			\item $\{F^{v}_{1},\dots, F^{v}_{r_{v}}\}\subseteq\{h^{-1}(F^{-}_{1}),\dots, h^{-1}( F^{-}_{r_{-}})\}$.
    		\end{itemize}
   
            In view of Lemma \ref{lem:pure-ortho-contain}, and the definition of $J_{|}$, we can strengthen the first of the bullet points above:
    		\[\{F^{|}_{1},\dots, F^{|}_{r_{|}}\}\subseteq \{h^{-1}(K_{i})\}_{i\in J_{-}\setminus J_{|}}.\]
            
    			Let $J=\{1,\dots,m\}$. Let $J^{abs}\subseteq J$ be the collection of $i\in J$ for which $K_{i}$ is properly absorbed on the left by $u_{-}$ into some	letter among the $F_{i}^{-}$. Write $J^{nabs}=J\setminus J^{abs}$. Since $u_{-}\perp^{*}u_{|}$, we must have $J^{abs}\subseteq J_{|}$, so that we have 
    			\begin{equation}\label{removing abs}
    				J^{nabs}\setminus J_{|}=J\setminus J_{|}.
    			\end{equation}

    	  Observe also that:
    		\begin{align*}
    			w\simeq w'K_{1}\dots K_{m}  &&
    			u_{-}=u_{-}'F^{-}_{1}\dots F^{-}_{r_{-}}\\
    			\uv=\uv'F^{|}_{1}\dots F^{|}_{r_{|}} &&
    			v=v'F^{v}_{1}\dots F^{v}_{r_{v}}.
    		\end{align*}

    		We conclude that $\zeta$ admits representatives as follows\\
    		
    		\begin{align*}
    			\left([g_{1}w'\prod_{i\in J^{nabs}\setminus J_{-}}K_{i}]*[u'_{-}h]\right)E
    			&& g_{2}v'g_{3}w'\left(\prod_{i\in J\setminus J_{|}}K_{i}\right)u_{|}'E .
    		\end{align*}

    		By Lemma \ref{ends}, we find that there exists an element $h'\in G[E]$ such that
    		
    		\begin{align}
    			\label{equivalence line}[g_{1}w'\prod_{i\in J^{nabs}\setminus J_{-}}K_{i}]*[
    			u'_{-}h]h'\simeq g_{2}v'g_{3}w'\left(\prod_{i\in J\setminus J_{|}}K_{i}\right)u_{|}'.
    		\end{align}

    		Now, consider the words
    		\begin{align*}
    			u^{0}_{-}=\left(\prod_{i\in (J_{|}\cap J^{nabs})\setminus J_{-}}K_{i}\right)u'_{-} &&  &u^{0}_{|}=
    			\left(\prod_{i\in J_{-}\setminus J_{|}}K_{i}\right)u'_{|} \\
    			w^{0}=w'\left(\prod_{i\in J^{nabs} \setminus(J_{-}\cup J_{|})}K_{i}\right) && &v^{0}=v'.
    		\end{align*}
    
    		Let us check that the tuple \[g_{1},g_{2},g_{3},hh',v^{0},w^{0},u^{0}_{|},u^{0}_{-}\]
    		satisfies the assumptions of the induction hypothesis. 
    		\begin{itemize}
    			\item The validity of \ref{uvuhinD} for the new tuples is clearly inherited from the old;
    			\item For \ref{uvperpuh}, we can derive the conclusion from the fact that for $j\in J_{-}$ the domain $K_{j}$ is orthogonal to $u'_{-}$, and the $K_{j}$, $j\in J$ are pairwise orthogonal.
    			\item For \ref{vperpwuv}, note that $v^0$ is a subword of $v$, and each letter in $w^{0}u^{0}_{|}$ is inherited from $w\uv$.  
    			\item The validity of  \ref{isreduced}, that is, the fact that the expression $v^{0}g_{3}w^{0}u^{0}_{|}$ is reduced is clear.
    			\item Item \ref{equality of two forms} follows from the following two equivalences. First of all,
    			that $g_{2}v^{0}g_{3}w^{0}\uv^{0}$ is equivalent to the right-hand side of \ref{equivalence line} follows from \ref{removing abs}. It remains to check that the left-hand side of \ref{equivalence line} is equivalent to 
    			$[g_{1}w^{0}]*[\uh^{0}hh']$. For this, it suffices to note that one can pass from	$g_{1}w^{0}\uh^{0}hh'$
    			to 
    			\[ g_{1}w'\left(\prod_{i\in J^{nabs}\setminus J_{-}}K_{i} \right)u'_{-}hh'\] 
    			by permuting the terms $K_{j}$ where $j$ ranges over 
    			\[
    			 J^{nabs}\setminus J_{-}=(J^{nabs}\setminus (J_{-}\cup J_{|}))\cup ((J_{|}\cap J^{nabs})\setminus J_{-})
    			\]
    		\end{itemize}

    		Likewise, the fact that $u^{0}_{-}$ and $u^{0}_{|}$ are strongly orthogonal is immediate from our
    		assumptions on the reduction of $[g_{1}w]*[u_{-}h]$ and that the product $wu_{|}$ was reduced.
    		
    		The induction hypothesis implies that $u^{0}_{|}$ is trivial, so $J_{-}\subseteq J_{|}$. This implies
            \[\emptyset\neq\{F^{|}_{1},\dots, F^{|}_{r_{|}}\}\subseteq \{h^{-1}(K_{i})\}_{i\in J_{-}\setminus J_{|}}=\emptyset,\]
            that is impossible (the right end of a non-empty word in $\D^{*}$ is always non-empty).
    	\end{proof}

    \begin{lemma}\label{lem-one step rigidity}
    Let $B$ be a weakly convex set and let $A$ be a one-step $D$--extension of $B$ through a point $a_{0}$ over the basepoint $b_{0}$, in the sense of Definition~\ref{def:one-step-extensions}. Suppose that there are $b\in B$ and $a,a'\in A\setminus B$ such that
    \[\Or(\pa{a,b})=\Or(\pa{a',b})=\Or(D).\]
    Then $aG=a'G$.
    \end{lemma}
    \begin{proof}
    We have $R_w(a,a')$, where $w=g_1ug_2$ for a suitable word $u\perp^*D$; this follows from the construction of a one-step $D$--extension, since points newly added over the same $D$--gate differ only by data strongly orthogonal to $D$, modulo the action of $G$. The boundary-avoidance step in Lemma~\ref{core qe} is precisely what upgrades the naive strict orthogonality to strong orthogonality here. By Lemma~\ref{gate property} and Lemma~\ref{core qe}, the class $\pa{a',b}$ has a representative of the form $h_1Dh_2$. It follows that $\pa{a,b}$ is a reduct of
    \[g_1ug_2h_1Dh_2.
    \]
    Since $g_2h_1\in G$ is pure, Lemma~\ref{lem:pure-ortho-contain} implies that no translate of $D$ by an element of $G$ can be contained in some domain in the word $u$, since $u\perp^*D$. Hence the only way for the resulting word to have associated ordinal $\Or(D)$ is for $u$ to be trivial. Thus $a$ and $a'$ lie in the same $G$--orbit. The same conclusion can also be obtained without explicitly invoking purity, but the pure case is the only one needed here.
    \end{proof}

We can now establish the following rigidity result for weakly convex sets.
    	
    	\begin{lemma}
    		\label{lem: incompressibility}
    		For a finite tuple $e\subset \hat N$ and $H\in\ti(e)$ we have:
    		\begin{enumerate}
    			\item \label{incompressibility} $H$ is strongly incompressible over $e$;
    			\item \label{automorphisms preserve orbits} Automorphisms of $H$ fixing $e$ preserve each of the $G$-orbits in $H$.
    		\end{enumerate}
    	\end{lemma}
    	\begin{proof}
    		The proof is by induction on the length $m$ of $e=(d_{i})_{i=1}^{m}$. Write
    		$e_{i}=[d_{i}]_{E_{i}}$, and fix $H\in\ti(e)$ such that
    		$H\in\mathcal{H}(p,H')$ for some \[H'\in\ti(e_{2},\dots, e_{m})\]
    		and some minimizing sequence $p$ from $d_{1}$ to a basepoint $b_{0}\in H'$.
    		We assume that the associated ordinal minimized over possible choices of $d_{1}$ and $H'$.
    		
    		Consider a homomorphism (see Definition \ref{d:homomorphism}) $\phi:H\to H$ fixing $e$. It suffices to show that in this situation, $\phi$ is injective and
    		preserves the $G$-orbits in $H$. Indeed, when checking incompressibility on a particular homomorphism $\phi$ over $e$ one can always post-compose with a retraction of $N$ onto $H$.
    		
    		By part \ref{incompressibility} of the inductive hypothesis, $\phi_{\restriction H'}$ is an isomorphic embedding. Together with the minimality of $\Or(\pa{d_{1},H'})$ among points in $e_{1}$, among which is $\phi(d_{1})$ and all choices of $H'\in \ti(e_{2},\dots, e_{m})$, this implies:
    		\begin{observation}\label{use of minimality}
    			We have $\Or(\pa{\phi(d_{1}),c})\geq\Or(\pa{d_{1},b_{0}})$ for any point $c\in \phi(H')$ or $c\in H'$.
    		\end{observation}

    		\subsection*{Part I: showing that $\phi(b_{0})\in H'$.} We begin by applying Lemma \ref{lem-convex-extensions} to $\phi(d_{1})\in[d_{1}]_{E_{1}}$. This yields
    		points $c_{0},c_{1},b_{1}\in H$, as well as $h\in G$ such that $c_{1}=\phi(d_{1})h$,
    	  and 
    	  \begin{enumerate}
    	  	\item \label{c0 in between}the point $c_{0}$ is on a strict sequence from $d_{1}$ to $b_{0}$,
    	  	\item $\pa{d_{1},c_{1}}=\pa{d_{1},c_{0}}*\pa{c_{0},c_{1}}$,
    	  	\item \label{perp assertion} $\pa{c_{0},c_{1}}=\pa{b_{0},b_{1}}\perp^{*}\pa{c_{0},b_{0}}=\pa{c_{1},b_{1}}$.
    	  \end{enumerate}
    		Observation \ref{use of minimality} combined with \ref{c0 in between} and the equality \[\Or(\pa{\phi(d_{1}),b_{1}})=\Or(\pa{c_{0},b_{0}}),\] which follows from \ref{perp assertion}, implies the existence of a group element $g$ such that $c_{0}=d_{1}g$. Write $u_{E}\in\D^{*}$ for a representative of $\pa{c_{0},c_{1}}$.
    		Next, let \[c'_{1}=\phi(b_{0})(h')^{-1},\quad c'_{0},\quad b'_{1},\] be the points resulting of applying
    		Lemma \ref{lem-convex-extensions} to $\phi(b_{0})$. Observation \ref{use of minimality}, together with the fact that $\phi$ is a homomorphism, implies 
    		\begin{equation}\label{equal after phi}
    			\pa{\phi(d_{1}),\phi(b_{0})}=\pa{d_{1},b_{0}}.
    		\end{equation}
    	    		
    		We can choose $u_{-},u_{\vert}\in\D^{*}$ such that
\[
    \pa{c'_0,c'_1}=\pa{b_0,b'_1}=[u_-],
    \qquad
    \pa{c'_0,b_0}=\pa{c'_1,b'_1}=[u_{\vert}].
\]
Then $u_{\vert}\perp^*u_-$, by Lemma~\ref{perp and equivalence} and the
choice of $c'_0,c'_1,b'_1$.
    		Write also $\pa{d_{1},c'_{0}}=g'w$ with $w\in\D^{*}$.
    		The aforementioned choice directly implies 
    		\begin{equation*}
    			\pa{d_{1},\phi(b_{0})}=[g'w]*[u_{-}h'].
    		\end{equation*}
    		
    		 On the other hand, we also know that the class $\pa{d_{1},\phi(b_{0})}$ is represented by a reduct of  
    		 	\begin{align*}
    		 		\pa{d_{1},c_{1}}\pa{c_{0},c_{1}}\pa{c_{1},\phi(d_{1})}\pa{\phi(d_{1}),\phi(b_{1})}\supseteq& gu_{E_{1}}h^{-1}\pa{\phi(d_{1}),\phi(b_{1})}=\\
    		 		&gu_{E_{1}}h^{-1}\pa{d_{1},b_{1}}\ni gu_{E_{1}}h^{-1}wu_{|},
    		 	\end{align*}
    		where in the last line we use \eqref{equal after phi} and item \ref{c0 in between}.

    	  By minimality of $\Or(\pa{d_1,b_0})$ and the fact that
$[gu_{E_1}h^{-1}]=\pa{d_1,\phi(d_1)}$
    		has representatives in $\wo(E_{1})\cup\{E_{1}\}$, the second reduction can only involve the absorption of conjugates	of letters in $u_{E_{1}}$ by letters in $u_{|}$.
    		Indeed, let $v=gu_{E_{1}}h^{-1}$ be in left normal form. If the conclusion above does not hold, then the first forbidden interaction in the reduction gives a terminal domain $D\subseteq E_1$ coming from $v$ and a word $Dv^{*}\in[wu_{|}]$ which can absorb it on the right. Moving across this initial $D$--piece replaces $d_1$ by a point $d'\in e_1=[d_1]_{E_1}$ with a strictly shorter word to $b_0$, contradicting the minimal choice of the sequence from $d_1$ to $H'$.  
    			
    		Therefore, we may assume that $gu'_{E_{1}}hwu_{|}$ is reduced, where here
    		$u'_{E_{1}}$ is an initial subword of $u_{E_{1}}$.

    			\begin{observation}\label{uE1 orthogonal}
    				We have $u_{E_{1}}\perp^{*}u_{|}$.
    			\end{observation}
    		\begin{subproof}
    			The class $\pa{d_{1},b_{0}}$ is represented by the word $g'u_{|}$ (by considering $(d_{1},c'_{0},b_{0})$), but also by $gu'_{|}$, where $u'_{|}\in\D^{*}\cap\pa{c_{0},b_{0}}$. We know that
    	   	$\pa{c_{0},c_{1}}\perp^{*}\pa{c_{0},b_{0}}$
    		  so in fact $u_{E_{1}}\perp^{*}u'_{|}$, by Lemma \ref{perp and equivalence}. Since $u_{|}\in\D^{*}$, it follows from the same lemma that also $u_{E_{1}}\perp^{*}u_{|}$.
    		\end{subproof}

    		Observation \ref{uE1 orthogonal} allows us to apply Lemma \ref{absorption-orthogonality} to conclude that the word $u_{|}$ is trivial, and so $\phi(b_{0})\in H'$. Post-composing with a retraction onto $H'$, the second inductive assumption yields some
    		$k\in G$ such that $\phi(b_{0})k=b_{0}$.
    		
    		\subsection*{Part II: showing $\phi(H')=H'$.} It suffices to show that $\phi(H')\subseteq H'$, since $H'$ consists of a finite number of $G$-orbits and we already know by induction the map $\phi$ to be inner injective on $H'$.
    		
    		Take some arbitrary $b\in H'$ and recall that $\phi(b)\in H$ by assumption. This time, we let
    		\[c''_{0},c''_{1},b''_{1}\] be the points obtained by applying Lemma \ref{lem-convex-extensions} to $\phi(b)$. If $\Or(\pa{c''_{0},b_{0}})=0$, i.e. if $c''_{0}\in b_{0}G$, then also $\phi(b)\in H'$ and we are done, so assume this is not the case. 
    	 
        Since $\phi(H')$ is weakly convex and $c''_{1}\in\phi(b)G\subseteq\phi(H')$, there must exist some point $c^{*}_{0}\in\phi(H')$ such that
    		$\pa{c^{*}_{0},c''_{1}}=\pa{c''_{0},c''_{1}}$ and $\pa{c^{*}_{0},b_{0}}=\pa{c''_{0},b_{0}}$.
    		Let $F\in\D^{*}$ be the system consisting of all the curves (common boundary components) that appear in some of all possible intersections between two domains, of which one appears in some representative of $\pa{c''_{0},c''_{1}}$ in $\D^{*}$ and the other appears in some representatives of $\pa{c''_{0},b_{0}}$ in $\D^{*}$.

    		It is easy to see that $R_{F}(c^{*}_{0},c''_{0})$. Indeed, for any representatives $v_{-}$ and $v_{|}$ of $\pa{c''_{0},c''_{1}}$ and $\pa{c''_{0},b_{0}}$ we must have $R_{v_{-}v_{-}^{-1}}(c''_{0},c^{*}_{0})$ and $R_{v_{|}v_{|}^{-1}}(c''_{0},c^{*}_{0})$.
    		Clearly \[\Or(F)<\Or(\pa{c''_{0},b_{0}})\] unless
    		$\pa{c''_{0},b_{0}}=[F]$; the latter conclusion can be ruled out since $\pa{c''_{0},b_{0}}$ and
    		$\pa{c''_{0},c''_{1}}$ are assumed to be strongly orthogonal. It follows that \[\Or(\pa{d_{1},c^{*}_{0}})<\Or(\pa{d_{1},b_{0}}),\]
    		contradicting the minimality of $\Or(\pa{d_{1},b_{0}})$.
    		
    		\subsection*{Part III: conclusion.} We now revisit $\phi(d_{1})$ and the associated points
    		$c_{0},c_{1},b_{1}$, where we recall that $\phi(d_{1})=c_{1}h$.
    		Since $\pa{b_{0},c_{0}}\perp^{*}\pa{c_{0},c_{1}}$, it follows that the concatenation \[h\pa{c_{1},c_{0}}\pa{c_{0},b_{0}}k^{-1}\] is reduced and thus a representative of $\pa{\phi(d_{1}),\phi(b_{0})}$. Since $\phi$ is an $e$--homomorphism, and
    		$\Or(\pa{b_{0},c_{0}})=\Or(\pa{b_{0},d_{1}})$, this gives
    		\[c_{0}=c_{1},\quad \phi(d_{1})\in d_{1}G,
    		\quad \Or(\pa{\phi(d_{1}),\phi(b_{0})})=\Or(\pa{d_{1},b_{0}}).\]

    	  To conclude, let  \[a_{0}=d_{1},\dots, a_{k}=b_{0}\] be the minimizing sequence from $d_{1}$ to $H'$ used to construct
    		$H$, where $R^{*}_{D}(a_{j-1},a_{j})$. Let $H_{k}=H'$, and for each $1\leq j\leq k$, write $H_{j-1}$ for the one-step $D_{j-1}$--extension of
    		$H_{j}$ through $a_{j-1}$ over the basepoint $a_{j}$ in the construction of $H$. Thus, $H_{0}=H$.
    		
    		We claim that $\phi$ restricts to an isomorphism of $H_{j}$ that preserves each of the $G$-orbits in $H_{j}$. We proceed
    		by reverse induction on $j$; begin by assuming that the desired conclusion holds for $j$.
    		Since $\phi(d_{1})\in d_{1}G$ and \[\pa{\phi(d_{1}),\phi(a_{j-1})}=\pa{d_{1},a_{j-1}},\] it follows that
    		$a_{j-1}\nin H_{j}$. Since $\phi(H)\subseteq H$ and  \[\pa{\phi(a_{j-1}),\phi(a_{j})}=\pa{a_{j-1},a_{j}},\]
    		it necessarily follows that $\phi(a_{j-1})\in H_{j-1}$. Lemma \ref{lem-one step rigidity}, together with
		the equality \[\phi(a_{j})\cdot G=a_{j}\cdot G,\] implies
    		$\phi(a_{j-1})G=a_{j-1}G$.
    		For any $b\in\cn_{D_{j-1}}(a_j,H_j)$ we can use the fact that
\[
    \Or(\phi(a_{j-1}),\phi(b))\leq \Or(a_{j-1},b)
\] to
    		show that its lift $\tilde{b}$ to $H_{j-1}\setminus H_{j}$ satisfies $\phi(\tilde{b})G=\tilde{b}G$, and complete the induction.
    		\end{proof}
    	
    	The following are straightforward via projections and incompressibility over $e$.
    	
    	\begin{corollary}\label{cor:permutation}
    		Let $e\subset \hat N$ be a finite tuple, and let $e'$ be a permutation of $e$. Then the sets $\ti(e)$ and $\ti(e')$ are equal. Moreover,
    		given $H\in \ti(e)$ and $H'\in\ti(e')$, there is an automorphism of $N$ that induces an isomorphism $H\to H'$ fixing $e$.
    	\end{corollary}
    	
    	We can now deduce the following strong homogeneity result, which will be crucial for characterizing types of triples; see
    	Corollary~\ref{types of triples} below.
    	
    	\begin{corollary}
    		\label{isomorphisms between tight sets}
    		Let $p$ and $p'$ be strict $w$-sequences between points $a,b$ and $a',b'$
    		respectively, and let \[H\in \mathcal{H}_{a,b},\quad H'\in \mathcal{H}_{a',b'}.\]
    		Then there is an automorphism of $\mathcal{N}$ sending $p$ to $p'$ and $H$ to $H'$.
    	\end{corollary}

    	\begin{proof}
    		We know there is an isomorphic embedding
    		$\phi:H\to N$ sending $p$ to $p'$, and let $\pi\colon N\to H$ be a homomorphic retraction.
		By Lemma \ref{lem: incompressibility}, we have that $\phi(H)$ is
    		strongly incompressible over $a',b'$.
    		It follows that $\pi$ is injective on $\phi(H)$, and thus $\pi\circ\phi$ an isomorphic embedding of $H$ into $H'$ sending
    		$a,b$ to $a',b'$. By construction, both $H$ and $H'$ consist of the same finite number of orbits.
    		It thus follows that $\pi\circ\phi$ is in fact an isomorphism between $H$ and $H'$.
    		The result then follows from relative quantifier elimination (Corollary \ref{cor:ext-qe}).
    	\end{proof}

  	\subsection{Wobbling, absorption, and definability}
    	We now adapt the concept of the \emph{wobbling} of a sequence from \cite{BPZ17}. Roughly, this is the ambiguity in expressing
	a reduced concatenation of words, consisting of domains that can be both left absorbed and right absorbed by consecutive
	factors.
	To set up the discussion, we have the following
    	calculus for cancellation of reduced words.
    	
	    \begin{observation}\label{obs:lad}
	    	Let $w\in\wo$ be reduced. There is a unique maximal possibly disconnected domain
	    	$D\in\Dvee$ with the following property: every word whose connected domain letters
	    	are supported in $D$ is left-absorbed by $w$, and every domain letter which is
	    	left-absorbed by $w$ is contained in $D$. Equivalently, for $u\in\D^{*}$ we have
	    	$[u]*[w]=[w]$ if and only if all connected domain letters of $u$ are contained in
	    	$D$. This domain depends only on the quasi-permutation class $[w]$.
	    	If $w$ is written in left normal form as $w\simeq vw'g$, where $v$ is the product of
	    	the components of the left end $E$ of $w$ and $w'\in\D^{*}$, then $D$ is described
	    	inductively by
	    	\begin{equation*}
	    		D=E\vee (E^{\perp}\cap LA(w')).
	    	\end{equation*}
	    	Here $E^{\perp}\cap LA(w')$ is understood componentwise. The same definition applies
	    	when group letters occur in $w$, after first putting $w$ in left normal form by the
	    	moves $(\mathrm{Jmp})$ and $(\mathrm{Cmp})$.
	    \end{observation}
	    \begin{proof}
	    	Clearly, any  $u\in\D^{*}$ is left-absorbed by $w$ if and only if each of its connected domain letters is, so the problem reduces to showing that $F\in\D$ is left-absorbed by $w$ if and only if $F\subseteq LA(w)$.
	    	The proof of the latter is by induction on the number of domain letters in a left-normal representative of $w$. The if part is trivial, so we focus on the only if part. Assume that $F$ is not contained in any of the components of $LA(w)$ as described in the statement. Note that here $LA(w)$ is given in terms of the chosen left-normal form, but the claim we are proving shows that it is independent of that choice.
	    	If $F$ is not contained in the left end $E$, and either $F\nsubseteq E^{\perp}$ or $F\subseteq E^{\perp}$ but $F\nsubseteq LA(w')$, then no sequence of transpositions and absorptions can move $F$ into a place where it is absorbed by $w$. Indeed, in the first case $F$ meets the left end essentially; in the second case, the induction hypothesis says that $F$ is not absorbed by the remaining word $w'$. Thus $F$ is left-absorbed exactly in the cases recorded in the displayed formula.
	    	Group letters do not change the domain which is left-absorbed: once $w$ is in left normal form, the only group letter is terminal, and the move $(\mathrm{Jmp})$ merely conjugates domain letters while preserving containment and orthogonality. This proves the observation.
	    \end{proof}
	    
	    \newcommand{\labs}[0]{LA}
	    \newcommand{\rabs}[0]{RA}
	    
	    \begin{definition}[Left absorption and right absorption]
	    	For a reduced word $w$, we denote the domain furnished by Observation~\ref{obs:lad}
	    	by $\labs(w)$, for \emph{left absorption}. Symmetrically, we define
	    	\[\rabs(w)=\labs(w^{-1}),\]
	    	and call this the \emph{right absorption} domain of $w$. These are elements of
	    	$\Dvee$ and are used only componentwise; they are not new domain letters in words.
	    \end{definition}

    	We have the following characterization of triples $(g_1,w,g_2)$ such that $g_1wg_2\simeq w$.
    	
    	\begin{lemma}
    		\label{l: element absorption}Let $w\in\wo\cap\D^{*}$ be reduced and $g_{1},g_{2}\in G$. Then $g_{1}wg_{2}\simeq w$ if and only if there is $g_{\perp}\perp w$, i.e.~supported on a domain orthogonal to all the domains of $w$, such that
    		\[g_{1}g_{\perp}\in G[LA(w)],\quad  \textrm{and}\quad g_{\perp}^{-1}g_{2}\in G[RA(w)].\] In particular,
    		$g_{\perp}$  has support orthogonal to both $g_1g_{\perp}$ and $g_{\perp}^{-1}g_{2}$, and thus commutes with $g_{1}$ and $g_{2}$.
    	\end{lemma}
    	\begin{proof}
    		The proof is by induction on the number of domains in $w$. It follows from $g_{1}wg_{2}\simeq w$
    		that there are points $x,y\in N$ and strict sequences $p,p'$ of type $w$ and $g_{1}wg_{2}$ from $x$ to $y$. The latter implies also the existence of a sequence
    		of type $g_{1}g_{2}g_{2}^{-1}(w)$ between those points. By concatenation we obtain a sequence of type
    		\[wg_{2}^{-1}(w^{-1})g_{1}^{-1}g_{2}^{-1}\] from $x$ to itself.
    		
    		Let $E_{1}$ be the right end of $w\simeq w_{0}E_{1}$, which for compactness of notation we view as a single letter.
    		Clearly $E_{2}:=g_{2}^{-1}(E_{1})$ is the right end of $g_{2}^{-1}(w)$. Fix sequences $p_{1}$ and $p_{2}$ of type $w$ and $g_{1}g_{2}g_{2}^{-1}(w)$ from $x$ to $y$ and let $y_{i}$ be the point in $p_{i}$ for which $R^{*}_{E_{i}}(y_{i},y)$. Lemma~\ref{uniqueness} and the uniqueness of right ends (Lemma~\ref{ends}) imply that $E_{1}=E_{2}=E$ and $R_{u}(y_{1},y_{2})$, where $u\in\wo(E)$ does not contain any full connected component of $E$.
    		We therefore have that $R_{E\vee E^{\perp}}(1,g_2)$, i.e.~$g_2\in G[E\vee E^{\perp}]$, since $G$ consists of pure mapping classes.

    		On the other hand we can write $u\simeq u'h$ where $u'\in\D^{*}$ and $h\in G[E]$. Let $z=y_{2}\cdot h^{-1}$. Clearly, $g_{1}w_{0}g_{2}\in\pa{x,y_{2}}$, and thus $g_{1}w_{0} g_{2}h^{-1}\in\pa{x,z}$. We claim that also $w_{0}\in\pa{x,z}$. On the one hand, we have
    		$\Or(\pa{x,z})=\Or(w_{0})$. On the other hand, no cancellation can occur in the reduction of $w_{0}u'$, since that would imply that $R_{w_{0}'E}(x,z)$ for some $w_{0}'$ with $\Or(w_{0}')<\Or(w_{0})$, contradicting the assumption that $R^{*}_{w}(x,y)$.
    		It follows that all the reduction consists simply of the absorption of $u'$ into $w_{0}$, whence we obtain $w_0\in\pa{x,z}$.
    		
    		By induction, there exists an element $g_{0}\perp w_{0}$ such that
            \[{g_{2}h^{-1}g_{0}^{-1}=g_{0}^{-1}g_{2}h^{-1}}\in RA(w_{0}),\quad g_{1}g_{0}\in LA(w_{0}),\]
    		where the equality is furnished by the final part of the induction statement, 
    		which states that the supports of $g_{0}$ and $g_{2}h^{-1}g_{0}^{-1}$
    		are orthogonal.
    		Since \[(g_{2}h^{-1}g_{0}^{-1})g_{0}=g_{2}h^{-1}\in \Stab(\partial E)=G[E\vee E^{\perp}],\]
    		where the equality follows from $G$ consisting of pure mapping classes, it follows that
    		$g_{0}\in \Stab(\partial E)$ as well. Hence, using $\supp$ to denote the support of a group element,
    		we can write $g_{0}=\bar{g}_{0}g_{\perp}$ where
    		$\bar{g}_{0}$ is supported on $E\cap\supp g_0$ and $g_{\perp}$ is supported on $E^{\perp}\cap\supp g_0$.
    		
    		Using that $\supp(h)\subseteq E\perp\supp(g_{\perp})$ and thus $h$ and $g_{\perp}$ commute, we have:
    		\begin{align*}
    			\supp(g_{1}g_{\perp})=\supp((g_{1}g_{0})\bar{g}_{0}^{-1})\subseteq \supp(g_{1}g_{0})\vee \supp(\bar{g}_{0})\\\subseteq LA(w_{0})\vee (w_{0}^{\perp}\cap E)=LA(w); \\
    			\supp(g_{2}g_{\perp}^{-1})=\supp(g_{2}h^{-1}g_{0}^{-1}(\bar{g}_{0}h))\subseteq\\
    			\supp(g_{2}h^{-1}g_{0}^{-1})\vee \supp(\bar{g}_{0}h)\subseteq (RA(w_{0})\cap(E\vee E^{\perp}))\vee E=RA(w);
    		\end{align*}
    		Conversely, suppose such a $g_{\perp}$ exists. Since $g_1g_{\perp}\in G[LA(w)]$, this group element is absorbed on the left by $w$; similarly, $g_{\perp}^{-1}g_2\in G[RA(w)]$ is absorbed on the right by $w$. Finally $g_{\perp}\perp w$, so the letter $g_{\perp}$ can be moved through every domain letter appearing in $w$. Thus
    		\[g_1wg_2\simeq (g_1g_{\perp})w(g_{\perp}^{-1}g_2)\simeq w,
    		\]
    		proving the reverse implication. This establishes the lemma.
    	\end{proof}
    	
    	\begin{corollary}\label{cor:absorption-general}
    		If $u,w$ are words in $\D^*$ and $g\in G$, and if $guw\simeq w$, then $gw\simeq w$ and $g$ satisfies the conclusions of
    		Lemma~\ref{l: element absorption}.
    	\end{corollary}
    	
    	\begin{definition}[Wobble]
    		For reduced words $w$ and $w'$,
    		we write \[w \wr w':=\labs(w^{-1})\cap\labs(w')\] for the \emph{wobble} of the product $ww'$.
    	\end{definition}

    	\begin{lemma}
    		\label{l: wiggling}
    		Let \[w=D_{1}D_{2}\cdots D_{k}\in\wo\cap\D^{*}\] be a reduced word, and let
    		\[a_{0},a_{1},\ldots, a_{k}\quad \textrm{and}\quad a'_{0},a'_{1},\ldots, a'_{k}\]
    		be two strict $w$-sequences between two points $a_{0}=a'_{0}$ and $a_{k}=a'_{k}$.
    		For each $1\leq i< k$, let \[u_{i}=D_{1}D_{2}\cdots D_{i},\quad
    		v_{i}=D_{i+1}D_{i+2}\cdots D_{k}.\] Then for all $1\leq i<k$, we have $R_{E_{i}}(a_{i},a'_{i})$ with $E_{i}=u_{i}\wr v_{i}$.
    	\end{lemma}
    	\begin{proof}
    		Following the proof of Lemma 6.19 in \cite{BPZ17}, we proceed by induction on $i<k$.
    		Suppose first that $i=1$. Then we have that $a_1$ and $a_1'$ are related to $a_0$ by
    		$D_1$, so that $R_{D_1D_1}(a_1,a_1')$. It follows that any
    		reduced word $w_1$ such that
    		$R^*_{w_1}(a_1,a_1')$ is absorbed by $D_1$.
    		Similarly, we have that
    		\[R_{D_2D_3\cdots D_k D_k\cdots D_2}(a_1,a_1'),\]
    		since $a_1$ and $a_1'$ are related to $a_k$ by $D_2D_3\cdots D_k$. Since the word
    		$D_2\cdots D_k$ is reduced and \[w_1D_2\cdots D_k\simeq D_2\cdots D_k,\]
    		we must have that $w_1$ is fully absorbed by $D_2\cdots D_k$. This claim is established for domains occurring in $w_1$ by
    		Observation~\ref{obs:lad} and for group elements by Lemma~\ref{l: element absorption}.
    		It follows that $a_1$ and $a_1'$
    		are related by an $E_1$ which is absorbed by both $D_1$ and $D_2D_3\cdots D_k$.
    		
    		Now suppose that $w_i$ and $w_{i+1}$ are
    		reduced words such that \[R_{w_{i}}(a_{i},a_{i}')\quad
    		\textrm{and}\quad R_{w_{i+1}}(a_{i+1},a_{i+1}').\]
    		By induction, $w_i$ is fully absorbed by
    		$D_1\cdots D_i$ and by $D_{i+1}\cdots D_k$. We may
    		(up to equivalence) write $w_i$ as a product
    		$w_i^1w_i^2$, where $w_i^1$ is left absorbed by $D_{i+1}$ and where $w_i^2$ is orthogonal
    		to $D_{i+1}$ and is left absorbed by $D_{i+2}\cdots D_k$. We therefore have that
    		$D_{i+1}w_iD_{i+1}$ reduces to $w_i^2D_{i+1}D_{i+1}$, which then further reduces to
    		$w_{i+1}$. It follows that $w_{i+1}$ is a reduct of $w_i^2y_i$, where $y_i$ is absorbed
    		by $D_{i+1}$.
    		
    		Since $R_{D_{i+2}\cdots D_k}(a_{i+1},a_k)$ and since $D_{i+2}\cdots D_k$ is reduced,
    		we have that $y_{i}$ must be left absorbed by $D_{i+2}\cdots D_k$. It follows then
    		that $w_{i+1}$ is left absorbed by $D_{i+2}\cdots D_k$. Since $w_i^2$ is orthogonal to
    		$D_{i+1}$ and is right absorbed  by $D_1\cdots D_i$ by induction, we have that
    		$w_i^2$ is left absorbed by $D_1\cdots D_{i+1}$. Since $y_i$ is absorbed by $D_{i+1}$,
    		we see that $w_{i+1}$ is right absorbed by $D_1\cdots D_{i+1}$. The lemma follows.
    	\end{proof}
    	
    	\newcommand\Osc{\mathrm{Osc}}
    	
    	\begin{definition}[Oscillation]
    		Let \[w=D_{1}D_{2}\cdots D_{k}\in\wo\cap\D^{*},\] and let $a_0,\ldots,a_k$ be a strict $w$--sequence.
    		As in Lemma~\ref{l: wiggling}, we write \[u_{i}=D_{1}D_{2}\cdots D_{i},\quad
    		v_{i}=D_{i+1}D_{i+2}\cdots D_{k}\] for each $1\leq i<k$, and $E_i=u_i\wr v_i$.
    		We define the oscillation \[\Osc_{w}(a_{0},a_{k})=\left(\bigcup_{i=1}^{k-1}\subg{[a_{i}]_{E_{i}}}_{G}\right).\]
    		
    		We write \[\Osc(a_0,a_k)=\bigcup_{w\in\pa{a_{0},a_{k}}\cap\D^{*}} \Osc_{w}(a_{0},a_{k}).\]
    		For general $a,b\in N$ we define $\Osc(a,b)$ by $\Osc(a',b)$ where $a'\in a G$
    		is chosen so that
    		$\pa{a',b}\cap\D^{*}\neq\emptyset$.
    	\end{definition}
    	
    	Lemma~\ref{l: wiggling} shows that $\Osc_w(a_0,a_k)$ is independent of the chosen strict $w$--sequence. Observe also that the definition of $\Osc(a,b)$ is independent of the choice of $a'$, by Lemma~\ref{l: element absorption}.
	The reader may recall the definition of $\subg{[a]_{D}}_{G}$ at the beginning of Subsection~\ref{ss:imagine-wc}.

    \begin{example}[Wobble in the running example \ref{ex:running-separating-curve}]
    \label{ex:running-wobble}
    Continue with Example~\ref{ex:running-separating-curve}. Work in a saturated
    model and consider a strict $DE$--sequence
    \[
        a_0,\ a_1,\ a_2.
    \]
    Thus $a_1$ is obtained from $a_0$ by a generic $D$--move, and $a_2$ is obtained
    from $a_1$ by a generic $E$--move. The intermediate point $a_1$ is not expected
    to be determined as a real point by the endpoints $a_0$ and $a_2$.

    The visible ambiguity comes from the boundary annulus $A=\{\alpha\}$. At the
    level of ordinary $DE$--sequences, one can insert a power of the boundary
    twist mid-point:
    \[
        a_0,\ a_1T_\alpha^n,\ a_2.
    \]
    The first step is still $D$--related and the second step is still $E$--related,
    because $T_\alpha$ lies in both $G[D]$ and $G[E]$. For strict sequences, the
    same phenomenon survives generically after excluding the lower-complexity
    exceptional relations: what is determined by the endpoints is not the real
    middle point, but the appropriate annular imaginary class associated to the
    middle point.

    This is the model-theoretic content of the wobble construction. The domains
    $u_i\wr v_i$ in Lemma~\ref{l: wiggling} record the possible sliding of
    intermediate points. The oscillation set $\Osc(a,b)$ is the union of the
    corresponding imaginary classes. Corollary~\ref{acl in dcl} says that these
    are precisely the imaginary classes, among the classes considered in $\hat N$,
    which are algebraic over the two real endpoints $a$ and $b$.
    \end{example}
    	
    	\begin{lemma}\label{lem:observation is definable}
    		We have $\Osc(a,b)\subseteq \dcl^{\EQ}(a,b)$
    	\end{lemma}

    	\begin{proof}
    		We may assume $\pa{a,b}\cap\D^{*}\neq\emptyset$.  If $e\in \Osc_{w}(a,b)$ then $e$ lies
    		in the definable closure of some $[a_{i}]_{E_{i}}$, where $a_{i}$ denotes the $i^{th}$
    		intermediate point of some strict sequence of type $w$ from $a$ to $b$. In turn,
    		$[a_{i}]_{E_{i}}$ can be uniquely characterized as the $E_{i}$ class of the $i^{th}$ intermediate point
    		of some strict sequence of type $w$ from $a$ to $b$. Since any sequence of type $w$ from
    		$a$ to $b$ is automatically strict, the latter
    		can be expressed by a first order formula.
    	\end{proof}
    	
    	Recall that $\acl^{\EQ}(A)$ stands for the collection of all imaginary classes whose orbit
    	under the point-wise stabilizer of $A$ is finite.
    	We obtain the following corollary:

    	\begin{corollary}
    		\label{acl in dcl}Let $a,b\in N$. Then \[\hat{N}\cap \acl^{\EQ}(a,b)=\Osc(a,b).\]
    		In particular, \[\hat{N}\cap \acl^{\EQ}(a,b)\subseteq \dcl^{\EQ}(a,b).\]
    	\end{corollary}
    	\begin{proof}
    		It suffices to prove that \[\widehat{\acl}(a,b)\subseteq \Osc(a,b).\]
    		Assume that \[\pa{a,b}\cap\D^{*}\neq\emptyset,\] and
    		let $e\in\hat{N}\setminus \Osc(a,b)$.
    		Let $p$ be a strict sequence from $a$ to $b$, and pick an arbitrary $H\in\mathcal{H}_{p}$.
    		Assume first that $e=[d]_{E}\notin\hat{H}$. Choose $d\in e$ and $b_{0}\in H$ minimizing $\Or(\pa{d,b_{0}})$. It suffices to show that there are infinitely many non $E$-related realizations of the type $p^{w}_{b_{0},H}$, which is satisfied by $d$. Assume that $F\in\D$ is the initial letter of $w$. We know that $F\nsubseteq E$. Let $\tau\in G[F]$ be non-trivial and pseudo-Anosov on $F$ if $F$ is not annular. It is easy to see that the elements $d_{n}=d\tau^{n}$, $n\in\N$ satisfy the type $p^{w}_{b_{0},H}$. Clearly $\models\neg R_{E}(d_{n},d_{m})$, hence the result.
    		
    		Assume now that $e=[c]_{E}$, with $c\in H$. By definition, there is some $c'$ of the form
    		$c h$
    		that occurs in some strict sequence of type $w_{1}*w_{2}\in\D^{*}$ from $a$ to $b$,
    		where \[\pa{a,c'}=[w_{1}],\quad \pa{c',b}=[w_{2}].\] Write $E'$ for the result of conjugating $E$ by $h$, and
    		let $D=w_{1}\wr w_{2}$. We have that $[c']_{E'}$ is interdefinable with $[c]_E$.
		Since $D$ is not contained in $E'$ by assumption, it follows from
    		Lemma \ref{l: wc and acl} that the orbit of $e$ under the action of
    		$G[D]$ on the right on $(c G)/R_{E}$ is infinite.
    		By Corollary \ref{isomorphisms between tight sets}, this implies that the orbit of
    		$e$ under $\Aut_{a,b}(\mathcal{N})$ is infinite and hence $e\notin \acl^{\EQ}(a,b)$.
    	\end{proof}

    	\begin{lemma}
    		\label{global invariant types} Let
    		$w,w'\in\wo$ be reduced words and let $n_{0},n'_{0}\in N$. The equality
    		$p^{w}_{n_{0},N}=p^{w'}_{n'_{0},N}$ holds if and only if there is an element $g\in G$
    		such that $w'\simeq wg$ and $R_{D,g}(n_{0},n'_{0})$, where here $D=RA(w)$ and $R_{D,g}$ denotes the relation represented by the two-letter word $(D,g)$.
    		In particular, we obtain \[\Or(w)=\Or(w*\pa{n_{0},n'_{0}}).\]
    	\end{lemma}
    	\begin{proof}
    		We will limit ourselves to proving the direction from left to right,
    		since the converse implication follows easily from the associativity of
    		reduction without cancellation (see Corollary~\ref{cor:associative}). 
    		
    		The type $p^{w}_{n_{0},N}=\tp(b/N)$ implies that \[\pa{b,n'_{0}}\simeq [w]*\pa{n_{0},n'_{0}},\]
    		and the same holds after exchanging the roles of $(w,n_{0})$ and $(w',n_{0})$, and replacing $b$ by a suitable $b'$. It follows that up to equivalence of words, each of $w,w'$ is an initial subword of the other. It follows immediately that
    		$\Or(w)=\Or(w')$.
    		As \[p^{wg}_{n_{0}\cdot g, N}=p^{w'}_{n'_{0},N},\] we may assume that
    			in fact $w$ is equal to $w'$, and then
    		the result follows immediately from the definition of $LA(w^{-1})$ and Lemma~\ref{l: element absorption}.
    		\qedhere
    	\end{proof}
    	
    	We will need a variation of Lemma~\ref{global invariant types} for classes. Let $a\in N$ and let $D\in\D$. Given a weakly convex set $H\subset N$, let $n_0\in H$ and $w\in\wo$ be such that $w\in\pa{c,n_0}$ for some $c\in [a]_D$, chosen in such a way so that $\Or(w)$ is minimized over all such possible choices.  We will refer to $n_0$ as a \emph{basepoint} for $[a]_D$ in $H$.
    	
    	For $w\in\wo$ we let $p^{D,w}_{n_0,H}$ be the type of the class $[a]_D$ over $H$ for some $a\in N$ satisfying 
    	$p^{w}_{n_{0},H}$, and we define its global extension $p^{D,w}_{n_{0},N}$ analogously, where $p^{w}_{n_{0},H}$ and $p^{w}_{n_{0},N}$ are as in Definition \ref{def:main-type}. 
    	
    	\begin{observation}
    		\label{obs: reduced} Let $n_{0}\in H$ be a basepoint for $e=[a]_{D}$ in the weakly convex set $H$. If 
    		$a$ minimizes $Or(\pa{a,n_{0}})$ over all points in $e$, then either $D\pa{a,n_{0}}$ is reduced, or 
    		$D$ is properly left-absorbed by $\pa{a,n_{0}}$. Indeed, otherwise the resulting reduction of $D\pa{a,n_{0}}$ would
		contradict the minimality of $Or(\pa{a,n_{0}})$.
    	\end{observation}
    	
    	\begin{lemma}\label{lem:global inv type class}
    		Let $D\in\D$, and let $w$ be a word such that either $Dw$ is reduced or such that $D$ is properly
    		left absorbed by $w$. Let $n_{0},n'_{0}$ be two points in a weakly convex set $H$. 
     		Then we have that $p^{D,w}_{n_0,H}=p^{D,w}_{n'_0,H}$ if and only if $R_u(n_0,n'_0)$,
    		with $u\subseteq RA(D\ast w)$.
    	\end{lemma}
    	\begin{proof}
    		For the only if direction, if the two types coincide, then they can be realized by the same $D$-class $e$. 
    		So, we pick such a class $e$ and $a,a'\in e$ with \[[w]=\pa{a,n_0}=\pa{a',n_0'}.\]
    		Let $u_D\subseteq D$ be such that $R_{u_D}(a,a')$, and
		write $u=\pa{n_0,n_0'}$. Observe that $u_D w u^{-1}$ reduces to $w$, and $u_D^{-1} w u$ also reduces
    		to $w$. Moreover, we have $w* u=u_D w$ and $w* u^{-1}=u_D^{-1} w$. If on the one hand
		$D* w$ is reduced, then we have that
    		\[Dw*u=Du_Dw\simeq Dw.\] On the other hand, if $D$ is properly absorbed by $w$ then since $u_D$ is subordinate to $D$,
    		we have $u_Dw\simeq w$. In either case, $u$ is right absorbed by $D*w$. 
    		
    		The if direction is clear.
    	\end{proof}
    	
	\begin{remark}
    	\label{rem: canonical bases}It follows from Lemma~\ref{lem:global inv type class} that the type $p^{D,w}_{n_{0},N}$ is invariant under an automorphism $\sigma$
    	of $\mathcal N$ if and only if $\sigma$ preserves the class $[n_0]_E$, where $E=RA(Dw)$, i.e., that	$[n_0]_{E}$ is a \emph{canonical base} for the global type $p^{D,w}_{n_0,N}$. 
    	
    	In view of Observation \ref{o: all global types are parametrized} and that each such type extends uniquely to a global type over $\hat{N}$ we conclude that in $\hat{N}$ every global $1$-type has a real canonical base. 
	\end{remark}

It follows from Lemma~\ref{acl in dcl} that given $a,b\in N$, any
    	\emph{$\widehat{\acl}(a,b)$--invariant} $1$-type, i.e.~ one invariant under all automorphisms fixing $\widehat{\acl}(a,b)$,
    	is actually $ab$ invariant. An analogous observation holds for types over single classes in $\hat{N}$, following Lemma~\ref{lem:global inv type class}. This, together with the previous remark yields:
    	
    	\begin{corollary}
    		\label{stationarity} Given $a,b\in N$ all types in $S^{1}(ab)$ are 
		stationary. Moreover, if $e\in \hat N$ then all types in $S^1(e)$ are stationary.
    	\end{corollary}

	  \subsection{Type of triples and quantifier elimination}
	    
	    The following lemma is the analogue, in the present setting, of Proposition 7.21 in
	    \cite{BPZ17}.  The statement is deliberately phrased in terms of projections to a
	    weakly convex side of a triangle, rather than as a purely formal decomposition of
	    three words.  In right-angled buildings these two formulations coincide.  In the
	    augmented Cayley graph of the mapping class group, however, group letters can be
	    absorbed into domain letters, and the closest point of the third vertex to the side
	    may lie in a parallel $G$--translate of the side.  The extra word $\epsilon$ below records
	    precisely this parallel offset.
	    Recall that the notation $w^{-1}$ for a word denotes the word obtained by writing the
	    letters occurring in it in reverse order, and inverting the group elements which appear.
	    
	    	\begin{figure}[h!]
	    		\psfrag{w}{$w$}
	    		\psfrag{c}{$C$}
	    		\psfrag{U}{$u_1$}
	    		\psfrag{X}{$xv_1$}
	    		\psfrag{v}{$v$}
	    		\psfrag{B}{$b_0$}
	    		\psfrag{s}{$s_{\alpha\beta}$}
	    		\psfrag{a}{$a$}
	    		\psfrag{u}{$u$}
	    		\psfrag{b}{$b$}
	    		\includegraphics[width=6.5cm]{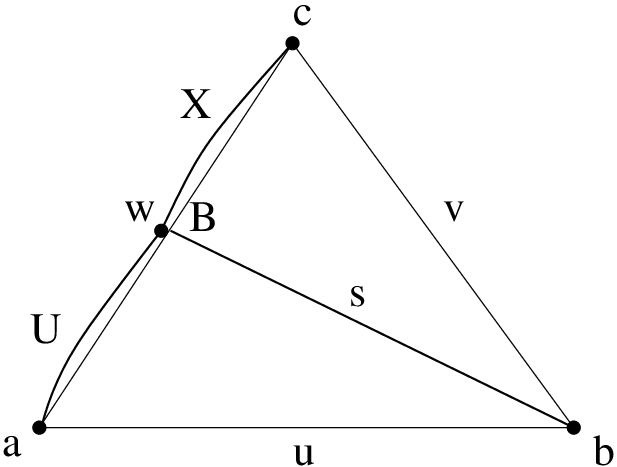}
	    		\caption{Lemma~\ref{triangle lemma}}
            \end{figure}
            
	    \begin{lemma}[Triangle projection lemma]
	    	\label{triangle lemma}
	    	Let $a,b,c\in N$, and write
	    	\[
	    		U=\pa{a,b},\qquad V=\pa{b,c},\qquad W=\pa{a,c}.
	    	\]
	    	After replacing $c$ by a point $c_{*}\in cG$, say $c_{*}=c\gamma$ for some
	    	$\gamma\in G$, there exist possibly trivial reduced words
	    	\[r,t,\epsilon\in\D^{*},\qquad y\in\wo,\]
	    	a strict sequence
	    	\[
	    		p:\quad a \xrightarrow{\,r\,} k
	    		\xrightarrow{\,\epsilon\,} b_{0}
	    		\xrightarrow{\,t\,} c_{*},
	    	\]
	    	and a weakly convex set $H\in\mathcal H_{p}$ such that the following hold.
	    	\begin{enumerate}
	    		\item\label{item:projection-basepoint} The point $b_{0}\in H$ is a basepoint for
	    		$b$ in $H$, and
	    		\[
	    			\pa{b,b_{0}}=[y].
	    		\]
	    		\item\label{item:projection-orthogonal} The parallel offset is strongly orthogonal
	    		to the terminal side:
	    		\[
	    			\epsilon\perp^{*}t.
	    		\]
	    		In particular, the displayed sequence from $a$ to $c_{*}$ is strict and reduced.
	    		\item\label{item:projection-equations} The three sides of the original triangle are
	    		obtained by reduction without cancellation from
	    		\[
	    			U=[r\epsilon y^{-1}],\qquad
	    			V=[yt\gamma^{-1}],\qquad
	    			W=[r\epsilon t\gamma^{-1}].
	    		\]
	    		\item\label{item:projection-minimal} The word $y$ is chosen with minimal
	    		associated ordinal among all choices satisfying the previous conditions. Equivalently,
	    		no nontrivial terminal subword of $y$ is supported in
	    		\[
	    			(r\epsilon)\wr t.
	    		\]
	    	\end{enumerate}
	    	Moreover, the resulting projection data
	    	\[
	    		([r\epsilon],[t],[y],\gamma)
	    	\]
	    	are determined by the triple of pairwise classes $(U,V,W)$, up to quasi-permutation
	    	of representatives and the replacement $b_{0}\mapsto b_{0}g$ by an element
	    	$g$ supported in $(r\epsilon)\wr t$. In particular, if two triples have the same three
	    	pairwise $\delta$--classes, then the projection data may be chosen with the same word
	    	classes and the same group element $\gamma$.
	    \end{lemma}
	    
	    \begin{proof}
	    	Choose a representative $c^{\circ}\in cG$ for which $\pa{a,c^{\circ}}$ has a
	    	representative in $\D^{*}$; this is obtained by putting a representative of
	    	$\pa{a,c}$ in right normal form and absorbing the terminal group letter into the
	    	choice of $c^{\circ}$.  Let
	    	\[
	    		p^{\circ}: a\longrightarrow c^{\circ}
	    	\]
	    	be a strict sequence of this domain-only type, and choose
	    	$H^{\circ}\in\mathcal H_{p^{\circ}}$.
	    	Among all such choices, and among all basepoints for $b$ in the chosen weakly
	    	convex hull, choose a basepoint $d\in H^{\circ}$ so that
	    	$\Or(\pa{b,d})$ is minimal. Write $\pa{b,d}=[y_{0}]$.
	    	
	    	This is the point at which Lemma~\ref{lem-convex-extensions} is  needed: it replaces the arbitrary basepoint $d$ by a point in a parallel copy of the side whose displacement from the side is strongly orthogonal to the remaining tail.  We apply that lemma to the side $p^{\circ}$, the terminal
	    	weakly convex set $c^{\circ}G$, and the point $d\in H^{\circ}$. The rectangle produced
	    	by that lemma gives points
	    	\[
	    		k\in H^{\circ},\qquad b_{0}\in dG,\qquad c_{*}\in c^{\circ}G,
	    	\]
	    	such that $k$ lies on a strict sequence from $a$ to $c^{\circ}$, and such that, after
	    	setting
	    	\[
	    		r=\pa{a,k},\qquad \epsilon=\pa{k,b_{0}},\qquad t=\pa{b_{0},c_{*}},
	    	\]
	    	we have $r,t,\epsilon\in\D^{*}$ and $\epsilon\perp^{*}t$.  The same construction gives
	    	$\pa{k,c^{\circ}}=[t]$ and $\pa{c^{\circ},c_{*}}=[\epsilon]$:  the displayed
	    	sequence from $a$ to $c_{*}$ is a strict sequence of type $r\epsilon t$.
	    	
	    	Since $b_{0}\in dG$ and $H^{\circ}$ is a union of $G$--orbits, $b_{0}$ lies in
	    	$H^{\circ}$.  The word from $b$ to $b_{0}$ is obtained from $y_{0}$ by multiplying by
	    	a group letter and reducing.  Its associated ordinal is therefore at most
	    	$\Or(y_{0})$, and by the minimality of $d$ it is exactly $\Or(y_{0})$.  Thus
	    	$b_{0}$ is again a basepoint for $b$ in the same weakly convex hull.  Replacing
	    	$H^{\circ}$ by the corresponding hull $H\in\mathcal H_{p}$ for the side
	    	$p:a\to k\to b_{0}\to c_{*}$, we have Item~\ref{item:projection-basepoint}.
	    	
	    	Now write $c_{*}=c\gamma$.  Concatenating the displayed paths gives the 
	    	word equations
	    	\[
	    		\pa{a,b}=[r\epsilon y^{-1}],\qquad
	    		\pa{b,c}=[yt\gamma^{-1}],\qquad
	    		\pa{a,c}=[r\epsilon t\gamma^{-1}],
	    	\]
	    	where $y=\pa{b,b_{0}}$.  This proves Items~\ref{item:projection-orthogonal} and
	    	\ref{item:projection-equations}.
	    	
	    	It remains to justify the minimality assertion. Suppose that a nontrivial terminal
	    	subword $z$ of $y$ is supported in $(r\epsilon)\wr t$.  Then $z$ is right-absorbed by
	    	$r\epsilon$ and left-absorbed by $t$.  If $y=y'z$, move the basepoint $b_{0}$ backwards
	    	along a strict $z$--sequence from $b$ to $b_{0}$.  The new point still lies in the
	    	same weakly convex hull, because the two adjacent side words absorb $z$ on the
	    	appropriate sides, and it is again a basepoint for $b$.  But the word from $b$ to
	    	the new basepoint has strictly smaller associated ordinal, contradicting the
	    	choice of $b_{0}$.  Hence no such terminal subword exists.
	    	
	    	Finally, we discuss uniqueness of the projection data.  The three displayed equations
	    	are equations in the reduced-word calculus.  Once $(U,V,W)$ is fixed, the common
	    	side $r\epsilon t\gamma^{-1}$ is fixed up to quasi-permutation.  The only ambiguity in
	    	recovering the middle projection is the possibility of moving a group element through
	    	the middle of the side without changing the two reduced products.  Lemma~\ref{l: element absorption}
	    	shows that any such group element decomposes into a left-absorbed part, a
	    	right-absorbed part, and a part supported orthogonally to the side.  The absorbed
	    	parts merely change representatives of $r\epsilon$ and $t$, while the orthogonal part
	    	amounts to replacing $b_{0}$ by $b_{0}g$ for $g$ supported in
	    	$(r\epsilon)\wr t$.  By Lemma~\ref{global invariant types} and its class version,
	    	Lemma~\ref{lem:global inv type class}, this replacement does not change the
	    	type over the side hull.  Thus the projection data are determined up to the
	    	 ambiguity stated in the lemma.
	    \end{proof}
	    
	    The following is the analogue of Corollary 7.22 in~\cite{BPZ17}.  The role of
	    Proposition 7.21 in that paper is played by the projection lemma above, while the
	    roles of Lemma 6.4 and Corollary 6.5 are played here by Lemma~\ref{core qe} and
	    Theorem~\ref{t:Mqe}.
	    \begin{corollary}
	    	\label{types of triples} Given $a,b,c\in N$, the type
	    	$\tp(abc)$ is uniquely determined by the triple
	    	\[(\pa{a,b},\pa{b,c},\pa{a,c}).\]
	    \end{corollary}
	    \begin{proof}
	    	Let $abc$ and $a'b'c'$ be two triples of points such that
	    	\[
	    	(\pa{a,b},\pa{b,c},\pa{a,c})=(\pa{a',b'},\pa{b',c'},\pa{a',c'}).
	    	\]
	    	Apply Lemma~\ref{triangle lemma} to both triples.  By the final assertion of the
	    	lemma, the projection data may be chosen so that the word classes
	    	\[
	    		[r\epsilon],\qquad [t],\qquad [y],\qquad \gamma
	    	\]
	    	agree.  Thus we obtain strict side sequences
	    	\[
	    		p:a\longrightarrow b_{0}\longrightarrow c_{*},
	    		\qquad
	    		p':a'\longrightarrow b'_{0}\longrightarrow c'_{*},
	    	\]
	    	of the same word, where $c_{*}=c\gamma$ and $c'_{*}=c'\gamma$.
	    	Choose $H\in\mathcal H_{p}$ and $H'\in\mathcal H_{p'}$ as in the lemma.  By
	    	Corollary~\ref{isomorphisms between tight sets}, there is an automorphism of
	    	$\mathcal N$ sending $H$ to $H'$ and 
	    	\[
	    		a\mapsto a',\qquad b_{0}\mapsto b'_{0},\qquad c_{*}\mapsto c'_{*}.
	    	\]
	    	Since $R_{\gamma^{-1}}$ is one of the named relations, this automorphism also sends
	    	$c=c_{*}\gamma^{-1}$ to $c'=c'_{*}\gamma^{-1}$.
	    	
	    	Finally, $b$ realizes the type $p^{y}_{b_{0},H}$ over $H$, and $b'$ realizes the
	    	corresponding type $p^{y}_{b'_{0},H'}$ over $H'$.  Iterating Lemma~\ref{core qe}
	    	along the letters of $y$ extends the above automorphism to one sending $b$ to
	    	$b'$.  Hence $\tp(abc)=\tp(a'b'c')$, by Theorem~\ref{t:Mqe}.
	    \end{proof}

    	If $\mathcal N$ is a model of $\Th{\mathcal M^G}$ and if $e_1,e_2\in\hat N$ are classes, we write $\pa{e_1,e_2}$ for the collection
    	of strict sequences between $e_1$ and $e_2$. We can finally establish quantifier elimination in $\Th{\mathcal{M}^G}$.
    	
    	\begin{theorem}\label{t:M-fullqe}
    		$\Th{\mathcal{M}^G}$ has absolute quantifier elimination. Moreover, suppose that $N$ is a model of
    		$\Th{\mathcal M^G}$. If \[\{e_1,\ldots,e_n\}\subset\hat N\] are classes, then $\tp(e_1,\ldots,e_n)$ is determined
    		by the tuple $\{\pa{e_i,e_j}\}_{1\leq i,j\leq n}$.
    	\end{theorem}
    	\begin{proof}
    		Suppose that $\{e_1,\ldots,e_n\}\subset\hat N$ are classes, and let
    		$\{e_1',\ldots,e_n'\}\subset\hat N$ satisfy
    		$\pa{e_i,e_j}=\pa{e_i',e_j'}$ for all $i$ and $j$. Choose representatives $a_1\in e_1$ and $a_1'\in e_1'$ so that
    		$a_1$ is independent from $e_2\cdots e_n$ over $e_1$, and so that $a_1'$ is independent from $e_2'\cdots e_n'$ over $e_1'$.
    		Note that we have the equalities \[\tp(a_1/e_1)=\tp(a_1'/e_1'),\quad \tp(e_1e_k)=\tp(e_1'e_k'), \,\, k\geq 2.\]
    		Stationarity of the relevant types (Corollary~\ref{stationarity}) implies then that $\tp(a_1e_1e_k)=\tp(a_1'e_1'e_k')$ for all $k$.
    		It follows that $\tp(a_1e_k)=\tp(a_1'e_k')$.
    		Next, choose $a_2\in e_2$ and $a_2'\in e_2'$ that are independent of $a_1,e_3,\ldots,e_n$ over $e_2$, and of
    		$a_1',e_3',\ldots,e_n'$ over $e_2'$, respectively.
		By stationarity, we see that $\tp(a_2e_2e_k)=\tp(a_2'e_2'e_k')$ for $k\geq 2$, so that
    		in particular $\tp(a_2e_k)=\tp(a_2'e_k')$, and that
    		$\tp(a_1a_2)=\tp(a_1'a_2')$. Recall that $\tp(a_1a_2)=\tp(a_1'a_2')$ if and only if $\pa{a_1,a_2}=\pa{a_1',a_2'}$.
    		By an easy induction, we see that there exist representatives $a_i\in e_i$ and $a_i'\in e_i'$ for
    		all indices $i$, such that $\tp(a_ia_j)=\tp(a_i'a_j')$ for all $i$ and $j$.
    		
    		By virtue of Corollary~\ref{qf characterization}, it now suffices to show
    		that for all pairs of $n$--tuples
    		$\{a,a'\}\subset N^n$ which satisfy $\pa{a_{i},a_{j}}=\pa{a'_{i},a'_{j}}$
    		for all
    		$1\leq i<j\leq n$, it is possible to construct weakly convex sets $A$ and $A'$ extending
    		$a$ and $a'$ respectively and an isomorphism from $A$ to $A'$ sending $a_{i}\mapsto a'_{i}$
    		for all $1\leq i\leq n$.
    		
    		It now suffices to show that given an
    		isomorphism $\phi$ between subsets $C$ and $\hat{C}$ of $N$ such that
    		$(a,b)\mapsto (a',b')$, and given a representative $w\in\pa{a,b}$, there are strict sequences $p$ and $p'$ of type $w$
    		from $a$ to $b$ and from $a'$ to $b'$ respectively such that $\phi$
    		extends to an isomorphism
    		\[C\cup p\to \hat{C}\cup p'.\]
    		Given a strict sequence $p_{0}$ of type $w$ from $a$ to $b$,
    		let $q(x,y,z)=\tp(p_{0},a,b)$. Note that this type is uniquely determined by $w$.
    		
    		We let $p$ be a realization of $q(x,a,b)$ independent from $C$,
    		and $p'$ a realization
    		of $q(x,a',b')$ that is independent from $\hat{C}$. Denote by $\psi$ the unique
    		isomorphism between $p$ and $p'$ taking $a$ to $a'$ and $b$ to $b'$.
    		For $c\in C$, we have that
    		$\tp(cab)=\tp(\phi(c)a'b')$, by Corollary \ref{types of triples}.
    		By Lemma \ref{stationarity}, both $\tp(c/ab)$ and $\tp(\phi(c)/a'b')$ are stationary,
    		which together with the fact that $\tp(abp)=\tp(a'b'p')$ implies that
    		$\tp(cp)=\tp(\phi(c)p')$. Since our language consists only of binary relations,
    		this suffices to show that $\phi\cup\psi$ is an isomorphism. Applying the statement just proved to equality classes, or equivalently to real elements regarded as trivial imaginaries, shows that the quantifier-free type of a finite real tuple determines its complete type. Hence $\Th{\mathcal M^G}$ has absolute quantifier elimination. This establishes the theorem.
    	\end{proof}

\subsection{$\omega$--stability}\label{ss:omega-stability}

	We now record the stability consequence of the preceding quantifier-elimination analysis.
	The statement was announced in Section~\ref{sec:interpretations}; placing the proof here avoids a separate short section and keeps the argument next to the absolute quantifier-elimination theorem on which the parametrization of types depends.

        We note the following corollary of Theorem \ref{t:M-fullqe}, which is the essential ingredient in bounding the number of types.
	
    	\begin{corollary}
    		\label{type parametrization}
    		Let $\mathcal{N}$ be a model of $\Th{\M^G}$, let
    		$A\subseteq N$  be weakly convex, and let $b,b'\in N$ such that there is an $a_{0}\in A$ that is a basepoint for both $b$ and $b'$ and $\pa{b,a_{0}}=\pa{b',a_{0}}$;
    		then $\tp(b/A)=\tp(b'/A)$.
    	\end{corollary}
    	
	\begin{definition}[The partial type $p^w_{a_0,A}$]\label{def:main-type}
    	Let $w=\pa{b,a_0}$. We write $p^w_{a_0,A}$ for the type over $A$ of a point $b$ such that $a_{0}$ is a basepoint for $b$ in $A$ and $w\in \pa{b,a_{0}}$.
    	We also denote by $p^{w}_{a_{0},N}$ the global type which is the union of $p^{w}_{a_{0},A}$, where $A$ ranges over all weakly convex sets containing $a_{0}$. 
	\end{definition}

   \begin{observation}
   	\label{o: all global types are parametrized}Notice that the same minimization argument that led us to conclude that any $1$-type over a weakly convex set $A$ is of the form $p^{w}_{a_{0},A}$ for some reduced word $w$ and some $a_{0}\in A$ generalizes to show that any global $1$-type is of the form $p^{w}_{a_{0},N}$. 
   \end{observation}
    	
    	\begin{proof}[Proof of Corollary~\ref{type parametrization}]
    		First we get \[\qftp(b/A)=\qftp(b'/A)\] by Corollary \ref{qf characterization}
    		and Lemma \ref{gate property}. By
    		Theorem~\ref{t:Mqe}, we finally conclude \[\tp(b/A)=\tp(b'/A).\]
    	\end{proof}

	The following result is part of Theorem~\ref{thm:stable} from the introduction.
    	\oldversion

    	\begin{proof}
    		Let $\mathcal{N}$ be a sufficiently saturated model of $\Th{\mathcal{M}^G}$.
    		By a standard argument (see, for instance, \cite{pillay}, Lemma 0.9) it suffices to show that given a countable set of parameters
    		$A\subset N$, there are at most countably
    		many distinct types in one variable over $A$.
    		If $A$ is weakly convex, this follows immediately from Corollary
    		\ref{type parametrization}, since
    		there are only countably many choices for
    		$a_{0}$ and countably many choices for $\pa{b,a_{0}}$.
    		For the general case, note that an arbitrary countable subset of $N$
    		is contained in a countable weakly convex set, by repeated applications
    		of the construction in Lemma \ref{core qe}.
    	\end{proof}
    	
    	\subsection{The flip graph}
	
	    	To complete the discussion in this section,
	we make a few brief remarks about the case $\D=\varnothing$, which is the collection of connected domains $\D_X$
    	we obtain when $X$ is the flip graph. In this case, the results in the preceding sections are more or less trivial, since the only
    	relations are $R_g$ for $g\in G$. In particular, any union of $G$--orbits is automatically weakly convex.
	In this case we obtain that $\mathcal M^G$ has absolute quantifier elimination. In particular,
    	absolute quantifier elimination holds in the case where $G=\Mod^{\pm}(\Sigma)$. As a consequence of the fact that
    	the natural bi-interpretation between the flip graph and $\M^G$ introduces an existential quantifier, we have the following:
    	
    	\begin{corollary}\label{cor:flip-qe}
    		The theory of a flip graph of a surface with at least one puncture
		admits relative quantifier elimination with respect to $\exists$--formulae.
    	\end{corollary}

\section{Morley rank}\label{sec:morley}
  	
  	We now turn our attention to the problem of finding upper and lower bounds
  	for the Morley rank
  	of types in $\Th{\M^G_\D}$ for a fixed $\D\subseteq\D_{0}$, as in Definition \ref{def-D}. In this section, the upper-bound arguments use the pure finite-index subgroup fixed in Section~\ref{sec:full-qe}, since they appeal to absolute quantifier elimination. Statements for the full extended mapping class group are obtained by passing to this finite-index pure subgroup; the interpreted curve graph and the relevant Morley-rank bounds are unchanged by this finite-index passage.
  	For compactness of notation, we will suppress $G$ and $\D$ where it does not cause confusion.
  	Throughout, we will use $\mathcal{N}$ to denote an appropriately saturated (usually at least $\omega$-saturated)
  	model of $\Th{\mathcal{M}}$.

  	\begin{definition}\label{def:mathcal S}
  		We write $\mathcal{S}$ for the collection of triples of the form
  		$(\bar{w},\bar{b},\bar{v})$, where
  		\[\bar{w}=(w_{j})_{j=1}^{r}, \quad \bar{v}=(v_{i,j})_{1\leq i<j\leq r}\]
  		are tuples of reduced words, and where $b_1\cdots b_r=\bar b\in N^r$ is an $r$--tuple of elements. To any
  		$\tau=(\bar{w},\bar b, \bar{v})\in\mathcal{S}$, we associate the $r$-variable formula
  		$$
  		\psi_{\tau}(x)\equiv\bigwedge_{j=1}^{r}R_{w_{j}}(x_{j},b_j)
  		\wedge\bigwedge_{1\leq i<j\leq r}R_{v_{i,j}}(x_{i},x_{j}),
  		$$
  		together with the ordinal
  		\[\Or(\tau):=\left(\left(\bigoplus_{1\leq j\leq r}\Or(w_{j})\right)
  		\oplus\left(\bigoplus_{1\leq i<j\leq r}\Or(v_{i,j})\right)\right).\]
  	\end{definition}

  	\begin{theorem}\label{thm:morley-upper-bound}
  		Let $G$ be a fixed finite index $G<\Mod^{\pm}(\Sigma)$
  		consisting of pure mapping classes. Then for any
  		$\tau=(\bar{w},\bar b, \bar{v})\in\mathcal S$, we have $RM(\psi_{\tau})\leq \Or(\tau)$.
  		In particular, for any $r>0$ the Morley rank of $M^{r}$ is at most
  		\[\binom{r+1}{2}\omega^{k(\Sigma)},\] where $k(\Sigma)$ is the
  		maximal length $k$ of a chain of connected domains
  		\[\emptyset\subsetneq D_{0}\subsetneq D_{1}\subsetneq D_{2}\cdots
  		\subsetneq D_{k}=\C \mbox{ in } \D \].
  	\end{theorem}
  	\begin{proof}
  		By quantifier elimination (Theorem~\ref{t:M-fullqe}), any type $q\in S^{r}(N)$ containing
  		$\psi_{\tau}$ (except at most one) must contain a formula of one of the following forms:
  		\begin{itemize}
  			\item $R_{w'_{j}}(x,b_j')$, where $b_j'\in N$ and $\Or(w'_{j})<\Or(w_{j})$
  			for some $1\leq j\leq r$;
  			\item $R_{v'_{i,j}}(x_{i},x_{j})$ for some proper reduct $v'_{i,j}$ of
  			$v_{i,j}$ and $1\leq i<j\leq r$.
  		\end{itemize}
  		In both cases, $q$ contains a formula of the form
  		$\psi_{\tau'}$ for some $\tau'\in\mathcal{S}$ with
  		$\Or(\tau')<\Or(\tau)$.
  		By induction,
  		\[RM(\psi_{\tau})\leq \Or(\tau')<\Or(\tau)\] and thus $RM(q)<\Or(\tau)$.
  		It follows from the characterization of the Morley rank of a type over $N$ as its
  		Cantor--Bendixson rank as a point in the space $S^{r}(N)$ and that of a formula as the
  		Cantor--Bendixson rank of the corresponding clopen subset of $S^{r}(N)$, that
  		$RM(\psi_{\tau})\leq \Or(\tau)$, as desired.
  	\end{proof}
  	
  	\begin{remark}
  		\label{rem: r=1}One can in fact show with relative ease that the result above is also valid in case $r=1$, without restrictions
  		on $G$.
  	\end{remark}
  	
  	The Morley rank can be bounded from below applying the same strategy as in
  	\cite{BPZ17}. To this end, we use $\mathcal{N}$ to denote the
  	monster model of $\Th{\M}$. For the lower-bound arguments, we continue to work with the same pure finite-index subgroup $G$. When applying the statements to the full extended mapping class group, we first pass to this subgroup; this does not change the curve graph or the relevant interpreted geometric graphs.

  	\begin{lemma}
  		\label{characterization of forking}
  		Let $D\in\D$ and let $e\in N/R_D$. Suppose that $B\subseteq C\subset N$ are weakly convex sets.
  		We have $\indep{e}{B}{C}$ if and only if some
  		basepoint for $e$ in $C$ lies in $B$.
  	\end{lemma}
  	\begin{proof}		
  		Pick a basepoint $c_{0}$ for $e$ in $C$, 
  		let $w$ be such that $[w]=\pa{a,c_0}$ has minimal ordinal among all $a\in e$, and let $F=RA(Dw)$.
  		By virtue of Lemma \ref{lem:global inv type class} and Remark \ref{rem: canonical bases}, the canonical base of 
  		the stationary type $p^{w}_{c_{0},C}$ is precisely $[c_{0}]_{F}$.
		Since $[c_{0}]_{F}\cap C$ is also the collection of basepoints of $e$ in $C$, it suffices to verify that
		\[[c_{0}]_{F}\cap B=\emptyset\quad \textrm{if and only if}\quad
  		[c_{0}]_{F}\notin \widehat{\acl}(B).\]
		The only if direction is clear. For the if direction, let $v$ be a minimizing word from $c_{0}$ to a basepoint $b_{0}$ for
		$c_{0}$ in $B$, let $q=p^{v}_{b_{0},N}$ be satisfied by $c_{0}$, and let
		$(c_{i})_{i\in\omega}$ be a Morley sequence of $q$ over $B$. By assumption, $v$ is not contained in $F$, and thus neither is
  		$[v]*[v^{-1}]$. However, we have that $\pa{c_{i},c_{j}}=[v]*[v^{-1}]$ for $i\neq j$,
		whence it follows that $[c_{i}]_{F}\neq[c_{j}]_{F}$ for $i\neq j$ and $[c_0]_F\notin \widehat\acl(B)$, the desired conclusion. 
  	\end{proof}

    Since a class in $\hat{N}$ intersecting a set $B\subseteq N$ is definable over $B$, we have:
    \begin{corollary}
    	Every $1$-type in $\hat{N}$ over a weakly convex set is stationary. 
    \end{corollary}
  	
  	\newcommand{\rd}[0]{\prec^{r}}
  	\newcommand{\rde}[0]{\preceq^{r}}
  	\begin{definition}
  		Given reduced words $v,w\in\wo$, we write $v\rd_{0}w$ if
  		$v\simeq v_{0}v_{1}$ and $w\simeq v_{0}D$, where $v_{1}\in\wo(D)\cap\D^{*}$,
  		and where none of the domains occurring as letters in $v_{1}$ are contained in $\partial D$.
  		We let $\rd$ be the transitive closure of $\rd_{0}$. Clearly $\rd\subset\prec$,
  		so it is well-founded, and we let $\Or^{r}$ be the corresponding foundation rank on $\wo/\simeq$.
  	\end{definition}
  	
  	\begin{definition}
  		Given $E\in\D$, let $\wo_{E}$ be the collection of $w\in\wo$ such that a reduction without cancellation of $Ew$ only involves
  		proper absorptions of $E$ into $w$. We denote by $\Or^{r}_{E}$ the foundation rank of $\rd$ when restricted to $\wo_{E}/\simeq$.
  	\end{definition}
  	Notice that $\wo_E$ and $\wo(E)$ are very different notions.
  	For any $w\in \wo$, we let $\Or^{+}(w)$ be equal to the (non-commutative) sum \[\Or(D_{1})+\Or(D_{2})+\dots+\Or(D_{k}),\] where
  	here $D_{1},\ldots, D_{k}$ are read in order. We set $\Or^{[+]}([w])$ to be the maximum of $\Or^{+}(w')$, taken over $w'\in[w]$.

  	We include the proof of the following easy fact for the sake of completeness:
  	\begin{lemma}
  		\label{l: ordinal sum lemma}For arbitrary ordinals $\alpha,\beta,\gamma$ with $\gamma<\omega^{\beta}$, we have
  		$\alpha\oplus\gamma<\alpha+\omega^{\beta}$.
  	\end{lemma}
  	\begin{proof}
  		Suppose we are given a counterexample, where by induction we may assume that
  		$\alpha$ is minimized. Then there must exist a well-founded linear order
  		$(J,<)$ of type $\alpha\oplus\gamma$ and a partition \[J=J_{\alpha}\coprod J_{\gamma},\] where $<$ is isomorphic to
  		$\alpha$ and to $\gamma$ when restricted to $J_{\alpha}$ and to $J_{\gamma}$, respectively, together with a linearly
  		ordered set $(K,<)$, which is the concatenation of $K_{\alpha}$ (isomorphic to $\alpha$) and $K_{\omega^{\beta}}$
  		(isomorphic to $\omega^{\beta}$), equipped with an isomorphism $f$ between $K$ and some initial segment of $J$.
  		
  		The key observation is that for arbitrary ordinals $\delta,\delta'<\omega^{\beta}$, we have $\delta\oplus\delta'<\omega^{\beta}$.
  		This implies that the restriction of the order to \[L=K_{\omega^{\beta}}\cap f^{-1}(J_{\alpha})\]
  		must be isomorphic to $\omega^{\beta}$.
  		
  		It follows in turn that
  		\[\alpha=\alpha'+\omega^{\beta}+\alpha''\] for suitable ordinals $\alpha'$ and $\alpha''$, where here $\alpha'$ corresponds
  		via the isomorphism between $J_{\alpha}$ and $\alpha$ to the intersection $f^{-1}(J_{\alpha})\cap K_{\alpha}$, and
  		$\omega^{\beta}$ to $f(K_{\omega^{\beta}}\setminus L)\subseteq J_{\gamma}$.
  		
  		On the one hand, the map $f$ witnesses the fact that
  		$\alpha\leq\alpha'\oplus\gamma'$ for some $\gamma'<\omega^{\beta}$. On the other hand, the induction hypothesis implies that $\alpha'\oplus\gamma'<\alpha'+\omega^{\beta}$ and thus $\alpha<\alpha'+\omega^{\beta}$, contradicting the previous paragraph.
  	\end{proof}
  	
  	\begin{lemma}
  		\label{l: calculating r-ank}Let $E\in\D$ and let $w\in\wo_{E}$. Then $\Or^{r}_{E}([w])=\Or^{[+]}([w])$.
  	\end{lemma}
  	\begin{proof}
  		We proceed by induction on $\Or(w)$. Let us show the inequality $\Or^{r}(w)\leq \Or^{[+]}(w)$ first.
  		Suppose that $w\simeq w'D$, and consider an arbitrary $z\in\wo(D)$ such that $w'z$ is reduced.
  		By the induction hypothesis, we have \[\Or^{r}([w'z])\leq \Or^{[+]}([w'z]).\] We claim that in fact
  		\[\Or^{[+]}([w'z])< \Or^{[+]}([w]).\] Indeed, for any $v\in[w'z]$, the partial sum of the ordinals corresponding to
  		letters coming from $w'$ (reading as a word $v'\in[w']$) is bounded by above by $\Or^{[+]}([w'])$.
  		
  		By Lemma \ref{l: ordinal sum lemma}, we obtain
  		$$\Or^{+}(v)\leq \Or^{+}(v')\oplus \Or(z)<\Or^{+}(v')+\Or(D)\leq \Or^{[+]}([w]),$$
  		and since the choice of $v$ is arbitrary, we conclude that \[\Or^{r}_{E}([w])\leq \Or^{r}([w])\leq \Or^{[+]}([w]).\]
  		
  		For the other inequality, we can assume that $w=w'D$ satisfies \[\Or^{+}(w)=\Or^{[+]}([w]).\] We may
  		also assume that the complexity of the domain $D$ is greater than one,
  		since otherwise the result follows immediately from the inductive assumption on $w'$.
  		Let $k=k(D)$. We make the following claim:
  		\begin{claim*}
  			For any $m\in\mathbb{N}$, there exists a word
  			\[D_{1}\cdots D_{m}\in\wo(D)\cap\D^{*}\] such that the complexity of each $D_i$ is $k-1$, and where
  			the reduction without cancellation of $Ew'D_{1}\cdots D_{m}$ only involves proper absorption of letters of $E$
  			into $w'D_{1}\cdots D_{m}$.
  		\end{claim*}
  		This is enough to finish the proof of the Lemma, since by induction we have:
  		$$\Or^{r}_{E}(w'D_{1}\cdots D_{m})\geq \Or^{[+]}([w'D_{1}\cdots D_{m}])\geq \Or^{[+]}([w'])+m\omega^{k-1}.$$
  		Since $m$ is arbitrary, it follows that \[\Or^{r}_{E}(w')\geq \Or^{[+]}(w')+\omega^{k}=\Or^{[+]}([w]).\]
  		
  		In order to prove the claim, recall first that for any word $w$, there exists a unique maximal domain $E_{0}\in\Dvee$ such that
  		whenever a reduced word $Ew'$ is equivalent to $w$ then $E$ is a union of components of $E_{0}$; this is the content of
  		Lemma~\ref{ends}.
  		Consider now $F:=\rabs(Ew')$. By assumption, $D$ is not absorbed by $F$.
  		
  		Using the infinite diameter of $\C(D)$ we can choose $\{D_{1},\ldots, D_{m}\}$ with $D_{i}$ of complexity $k-1$, and with
  		\[D_{i}\tv D_{i+1},\quad E_{0}\nsubseteq D_{i},\quad D_{i}\nsubseteq F.\] It follows that the only possible absorption in the reduction without cancellation of $Ew'D_{1}\cdots D_{m}$ is a proper absorption of $E$ into some letter of $w'$.
  	\end{proof}

  	Observe that any finite $\prec^{r}$ descending chain starting at a word $w$ gives rise to a chain of
  	extensions of any type $p^{w}_{b_{0},B}$, where $B$ is weakly convex, which follows from an iterated
  	application of the following lemma.

  	\begin{lemma}
  		\label{approaching extensions}Let $a\in N$, and let $B\subseteq N$
  		be a weakly convex set such that there is a minimizing sequence
  		of type $w D$ from $a$ to $B$. Let $v\in\wo(D)$ be a (possibly empty) word
  		such that $wv$ is reduced. Then there is a weakly convex set $\hat{B}$ containing $B$
  		and a minimizing sequence of type $wv$ from $a$ to $\hat{B}$.
  	\end{lemma}
  	\begin{proof}
  		We may assume that $w\in\wo\cap\D^{*}$.
  		Suppose that \[a=a_{0},a_{1},a_{2},\ldots, a_{r}=b_{0}\in B\] is a minimizing
  		sequence of type $w$ from $a$ to $B$ and pick $a'\in N$ such that $\pa{a,a'}=[v]$.
  		Notice that $a'$ must be one $D$-step away from
  		$B$ with basepoint $b_{0}$, just like $a_{r-1}$. Let $\hat{B}$ be a one-step extension
  		of $B$ through $a'$. We claim that for any strict $w$-sequence
  		\[a_{r-1}=a'_{r-1},a'_{r},\ldots, a'_{s}=a'\] from $a_{r-1}$ to $a'$, the resulting sequence
  		\[a_{0},a_{2},\ldots, a_{r-1},a'_{r},\ldots, a'_{s}\] is a minimizing
  		sequence from $a$ to $\hat{B}$.
  		
  		Indeed, suppose not. Since $wv$ is reduced, we have that there exists a point
  		$b_{1}\in \hat{B}$ such that
  		$\pa{a,b_{1}}=[w']$, with $\Or(w')<\Or(wv)$. Since
  		\[[w']= [w]\ast [v]\ast\pa{a',b_{1}},\] this can only take place if at least one cancellation move is
  		used in the reduction process. The fact that $v\perp D$, together with the assumption
  		that none of the letters in $v_{1}$ are contained in $\partial D$, implies that such a
  		cancellation involves a letter in $\pa{a',b_{1}}$, together with a letter
  		in $w$. This contradicts the minimality of the sequence
  		$a_{0},a_{1},\ldots, a_{r}$.
  	\end{proof}
  	
  	Recall that in the $\omega$-stable context, a type extension $p\subset q$ is
  	non-forking if and only if $RM(p)=RM(q)$. It follows that
  	the foundation rank (on the class of complete types over varying sets of parameters) of
  	the relation $\prec^{f}$, given by $q\prec^{f}p$ if and only if $q$ is a
  	forking extension of $q$ (known as the \emph{Lascar $U$-rank})
  	bounds the Morley rank of a type
  	from below.
  	
  	\begin{lemma}
  		\label{lower bound on Morley rank-Lascar}
  		Let $B\subseteq N$ be weakly convex, let $E\in\D$, and let $w$ be a word such that either $Ew$ is reduced or $E$
  		is properly left-absorbed by $w$. Then the Lascar $U$--rank satisfies
  		\[U(p^{E,w}_{b,B})= \Or^{r}_{E}(w).\]
  	\end{lemma}
  	\begin{proof}
  		Without loss of generality, we assume $w\in\D^{*}$. Pick any realization $a$ of the type $p^{w}_{b,B}$,
  		so that $e=[a]_{E}$ is a realization of $p^{E,w}_{b,B}$. Pick any finite descending chain
  		\[w=w_{0}\succ^{r}w_{1}\succ^{r}\cdots \succ^r w_{k}\] such that for each $1\leq i\leq k$, either $Ew_{i}$ is reduced or $E$
  		is properly left-absorbed by $w_{i}$. Iteratively applying Lemma \ref{approaching extensions}, we find an
  		ascending chain of weakly convex sets $B=B_{0}\subset B_{1}\subset \cdots\subset B_{k}$ such that there is a
  		minimizing sequence of type $w_{i}$ from $e$ to $B_{i}$.
  		
  		Our assumption on the interaction between $w_{i}$ and $E$ implies that for all \[0\leq i<j\leq k,\]
  		no reduct of a word of the form $u_{E}w_{i}$, where here $u_{E}\in\wo(E)\cup\{E\}$, can be equal to a reduct of a word of the
  		form $u'_{E}w_{j}$, where $u'_{E}\in\wo(E)\cup\{E\}$. Lemma \ref{characterization of forking}
  		then implies that $\tp(e/B_{j})$ forks over $B_{i}$. The $\geq$ inequality follows by induction on
  		$<^{r}$, using Lemma~\ref{l: calculating r-ank}.
  		
  		For the reverse inequality, begin by noticing that for any set $B$ of parameters, we can consider a weakly convex
  		$\hat{B}$ containing $B$ that is independent from $e$ over $B$. Then, $\tp(e/\hat{B})$ does not fork over $\tp(e/B)$,
  		and thus the Lascar $U$--rank satisfies \[U(\tp(e/\hat{B}))=U(\tp(e/B)).\]
  		On the other hand, if $B\subseteq C$ are weakly convex then we have
  		$\tp(e/C)$ forks over $\tp(e/B)$, or equivalently \[U(\tp(e/C))<U(\tp(e/B)),\] if and only if
  		$\Or_{E}^{r}([w])<\Or_{E}^{r}([v])$, where here $v$ and $w$ are minimizing words from $e$ to $B$ and $C$, respectively.
  		The inequality then follows by induction on $\prec^{r}$.
  	\end{proof}

  	It is useful to remark that Lemma~\ref{lower bound on Morley rank-Lascar} can be applied to the case where $w$ is a single domain $D$,
  	in which case $E$ must be distinct from a component of $D$. We then get that $U(p^{E,D}_{b,B})= \Or^{r}_{E}(D)$.
  	Putting this information together with Lemmas \ref{characterization of forking} and
  	\ref{approaching extensions} yields lower bounds on Morley rank.
  	
  	\begin{corollary}
  		\label{lower bound on Morley rank}
  		Let $b\in N$ and let $B\subseteq N$ be weakly convex. Then
  		we have $RM(p^{w}_{b,B})\geq \Or^{r}(w)$. In particular,
  		\[RM(R^{*}_{D}(x,a))=RM(R_{D}(x,a))\geq\Or(D)\] for any $D\in\D\setminus\{\C\}$, and
  		\[RM(x=x)=\omega^{k(\D)}.\]
  	\end{corollary}
  	
  	The reader may recall the definition of the length $k(\D)$ of a downward closed collection of connected domains from Section~\ref{sec:geometric_graphs}.

    \begin{example}[Rank intuition from the running example]
    \label{ex:running-rank}
    In Example~\ref{ex:running-separating-curve}, we have chains of domains
    \[
        A\subset D\subset \mathcal C_0(\Sigma),\qquad
        A\subset E\subset \mathcal C_0(\Sigma).
    \]
    The annular domain $A$ gives one level of freedom, while the non-annular
    domains $D$ and $E$ give higher-rank directions corresponding to the curve
    graphs of $Y$ and $Z$. Morley rank detects the longest possible chains of
    such domain directions.

    This also explains why different geometric graphs associated to the same
    surface can have different ranks. The curve graph sees non-annular
    subsurface directions such as $D$ and $E$. The pants graph, by contrast,
    primarily sees annular directions coming from individual curves in a pants
    decomposition. Thus the pants graph can be interpretable in the curve graph
    while still having strictly smaller Morley rank.
    \end{example}
  	
  	\begin{corollary}
  		Let $X(\Sigma)$ be a geometric graph satisfying the assumptions of Lemma
  		\ref{interpretation in B}. Then $RM(X(\Sigma))$,
  		i.e.~the Morley rank of the formula $x=x$ in $\Th{X(\Sigma)}$, is at least $\omega^{k(\D_{X(\Sigma)})}$.
  	\end{corollary}
  	\begin{proof}
  		We fix notation \[X=X(\Sigma),\quad \D=\D_X,\] and let $G<\Mod^{\pm}(\Sigma)$ be the pure finite-index subgroup used in Section~\ref{sec:full-qe}; set $\M=\M_{\D}^{G}$. Passing from the full extended mapping class group to this finite-index subgroup does not change the interpreted graph $X$ or the domain length $k(\D)$.
  		By Lemma \ref{interpretation in M}, there is an interpretation $\bar{\zeta}$ of $X$ into $\M$ which sends the universe of
  		$X$ into a finite disjoint union of imaginary sets of the form
  		$M/R_{D}$, where $D$ is a proper subdomain.
  		Choose an orbit of vertices whose associated domain has length $k(\D)$, and let $Y$ be the corresponding imaginary sort in this finite disjoint union. Since this domain is proper and realizes the maximal length in $\D$, the remarks immediately after Lemma~\ref{lower bound on Morley rank-Lascar} give
  		$RM(Y)\geq\omega^{k(\D)}$.
  		
  		Now, by Lemma \ref{interpretation in B} there is also an interpretation $\bar{\eta}$ of $\M$ in $X$ such that
  		$(\bar{\eta},\bar{\zeta})$ form a bi-interpretation between $X$ and $\M$. Let $Z$ be the image of the definable
  		imaginary set $Y$ by $\bar{\eta}$. The interpretation $\bar{\eta}$ sends $Y$ to $Z$, so $RM(Z)=RM(Y)$.
  		On the other hand, the fact that the pair is a bi-interpretation provides a $\emptyset$-definable bijection
  		between $Z$ and a $1$-variable definable subset of $X$, so that $RM(X)\geq\omega^{k(\D)}$.
  	\end{proof}
  	
  	In the context of forking, we collect the following two facts, which will be used in the sequel. A theory is \emph{totally trivial} if for
  	an arbitrary model $N$, an arbitrary set of parameters $C$, and an arbitrary tuple $(a,b,c)$, we have that if $a$ is independent from both
  	$b$ and $c$ over $C$ then it is also independent from $(b,c)$ over $C$ (see~\cite{trivial-goode}).
  	\begin{lemma}
  		\label{total triviality} $\Th{\hat{\mathcal{M}}}$ is totally trivial.
  	\end{lemma}
  	\begin{proof}
  		The proof is an almost verbatim reprisal of that of Proposition 7.26 in \cite{BPZ17}, using Theorem~\ref{t:M-fullqe}.
  	\end{proof}
  	
  	 Recall that a theory has \emph{weak elimination of imaginaries} if every imaginary is in
	the definable closure of a finite tuple in the home sort, and the tuple lies in the algebraic closure of the imaginary.
	Recall that a \emph{canonical base} of a stationary type $p$ is a set $C$ such that for any automorphism $\alpha$ of
  	a sufficiently saturated model $\mathcal N$, we have
  	$\alpha$ fixes $C$ if and only if it fixes $p$. A canonical base is unique up to interdefinability.
  	An $\omega$--stable theory will have weak elimination of imaginaries if the canonical base of a stationary
  	type can always be chosen in the home sort.
  	
  	\begin{corollary}
  		\label{c: weak elimination of imaginaries}$\Th{\hat{\mathcal{M}}}$ has weak elimination of imaginaries.
  	\end{corollary}
  	\begin{proof}
  		This follows from stationarity (Corollary \ref{stationarity}) and total triviality, by a standard argument
  		as in the proof of Corollary 7.28 in \cite{BPZ17}.
  	\end{proof}

	\section{Interpretation rigidity of the curve graph}\label{sec:int-rigid}
  	
  	The goal of this section is to prove \emph{interpretation rigidity} for curve graphs. Namely, if the curve graphs of surfaces $\Sigma_1$ and
  	$\Sigma_2$ are mutually interpretable, then except for some sporadic low-complexity cases, it must be the case that
  	$\Sigma_1$ and $\Sigma_2$ are homeomorphic.
  	
  	\subsection{Algebraic closure and definable closure}
    	Consider the theory $\Th{\M_D^G}$, where here as before
    	$G<\Mod^{\pm}(\Sigma)$ denotes the pure subgroup consisting of mapping classes
    	acting trivially on homology modulo $3$. As before, we use $\mathcal N$ to denote a sufficiently saturated model of $\Th{\M_D^G}$,
    	and $\D\subset \D_0$ is downward closed, $G$--invariant, and consists of
connected domains.
    	
    	Our first goal will be to show that for a finite tuple of elements in $\hat N$, algebraic and definable closures coincide.
    	The main input from Section~\ref{sec:full-qe} is Lemma~\ref{lem-convex-extensions}: it allows us to compare different weakly convex hulls by sliding points in parallel copies of minimizing sequences, while keeping track of the boundary directions where wobbling can occur through the chosen sets $\cn_{D}(a,H)$.  The first lemma below records the precise form of this comparison needed in the algebraic-closure argument.

    \begin{lemma}\label{lem: refined convex extensions}
    Let $H_0$ be a weakly convex set with finitely many $G$--orbits, let
    \[
        p:a=a_0,a_1,\ldots,a_k=b_0
    \]
    be a minimizing sequence from $a$ to the basepoint $b_0\in H_0$, and let
    $H\in\mathcal H(p,H_0)$.
    Then, for $a'\in H$, the following are equivalent:
    \begin{enumerate}[(a)]
        \item there is a minimizing sequence $p'$ from $a'$ to $H_0$ such that
        $H\in\mathcal H(p',H_0)$;
        \item for some, equivalently for every, basepoint $b'_0$ for $a'$ in
        $H_0$, one has
        \begin{equation}\label{ordinal equality}
            \Or(\pa{a',b'_0})=\Or(\pa{a,b_0}).
        \end{equation}
    \end{enumerate}
    Moreover, suppose that $a^1,\ldots,a^r\in H$ satisfy these equivalent
    conditions.  Then there are representatives
    \[
        \tilde a^i\in a^iG,
        \qquad
        b_0^i\in H_0,
        \qquad 1\leq i\leq r,
    \]
    such that:
    \begin{enumerate}
        \item\label{one} $b_0^i$ is a basepoint for $\tilde a^i$ in $H_0$;
        \item\label{two} there is a word $u_{\vert}^{\max}\in\D^*$, independent
        of $i$, such that
        \[
            \pa{\tilde a^i,b_0^i}=[u_{\vert}^{\max}];
        \]
        \item\label{three} for all distinct $i,j$,
        \[
            \pa{\tilde a^i,\tilde a^j}=\pa{b_0^i,b_0^j},
            \qquad
            [u_{\vert}^{\max}]\perp^*\pa{\tilde a^i,\tilde a^j};
        \]
        \item\label{four} for every $c\in H$, there exist
        \[
            c_1\in cG,
            \qquad
            b_1\in H_0,
            \qquad
            c_0^i\in H,
            \qquad 1\leq i\leq r,
        \]
        such that:
        \begin{itemize}
            \item $c_0^i$ lies on a strict sequence from $\tilde a^i$ to
            $b_0^i$;
            \item the initial classes are independent of $i$:
            \[
                \pa{\tilde a^i,c_0^i}=\pa{\tilde a^j,c_0^j};
            \]
            \item
            \[
                \pa{\tilde a^i,c_1}
                =\pa{\tilde a^i,c_0^i}*\pa{c_0^i,c_1};
            \]
            \item
            \[
                \pa{b_0^i,b_1}=\pa{c_0^i,c_1},
                \qquad
                \pa{c_0^i,c_1}\perp^*\pa{c_0^i,b_0^i},
            \]
            and
            \[
                \pa{c_1,b_1}=\pa{c_0^i,b_0^i}.
            \]
            \item $\pa{c_1,b_1}$ admits a representative in $\D^*$.
        \end{itemize}
    \end{enumerate}
\end{lemma}
\begin{proof}
    We first prove the equivalence.  Suppose that $H\in\mathcal H(p',H_0)$
    for a minimizing sequence $p'$ from $a'$ to $H_0$.  The weakly convex hull
    $H$ is strongly incompressible both over $aH_0$ and over $a'H_0$
    (Lemma~\ref{lem: incompressibility}).  A homomorphic retraction comparing
    these two hull structures cannot decrease one of the two minimizing
    ordinals without decreasing the other.  Hence the ordinal of a minimizing
    sequence from $a'$ to $H_0$ is equal to $\Or(\pa{a,b_0})$.

    Conversely, assume that $a'$ has a basepoint $b'_0\in H_0$ satisfying
    \eqref{ordinal equality}.  Apply Lemma~\ref{lem-convex-extensions} to the
    hull $H\in\mathcal H(p,H_0)$ and to the point $a'\in H$.  After replacing
    $a'$ by an element of its $G$--orbit and choosing $b'_0$ appropriately, we
    obtain representatives in $\D^*$ with
    \[
        \pa{a',b'_0}=\pa{a,b_0},
        \qquad
        \pa{a,a'}=\pa{b_0,b'_0},
        \qquad
        \pa{a,a'}\perp^*\pa{a,b_0}.
    \]
    Thus $a'$ lies in a parallel copy of the side from $a$ to $H_0$, and the
    parallel displacement is strongly orthogonal to that side.  If
    $H'\in\mathcal H(p',H_0)$ is any hull built from a minimizing sequence
    $p'$ from $a'$ to $H_0$, weak convexity supplies a point $\tilde a\in H'$ with
    \[
        \pa{\tilde a,a'}=\pa{a,a'},
        \qquad
        \pa{\tilde a,b_0}=\pa{a,b_0}.
    \]
    The displacement from $a$ to $\tilde a$ is supported only in boundary
    directions common to the displayed words; otherwise one would shorten one
    of the minimizing sequences.  Hence $b_0$ is also a basepoint for
    $\tilde a$ over $H_0$, and quantifier elimination gives
    \[
        \tp(a,a'H_0)=\tp(\tilde a,a'H_0).
    \]
    Moving $\tilde a$ to $a$ by an automorphism fixing $a'H_0$ sends $H'$ to a
    hull over $H_0$ which contains $a$ and is isomorphic to $H$.  By
    incompressibility, this hull is $H$.  Therefore $H\in\mathcal H(p',H_0)$.

    We now prove the simultaneous statement for $a^1,\ldots,a^r$.  Fix once
    and for all a construction
    \[
        H=K_0\supset K_1\supset\cdots\supset K_k=H_0
    \]
    realizing $H$ as a weakly convex $p$--hull of $H_0$, so that $K_\ell$ is
    obtained from $K_{\ell+1}$ by the one-step $D_\ell$--extension through
    $a_\ell$ over $a_{\ell+1}$, using the chosen cone
    \[
        \cn_{D_\ell}(a_{\ell+1},K_{\ell+1}).
    \]
    Replacing each $a^i$ by a suitable element of its $G$--orbit, we may assume
    that the words $\pa{a,a^i}$ have representatives in $\D^*$.

    We argue by induction on $k$.  At the first layer, Lemma~\ref{lem-convex-extensions}
    expresses each $a^i$ either as a point of $K_1$ or as the $D_0$--lift of a
    point of the same chosen cone
    \[
        \cn_{D_0}(a_1,K_1).
    \]
    The equality of the minimizing ordinals forces the same first-step behavior
    for all the points satisfying the equivalent conditions above.  The
    displacements between the resulting points are strongly orthogonal to the
    common side, again by Lemma~\ref{lem-convex-extensions}.  Applying the
    induction hypothesis in $K_1$ gives representatives $\tilde a^i$ and
    basepoints $b_0^i$ satisfying \ref{one}, \ref{two}, and \ref{three}.

    It remains to prove \ref{four}.  Fix $c\in H$.  Apply the refined
    Part~\ref{part3} of Lemma~\ref{lem-convex-extensions} simultaneously to
    the strict sequences from $\tilde a^i$ to $b_0^i$.  The preceding paragraph
    ensures that these sequences use the same layer filtration and differ only
    by the parallel translations allowed in Lemma~\ref{lem-convex-extensions}.
    Consequently the same terminal translate $c_1\in cG$ and the same point
    $b_1\in H_0$ work for all $i$, while the side point may depend on $i$; call
    it $c_0^i$.  Lemma~\ref{lem-convex-extensions} gives
    \[
        \pa{\tilde a^i,c_1}
        =\pa{\tilde a^i,c_0^i}*\pa{c_0^i,c_1},
    \]
    and
    \[
        \pa{b_0^i,b_1}=\pa{c_0^i,c_1}\perp^*\pa{c_0^i,b_0^i},
        \qquad
        \pa{c_1,b_1}=\pa{c_0^i,b_0^i}.
    \]
    Since the side words $\pa{\tilde a^i,b_0^i}$ are all equal to
    $[u_{\vert}^{\max}]$, the initial segments
    $\pa{\tilde a^i,c_0^i}$ are independent of $i$.  The final assertion that
    $\pa{c_1,b_1}$ has a representative in $\D^*$ is also part of
    Lemma~\ref{lem-convex-extensions}.  This proves \ref{four} and completes
    the proof.
\end{proof}

    \begin{lemma}\label{lem:acl=dcl imaginary}
    		Let $e=([a_i]_{D_i})_{i=1}^k$ be a finite tuple of imaginaries in $\hat N$, and let $H\in\ti(e)$.
    		For all $a\in H$, there exists a (not necessarily connected) $D_{a}\in\Dvee$ such that for all $a'\in a\cdot G$,
    		there exists an automorphism $\phi\in\Aut_{e}(H)$ with $\phi(a)=a'$ if and only if $R_{D_{a}}(a,a')$.
    		As a result, $\hat\acl(e)=\hat\dcl(e)$.
    		\end{lemma}
    	
    	Here, $\hat\acl$ and $\hat\dcl$ denote the algebraic and definable closure in $\hat N$.
    	
    	\begin{proof}[Proof of Lemma~\ref{lem:acl=dcl imaginary}]
    		Let $H\in\ti(e)$ and $\mathfrak{A}=\Aut_{e}(H)$.
    		Using Lemma~\ref{l: wc and acl}, it is easy to show that $f\in\hat{N}$ can only be algebraic over $e$ if $f\cap H\neq\emptyset$.
    		Moreover, it follows from the incompressibility of $H$ and the fact that homomorphic retractions preserve equivalence classes,
    		that if $f_{1},f_{2}\in \widehat{\acl}(e)$ are in the same orbit over $e$ then there exists an automorphism $\phi\in \mathfrak{A}$
    		and points $a_{i}\in f_{i}\cap H$ such that $\phi(a_{1})=a_{2}$.
    		
    		We claim that the equality $\hat\acl(e)=\hat\dcl(e)$ follows from the first statement of the lemma.
    		Indeed, take $[a]_{E}\in\widehat{\acl}(e)$. We have that any element in the $\mathfrak A$-orbit of
    		$[a]_{E}$ is of the form $[a\cdot g]_{E}$, with $D_a$ and $g\in G[D_{a}]$ chosen appropriately.
    		If $D_{a}\subseteq E$, then $[a\cdot g]_E=[a]_E$ and $[a]_{E}\in\hat{\dcl}(e)$.
    		If $D_{a}\nsubseteq E$, then the collection of cosets $|G[D_{a}]/G[E\cap D_{a}]|$ is infinite.
    		Since $[a\cdot g]_{E}=[a\cdot h]_{E}$ if and only if $gG[E]=hG[E]$, it follows
    		that the orbit of $[a]_{E}$ under $\mathfrak{A}$ and thus under $\Aut_{e}(\mathcal{N})$ is infinite,
    		and consequently that $[a]_{E}\nin\widehat{\acl}(e)$.
    		
    		It therefore suffices to prove just the first statement of the lemma.
    		We proceed by induction on the number of $G$-orbits comprising $H\in\ti(e)$.
    		If $H$ is a single orbit then we may assume $k=1$, in which case the result is clear.
    		Indeed, one can replace $e_{i}$ with some $e'_{i}$ that is interdefinable with it, and such that
    		\[e_{0}:=\bigcap_{1\leq i\leq k}e_{i}\neq\emptyset.\] Then $e_{0}\in\hat{N}$ is interdefinable with $e$.
    		
    		Assume now $H$ consists of more than one $G$--orbit. Recall that by Corollary \ref{cor:permutation}
    		the expression $\ti(e_{1},\dots, e_{k})$ does not depend on the order of the arguments. Therefore after permuting the $e_{i}$, we may assume that the number of $G$-orbits in
    		any member of $\ti(e_{r},\dots, e_{k})$ is strictly less than in $H$,  but the number of $G$--orbits in $H$ is equal to the number
    		of orbits in any \[H'\in\ti(e_{i},e_{r},\dots, e_{k})\quad \textrm{for}\quad  1\leq i\leq r-1,\] i.e.~$H\in\ti(e_{i},e_{r},\dots, e_{k})$.
    		
    		For $1\leq i\leq r-1$, let
    		\[H_{0}^{i}\in\ti(e_{r},\dots, e_{k}),\quad  H_{0}^{i}\subseteq H,\quad a_{i}\in e_{i}\cap H\] be such that
    		$\Or(\pa{a_{i},H_{0}^{i}})$ is minimal among all the choices of
    		$a_{i}$ and $H_{0}$. Note that this value is the same as the one we would obtain minimizing over all $a_{i}\in e_{i}$ and $H_{0}\in\ti(e_{r},\dots, e_{k})$. Using Lemma \ref{lem: incompressibility}, one readily sees that we may assume
    		$H\in\mathcal{H}(q_{i},H_{0}^{i})$, where $q_{i}$ is a sequence from $a_{i}$ to $H_{0}^{i}$.
    		
    		For all $1\leq i\leq r-1$, let \[\alpha_{i}=\Or(\pa{a_{i},H_{0}^{i}}),\quad \beta_{i}=\Or(\pa{a_{i},H_{0}^{1}}).\] For $2\leq i\leq r-1$,
    		Lemma \ref{lem-convex-extensions} applied to $H\in\mathcal{H}(q_{1},H_{0}^{1})$  implies that $\pa{a_{i},H_{0}^{1}}$ is a
    		terminal subword of some representative of $\pa{a_{1},H_{0}^{1}}$, so that $\alpha_{1}\geq\beta_{i}$.
    		A symmetric argument provides $\alpha_{i}\geq\beta_{1}$. By definition, we must have $\beta_{i}\geq\alpha_{i}$ for
    		$1\leq i\leq r-1$. Therefore, for all $1\leq i\leq r-1$, we have  $\alpha_{i}=\beta_{i}=\alpha_{1}$. Consequently, there exist
    		minimizing sequences $p_{i}$ from $a_{i}$ to any basepoint in $H_{0}:=H_{0}^{1}$, and $H\in\mathcal{H}(p_{i},H_{0})$.

    		It follows from the above discussion that we can apply Lemma \ref{lem: refined convex extensions} to $H_{0}$ and $a_{1},\dots, a_{r-1}$. Up to replacing $e_{i}$ with an interdefinable class, we may assume that $a'_{i}=a_{i}$ in the conclusion of Lemma
    		\ref{lem: refined convex extensions}. For each $1\leq i\leq r-1$, let $b_{0}^{i}\in H_{0}$ be the resulting basepoint.
    		Recall that according to Lemma \ref{lem: refined convex extensions}, we have $\pa{a_{i},b_{0}^{i}}=[u_{|}^{\max}]$ for some
    		$u_{|}^{\max}\in\D^{*}$ strongly orthogonal to $\pa{a_i,a_j}$ for any distinct $1\leq i,j\leq r-1$.
    		
    		It follows by a straightforward argument, using the minimality of the choice of $p_i$, that
    		no component of the left end of $u_{|}^{\max}$ can be contained in $D_{i}$ for $1\leq i\leq r-1$.
    		Let $e'=(e_{\ell})_{\ell=r}^{k}$. For each $1\leq i\leq r-1$, let $e^{0}_{i}=[b_{0}^{i}]_{D'_{i}}$, where here
    		$D'_{i}=RA(D_{i}u_{|}^{\max})$, and let $\mathfrak{A}_{0}=\Aut_{(e',e^{0})}(H_{0})$.
    		The inductive hypothesis applies to $H_{0}$ and $(e',e^{0})$, and so we have that for all
    		$a\in H_{0}$, there exists some $E_{a}\in\Dvee$ such that the orbit of $a$ under $\mathfrak{A}_{0}$ is precisely
    		$[a]_{E_{a}}\cap a\cdot G$. We will see that the assignment $a\mapsto E_a$
    		can be extended to any point $a\in H$, so as to satisfy the claim of the lemma.
    		
    		Fix an element $c\in H$. The last item in Lemma \ref{lem: refined convex extensions} implies the existence of
    		\[c_{1}\in c\cdot G,\quad b_{1}\in H_{0},\quad \{c_{0}^{i}\}_{1\leq i\leq r-1}\]
    		such that for all $1\leq i,j\leq r-1$ we have:
    		\newcommand{\uvu}[0]{u_{|}^{*}}
    		\newcommand{\uvd}[0]{u_{|}}
    		\begin{itemize}
    			\item $[u_{|}]:=\pa{c^{i}_{0},b_{0}^{i}}=\pa{c_{1},b_{1}}$;
    			\item $\pa{b_{0}^{i},b_{1}}=\pa{c_{0}^{i},c_{1}}\ni u_{-}^{i}$;
    			\item $[u_{|}]\perp^* u_-^i$;
    			\item $\pa{a_{i},c^{i}_{0}}=\pa{a_{j},c^{j}_{0}}:=[u_{|}^{*}]$;
    			\item $[u_{|}^{\max}]=\pa{a_{i},b_{0}^{i}}=[u_{|}^{*}u_{|}]$;
    			\item 	$\pa{a_{i},c}=[u_{|}^{*}]*[u_{-}^{i}]$.
    		\end{itemize}
    		Let \[E^{*}=\bigcap_{i=1}^{r-1}RA(D_{i}u_{|}^{*})=RA\left(\left(\bigcap_{i=1}^{r-1}D_{i}\right)u_{|}^{*}\right).\]
    		Consider the domain furnished by the assignment $b_1\mapsto E_{b_1}$.
    		The domain $E_{b_{1}}$ contains the two orthogonal subdomains
    		\[
    		E_{b_{1}}^{-}:=E_{b_{1}}\cap u_{|}^{\perp},\quad
    		E_{b_{1}}^{|}:=E_{b_{1}}\cap RA(u_{|}).
    		\]
    		We do not need these subdomains to exhaust $E_{b_1}$; only the indicated support conditions will be used below.
    		In particular, note that $E^{-}_{b_{1}}\perp LA(u_{|})$.
    		
    		As $E_{c}$, we take \[E^{*}\cap (E_{b_{1}}^{-}\vee LA(u_{|})),\]
    		and we claim that the resulting domain satisfies the properties required by the lemma.
    		
    		{\bf ``If'' direction:}
    		We first assert that $c':=c\cdot g\in\mathfrak{A}\cdot c$ for any $g\in G[E_{c}]$.
    		Writing $g=g_{|}g_{-}$, where
    		\[g_{|}\in G_{LA(u_{|})\cap E^{*}},\quad g_{-}\in G_{E_{b_{1}}^{-}\cap E^{*}},\]
    		it suffices to show that $\tp(c/e)=\tp(c'/e)$. Quantifier elimination reduces this to $\pa{c,e_{i}}=\pa{c',e_{i}}$ for $1\leq i\leq m$.
    		
    		Choose an element $\sigma\in\mathfrak{A}_{0}$ such that $\sigma(b_{1})=b_{1}\cdot g_{-}$.
    		Since $b_{1}\cdot g_{-}$ is a basepoint for $c'$ in $H_{0}$ and since $g_{|}$ is absorbed by $u_{|}$,
    		we have that for each $r\leq i\leq m$:
    			\begin{align*}
    				\pa{c',e_{i}}
    				&=\pa{c',\sigma(b_{1})}*\pa{\sigma(b_{1}),e_{i}} \\
    				&=\pa{c',\sigma(b_{1})}*\pa{b_{1},e_{i}} \\
    				&=g_{-}^{-1}u_{|}g_{-}*\pa{b_{1},e_{i}} \\
    				&=u_{|}\pa{b_{1},e_{i}}=\pa{c,e_{i}}.
    			\end{align*}
    		We briefly clarify the meaning of this last series of displayed equations, since expressions involving sequences between a point and
    		a class contained in a weakly convex set have not been specifically spelled out. We choose $d_i\in e_i$ and carry out the sequence
    		of equalities with $d_i$ in place of $e_i$, and the displayed equations are a succinct shorthand for such a choice.
    		Since $e_i$ is $\sigma$-invariant, we have that $\sigma(d_i)$ lies in the same class
    		as $e_i$, and so a sequence from a point to $\sigma(d_i)$ is still a sequence from that point to the class $e_i$.
    		
    		For $1\leq i\leq r-1$ write $u^{i}_{-}\simeq \ell_{i}v^{i}_{-}$, where $\ell_{i}\in G$ and $v^{i}_{-}\in\D^{*}$.
    		
    		\begin{claim}\label{claim-12}
    			Let $i$ be fixed. There is a decomposition $g_{-}=g_{-}^{1}g_{-}^{2}$, where
    			\[g_{-}^{1}\in G[(v_{-}^{i})^{\perp}\cap u_{|}^{\perp}\cap \ell_{i}^{-1}(E_{b_{0}^{i}})],\quad g_{-}^{2}\in G[RA(u^{i}_{-})].\]
    		\end{claim}
    		\begin{subproof}[Proof of Claim~\ref{claim-12}]
    			Notice that $\sigma(b_{0}^{i})=b_{0}^{i}\cdot k$ for a suitable element $k\in G[E_{b_{0}^{i}}]$. It follows that
    			$$
    			[\ell_{i}v^{i}_{-}]=[u^{i}_{-}]=\pa{b_{0}^{i},b_{1}}=[k\pa{\sigma(b_{0}^{i}),b_{1}\cdot g_{-}}g_{-}^{-1}]=[ku_{-}^{i}g_{-}^{-1}]=
    			[\ell_{i}k^{\ell_{i}}v^{i}_{-}g_{-}^{-1}]
    			$$
    			Lemma \ref{l: element absorption} yields a decomposition of $g_{-}$ into a component supported on $RA(u_{-}^{i})$, and a
    			component supported in
    			$(v_{-}^{i})^{\perp}\cap \ell_{i}^{-1}(E_{b_{0}^{i}})$. Note that since $g_{-}^{2}$ is orthogonal to $u_{|}$, so must $g_{-}^{1}$.
    		\end{subproof}
    		
    		For a word $u\in\wo$, write $u^{\perp}$ for the intersection of the orthogonal complements of the letters occurring
    		in $u$. Note that $u_{|}^{\perp}\cap E_{b_{0}^{i}}\subseteq E^{*}$. In particular
    		$(g^{1}_{-})^{\ell_{i}}\subseteq E^{*}$.
    		Using Claim~\ref{claim-12}, we get that for $1\leq i\leq r-1$,
    		$$
    		\pa{e_{i},c'}\simeq D_{i}u_{|}^{*}*u_{-}^{i}g_{|}g_{-}\simeq(D_{i}u_{|}^{*}g_{|}(g^{1}_{-})^{\ell_{i}})*(u_{-}^{i}g^{2}_{-})
    		\simeq(D_{i}u_{|}^{*})*u^{i}_{-}\simeq\pa{e_{i},c}.
    		$$
    		
    		{\bf ``Only if'' direction:}
    		It remains to show the opposite inclusion. Fix an automorphism $\phi\in\mathfrak{A}$.
    		We know by Lemma \ref{lem: incompressibility} that $\phi$ restricts to an orbit-preserving automorphism of
    		$H_{0}$. For each $c\in H$ we will let $g_{c}$ stand for the unique group element such that $\phi(c)=c\cdot g_{c}$.
    		
    		Let us begin by showing that $\phi_{\restriction H_{0}}\in\mathfrak{A}_{0}$, so that $\phi$ preserves $e'_{i}$ for $1\leq i\leq r-1$.
    		For $1\leq i\leq r-1$, we have
    		$$
    		u_{|}^{\max}\simeq g_{a_{i}}u_{|}^{\max}g_{b_{0}^{i}}^{-1}.
    		$$
    		%  		$$\pa{a_{i},b_{0}^{i}}=g_{a_{i}}\pa{\phi(a_{i}),\phi(b_{0}^{i})}g_{b_{0}^{i}}^{-1}=g_{a_{i}}\pa{a_{i},b_{0}^{i}}g_{b_{0}^{i}}^{-1}$$
    		It follows from Lemma \ref{l: element absorption} that $g_{b_{0}^{i}}$ decomposes as $h_{1}h_{2}$, where
    		here \[h_{1}\in G[{D_{i}}]\cap (u_{|}^{\max})^{\perp},\quad h_{2}\in RA(u_{|}^{\max}).\] Equivalently,
    		\[g_{b_{0}^{i}}\in G[{RA(D_{i}u_{|}^{\max})}]=G[{D'_{i}}].\] Hence, $\phi_{\restriction H_{0}}\in\mathfrak{A}_{0}$ and thus
    		$g_{d}\in G[{D_{d}}]$, for all $d\in H_{0}$.
    		
    		Similarly,
    		\[g_{c_{0}^{i}}\in G[{RA(D_{i}u_{|}^{*})}], \quad g_{c_{0}^{i}}\in G[{LA(u_{|}E_{b_{0}^{i}})}]\] for $1\leq i\leq r-1$.
    		By straightforward applications of Lemma~\ref{l: element absorption} we have:
    		\begin{itemize}
    			\item $E_{b_{1}}\subseteq RA(u_{-}^{i})\vee ((v_{-}^{i})^{\perp}\cap \ell_{i}^{-1}(E_{b_{0}^{i}}))$ for any $1\leq i\leq r-1$;
    			\item The element $g_{c}$ decomposes into a component $g_{-}$ that is orthogonal to $u_{|}$, and a component
    			$g_{|}$ that is supported in the join of the letters appearing in $u_{|}$;
    			\item For each $1\leq i\leq r-1$, the element $g_{c}$ decomposes into a component in $RA(u_{-}^{i})$ and a component in
    			$(v_{-}^{i})^{\perp}\cap \ell_{i}^{-1}(E_{c_{0}^{i}})$, which implies that
    			$g_{|}$ is supported in $E^{*}\cap LA(u_{|})$;
    			\item The element $g_{c}$ decomposes into a component supported in $LA(u_{|})$ and a component supported in
    			$u_{|}^{\perp}\cap E_{b_{1}}$, which implies that $g_{-}$ is supported in
    			$(E_{b_{1}}\cap u_{|}^{\perp})=E^{-}_{b_{1}}$.
    		\end{itemize}
    		
    		This proves that  $g_c$ is supported on $E_c$, completing the proof of the lemma.
    	\end{proof}
    	
  	\subsection{Using interpretations to build injective homomorphisms between mapping class groups}

    	\renewcommand{\N}[0]{\mathcal{N}}
    	
    	Recall that if $\Gamma$ is an undirected graph, the \emph{clique number} of $\Gamma$ is the
    	maximum size of a complete subgraph of $\Gamma$.
    	For the curve graph of a surface of genus $g$ with $n$ punctures, it is a standard fact that the clique number is $3g-3+n$.

\begin{lemma}
    \label{l: rare cancellation}
    Let $\Sigma$ be a connected surface, and let $\chi$ be the clique number of
    $\C=\C(\Sigma)$. Let $D\in\Dvee$, let $w\in\wo$ be reduced, and let
    \[
        a,b_0,b_1,\ldots,b_m\in N
    \]
    satisfy $R_D^*(b_i,b_j)$ for $i\neq j$ and $R_w^*(a,b_0)$. Here, if
    $D=\bigvee_{\ell=1}^k D_\ell$ is disconnected, $R_D^*$ is understood
    componentwise, with respect to the connected components $D_\ell$.

    Then there are at most $\chi$ indices $j\in\{1,\ldots,m\}$ for which the
    reduction of $wD$ to a representative of $\pa{a,b_j}$ uses a cancellation
    involving a component of $D$. Equivalently, if $m>\chi$, then for some
    $j\in\{1,\ldots,m\}$ the class $\pa{a,b_j}$ is obtained from $wD$ by
    reduction without cancellation.
\end{lemma}

\begin{proof}
    Let $D_1,\ldots,D_k$ denote the connected components of $D$. Clearly
    $k\leq\chi$. Suppose that, for some index $j$, the reduction of $wD$ to
    $\pa{a,b_j}$ uses a cancellation involving the component $D_i$ of $D$.
    We show that this can happen for at most one value of $j$ for each fixed
    $i$.

    If $D_i$ is not a component of the right end of $w$, then such a
    cancellation can never occur. Thus write $w\simeq uD_i$. If the final
    $D_i$ cancels in the reduction to $\pa{a,b_j}$, then
    \[
        \pa{a,b_j}=[u'v],
    \]
    where $v\in\wo(D_i)$ and $u'$ is a reduction of $u\hat D$, with
    \[
        \hat D=\bigvee_{\ell\neq i}D_\ell .
    \]
    Since $w$ is reduced, the component $D_i$ cannot cancel in any reduction
    of $u'vD_i$: such a cancellation would give, after reversing the preceding
    reduction, a cancellation of the displayed terminal copy of $D_i$ in the
    original word $uD_i=w$.

    Now let $j'\neq j$. Since $R_D^*(b_j,b_{j'})$ holds, a representative of
    $\pa{a,b_{j'}}$ is obtained by reducing $\pa{a,b_j}D$. In particular, the
    component $D_i$ enters through a terminal copy appended to $u'v$. By the
    preceding paragraph this copy of $D_i$ cannot cancel. Hence $D_i$ cannot be
    the cancelling component for both $j$ and $j'$. Therefore each component of
    $D$ accounts for at most one bad index, and there are at most
    $k\leq\chi$ bad indices in total.
\end{proof}

    	There is a natural action of $G$ on $M\times M$, given by the diagonal. For fixed $g,h\in G$, write $\lambda=\lambda_{g,h}$ for
    	the action of $\bZ$ on $M\times M$ given by \[\lambda(k)(a,b)=(a\cdot g^{k},b\cdot h^{k}).\]
    	The following technical result can be thought of as saying that, outside a lower-rank exceptional set, an orbit for the action defined by $\lambda$ cannot visit a fixed proper relation infinitely often.
    	
    	\begin{lemma}
    		\label{l: orbits avoid subvarieties}
    		Let $v\precneq w$ be reduced words, and let $g,h\in G$ be fixed.
    		There exists a finite fragment $\phi(x,y)$ of $R^{*}_{w}(x,y)$ containing
    		$\neg R_{v}(x,y)$ and an integer $K$ such that if $M\models \phi(a,b)$, then
    		\[|\{k\in\mathbb{Z}\mid M\models R_v(\lambda(k)(a,b))\}|<K.\]
    	\end{lemma}
    	\begin{proof}
    		Let \[D=\supp(g),\quad E=\supp(h).\] These supports are regarded as elements of $\Dvee$. Let $\chi$ be the clique number of $\C$ and let $n=2\chi+1$.
    		If the formula $\phi$ and integer $K$ claimed by the lemma do not exist, then the following type is consistent:
    		$$
    		q(x,y)=R^{*}_{w}(x_{0},y_{0})\cup\bigcup_{0\leq i<j\leq n}R^{*}_{D}(x_{i},x_{j})\cup
    		\bigcup_{0\leq i<j\leq n}R^{*}_{E}(y_{i},y_{j})\cup\{R_{v}(x_{i},y_{i})\mid 1\leq i\leq n\},
    		$$
    		as follows from Corollary~\ref{c: consistency} and Lemma~\ref{l: displacement}.
    		Let $(a_{i},b_{i})_{i=0}^{n}$ be a realization of $q$. Then $\pa{a_{0},b_{0}}=w$.
    		For each $1\leq i\leq n$ the type $\pa{a_{i},b_{i}}$ is a reduction of
    		$DwE$. From Lemma \ref{l: rare cancellation}, we deduce the existence of an index $1\leq i_{0}\leq n$ such that
    		$\pa{a_{0},b_{i_{0}}}$ and $\pa{a_{i_{0}},b_{0}}$ are reducts without cancellation of $wE$ and $Dw$ respectively.
    		This implies that $w\preceq\pa{a_{i_{0}},b_{i_{0}}}$, contradicting the fact that $R_{v}(a_{i_{0}},b_{i_{0}})$.
    	\end{proof}

    	\newcommand{\G}[0]{G_{1}}
    	\newcommand{\Gg}[0]{G_{2}}

    	We let $\Sigma_{1},\Sigma_{2}$ be finite type orientable surfaces,
    	and write $G_1=\Mod^{\pm}(\Sigma_1)$ and $G_2$ for the pure mapping class group of $\Sigma_2$. For $i\in \{1,2\}$, let $\D_i$ be
    	a downward closed, $G_i$--invariant collection of connected domains, and let $\M_i=\M_{\D_i}^{G_i}$. Thus the source group may be the full extended mapping class group, while the target group is taken pure so that the support and canonical-base arguments from Section~\ref{sec:full-qe} apply.
    	
    	We suppose now that there is an interpretation with parameters $\bar{\zeta}=(\zeta,X,\mathcal E)$ of $\M_{1}$ into $\M_{2}$,
    	where $X$ is a definable set of $\M_{2}$ in $m$ variables. We work in sufficiently saturated elementary extensions
    	$\N_{1}$ and $\N_{2}$ of $\M_{1}$ and $\M_{2}$, wherein $\bar{\zeta}$ extends to an interpretation of $\N_{1}$ into $\N_{2}$, and where
    	the parameters remain in the standard models.

        For convenience, we will rather work in $\mathrm{Th}(\hat{\M}_{2})$, which we know to have weak elimination of imaginaries by Corollary \ref{c: weak elimination of imaginaries}. Thus, take $X$ to be of the form
        \[X=(X_{1}\sqcup X_{2}\sqcup\dots \sqcup X_{k}),\]
    	where each $X_{j}$ is a definable set in a Cartesian
    	product of sorts in $\mathrm{Th}(\hat{\M}_{2})$. 
        Weak elimination of imaginaries, together with 
         \cite[Lemme 4]{poizat1983theorie}
        allows us to assume that $\mathcal{E}$ is a $\emptyset$-definable equivalence relation with finite equivalence classes. In fact, we can assume that each $\mathcal{E}$-class is contained in one of the components of the disjoint union above. The latter follows easily from the observation that given a
    	definable equivalence relation $\mathcal{E}$ on a definable set $Z$ that admits a partition into definable sets
    	$Z=Z_{1}\coprod Z_{2}$, the equivalence relation whose classes are of the form $[c]_{\mathcal{E}}\cap Z_{i}$
    	for $i=1,2$ is again definable.

    Note that the disjoint union can be canonically encoded as a set of imaginaries,
	though the encoding may involve the addition of either real constant or components which are imaginaries defined by equivalence relations with infinite classes, though this does not affect our argument below. 
    
    	Given $a\in X$, we write $[a]_{\mathcal E}$, or simply $[a]$ when no confusion can arise, for the class of $a$ with respect to the equivalence relation $\mathcal E$.
    	Given $D\in\D_{1}$ we let \[R_{\hat D}=\zeta[R_{D}]\subseteq X\times X\] be the interpretation of $R_D$.
    	We will occasionally abbreviate $R_{\hat D}(a,a')$ by $a\sim_{\hat{D}}a'$.

    	Let $\mathcal{P}=\{p_{1},p_{2},\dots, p_{r}\}$ denote the complete types of maximal Morley rank containing $X$.
    	By further subdividing the components $X_{i}$, we may assume that each $X_{i}$ contains at most one of the types $p_{k}$
    	and that the $\mathcal{E}$-equivalence class of a tuple in $p_{i}$ is contained in $p_{i}$;
    	this is only possible because the $\mathcal{E}$-classes were finite to begin.
    	We may also assume that for any $\mathcal{E}$-class $\epsilon$, distinct elements $e,e'\in\epsilon$ have the same type over $\epsilon$.
    	
    	From now on we will focus on the type $p=p_{1}$, which we assume contained in \[X_{1}\subseteq S_{D_{1}}\times\dots\times S_{D_{m}}=S_{\bar{D}},\] where $S_{D_i}$ is the sort consisting of $R_{D_i}$ equivalence
    	classes.
    	
    	\newcommand{\Fun}{\mathrm{Fun}}
    	\begin{definition}
    		Let $p,q$ be types in a first order structure. Let $\Fun(p,q)$ be the collection of definable functions
    		$f:X\to Y$ from definable sets containing $p$ (i.e.~where the defining formula is contained in $p$)
    		to definable sets containing $q$.
    		A \emph{germ} of definable functions between $p$ and $q$ is an equivalence class of functions in $\Fun(p,q)$,
    		where the equivalence relation is given as follows: $f:X\to Y$ and $f':X'\to Y'$ are equivalent if there exists a
    		definable subset of $X\cap X'$ containing $p$ on which the restrictions of $f$ and $f'$ coincide.
    		It is not difficult to see that a germ of definable functions between $p$ and $q$ determines a map from the collection of
    		realizations of $p$ to the collection of realizations of $q$.
    	\end{definition}

    	\begin{observation}
    		\label{o: germs from definability} Let $p$ and $q$ be complete types, together with realizations $c$ and $d$ of
    		$p$ and $q$ respectively, and suppose that $d\in \dcl(c)$. Then, there exists a unique germ of definable functions
    		from $p$ to $q$ that sends $c$ to $d$.
    	\end{observation}
    	
    	\begin{definition}
    		Let $e=(e_{j})_{j=1}^{k}$ be a tuple with $e_{i}=[a_{i}]_{D_{i}}$. We say that $e$ is \emph{fine} if for $1\leq i\leq k$ and any $E_{i}\in\D$ properly containing $D_{i}$ we have \[e_{i}\nin \dcl(e\setminus e_{i}\cup[a_{i}]_{E_{i}}).\]
    	\end{definition}
    	
    	\begin{lemma}
    		\label{lem: definable functions permute} Let $p(x_{1},\dots, x_{k})$ be a complete type of fine tuples in
    		$S_{D_{1}}\times\cdots\times S_{D_{k}}$ over a finite tuple $c$ of parameters, and let $\mathfrak{G}$ be a group of germs
    		of definable functions from $p$ to itself. Then there exists a finite index subgroup
    		$\mathfrak{G}_{0}\leq\mathfrak{G}$ and homomorphisms \[\lambda_{i}:\mathfrak{G}_{0}\to G[D_{i}^{\perp}],\quad 1\leq i\leq k\]
    		such that for all $f\in\mathfrak{G}_{0}$, we have $f(e_{i})=[a_{i}\lambda_{i}(f)]_{D_{i}}$.
    	\end{lemma}
    	\begin{proof}
    		By subsuming the constants in the tuples into the underlying language, and redefining $p$ so that each $e\models p$ starts with some fixed enumeration of $c$, we may assume that $c=\emptyset$. 
            
            Fix an instance $e$ of $p$ and
    		$H\in\ti(e)$,\footnote{Note that this actually means $H\in\ti(e)$ in the original notation.} and let $f\in\mathfrak{G}$. Put $e':=f(e)$. Since $\tp(e)=\tp(e')$, we can fix an automorphism $\sigma_{f}$ of
    		$\hat{\mathcal{N}}$ that sends $e$ to $e'$. Let $\pi$ be a retraction of $\mathcal{N}$ onto
    		$H$, and let $\rho_{f}:=(\pi\sigma_{f})_{\restriction H}$.
    		
    		Since $e'\subseteq \dcl(e)$, each component of $e'$ must intersect $H$; therefore, $\pi$ must fix every component of $e'$.
    		Clearly, we have \[H':=\sigma_{f}(H)\in\ti(e').\] By incompressibility of $H$ over $e$ (see Lemma \ref{lem: incompressibility})
    		and the finiteness of the number of $G$-orbits in $H$, we obtain that $\rho_{f}$ is an automorphism of $H$. Applying
    		incompressibility again, we see that $\pi$ must restrict to an isomorphic embedding of $H'$ into $H$.
    		The sets $H$ and $H'$ must contain the same number of $G$-orbits, and so $\rho_{f}$ is onto and $H\in\ti(e')$.
    		
    		Now, for any other $f'\in\mathfrak{G}$,
    		since $f'$ is a germ of definable functions and therefore invariant, we have
    		\[\rho_{f}(f'(e))=\sigma_{f}(f'(e))=f'(f(e)).\] 
        
            By the second part of Lemma \ref{lem: incompressibility}, $\rho_{f}\circ\rho_{f'}$ and
    		$\rho_{f'f}$ must permute the orbits in $H$ in the same way, since they are two automorphisms of $H$ mapping $e$ to the same image. In other words $\rho_{-}$ determines an action of $\mathfrak{G}$ on the right on the set of orbits of $H$.
            It follows that there exists a finite index subgroup
    		$\mathfrak{G}_{0}\leq\mathfrak{G}$ such that for any $f\in\mathfrak{G}_{0}$, the map $\rho_{f}$ preserves each $G$-orbit in $H$.
    		In particular, for any given $f\in\mathfrak{G}_{0}$ and $1\leq i\leq k$, there exist elements
    		$h_{i}\in G$ such that $\rho_{f}(a_{i})=a_{i}h_{i}$. We claim that we may take $h_{i}\in G[D_{i}^{\perp}\vee D_{i}]$.
    		
    		First of all, recall that by Lemma~\ref{lem:acl=dcl imaginary},
    		for all
    		$a\in H$ there exists a domain $D_{a}\in\Dvee$ such that
    		the orbit of $a$ under the group of automorphisms $\Aut_e(H)$ coincides with the orbit $aG[D_{a}]$.
    		Moreover, $[a]_{E}\subseteq \dcl(e)$ for all $E\subseteq D_{a}$.
    		It follows that $D_{a_{i}}\subseteq D_{i}$. The fact that we chose $e$ to be fine implies that $D_{a_{i}}= D_{i}$.
    		
    		Now, it is easy to see from the definitions that for any $h\in G$
    		and $a\in H$, we have $D_{ah}=h(D_{a})$. Since $[a_{i}h_{i}]_{D_{i}}\in \dcl(e)$, necessarily
    		\[D_{i}\subseteq D_{a_{i}h_{i}}=h_{i}(D_{a_{i}})=h_{i}(D_{i}).\] It follows that $h_{i}(D_{i})=D_{i}$, and
    		therefore $h_{i}\in G[D_{i}\vee D_{i}^{\perp}]$. 

            Since $G[D_{i}]$ acts trivially on $\N_{D_{i}}$, the discussion above shows the existence of $h'_{i}\in G[D_{i}^{\perp}]$ such that 
            \[f(e)=f(([a_{i}]_{D_{i}})_{1\leq i\leq k})=([a_{i}h_{i}']_{D_{i}})_{1\leq i\leq k}.\]
            Since $f$ is a germ of definable functions at $p$, it follows that  
            \[
            f(e')=([a_{i}h'_{i}]_{D_{i}})_{1\leq i\leq k}
            \]
            for any other $e'\models p$. The result follows, letting $\lambda_{i}(f)=h'_{i}$.
    	\end{proof}
    	
    	\begin{remark}
    		In general, we cannot take $\mathfrak G_0=\mathfrak G$.
    		Consider two orthogonal domains $D,E$ without common boundary components and let $p(x,y)$ be the real type stating that $\pa{x,y}=DE$. Then for any instance $(a_{1},a_{2})$ of $p$ there exists a unique instance $(a'_{1},a'_{2})$ of $p$ such that
    		\begin{align*}
    			R^{*}_{D}(a_{1},a'_{1}), && R^{*}_{E}(a'_{1},a_{2}),   && R^{*}_{E}(a_{1},a'_{2}), && R^{*}_{D}(a'_{2},a_{2})
    		\end{align*}
    		This gives a germ of definable functions $f$ whose domain is some definable set $X$ in $p$ and such that $f^{2}$ is the identity.
    	\end{remark}
    	
    	The following technical result is key to transform interpretations between curve graphs into
    	homomorphisms between mapping class groups.
    	
    	\begin{lemma}
    		\label{lem: nice action}
    		There exists a finite index subgroup $G'_{1}\leq G_{1}$ such that:
    		\begin{enumerate}
    			\item \label{nice-action-first-item}Each element
    			$g\in G'_{1}$ permutes the collection of $\mathcal{E}$-classes of realizations of $p_{1}$;
    			\item \label{nice-action-second-item}The action of $G_{1}'$ on the collection of $\mathcal{E}$-classes of realizations of $p_{1}$  lifts to an action of $G_{1}'$ on the collection of realizations of $p_{1}$;
    			\item \label{third item}For each $1\leq i\leq m$ and $g\in G'_{1}$, there is a unique element
    			$\mu_{i}(g)\in G_2[D_{i}^{\perp}]$ satisfying
    			\[(e_{i})_{i=1}^{m}\circ g=(e_{i}\cdot\mu_{i}(g))_{i=1}^{m}\] for any $(e_{i})_{i=1}^{m}\models p_1$;
    			\item \label{fourth item}The map
    			\[\mu=(\mu_{i})_{i=1}^{m}:\G'\longrightarrow
    			\bigtimes_{i=1}^{m}G_2[D_{i}^{\perp}]\] is a homomorphism, at least one of whose components
    			is injective.
    		\end{enumerate}
    	\end{lemma}
    	\begin{proof}
    		As in Lemma~\ref{lem: definable functions permute}, we will assume all parameters have been subsumed by the language.
    		Since two tuples that are interdefinable must have the same Morley rank, it follows that the action of $G_{1}$ on the interpreted universe
    		induces a permutation of $\mathcal{P}$. We set
    		$G_{1}^{0}\leq G_{1}$ to be the stabilizer of $p_1$, which clearly has finite index. The group
    		$G_1^0$ then acts on the set of realizations of $p_{1}$.
    		
    		Now, fix $g\in G^{0}_{1}$ and an $\mathcal E$--class $\delta$ represented by an instance $d\models p_1$. Let $\epsilon$ be the
    		$\mathcal E$--class represented by the image of $d$ under the action of $g$.
    		Choose $e\in\epsilon$. By Lemma \ref{lem:acl=dcl imaginary}, the elements $d$ and $e$ are interdefinable, since they are interalgebraic, as \[e\in \acl(\delta\cdot g)\subseteq \acl(\delta),\] and vice versa.
    		
    		Let $f_{d,e}$ be the unique invertible germ of definable functions from  $p_{1}$ to itself, sending $d$ to $e$; cf.~Observation
    		\ref{o: germs from definability}.
    		We claim that $f_{d,e}$ sends $\delta$ to $\epsilon$. Indeed, let $d'\in\delta$. Since $d$ and $d'$ have the same type and $\delta$
    		is defined from $d$,
    		we may conclude that
    		$\tp(d/\delta)=\tp(d'/\delta)$, which
    		implies the existence of an automorphism $\sigma$ of $\mathcal{N}$ such that
    		$\sigma(d)=d'$. By invariance of $f_{d,e}$ we have
    		\[f_{d,e}(d')=f_{d,e}(\sigma(d))=\sigma(f_{d,e}(d))=\sigma(e).\] However,
    		$\sigma$ must fix $\epsilon$ (it fixes $\delta$) and thus $\sigma(e)\in\epsilon$.
    		
    		By invariance, it follows that $f_{d,e}$ is an $\mathcal{E}$-compatible permutation of the class of realizations of $p_{1}$ sending each class to its image by the action of $g$. Let $\mathfrak{G}_{g}$ be the collection of all the germs of definable functions from $p_1$
    		to itself with that property. Clearly for $h,g\in G^{0}_{1}$ the composition of all germs in $\mathfrak{G}_{h}$ and all germs in $\mathfrak{G}_{g}$ is a germ in $\mathfrak{G}_{gh}$ and the inverse of a germ in $\mathfrak{G}_{g}$ is a germ in $\mathfrak{G}_{g^{-1}}$. Hence, the following is a group:
    		\[\mathfrak{G}=\bigcup_{g\in G^{0}_{1}}\mathfrak{G}_{g}~.\] 
    		
    		By Lemma \ref{lem: definable functions permute}, there exists some finite index subgroup $\mathfrak{G}^{0}$ of $\mathfrak{G}$ such that any $f\in\mathfrak{G}^{0}$ acts on each coordinate by right multiplication by an element of $G_{2}$.
    		
    		The natural projection which sends a germ in $\mathfrak G_g$ to $g$ maps $\mathfrak{G}^{0}$ homomorphically onto a finite index subgroup
    		$G'_{1}\leq G^{0}_{1}$. We claim that this map is in fact injective. Otherwise, pick a non-identity
    		$f\in\mathfrak{G}^0_{1_G}$ in the kernel. For any
    		$e=(e_{i})_{i=1}^{m}$ satisfying $p_{1}$, we have \[f(e)=(e_{i}h_{i})_{i=1}^{m},\quad h_{i}\in G_2[D_{i}^{\perp}]\setminus
    		G_2[\partial D_i],\]
    		and $h_{i}\neq 1$ for at least one value of $i$. This contradicts the fact that $G_{2}[D_i^{\perp}]/G_2[\partial D_i]$
    		is torsion-free,
    		the fact that $f$ preserves $[e]_{\mathcal{E}}$, and the fact that $[e]_{\mathcal{E}}$ is finite. The fact that
    		$G_{2}[D_i^{\perp}]/G_2[\partial D_i]$ is torsion-free makes critical use of the fact that $G_2$ consists of pure mapping classes.
    		
    		This settles item (\ref{nice-action-second-item}), and furnishes maps \[\mu_{i}:G'_{1}\longrightarrow G_{2},\quad 1\leq i\leq m\]
    		as in (\ref{third item}). It follows immediately from the properties of the interpretation that the product $\mu$ of the maps $\{\mu_i\}_{i=1}^m$ is an injective homomorphism.
    		A standard argument then proves that $\mu_{i}$ must be injective for at least one value of $1\leq i\leq m$;
    		this is well known (see for instance \cite{ivanov-book}), and we provide a proof for the sake of completeness.
    		Argue by induction on $m$. If $\mu_1$ is injective, there is nothing to prove. Otherwise, let $K_1=\ker(\mu_1)$.
    		If the product map $\mu_2\times\cdots\times\mu_m$ is injective, the induction hypothesis applied to this product gives an injective component among $\mu_2,\ldots,\mu_m$. Thus suppose that
    		$K_2=\ker(\mu_2\times\cdots\times\mu_m)$ is also nontrivial. Since $K_1$ and $K_2$ are nontrivial normal subgroups of the finite index subgroup $G'_1$, each contains a pseudo-Anosov mapping class~\cite{ivanov-book}. Moreover,
    		$K_1\cap K_2=\{1\}$, because $\mu$ is injective, and every element of $K_1$ commutes with every element of $K_2$: the commutator maps trivially under every coordinate of $\mu$.
    		This contradicts the fact that the centralizer of a pseudo-Anosov mapping class is virtually cyclic~\cite{FLP}. It follows that
    		$\mu_i$ is injective for some $i$, and the lemma
    		thus follows.
    	\end{proof}

    	In light of Lemma~\ref{lem: nice action}, we will abuse notation and use the same symbol $\circ$ to denote the lifted action.
    	For any $\alpha\in\C(\Sigma_1)$ choose a generator $\tau_{\alpha}$ of the cyclic group $G_1'\cap\langle T_{\alpha}\rangle$, where $T_{\alpha}$ is the Dehn twist about $\alpha$. If $G_1'$ is chosen to be normal, we may assume that $\tau_{\alpha}$ is chosen $G_1$--equivariantly.
    	
    	\begin{corollary}\label{cor:induced-homo-body}
    		Suppose $\Sigma_1$ and $\Sigma_2$ are surfaces as above, and suppose that $\C(\Sigma_1)$ is interpretable in $\C(\Sigma_2)$.
    		Then, there is a finite index subgroup $G<\Mod^{\pm}(\Sigma_1)$ and an injective homomorphism
    		$G\longrightarrow \Mod^{\pm}(\Sigma_2)$.
    	\end{corollary}
    	
    	In fact, Corollary~\ref{cor:induced-homo-body} does not use anything specific to the curve graph, and only requires bi--interpretability
    	between the curve graph and the corresponding augmented Cayley graph. Thus, interpretations between suitable
    	pants graphs, arc graphs, nonseparating curve graphs, and flip graphs
    	also induce homomorphisms between the underlying mapping class groups, excluding sporadic surfaces.
    	
  	\subsection{Interpretation rigidity for curve graphs}
	In this subsection, all supports on the target side are computed inside the pure group $G_2$ fixed above.
    	We now show that in the induced homomorphism $G\longrightarrow \Mod^{\pm}(\Sigma_2)$, powers of Dehn twists in $G$ get mapped
    	to powers of Dehn twists in $\Mod^{\pm}(\Sigma_2)$. This will be the last fact needed to establish interpretation rigidity.
    	
    	\newcommand{\perptw}{\perp^{\mathrm{tw}}}
    	\begin{definition}
    		For $w\in\wo$ and $D\in\D$, we will say that $D$ is \emph{skew-orthogonal} to $w$, written as
    		$w\perptw D$ if there exists $gu\in[w]$ in left normal form such that $u\perp D$.
    		It is easy to check that the property is independent of the left normal form chosen.
    		We define $w(D)=g(D)$, which again is independent of the normal form chosen.
    		For $h\in G$, we say that $w\perptw h$ if $w\perptw \supp(h)$, and define $w(h)=h^{w^{-1}}$ similarly.
    		The definition does not depend on the choice of left normal form.
    	\end{definition}

    	For the following discussion it is worth making the constant $c\in \hat M_{2}$ used in the interpretation explicit.
        Let $x_i$ be a variable in the sort $S_{D_i}$. Write
    	$q(x_{1},\dots, x_{m})$  for the type consisting of
    	$\pa{x_{i},x_{j}}=[v_{i,j}]$ and $\pa{c,x_i}=[v_{i}]$.
    	Let $w_{i}$ and $w_{i,j}$ be the reducts (necessarily without cancellation) of $v_{i}D_i$ and $D_{i}v_{i,j}D_{j}$ respectively.
    	
    	We start with some general observations that will be useful in setting up the next lemma.
    	First, suppose $g\in G_1'$. The fact that the action of $G'_{1}$ leaves $p$ invariant implies that \[w_{i}\mu_{i}(g)\simeq w_{i},\quad
    	\mu_{i}(g)^{-1}w_{i,j}\mu_{j}(g)\simeq w_{i,j}\] for
    	distinct indices $1\leq i,j\leq m$. Let $\mathcal{S}_{j}$ denote the collection of connected components of
    	$\supp(\mu_{j}(g))$.  Lemma \ref{l: element absorption} implies that
    	$\supp(\mu_{i}(g))\subseteq RA(w_{i})$, and for each distinct $i,j$ there are partitions
    	\[\mathcal{S}_{i,j}^{\perp}\sqcup\mathcal{S}^{abs}_{i,j}=\mathcal{S}_{i}\] and
    	\[\mathcal{S}_{j,i}^{\perp}\sqcup\mathcal{S}^{abs}_{j,i}=\mathcal{S}_{j}\] such that
    	\begin{itemize}
    		\item For all $D\in\mathcal{S}^{abs}_{i,j}$, we have $D\subseteq LA(w_{i,j})$;
    		\item For all $D\in\mathcal{S}^{abs}_{j,i}$, we have
    		$D\subseteq RA(w_{i,j})$;
    		\item For all $D\in\mathcal{S}^{\perp}_{j,i}$, we have $w_{i,j}\perptw D$ and $w_{i,j}\nsubseteq RA(D)$;
    		\item $\mathcal{S}_{i,j}^{\perp}=\{w_{i,j}(D)\,:\,D\in\mathcal{S}^{\perp}_{j,i}\}$.
    	\end{itemize}

    	Consider the preimage $S$ of $R_{\widehat{\supp(g)}}$ in $X$, where here $\widehat{\supp(g)}$ means $\hat D$ for the
	domain $D=\supp(g)$. Tautologically, its restriction to the set of realizations of $p$ coincides with the  fiber partition of a germ $\bar{\psi}$ of definable maps from $p$ to some type $r$ in some imaginary sort in $\hat{\mathcal{N}}^{eq}$. By Corollary \ref{c: weak elimination of imaginaries} we may assume said sort is of the form
    \begin{equation*}\label{cartesian product}
    	S=(S_{D_{1}'}\times\cdots \times S_{D'_{n}})/\mathcal{F},
     \end{equation*}   
        where $\mathcal{F}$ has finite equivalence classes. Let $r_{0}$ be the lift of the type $r$ to the $S$ above.
    	Pick some realization $e$ of $p$ and choose some $f\in \bar{\psi}(e)$. By Lemma \ref{lem:acl=dcl imaginary}, we have
    	$f\in \dcl(e)$. If we denote by $\psi$ the corresponding germ of definable functions from $p$ to $\tp(f)$, then equivalence classes by $S_{\restriction p}$ are finite unions (of the same cardinality) of fibers of the map $\psi$.
   
    	Given a realization $e$ of $p$, each component \[\psi_{\ell}((e_{i})_{i=1}^{m})\in S_{F_{\ell}}\] is determined by the sequence types
    	$(\pa{e_{i},\psi_{\ell}((e_{i})_{i=1}^{m})})_{1\leq i\leq m}$, by Theorem \ref{t:M-fullqe}. Since $\psi(e)=\psi(e\circ g)$, we have that
    	\[\supp(\mu_{i}(g))\subseteq RA(F_{\ell}\pa{\psi_{\ell}((e_{i})_{i=1}^{m}),e_{i}}D_{i})\]
    	for $1\leq i\leq m$.
    	
    	\begin{lemma}\label{lem:dt-support}
    		Let $\alpha\in\C(\Sigma_1)$, and suppose that $\mu_{1}(\tau_{\alpha})\neq 1$. Then $\mu_{1}(\tau_{\alpha})$ is a
		 power of a Dehn twist about a curve $\phi(\alpha)\in \C(\Sigma_2)$.
    		Moreover, for each $2\leq j\leq m$, the set
    		$\mathcal{S}_{j,1}^{abs}$ is empty. In particular, the annular domain with core $\phi(\alpha)$ lies in $\mathcal{S}^{\perp}_{1,j}$, or else\[\mu_{j}(\tau_{\alpha})=1.\]
    	\end{lemma}
    	\begin{proof}
    		\newcommand{\g}[0]{\tau_{\alpha}}
    		
    		Take $\alpha\in\C(\Sigma_1)$, and suppose that either $\mu_{1}(\tau_{\alpha})$ is not a power of a Dehn twist,
    		i.e.~$F=\supp(\mu_{1}(\tau_{\alpha}))$ is not annular, or that there exists an index $i\geq 2$ such that some component
    		$F'$ of $\supp(\mu_{i}(\tau_{\alpha}))$ satisfies $F'\subseteq LA(w_{i,1})$.
    		
    		We will show that this hypothesis implies the existence of infinitely many different $\subg{\tau_{\alpha}}$--orbits in the interpreted universe $X/\mathcal E$ that are $\hat{\alpha}$--equivalent, contradicting the fact that $\bar{\zeta}$ is an interpretation of $\M_{1}$ in $\M_{2}$;
    		indeed, two group elements in $\M_1$ that are $\alpha$--related differ by a power of a Dehn twist, and so
    		there are only finitely many $\tau_{\alpha}$ orbits
    		that are $\alpha$--related in $\M_1$.
    		
    		Let $O(a)$ denote the orbit of $a\in X_1$ under the action by right multiplication of $\mu(\tau_{\alpha})$ on the Cartesian product $S_{D_{1}}\times S_{D_{2}}\times\cdots\times S_{D_{m}}$.  We will write
    		$$\bar{X}_{1}:=\bigwedge_{1\leq i< j\leq m} R_{v_{i,j}}(x_{i},x_{j})\wedge \bigwedge_{1\leq i\leq m} R_{v_{i}}(c,x_{i}).$$
    		\begin{claim*}\label{claim:12.6}
    			There are definable subsets $Z_{1}\subseteq Z_{0}\subseteq \bar{X}_{1}$ in $q$ such that:
    			\begin{enumerate}
    				\item \label{group action is generic}$e\circ \tau_{\alpha}=e\cdot\mu(\tau_{\alpha})$ for any $e\in Z_{0}$;
    				
    				\item \label{orbits are don't get swallowed} $|O(e)\setminus Z_{0}|$ is finite for any $e\in Z_{1}$.
    			\end{enumerate}
    			Moreover, we may assume that if $e=(e_i)_{i=1}^m\in Z_{1}$ with $e_i=[a_i]_{D_i}$, then the following conclusions hold:
    			\enum{(i)}{
    				\item \label{first case}\label{flexibility with the first coordinate} If $F$ is not annular, then
    				there exists an infinite sequence $(e^{\ell}_{1})_{1\leq \ell<\infty}$ of classes
    				that lie in different orbits under the action of
    				$\mu_{1}(\tau_{\alpha})$, such that for any $\ell\geq 1$ we can extend $e^{\ell}_{1}$ to a tuple
    				$e^{\ell}=(e^{\ell}_{j})_{j=1}^{m}$ satisfying
    				$\psi(e)=\psi(e^{\ell})$.
    				\item \label{second case} If $2\leq i\leq m$ and $\mu_i(\tau_\alpha)$ is not
supported on $\bigvee\mathcal{S}^{\perp}_{1,i}$, then there exist infinitely many
choices $\{e_i^\ell\}_{\ell\in\mathbb N}$ such that $e_1$ and $e_i^\ell$ are the
first and $i^{th}$ coordinate respectively of a tuple
    				$e^{\ell}\in Z_1$ satisfying $\psi(e)=\psi(e^{\ell})$.
    			}
    		\end{claim*}
    		\begin{subproof}[Proof of the Claim]
    			By compactness and the definition of \(\mu\), there is a finite fragment
\(Z_0\) of \(p\) on which Item~\ref{group action is generic} holds.  Now, for any finite fragment of $p$ defining $Z_0$, we have that $\bar{X}_{1}\setminus Z_{0}$
    			is contained in a union of definable sets of the form \[R_{v'_{i,j}}(x_{i},x_{j})\,\, \textrm{for}\,\, 1\leq i<j\leq m,\,\,
    			v'_{i,j}\precneq v_{i,j},\] or of the form \[R_{v'_{i}}(c,x_{i})\,\, \textrm{for}\,\,
    			v'_{i}\precneq v_{i}.\] Lemma \ref{l: orbits avoid subvarieties}, yields a finite fragment $\phi$ of $p$ containing
    			\[\{R_{v_{i,j}}(x_{i},x_{j})\}_{1\leq i< j\leq m},\quad  \{R_{v_{i}}(c,x_{i})\}_{1\leq i\leq m},\]
    			such that for any tuple $e$ satisfying $\phi$, the orbit of $e$ by the action of $\mu(\tau_{\alpha})$
    			on the right has finite intersection with the set $\bar{X}_{1}\setminus Z_{0}$.
                We take $Z_1$ to be the definable subset of $\bar X_1$ cut out by this finite fragment.
                This proves
    			Item~\ref{orbits are don't get swallowed}.
                Below we shrink $Z_{1}$ further for additional properties.
                
    			Now, let $k$ be the number of components of $\psi$. For
    			$1\leq i\leq m$ and $1\leq j\leq k$, let $r_{i,j}(x,y)$ be the type expressing that
    			$\pa{x_{i},y_{j}}=\pa{e_{i},\psi_{j}(e)},$ where $y_{j}$ is of sort $F_{j}$ and $e$ is an instance of $p$.
    			By compactness, there is a finite approximation $p_{0}(x)\subset p(x)$ and a finite approximation \[r_{0}(x,y)
    			\subseteq\bigcup_{\substack{1\leq i\leq m\\1\leq j\leq k}}r_{i,j},\] such that for any tuple $e$ satisfying
    			$p_{0}$, we have $\psi(e)$ is characterized by $r_{0}(e,\psi(e))$.

    			To show part (\ref{first case}) of the moreover part, fix some $h_{1}\in G[F]$ that acts as a pseudo-Anosov
    			mapping class on all of the components of $F$. For $2\leq j\leq m$, let $h_{j}=v_{1,j}^{-1}(h_{1})$ if $F$
    			is contained in $\mathcal{S}_{j,1}^{\perp}$, and let $h_j$ be the identity otherwise. For each integer $\ell\geq 1$, set
    			$e^{\ell}=(e_jh_j^{\ell})_{j=1}^m.$
    			
    			After further shrinking $Z_{1}$, we may assume it contains only realizations of $p_{0}$.
    			Since $e\in Z_{1}$ by assumption, by an application of Lemma \ref{l: orbits avoid subvarieties}
    			analogous to that used to show item (\ref{group action is generic}), we conclude that $e^{\ell}\in Z_{1}$
    			for all but finitely many values of $\ell$ and in fact $\psi(e^{\ell})=\psi(e)$ for all but finitely many of those values.
    			This shows (\ref{first case}), since the cosets $\{h^{l}_{1}\langle\mu_{1}(\tau_{\alpha})\rangle\}_{l\in \Z}$ are all distinct.
    			
    			For Item (\ref{second case}), let $\beta$ be a curve in $\supp(\mu_{i}(\tau_{\alpha}))$ that is
    			not contained in a component of $\mathcal{S}_{1,i}^{\perp}$. Let $\tau_\beta$ be a Dehn twist about $\beta$. For $1\leq j\leq m$ define
    			$h_{j}$ by:
    			\begin{itemize}
    				\item $\tau_\beta$ if $j=i$;
    				\item the identity if $\beta$ is not contained in a component of $\mathcal{S}^{\perp}_{i,j}$;
    				\item $w_{i,j}^{-1}(h_{i})$ otherwise.
    			\end{itemize}
                In particular, $h_{1}=1$.
    			For $\ell\geq 1$ set $e^{\ell}=(e_jh_j^{\ell})_{j=1}^m$: the proof goes as in Case (\ref{first case}).
    		\end{subproof}

    		Let $e\in Z_{1}(M_{2})$. The claim, together with the supposition made at the start of the proof, implies the existence of
    		an infinite sequence \[\{e^{0}=e,e^{1},e^2\dots\} \subset Z_{1}(M_{2})\] such that \[O(e^{i})\cap O(e^{j})=\emptyset\] for
    		$i\neq j$ and such that $[e^i]_{\mathcal E}$ and $[e^j]_{\mathcal E}$ are $\hat\alpha$--related.
    		
    		Since $O(e^{i})\setminus Z_{0}$ is finite for all $i$, there is a representative
    		\[\tilde{e}^{i}\in O(e^{i})\subseteq S_{D_{1}}(M_{2})\times\cdots\times S_{D_{m}}(M_{2})\] such that
    		$\tilde{e}^{i}\mu(\tau_{\alpha}^{t})\in Z_{0}(M_{2})$ for any $t\geq 0$, and therefore
    		\[\{[\tilde{e}^{i}\mu(\tau_{\alpha}^{t})]_{\mathcal E}\,:\,t\geq 0\}=\{[\tilde{e}^{i}]_{\mathcal E}\circ \tau_{\alpha}^{t}\,:\,t\geq 0\}.\]
    		Since $\mathcal E$--fibers are finite and 
    		$\{[\tilde{e}^{i}]_{\mathcal E}\}_{i\in\mathbb{N}}$ is a set of representatives of $\infty$-many distinct orbits of the action of
    		$\subg{\tau_{\alpha}}$ on $X$, we deduce that there exist $\infty$-many distinct orbits of the action of
    		$\subg{\tau_{\alpha}}$ on $X/\mathcal E$. Thus, we have  $[\tilde{e}^{i}]\sim_{\hat{\alpha}}[\tilde{e}^{j}]$,
    		in contradiction with the fact that in $M_{1}$ any $\alpha$--class is the union of finitely many $\subg{\tau_{\alpha}}$ orbits.
    	\end{proof}

    	If $\Lambda_1$ and $\Lambda_2$ are graphs, then a map $\phi\colon\Lambda_1\to\Lambda_2$ is called a \emph{full embedding} if
    	$\phi$ is injective on vertices, and if $\phi$ preserves adjacency and non-adjacency in $\Lambda_1$. Maps between graphs that
    	preserve both adjacency and non-adjacency are also sometimes called \emph{superinjective}. Irmak~\cite{irmak-superinjective}
    	shows that a superinjective map of graphs $\C(\Sigma)\to\C(\Sigma)$, under the assumption that $\Sigma$ is not one of finitely
    	many sporadic surfaces, is automatically induced by a homeomorphism of $\Sigma$; cf.~\cite{Shackleton}.
    	It is a well-known fact that two non-trivial powers of Dehn twists about
essential closed curves \(\alpha\) and \(\beta\) commute if and only if
\(i(\alpha,\beta)=0\); see Fact~3.9 in~\cite{farb_margalit_11}. Thus
Lemma~\ref{lem:dt-support} associates to each curve
\(\alpha\in\C(\Sigma_1)\) a curve
\(\phi(\alpha)\in\C(\Sigma_2)\). Since the homomorphism \(\mu\) preserves
commutation and non-commutation of the corresponding twist powers, the map
\(\phi\) preserves both disjointness and non-disjointness. Hence
\(\phi:\C(\Sigma_1)\to\C(\Sigma_2)\) is superinjective, which gives the
following corollary and the main result of this section. 
    	
    	\begin{corollary}\label{cor:cc-mutual-interpret}
    		Suppose that for $i\in\{1,2\}$, we have that $\Sigma_i$ is not a sphere with six or fewer punctures, a
    		torus with two or fewer punctures, or a closed surface of genus two.
    		If $\C(\Sigma_{1})$ is interpretable in $\C(\Sigma_{2})$ then there is a full embedding $\C(\Sigma_{1})\to \C(\Sigma_{2})$.
    		If $\C(\Sigma_{2})$ is also interpretable in $\C(\Sigma_{1})$ then $\C(\Sigma_{1})\cong \C(\Sigma_{2})$. In particular,
    		$\Sigma_1$ and $\Sigma_2$ are homeomorphic.
    	\end{corollary}
    	\begin{proof}
    		The interpretation of $\C(\Sigma_1)$ into $\C(\Sigma_2)$ interprets $\M_{\D_1}^{G_1}$ in $\M_{\D_2}^{G_2}$, and we
    		obtain a homomorphism $\mu$ from Lemma~\ref{lem: nice action}. Lemma~\ref{lem:dt-support} shows that $\mu$
    		induces a superinjective map between $\C(\Sigma_1)$ and $\C(\Sigma_2)$.
    		
    		Since $\C(\Sigma_2)$ is also interpretable in $\C(\Sigma_1)$, we see that the complexities $k(\Sigma_1)$ and $k(\Sigma_2)$
    		have to be the same, since these are determined by the maximal number of pairwise non-isotopic, nonperipheral simple closed curves,
    		and since these are in turn equal to the maximal rank of a torsion-free abelian subgroup of the corresponding mapping class group,
    		by~\cite{BLM-duke}. It follows that $\Sigma_1\cong\Sigma_2$; see Theorem A of~\cite{erl-fan-2017}, for instance.
    	\end{proof}

    	\bibliography{literature}{}

@misc{zahra-katrin,
Author = {Zahra Mohammadi Khangheshlaghi and Katrin Tent},
Title = {On the model theory of the Free Factor Complex of rank 2},
Year = {2026},
Eprint = {arXiv:2603.17084},
}

@article {tent-jsl,
    AUTHOR = {Tent, Katrin},
     TITLE = {The free pseudospace is {$N$}-ample, but not {$(N+1)$}-ample},
   JOURNAL = {J. Symb. Log.},
  FJOURNAL = {The Journal of Symbolic Logic},
    VOLUME = {79},
      YEAR = {2014},
    NUMBER = {2},
     PAGES = {410--428},
      ISSN = {0022-4812,1943-5886},
   MRCLASS = {03C45 (51E24)},
  MRNUMBER = {3224974},
MRREVIEWER = {John\ T.\ Baldwin},
       DOI = {10.1017/jsl.2014.42},
       URL = {https://doi.org/10.1017/jsl.2014.42},
}

@article {mcleay19,
    AUTHOR = {McLeay, Alan},
     TITLE = {Geometric normal subgroups in mapping class groups of
              punctured surfaces},
   JOURNAL = {New York J. Math.},
  FJOURNAL = {New York Journal of Mathematics},
    VOLUME = {25},
      YEAR = {2019},
     PAGES = {839--888},
   MRCLASS = {57M07},
  MRNUMBER = {4012570},
MRREVIEWER = {Robert W. Bell},
}

@article {kordek-margalit,
    AUTHOR = {Kordek, Kevin and Margalit, Dan},
     TITLE = {Homomorphisms of commutator subgroups of braid groups},
   JOURNAL = {Bull. Lond. Math. Soc.},
  FJOURNAL = {Bulletin of the London Mathematical Society},
    VOLUME = {54},
      YEAR = {2022},
    NUMBER = {1},
     PAGES = {95--111},
      ISSN = {0024-6093},
   MRCLASS = {20F36},
  MRNUMBER = {4396925},
MRREVIEWER = {Mustafa Korkmaz},
       DOI = {10.1112/blms.12560},
       URL = {https://doi.org/10.1112/blms.12560},
}

@incollection {farb-survey,
    AUTHOR = {Farb, Benson},
     TITLE = {Some problems on mapping class groups and moduli space},
 BOOKTITLE = {Problems on mapping class groups and related topics},
    SERIES = {Proc. Sympos. Pure Math.},
    VOLUME = {74},
     PAGES = {11--55},
 PUBLISHER = {Amer. Math. Soc., Providence, RI},
      YEAR = {2006},
   MRCLASS = {57M50 (20F65 30F40 37D20)},
  MRNUMBER = {2264130},
       DOI = {10.1090/pspum/074/2264130},
       URL = {https://doi.org/10.1090/pspum/074/2264130},
}

@article {harvey-korkmaz,
    AUTHOR = {Harvey, William J. and Korkmaz, Mustafa},
     TITLE = {Homomorphisms from mapping class groups},
   JOURNAL = {Bull. London Math. Soc.},
  FJOURNAL = {The Bulletin of the London Mathematical Society},
    VOLUME = {37},
      YEAR = {2005},
    NUMBER = {2},
     PAGES = {275--284},
      ISSN = {0024-6093},
   MRCLASS = {57N05 (20F38 30F10)},
  MRNUMBER = {2119027},
MRREVIEWER = {Athanase Papadopoulos},
       DOI = {10.1112/S0024609304003911},
       URL = {https://doi.org/10.1112/S0024609304003911},
}

@article {castel16,
    AUTHOR = {Castel, Fabrice},
     TITLE = {Geometric representations of the braid groups},
   JOURNAL = {Ast\'{e}risque},
  FJOURNAL = {Ast\'{e}risque},
    NUMBER = {378},
      YEAR = {2016},
     PAGES = {vi+175},
      ISSN = {0303-1179},
      ISBN = {978-2-85629-835-0},
   MRCLASS = {20F36 (20F34 30F60 57M27)},
  MRNUMBER = {3499033},
MRREVIEWER = {Eric C. Rowell},
}

@article {DPTrans19,
    AUTHOR = {Disarlo, Valentina and Parlier, Hugo},
     TITLE = {The geometry of flip graphs and mapping class groups},
   JOURNAL = {Trans. Amer. Math. Soc.},
  FJOURNAL = {Transactions of the American Mathematical Society},
    VOLUME = {372},
      YEAR = {2019},
    NUMBER = {6},
     PAGES = {3809--3844},
      ISSN = {0002-9947},
   MRCLASS = {05E45 (30F10 30F60 32G15 57M50 57M60)},
  MRNUMBER = {4009420},
MRREVIEWER = {Sang-hyun Kim},
       DOI = {10.1090/tran/7356},
       URL = {https://doi.org/10.1090/tran/7356},
}

@article {HuangTshi,
    AUTHOR = {Huang, Junzhi and Tshishiku, Bena},
     TITLE = {Finite rigid sets and the separating curve complex},
   JOURNAL = {Topology Appl.},
  FJOURNAL = {Topology and its Applications},
    VOLUME = {312},
      YEAR = {2022},
     PAGES = {Paper No. 108078, 19},
      ISSN = {0166-8641},
   MRCLASS = {57K20},
  MRNUMBER = {4405485},
       DOI = {10.1016/j.topol.2022.108078},
       URL = {https://doi.org/10.1016/j.topol.2022.108078},
}

@article {trivial-goode,
    AUTHOR = {Goode, John B.},
     TITLE = {Some trivial considerations},
   JOURNAL = {J. Symbolic Logic},
  FJOURNAL = {The Journal of Symbolic Logic},
    VOLUME = {56},
      YEAR = {1991},
    NUMBER = {2},
     PAGES = {624--631},
      ISSN = {0022-4812},
   MRCLASS = {03C45},
  MRNUMBER = {1133089},
MRREVIEWER = {O. V. Belegradek},
       DOI = {10.2307/2274704},
       URL = {https://doi.org/10.2307/2274704},
}

@book {ivanov-book,
    AUTHOR = {Ivanov, Nikolai V.},
     TITLE = {Subgroups of {T}eichm\"{u}ller modular groups},
    SERIES = {Translations of Mathematical Monographs},
    VOLUME = {115},
      NOTE = {Translated from the Russian by E. J. F. Primrose and revised
              by the author},
 PUBLISHER = {American Mathematical Society, Providence, RI},
      YEAR = {1992},
     PAGES = {xii+127},
      ISBN = {0-8218-4594-2},
   MRCLASS = {57M50 (20F38 30F60 57N05)},
  MRNUMBER = {1195787},
MRREVIEWER = {Athanase Papadopoulos},
       DOI = {10.1090/mmono/115},
       URL = {https://doi.org/10.1090/mmono/115},
}

@article {ShinkleAGT,
    AUTHOR = {Shinkle, Emily},
     TITLE = {Finite rigid sets in arc complexes},
   JOURNAL = {Algebr. Geom. Topol.},
  FJOURNAL = {Algebraic \& Geometric Topology},
    VOLUME = {20},
      YEAR = {2020},
    NUMBER = {6},
     PAGES = {3127--3145},
      ISSN = {1472-2747},
   MRCLASS = {20F38 (20F65 57K20 57M60)},
  MRNUMBER = {4185937},
MRREVIEWER = {Vassilis Metaftsis},
       DOI = {10.2140/agt.2020.20.3127},
       URL = {https://doi.org/10.2140/agt.2020.20.3127},
}

@article {Shackleton,
    AUTHOR = {Shackleton, Kenneth J.},
     TITLE = {Combinatorial rigidity in curve complexes and mapping class
              groups},
   JOURNAL = {Pacific J. Math.},
  FJOURNAL = {Pacific Journal of Mathematics},
    VOLUME = {230},
      YEAR = {2007},
    NUMBER = {1},
     PAGES = {217--232},
      ISSN = {0030-8730},
   MRCLASS = {57M50},
  MRNUMBER = {2318453},
MRREVIEWER = {Mustafa Korkmaz},
       DOI = {10.2140/pjm.2007.230.217},
       URL = {https://doi.org/10.2140/pjm.2007.230.217},
}

@book {FLP,
     TITLE = {Travaux de {T}hurston sur les surfaces},
      NOTE = {S\'{e}minaire Orsay,
              Reprint of {{\i}t Travaux de Thurston sur les surfaces}, Soc.
              Math. France, Paris, 1979 [ MR0568308 (82m:57003)],
              Ast\'{e}risque No. 66-67 (1991) (1991)},
 PUBLISHER = {Soci\'{e}t\'{e} Math\'{e}matique de France, Paris},
      YEAR = {1991},
     PAGES = {1--286},
      ISSN = {0303-1179},
   MRCLASS = {57-06 (32G15 57M99)},
  MRNUMBER = {1134426},
}

@incollection {ara-souto-rigid,
    AUTHOR = {Aramayona, Javier and Souto, Juan},
     TITLE = {Rigidity phenomena in the mapping class group},
 BOOKTITLE = {Handbook of {T}eichm\"{u}ller theory. {V}ol. {VI}},
    SERIES = {IRMA Lect. Math. Theor. Phys.},
    VOLUME = {27},
     PAGES = {131--165},
 PUBLISHER = {Eur. Math. Soc., Z\"{u}rich},
      YEAR = {2016},
   MRCLASS = {20F34 (57M07)},
  MRNUMBER = {3618188},
MRREVIEWER = {Funda G\"{u}ltepe},
}

@article {Be,
    AUTHOR = {Behrstock, Jason A.},
     TITLE = {Asymptotic geometry of the mapping class group and
              {T}eichm\"{u}ller space},
   JOURNAL = {Geom. Topol.},
  FJOURNAL = {Geometry and Topology},
    VOLUME = {10},
      YEAR = {2006},
     PAGES = {1523--1578},
      ISSN = {1465-3060},
   MRCLASS = {20F69 (20F67 30F60 32G15 57M07)},
  MRNUMBER = {2255505},
MRREVIEWER = {Gabriela Schmith\"{u}sen},
       DOI = {10.2140/gt.2006.10.1523},
       URL = {https://doi.org/10.2140/gt.2006.10.1523},
}

@article {ShinkleTAMS,
    AUTHOR = {Shinkle, Emily},
     TITLE = {Finite rigid sets in flip graphs},
   JOURNAL = {Trans. Amer. Math. Soc.},
  FJOURNAL = {Transactions of the American Mathematical Society},
    VOLUME = {375},
      YEAR = {2022},
    NUMBER = {2},
     PAGES = {847--872},
      ISSN = {0002-9947},
   MRCLASS = {57K20 (05E16 20F65 57M15)},
  MRNUMBER = {4369237},
       DOI = {10.1090/tran/8407},
       URL = {https://doi.org/10.1090/tran/8407},
}

@article {MaungchaungJKTR,
    AUTHOR = {Maungchang, Rasimate},
     TITLE = {Exhausting pants graphs of punctured spheres by finite rigid
              sets},
   JOURNAL = {J. Knot Theory Ramifications},
  FJOURNAL = {Journal of Knot Theory and its Ramifications},
    VOLUME = {26},
      YEAR = {2017},
    NUMBER = {14},
     PAGES = {1750105, 11},
      ISSN = {0218-2165},
   MRCLASS = {57M50},
  MRNUMBER = {3735406},
MRREVIEWER = {Jiming Ma},
       DOI = {10.1142/S021821651750105X},
       URL = {https://doi.org/10.1142/S021821651750105X},
}

@article {HHLM19,
    AUTHOR = {Hern\'{a}ndez Hern\'{a}ndez, Jes\'{u}s and Leininger, Christopher J. and
              Maungchang, Rasimate},
     TITLE = {Finite rigid subgraphs of pants graphs},
   JOURNAL = {Geom. Dedicata},
  FJOURNAL = {Geometriae Dedicata},
    VOLUME = {212},
      YEAR = {2021},
     PAGES = {205--223},
      ISSN = {0046-5755},
   MRCLASS = {20F65 (57K20)},
  MRNUMBER = {4251670},
MRREVIEWER = {Lvzhou Chen},
       DOI = {10.1007/s10711-020-00555-1},
       URL = {https://doi.org/10.1007/s10711-020-00555-1},
}

@article {BLM-duke,
    AUTHOR = {Birman, Joan S. and Lubotzky, Alex and McCarthy, John},
     TITLE = {Abelian and solvable subgroups of the mapping class groups},
   JOURNAL = {Duke Math. J.},
  FJOURNAL = {Duke Mathematical Journal},
    VOLUME = {50},
      YEAR = {1983},
    NUMBER = {4},
     PAGES = {1107--1120},
      ISSN = {0012-7094},
   MRCLASS = {20F38 (30F20 57N05)},
  MRNUMBER = {726319},
MRREVIEWER = {William Harvey},
       DOI = {10.1215/S0012-7094-83-05046-9},
       URL = {https://doi.org/10.1215/S0012-7094-83-05046-9},
}

@article {BestBromFuji,
    AUTHOR = {Bestvina, Mladen and Bromberg, Ken and Fujiwara, Koji},
     TITLE = {Constructing group actions on quasi-trees and applications to
              mapping class groups},
   JOURNAL = {Publ. Math. Inst. Hautes \'{E}tudes Sci.},
  FJOURNAL = {Publications Math\'{e}matiques. Institut de Hautes \'{E}tudes
              Scientifiques},
    VOLUME = {122},
      YEAR = {2015},
     PAGES = {1--64},
      ISSN = {0073-8301},
   MRCLASS = {20E08 (20F34 20F65 20F69)},
  MRNUMBER = {3415065},
MRREVIEWER = {Dmytro M. Savchuk},
       DOI = {10.1007/s10240-014-0067-4},
       URL = {https://doi.org/10.1007/s10240-014-0067-4},
}

@article {Baudisch-ample,
    AUTHOR = {Baudisch, Andreas and Martin-Pizarro, Amador and Ziegler,
              Martin},
     TITLE = {Ample hierarchy},
   JOURNAL = {Fund. Math.},
  FJOURNAL = {Fundamenta Mathematicae},
    VOLUME = {224},
      YEAR = {2014},
    NUMBER = {2},
     PAGES = {97--153},
      ISSN = {0016-2736},
   MRCLASS = {03C45},
  MRNUMBER = {3180584},
MRREVIEWER = {Jonathan Kirby},
       DOI = {10.4064/fm224-2-1},
       URL = {https://doi.org/10.4064/fm224-2-1},
}

@book {tent_ziegler,
    AUTHOR = {Tent, Katrin and Ziegler, Martin},
     TITLE = {A course in model theory},
    SERIES = {Lecture Notes in Logic},
    VOLUME = {40},
 PUBLISHER = {Association for Symbolic Logic, La Jolla, CA; Cambridge
              University Press, Cambridge},
      YEAR = {2012},
     PAGES = {x+248},
}

@unpublished{depool, 
  author = {de Pool, Rodrigo}, 
  title = {Finite rigid sets of the non-separating curve complex},
  journal = {arXiv:2210.05317},
  year = 2022
}

@unpublished{mcleay18, 
  author = {Alan McLeay}, 
  title = {Normal subgroups of the braid group and the metaconjecture of Ivanov},
  journal = {arXiv:1801.05209},
  year = 2018
}

@article {AraSoutoGT12,
    AUTHOR = {Aramayona, Javier and Souto, Juan},
     TITLE = {Homomorphisms between mapping class groups},
   JOURNAL = {Geom. Topol.},
  FJOURNAL = {Geometry \& Topology},
    VOLUME = {16},
      YEAR = {2012},
    NUMBER = {4},
     PAGES = {2285--2341},
      ISSN = {1465-3060},
   MRCLASS = {57M07 (20F65)},
  MRNUMBER = {3033518},
MRREVIEWER = {Sang-hyun Kim},
       DOI = {10.2140/gt.2012.16.2285},
       URL = {https://doi.org/10.2140/gt.2012.16.2285},
}

@article {schmutz_schaller,
    AUTHOR = {Schmutz Schaller, Paul},
     TITLE = {Mapping class groups of hyperbolic surfaces and automorphism
              groups of graphs},
   JOURNAL = {Compositio Math.},
  FJOURNAL = {Compositio Mathematica},
    VOLUME = {122},
      YEAR = {2000},
    NUMBER = {3},
     PAGES = {243--260},
  }

@book {pillay,
    AUTHOR = {Pillay, Anand},
     TITLE = {An introduction to stability theory},
    SERIES = {Oxford Logic Guides},
    VOLUME = {8},
 PUBLISHER = {The Clarendon Press, Oxford University Press, New York},
      YEAR = {1983},
     PAGES = {xii+146},
}

@article {aramayona_koberda_parlier,
    AUTHOR = {Aramayona, Javier and Koberda, Thomas and Parlier, Hugo},
     TITLE = {Injective maps between flip graphs},
   JOURNAL = {Annales de l'Institut Fourier},
    VOLUME = {65},
      YEAR = {2015},
    NUMBER = {5},
     PAGES = {2037--2055},
  }

@article {disarlo_parlier,
    AUTHOR = {Disarlo, Valentina and Parlier, Hugo},
     TITLE = {Simultaneous flips on triangulated surfaces},
   JOURNAL = {Michigan Math. J.},
  FJOURNAL = {Michigan Mathematical Journal},
    VOLUME = {67},
      YEAR = {2018},
    NUMBER = {3},
     PAGES = {451--464},
}

@article {KKIMRN,
    AUTHOR = {Kim, Sang-Hyun and Koberda, Thomas},
     TITLE = {An obstruction to embedding right-angled {A}rtin groups in
              mapping class groups},
   JOURNAL = {Int. Math. Res. Not. IMRN},
  FJOURNAL = {International Mathematics Research Notices. IMRN},
      YEAR = {2014},
    NUMBER = {14},
     PAGES = {3912--3918},
      ISSN = {1073-7928},
   MRCLASS = {20F65},
  MRNUMBER = {3239092},
MRREVIEWER = {Kevin D. Schreve},
       DOI = {10.1093/imrn/rnt064},
       URL = {https://doi.org/10.1093/imrn/rnt064},
}

@article {KKOsaka,
    AUTHOR = {Kim, Sang-Hyun and Koberda, Thomas},
     TITLE = {Right-angled {A}rtin groups and finite subgraphs of curve
              graphs},
   JOURNAL = {Osaka J. Math.},
  FJOURNAL = {Osaka Journal of Mathematics},
    VOLUME = {53},
      YEAR = {2016},
    NUMBER = {3},
     PAGES = {705--716},
      ISSN = {0030-6126},
   MRCLASS = {20F36 (20F34 20F65 57M15)},
  MRNUMBER = {3533464},
MRREVIEWER = {Daniel Matei},
       URL = {http://projecteuclid.org/euclid.ojm/1470413985},
}

@article {bering_conant_gaster,
    AUTHOR = {Bering, IV, Edgar A. and Conant, Gabriel and Gaster, Jonah},
     TITLE = {On the complexity of finite subgraphs of the curve graph},
   JOURNAL = {Osaka J. Math.},
  FJOURNAL = {Osaka Journal of Mathematics},
    VOLUME = {55},
      YEAR = {2018},
    NUMBER = {4},
     PAGES = {795--808},
}

@incollection {ivanov06,
    AUTHOR = {Ivanov, Nikolai V.},
     TITLE = {Fifteen problems about the mapping class groups},
 BOOKTITLE = {Problems on mapping class groups and related topics},
    SERIES = {Proceedings of Symposia in Pure Mathematics},
    VOLUME = {74},
     PAGES = {71--80},
 PUBLISHER = {American Mathematical Society},
      YEAR = {2006},
}

@incollection {disarlo_0,
    AUTHOR = {Disarlo, Valentina},
     TITLE = {On the coarse geometry of the complex of domains},
 BOOKTITLE = {Handbook of {T}eichm\"{u}ller theory. {V}olume {III}},
    SERIES = {IRMA Lect. Math. Theor. Phys.},
    VOLUME = {17},
     PAGES = {425--439},
 PUBLISHER = {Eur. Math. Soc., Z\"{u}rich},
      YEAR = {2012},
      }

@incollection {harvey_1978,
    AUTHOR = {Harvey, William J.},
     TITLE = {Boundary structure of the modular group},
 BOOKTITLE = {Riemann surfaces and related topics: {P}roceedings of the 1978
              {S}tony {B}rook {C}onference},
    SERIES = {Annals of Mathematics Studies},
    VOLUME = {97},
     PAGES = {245--251},
 PUBLISHER = {Princeton University Press},
      YEAR = {1981},
}

@article {hatcher_thurston1980,
    AUTHOR = {Allen Hatcher and William P. Thurston},
     TITLE = {A presentation for the mapping class group of a closed
              orientable surface},
   JOURNAL = {Topology},
    VOLUME = {19},
      YEAR = {1980},
    NUMBER = {3},
     PAGES = {221--237},
}

@article {irmak_mccarthy10,
    AUTHOR = {Irmak, Elmas and John D. McCarthy},
     TITLE = {Injective simplicial maps of the arc complex},
   JOURNAL = {Turkish Journal of Mathematics},
      VOLUME = {34},
      YEAR = {2010},
    NUMBER = {3},
     PAGES = {339--354},
      ISSN = {1300-0098},
}

@incollection {minsky-survey,
    AUTHOR = {Minsky, Yair N.},
     TITLE = {A geometric approach to the complex of curves on a surface},
 BOOKTITLE = {Topology and {T}eichm\"{u}ller spaces ({K}atinkulta, 1995)},
     PAGES = {149--158},
 PUBLISHER = {World Sci. Publ., River Edge, NJ},
      YEAR = {1996},
   MRCLASS = {32G15 (57M50)},
  MRNUMBER = {1659683},
MRREVIEWER = {Yoichi Imayoshi},
}

@article {margalit,
    AUTHOR = {Margalit, Dan},
     TITLE = {Automorphisms of the pants complex},
   JOURNAL = {Duke Mathematical Journal},
    VOLUME = {121},
      YEAR = {2004},
    NUMBER = {3},
     PAGES = {457--479},
 }

@article {korkmaz_papadopoulos,
    AUTHOR = {Korkmaz, Mustafa and Papadopoulos, Athanase},
     TITLE = {On the ideal triangulation graph of a punctured surface},
   JOURNAL = {Annales de l'Institut Fourier},
    VOLUME = {62},
      YEAR = {2012},
    NUMBER = {4},
     PAGES = {1367--1382},
}

@book{farb_margalit_11,
    Author = {Benson {Farb} and Dan {Margalit}},
    Title = {{A primer on mapping class groups}},
    ISBN = {978-0-691-14794-9/hbk; 978-1-400-83904-9/ebook},
    Pages = {xiv + 492},
    Year = {2011},
    Publisher = {Princeton, NJ: Princeton University Press},
    MSC2010 = {57-01 30-01 30F10 30F60 57M35 20F38 32G15 57M07 57N05 20F36 14H15},
    Zbl = {1245.57002}
}

@article {irmak-superinjective,
    AUTHOR = {Irmak, Elmas},
     TITLE = {Superinjective simplicial maps of complexes of curves and
              injective homomorphisms of subgroups of mapping class groups},
   JOURNAL = {Topology},
  FJOURNAL = {Topology. An International Journal of Mathematics},
    VOLUME = {43},
      YEAR = {2004},
    NUMBER = {3},
     PAGES = {513--541},
      ISSN = {0040-9383},
   MRCLASS = {57M50 (20F38 57N05)},
  MRNUMBER = {2041629},
MRREVIEWER = {Mustafa Korkmaz},
       DOI = {10.1016/j.top.2003.03.002},
       URL = {https://doi.org/10.1016/j.top.2003.03.002},
}

@article {erl-fan-2017,
    AUTHOR = {Erlandsson, Viveka and Fanoni, Federica},
     TITLE = {Simplicial embeddings between multicurve graphs},
   JOURNAL = {Michigan Math. J.},
  FJOURNAL = {Michigan Mathematical Journal},
    VOLUME = {66},
      YEAR = {2017},
    NUMBER = {3},
     PAGES = {549--567},
      ISSN = {0026-2285},
   MRCLASS = {57M15 (05C10 30F99)},
  MRNUMBER = {3695351},
MRREVIEWER = {Sang-Eon Han},
       DOI = {10.1307/mmj/1496995339},
       URL = {https://doi.org/10.1307/mmj/1496995339},
}

@article {masur_minsky_99,
    AUTHOR = {Masur, Howard and Minsky, Yair N.},
     TITLE = {Geometry of the complex of curves. {I}. {H}yperbolicity},
   JOURNAL = {Inventiones Mathematicae},
    VOLUME = {138},
      YEAR = {1999},
    NUMBER = {1},
     PAGES = {103--149},
 }

@article {irmak,
    AUTHOR = {Irmak, Elmas},
     TITLE = {Complexes of nonseparating curves and mapping class groups},
   JOURNAL = {Michigan Math. J.},
  FJOURNAL = {Michigan Mathematical Journal},
    VOLUME = {54},
      YEAR = {2006},
    NUMBER = {1},
     PAGES = {81--110},
}

@article {looijenga,
    AUTHOR = {Looijenga, Eduard},
     TITLE = {Connectivity of complexes of separating curves},
   JOURNAL = {Groups Geom. Dyn.},
  FJOURNAL = {Groups, Geometry, and Dynamics},
    VOLUME = {7},
      YEAR = {2013},
    NUMBER = {2},
     PAGES = {443--450},
}

@article {kida,
    AUTHOR = {Kida, Yoshikata},
     TITLE = {Automorphisms of the {T}orelli complex and the complex of
              separating curves},
   JOURNAL = {J. Math. Soc. Japan},
  FJOURNAL = {Journal of the Mathematical Society of Japan},
    VOLUME = {63},
      YEAR = {2011},
    NUMBER = {2},
     PAGES = {363--417},
}

@article {irmak_korkmaz,
    AUTHOR = {Irmak, Elmas and Korkmaz, Mustafa},
     TITLE = {Automorphisms of the {H}atcher-{T}hurston complex},
   JOURNAL = {Israel J. Math.},
  FJOURNAL = {Israel Journal of Mathematics},
    VOLUME = {162},
      YEAR = {2007},
     PAGES = {183--196},
      ISSN = {0021-2172},
   MRCLASS = {57M60 (57M07 57M20 57N05)},
  MRNUMBER = {2365859},
MRREVIEWER = {Stephen P. Humphries},
       DOI = {10.1007/s11856-007-0094-7},
       URL = {https://doi.org/10.1007/s11856-007-0094-7},
}

@article {BPZ17,
    AUTHOR = {Baudisch, Andreas and Martin-Pizarro, Amador and Ziegler,
              Martin},
     TITLE = {A model-theoretic study of right-angled buildings},
   JOURNAL = {J. Eur. Math. Soc. (JEMS)},
  FJOURNAL = {Journal of the European Mathematical Society (JEMS)},
    VOLUME = {19},
      YEAR = {2017},
    NUMBER = {10},
     PAGES = {3091--3141},
      ISSN = {1435-9855},
   MRCLASS = {03C45 (05A05 05C90 20F55 51E24)},
  MRNUMBER = {3713002},
MRREVIEWER = {S. Comer},
       DOI = {10.4171/JEMS/736},
       URL = {https://doi.org/10.4171/JEMS/736},
}

@article {masur_minsky_00,
    AUTHOR = {Masur, H. A. and Minsky, Y. N.},
     TITLE = {Geometry of the complex of curves. {II}. {H}ierarchical
              structure},
   JOURNAL = {Geom. Funct. Anal.},
  FJOURNAL = {Geometric and Functional Analysis},
    VOLUME = {10},
      YEAR = {2000},
    NUMBER = {4},
     PAGES = {902--974},
}

@article {ivanov_97,
    AUTHOR = {Ivanov, Nikolai V.},
     TITLE = {Automorphism of complexes of curves and of {T}eichm\"{u}ller
              spaces},
   JOURNAL = {International Mathematics Research Notices},
      YEAR = {1997},
    NUMBER = {14},
     PAGES = {651--666},
      ISSN = {1073-7928},
   MRCLASS = {57M99 (20F38 30F60 57N05)},
MRREVIEWER = {Darryl McCullough},
       DOI = {10.1155/S1073792897000433},
       URL = {https://doi.org/10.1155/S1073792897000433},
}

@article {brendle_margalit,
    AUTHOR = {Brendle, Tara E. and Margalit, Dan},
     TITLE = {Normal subgroups of mapping class groups and the
              metaconjecture of {I}vanov},
   JOURNAL = {J. Amer. Math. Soc.},
  FJOURNAL = {Journal of the American Mathematical Society},
    VOLUME = {32},
      YEAR = {2019},
    NUMBER = {4},
     PAGES = {1009--1070},
      ISSN = {0894-0347},
   MRCLASS = {20F36 (57M07)},
  MRNUMBER = {4013739},
MRREVIEWER = {Sebastian Wolfgang Hensel},
       DOI = {10.1090/jams/927},
       URL = {https://doi.org/10.1090/jams/927},
}

@article {korkmaz,
    AUTHOR = {Korkmaz, Mustafa},
     TITLE = {Automorphisms of complexes of curves on punctured spheres and
              on punctured tori},
   JOURNAL = {Topology and its Applications},
    VOLUME = {95},
      YEAR = {1999},
    NUMBER = {2},
     PAGES = {85--111},
      }

@article {Luo,
    AUTHOR = {Luo, Feng},
     TITLE = {Automorphisms of the complex of curves},
   JOURNAL = {Topology},
    VOLUME = {39},
      YEAR = {2000},
    NUMBER = {2},
     PAGES = {283--298},
   }

@article{masur_schleimer_13,
    AUTHOR = {Masur, Howard and Schleimer, Saul},
     TITLE = {The geometry of the disk complex},
   JOURNAL = {Journal of the American Mathematical Society},
    VOLUME = {26},  
    AUTHOR = {Funar, Louis and Kapoudjian, Christophe and Sergiescu, Vlad},
     TITLE = {Asymptotically rigid mapping class groups and {T}hompson's
              groups},
 BOOKTITLE = {Handbook of {T}eichm\"{u}ller theory. {V}olume {III}},
    SERIES = {IRMA Lect. Math. Theor. Phys.},
    VOLUME = {17},
     PAGES = {595--664},
 PUBLISHER = {Eur. Math. Soc., Z\"{u}rich},
      YEAR = {2012},
   MRCLASS = {57-02 (20F65 57M50)},
  MRNUMBER = {2952772},
MRREVIEWER = {Bruno P. Zimmermann},
       DOI = {10.4171/103-1/11},
       URL = {https://doi.org/10.4171/103-1/11},
}

@article{bell_disarlo_tang_2018, 
title={Cubical geometry in the polygonalisation complex}, 
DOI={10.1017/S0305004118000130}, 
journal={Mathematical Proceedings of the Cambridge Philosophical Society}, 
publisher={Cambridge University Press}, 
author={Bell, Mark C. and Disarlo, Valentina and Tang, Robert}, 
year={2018}, 
pages={1--22}
}

@article{behrstock2006asymptotic,
  title={Asymptotic geometry of the mapping class group and Teichm{\"u}ller space},
  author={Behrstock, Jason A},
  journal={Geometry \& Topology},
  volume={10},
  number={3},
  pages={1523--1578},
  year={2006},
  publisher={Mathematical Sciences Publishers}
}

@book {marker2006model,
    AUTHOR = {Marker, David},
     TITLE = {Model theory},
    SERIES = {Graduate Texts in Mathematics},
    VOLUME = {217},
      NOTE = {An introduction},
 PUBLISHER = {Springer-Verlag, New York},
      YEAR = {2002},
     PAGES = {viii+342},
      ISBN = {0-387-98760-6},
   MRCLASS = {03Cxx (03-01 03-02 20A15)},
  MRNUMBER = {1924282},
MRREVIEWER = {H. Dugald Macpherson},
}

@article{aramayona2016exhausting,
  title={Exhausting curve complexes by finite rigid sets},
  author={Aramayona, Javier and Leininger, Christopher},
  journal={Pacific Journal of Mathematics},
  volume={282},
  number={2},
  pages={257--283},
  year={2016},
  publisher={Mathematical Sciences Publishers}
}

@article{morley1965categoricity,
	title={Categoricity in power},
	author={Morley, Michael},
	journal={Transactions of the American Mathematical Society},
	volume={114},
	number={2},
	pages={514--538},
	year={1965},
	publisher={American Mathematical Society}
}

@article{poizat1983theorie,
  title={Une th{\'e}orie de Galois imaginaire},
  author={Poizat, Bruno},
  journal={The Journal of Symbolic Logic},
  volume={48},
  number={4},
  pages={1151--1170},
  year={1983},
  publisher={Cambridge University Press}
}
   	\bibliographystyle{alpha}
    	
\end{document}